\documentclass[a4paper]{article}

\usepackage[T1]{fontenc}
\usepackage[utf8]{inputenc}
\usepackage[width=15cm,height=22cm]{geometry}
\usepackage[dvipsnames]{xcolor}
\usepackage{amsmath,amsfonts,amssymb,amsthm}
\usepackage[hidelinks,pdfusetitle]{hyperref}
\usepackage[capitalise]{cleveref}
\usepackage[usenames,dvipsnames]{pstricks}
\usepackage{epsfig}
\usepackage{bbm}
\usepackage{stmaryrd}
\usepackage{tikz}
\usepackage{makecell}
\usepackage{shuffle}

\newtheorem{proposition}{Proposition}
\newtheorem{theorem}[proposition]{Theorem}
\newtheorem{definition}[proposition]{Definition}
\newtheorem{corollary}[proposition]{Corollary}
\newtheorem{open}[proposition]{Open problem}
\newtheorem{false}[proposition]{False proposition}
\newtheorem{lemma}[proposition]{Lemma}
\newtheorem{remark}[proposition]{Remark}
\newtheorem{example}[proposition]{Example}

\newcommand*\samethanks[1][\value{footnote}]{\footnotemark[#1]}
\title{On expansions for nonlinear systems, \\ error estimates and convergence issues}
\author{Karine Beauchard\texorpdfstring{\thanks{Univ Rennes, CNRS, IRMAR - UMR 6625, F-35000 Rennes, France}}{}, J\'er\'emy Le Borgne\texorpdfstring{\samethanks}{},  Fr\'ed\'eric Marbach\texorpdfstring{\samethanks}{}}

\newcommand{\N}{\mathbb{N}}

\newcommand{\K}{\mathbb{K}}
\newcommand{\C}{\mathbb{C}}
\newcommand{\R}{\mathbb{R}}
\newcommand{\dd}{\,\mathrm{d}}
\newcommand{\CC}{\mathcal{C}}

\newcommand{\bset}{{B}}
\newcommand{\basis}{{\mathcal{B}}}

\newcommand{\GHB}{{Hall basis}}
\newcommand{\GHBS}{{Hall bases}}

\DeclareMathOperator{\Br}{Br}
\newcommand{\ad}{\operatorname{ad}}
\newcommand{\dad}{\underline{\operatorname{ad}}}
\newcommand{\eval}{\textsc{e}}
\newcommand{\val}{\operatorname{val}}
\DeclareMathOperator{\vect}{span}
\DeclareMathOperator{\CBHD}{CBHD}

\newcommand{\intset}[1]{\llbracket #1 \rrbracket}
\newcommand{\opnorm}[1]{\left|\mkern-1.5mu\left|\mkern-1.5mu\left| #1 \right|\mkern-1.5mu\right|\mkern-1.5mu\right|}

\begin{document}

\maketitle

\begin{abstract}
 Explicit formulas expressing the solution to non-autonomous differential equations are of great importance in many application domains such as control theory or numerical operator splitting. In particular, intrinsic formulas allowing to decouple time-dependent features from geometry-dependent features of the solution have been extensively studied.  

 First, we give a didactic review of classical expansions for formal linear differential equations, including the celebrated Magnus expansion (associated with coordinates of the first kind) and Sussmann's infinite product expansion (associated with coordinates of the second kind). Inspired by quantum mechanics, we introduce a new mixed expansion, designed to isolate the role of a time-invariant drift from the role of a time-varying perturbation.

 Second, in the context of nonlinear ordinary differential equations driven by regular vector fields, we give rigorous proofs of error estimates between the exact solution and finite approximations of the formal expansions. In particular, we derive new estimates focusing on the role of time-varying perturbations. For scalar-input systems, we derive new estimates involving only a weak Sobolev norm of the input.
 
 Third, we investigate the local convergence of these expansions. We recall known positive results for nilpotent dynamics and for linear dynamics. Nevertheless, we also exhibit arbitrarily small analytic vector fields for which the convergence of the Magnus expansion fails, even in very weak senses. We state an open problem concerning the convergence of Sussmann's infinite product expansion.
 
 Eventually, we derive approximate direct intrinsic representations for the state and discuss their link with the choice of an appropriate change of coordinates.
\end{abstract}

\setcounter{tocdepth}{2}
\tableofcontents

\section{Introduction}

\subsection{Motivations}

There are multiple situations in which one desires to compute the solution to a differential equation whose dynamics depend on time. One often looks for explicit formulas, depending preferentially on intrinsic quantities, which describe the composition of flows, or even the continuous composition of flows. Some important applications are listed below.
\begin{itemize}
    \item \textbf{Control theory.} Here, the dynamics depend on time mostly through the choice of time-varying controls. One looks for explicit formulas of the continuous product of flows in order to be able to construct controls for which this resulting flow drives a given initial state to a desired target state. In order to establish necessary and sufficient conditions for controllability, one is interested in intrinsic formulas. It is our main motivation. 

    \item \textbf{Numerical splitting methods.} Here, the splitting algorithm applies sequentially a succession of basic flows, composed with appropriate time steps. One is interested in choosing correctly the base flows and the time steps in order to approximate the most precisely possible the solution to the true complex flow. Formulas concerning the composition of flows are essential to compute the order of the resulting numerical scheme. We refer to the survey~\cite{zbMATH06057121} and the introduction books~\cite{MR3642447,MR2840298}. Composition of flows formulas are also very useful in particular settings like Hamiltonian systems \cite{MR1843609} or in the presence of a small perturbation of a reference flow \cite{MR1429017}.
    {Concerning numerical methods, more generally, we refer to \cite{mclachlan2015butcher} (respectively \cite{curry2017post}) for a survey on Butcher series (resp.\ post-Lie algebras), algebraic tools related to the algebras manipulated in the sequel.} 
    
    \item \textbf{Stochastic differential equations.} Here, the dynamics depend on time through the sources of randomness, say Brownian motions. One wishes to investigate the influence of the randomness on the final state and thus looks for explicit formulas involving iterated Stratanovich integrals to construct a representation of the flow, see e.g.\ \cite{MR2154760,MR981567,MR1721138,MR1227033}.
    
    \item \textbf{Differential equations on Lie groups.} Sometimes, the state itself of the differential equation belongs to a Lie group, as in \cite{MR1883629}. Then, looking for an intrinsic approximation of the state helps to preserve structure which would be lost otherwise. In particular, writing the product of multiple flows as a single flow is important. There are also control problems for differential equations set on Lie groups, as in \cite{MR0331185}.
    {Some works, e.g.\ \cite{casas2006explicit}, also tackle the hard question of obtaining Magnus-type expansions, which are intrinsically linear, for nonlinear equations within matrix Lie groups.}
    
    \item {\textbf{Analysis of time-periodic systems.}
    When investigating the behavior of time-periodic systems, some authors borrow tools from ``chronological calculus'' or expressions of the ``logarithm of the flow'' (described below).
    For example, such techniques are used to study stability and asymptotic stability of time-periodic systems of ODEs; see the non linear Floquet Theorem 3.2 and the high-order averaging procedure Theorem 7.1 in \cite{zbMATH01635357}, or the recent higher-order averaging results of \cite{maggia2020higher}.}
\end{itemize}

\subsection{Short historical survey}

We start with a short survey of some of the many approaches related with the computation of solutions to formal linear differential equations, say
\begin{equation} \label{ydot.xy}
 \dot{x}(t) = X(t) x(t),
\end{equation}
together with some initial condition $x(0)$.
We recall in \cref{sec:intro-nonlinear} the consequences of such results for nonlinear ordinary differential equations.

\subsubsection{Iterated integration and Chen-Fliess expansion}

A straightforward approach to solving \eqref{ydot.xy} consists in what can be seen as {a Picard iteration}. 
For small times, starting from the initial approximation $x(t) \approx x(0)$, one then enhances the approximation {iteratively by substituting it} in the equation and obtains successively $x(t) \approx x(0) + \int_0^t X(s) x(0) \dd s$, then $x(t) \approx x(0) + \int_0^t X(s) x(0) \dd s + \int_0^t X(s) \int_0^s X(s') x(0) \dd s' \dd s$ and so on.

In the context of control theory, this expansion is known as the Chen-Fliess expansion, after being popularized by the works \cite{MR0073174, MR613847}.
Its main advantages are its simplicity and nice convergence properties (see \cref{sec:conv-chen}). However, it also has some strong drawbacks, which we detail in \cref{rk:cf-drawbacks} and \cref{rk:chen.gb.fb} and motivate the investigation of other expansions.

\subsubsection{Magnus expansion}

When $X(t)$ is piecewise constant, for example with values $X_1$ for $t \in [0,1]$ and $X_2$ for $t \in [1,2]$, one has formally, $x(2) = e^{X_2}e^{X_1}x(0)$. 
Hence, the computation of solutions to \eqref{ydot.xy} has a deep link with the famous Campbell \cite{MR1576434}, Baker \cite{MR1575931}, Hausdorff \cite{hausdorff1906symbolische}, Dynkin \cite{MR0030962} formula (``CBHD formula'' in the sequel).

This formula has a long and rich history which involves forgotten contributions of other authors such as Schur, Poincaré, Pascal or Yosida. 
As noted by Bourbaki in \cite{MR782480}, ``\emph{chacun considère que les démonstrations de ses prédécesseurs ne sont pas convaincantes}'' {(each one considers that the proofs of his predecessors are not convincing)}.
We therefore encourage the reader to dive into the fascinating retrospectives~\cite{MR2913666} and~\cite{MR2883818} to understand the progressive construction of its proof throughout the decades. 
This formula is a formal identity expressing the product of the exponentials of two (non-commutative) indeterminates $X_1$ and $X_2$ as the single exponential of a series of Lie brackets (i.e.\ nested commutators) of these indeterminates, of which the first terms are well-known:
\begin{equation} \label{cbh.1}
 e^{X_2} e^{X_1} = \exp \left( X_1 + X_2 + \frac 1 2 [X_2, X_1] + \dotsc \right).
\end{equation}
When more than two exponentials are multiplied, say $e^{X_1}$ through $e^{X_n}$, one can of course iterate the formula \eqref{cbh.1} with itself to formally express the product of $n$ exponentials as the single exponential of a complicated series. Letting $n \to +\infty$, one is lead to computing a continuous product of exponentials, which corresponds, heuristically, to solving \eqref{ydot.xy}.

Magnus performed a breakthrough by deriving in \cite{MR0067873} the first formal representation of the solution to \eqref{ydot.xy} as the exponential of a series, of which the first terms are
\begin{equation} \label{magnus.1}
 x(t) = \exp \left( \int_0^t X(\tau_1) \dd \tau_1 + \frac 1 2 \int_0^t \int_0^{\tau_1} \left[ X(\tau_1),  X(\tau_2)\right] \dd \tau_2 \dd \tau_1 + \dotsb \right) x(0).
\end{equation}
This formula can be seen as the continuous counterpart of the CBHD formula and highlights important structural properties of the solutions to \eqref{ydot.xy} (see \cref{sec:formal.log.flows}).

\subsubsection{Infinite products}

The CBHD formula and the Magnus formula share the goal of expressing the desired quantity as the exponential of a single, although complicated, object.
Other approaches go the other way around and try to express the desired quantity as a long (infinite) product of exponentials of very simple objects.

A well-known example is the Lie-Trotter product formula (see e.g.\ \cite{zbMATH03161502}), often used for numerical splitting methods which attempts to give a meaning to the equality
\begin{equation}
    e^{X_1+X_2} = \lim_{n\to + \infty} \left(e^{\frac{X_1}{n}} e^{\frac{X_2}{n}}\right)^n,
\end{equation}
the interest relying on the fact that the exponentials of $X_1$ and $X_2$ are assumed to be easier to compute in some sense than the direct exponential of $X_1+X_2$.

Another related formula is the Zassenhaus expansion, described by Magnus in \cite{MR0067873}, which allows to decompose the same quantity $e^{X_1+X_2}$ as an infinite product of exponentials of linear combinations of nested commutators of strictly increasing lengths, whose first terms are
\begin{equation}
    e^{X_1+X_2} = e^{X_1} e^{X_2} \exp\left(-\frac{1}{2}[X_1,X_2]\right)
    \exp\left(\frac{1}{3} [X_2,[X_1,X_2]] + \frac{1}{6} [X_1,[X_1,X_2]]\right) \dotsb
\end{equation}

In the context of differential equations such as \eqref{ydot.xy}, a nice formula is Sussmann's infinite product expansion, introduced in \cite{MR935387}. When $X(t)$ is given as a linear combination of elementary generators, e.g.\ $X(t) = a_1(t) X_1 + a_2(t) X_2$, Sussmann's infinite product expansion is given by a product of exponentials of Lie monomials, such as
\begin{equation}
    x(t) = e^{\xi_1 X_1} e^{\xi_2 X_2} e^{\xi_{12} [X_1,X_2]} e^{\xi_{112} [X_1,[X_1,X_2]]} e^{\xi_{212} [X_2,[X_1,X_2]]} \dotsb x(0),
\end{equation}
where the $\xi_i$ are scalar functions of time given by explicit formulas from the functions $a_1$ and $a_2$.
Compared to other expansions, this formula is both intrinsic (such as the Magnus expansion) and involves coefficients which are easily computed by induction (such as the Chen-Fliess expansion).
 
\subsubsection{Consequences for nonlinear ordinary differential equations}
\label{sec:intro-nonlinear}

Although the expansions mentioned above concern linear formal differential equations, they can be adapted to ordinary nonlinear differential equations on smooth manifolds governed by smooth vector fields. 
Indeed, one can identify vector fields with linear operators acting on smooth functions, and points of the manifold with the linear operator on smooth functions corresponding to evaluation at this point. 
This method allows to recast the nonlinear equation into a linear equation set on a larger space, for which the formal linear expansions can be used (see \cref{subsec:linear_trick}). 

This linearization technique has been used by Sussmann in \cite[Proposition 4.3]{MR710995} to prove the convergence of the Chen-Fliess expansion for nonlinear ordinary differential equations driven by analytic vector fields, by Agrachev and Gamkrelidze in the context of control theory (see \cite{MR524203,MR579930,MR535331} in which they derive an exponential representation of flows, very similar to Magnus' expansion, using the \emph{chronological calculus} framework) and by Strichartz (see \cite{MR886816} and his derivation of the \emph{generalized CBHD formula}, with applications related to sub-Riemannian geometry).

At a formal level, all identities mentioned above (almost) always make sense. 
However, if the indeterminates are replaced by {objects coming from analysis (say vector fields, matrices or differential operators)}, convergence issues arise.
Generally speaking, convergence often requires that one either assumes that the objects are small enough or that the generated Lie algebra has additional structure, like nilpotence.

\subsection{Main goals and organization of this article}

This article is both a survey on some classical expansions for nonlinear systems, a research article containing new results and counter-examples and a toolbox for future works. In particular, we aim at the following goals.

\begin{itemize}
    
    \item In \cref{sec:formal} we give a \textbf{didactic review of classical expansions for formal linear differential equations}. Our introduction to this algebraic topic is written with a view to making it understandable by readers with minimal algebraic background. We review the following classical expansions:  
    \begin{enumerate}
        \item the Chen-Fliess formula,
        \item the Magnus or generalized CBHD formula (associated with coordinates of the first kind),
        \item Sussmann's infinite product formula (associated with coordinates of the second kind).
    \end{enumerate}
    
    \item We introduce a \textbf{new formal mixed expansion}, inspired by quantum mechanics, designed to isolate the role of a time-invariant drift from the role of a time-varying perturbation {(see \cref{thm:Magnus1.0_formel})}, which we name \emph{Magnus expansion in the interaction picture} and for which we define coordinates of the \emph{pseudo-first kind} by analogy with first and second kind coordinates.
    
    \item We recall in \cref{sec:technical} \textbf{classical well-posedness results and estimates} for products and Lie brackets of analytic vector fields, which are used throughout the article.
    
    \item In the context of nonlinear ordinary differential equations driven by regular vector fields, we give in \cref{sec:estimates} \textbf{rigorous proofs of error estimates} between the exact solution and finite approximations of each of these four formal expansions. These estimates are part of the mathematical folklore for the Chen-Fliess and Magnus expansions, but are new for our mixed expansion {(see \cref{Prop:Magnus_2})} and for Sussmann's infinite product expansion {(see \cref{Prop:Error_Sussman})}. We strive towards providing estimates with similar structures for the four expansions and which are valid under parsimonious regularity assumptions.
    
    \item We investigate the convergence of these expansions in \cref{sec:convergence}. We recall \textbf{known positive convergence results} for smooth vector fields generating nilpotent Lie algebras and for small linear dynamics (matrices). For our new expansion, we investigate the subtle convergence under a natural partial {nilpotency assumption} {(see \cref{thm:analytic+brackets-nilpotent=>magnus11-equality})}. In this case, convergence requires analyticity, contrary to the proofs we give for the other expansions under a full {nilpotency assumption}.
    
    \item For analytic vector fields, only the Chen-Fliess expansion is known to converge. We give in \cref{sec:conv-magnus} \textbf{new strong counter examples} to the convergence of {CBHD and} Magnus expansions, which \textbf{disprove the convergence} of these expansions even for analytic vector fields and in very weak senses {(see \cref{prop:cbh.c-ex})}. We state an open problem concerning the convergence of Sussmann's infinite product for analytic vector fields {(see \cref{open:sussmann})}. 
    
    \item When the system involves a time-invariant drift and a time-varying perturbation, we show in \cref{sec:estimates-control} that only the Magnus expansion fails to provide \textbf{well-behaved estimates with respect to the perturbation size}. For the three other expansions, it turns out to be possible to obtain such estimates by summing well-defined infinite partial series which converge for analytical vector fields {(see mostly \cref{thm:Magnus_2_u,Prop:error_u_Prod})}. 
    
    \item In the particular case of scalar-input systems, we prove in \cref{sec:u1} \textbf{new errors estimates involving a negative Sobolev norm of the time-varying input} {(see mostly \cref{Magnus3_Cor1,Prop:error_u1_Prod})}. Such estimates are the best compatible with the regularity of the input-to-state map and can be helpful for specific applications.
    
    \item Eventually, we derive in \cref{sec:diffeo} \textbf{approximate direct intrinsic representations of the state} for nonlinear systems, which don't require the computation of flows {(see \cref{Prop:Approx_intrinsic_repr})}. Our formulas can be viewed as almost-diffeomorphisms and might be useful for applications in control theory. Unfortunately, we also study a counter-example which demonstrates that one cannot obtain an exact representation through a diffeomorphism.
\end{itemize}


\section{Formal expansions for linear dynamics}
\label{sec:formal}

In this section, we consider formal linear differential equations, recall classical expansions valid in this formal setting (for which there is no convergence issue{; see nevertheless \cref{rk:topo-a-hat}}) and introduce a new mixed expansion which isolates the role of a perturbation in the dynamics.
{Here and in the sequel, the adjective \emph{formal} denotes situations in which we work within the realm of formal power series (see \cref{-def:formalseries}).}

\subsection{Notations}
\label{sec:notations.algebra}

We recall classical definitions and notations for usual algebraic objects.
In the sequel, $\K$ denotes the field $\R$ or $\C$. All statements and proofs hold for both base fields. It will be implicit that all vector spaces and algebras are constructed from the base field $\K$. 

\subsubsection{Free algebras}

We refer to the books \cite{MR559927, MR1231799} for thorough introductions to Lie algebras and free Lie algebras.

\begin{definition}[Indeterminates]
 Let $I$ be a finite set. At the formal level, we consider a set $X := \{ X_i; \enskip i \in I \}$ of \emph{indeterminates}, indexed by $I$. For applications, we will substitute in their place matrices or vector fields. Most often, we will write $I = \intset{1,q}$ for some $q \in \N^*$, or $I = \intset{0,q}$ when we want to isolate the role of the indeterminate $X_0$.
\end{definition}

\begin{definition}[Free monoid] \label{def:free.monoid}
    For $I$ as above, we denote by $I^*$ the \emph{free monoid} over $I$, i.e.\ the set of finite sequences of elements of $I$ endowed with the concatenation operation. 
    {More precisely, if $\sigma = (\sigma_i)_{1 \leq i \leq \ell}$ and $\sigma' = (\sigma_i')_{1 \leq i \leq \ell'}$ are elements of $I^*$, then the concatenation of $\sigma$ and $\sigma'$ is the sequence $\sigma \cdot \sigma' = (\sigma_i'')_{1 \leq i \leq \ell+\ell'}$ where $\sigma_i'' = \sigma_i$ if $1 \leq i \leq \ell$ and $\sigma_i'' = \sigma'_{i-\ell}$ if $\ell+1 \leq i \leq \ell + \ell'$. 
    It is common to write the elements of $I^*$ as words whose letters are elements of $I$, by juxtaposition of the elements of the sequence. With this point of view, the concatenation operation is the juxtaposition of words. 
    For a more detailed exposition, see \cite[§7.2]{MR1727844}.}

    For $\sigma = (\sigma_1, \dotsc \sigma_k) \in I^*$, where $k$ is the length of $\sigma$ also denoted by $|\sigma|$, we let $X_\sigma := X_{\sigma_1} \dotsb X_{\sigma_k}$. 
    This operation defines an homomorphism from $I^*$ to $X^*$, the free monoid over $X$ (monomials over $X$).
\end{definition}

\begin{definition}[Free algebra] \label{def:free.algebra}
 For $X$ as above, we consider $\mathcal{A}(X)$ the \emph{free associative algebra} generated by $X$ over the field $\K$, i.e.\ the unital associative algebra of polynomials of the non commutative indeterminates $X$ (see also \cite[Chapter~3, Section~2.7, Definition~2]{MR1727844}). $\mathcal{A}(X)$ can be seen as a graded algebra:
 \begin{equation}
    \mathcal{A}(X) = \underset{n \in \N}{\bigoplus} \mathcal{A}_n(X),
 \end{equation}
 where $\mathcal{A}_n(X)$ is the finite-dimensional $\K$-vector space spanned by monomials of degree $n$ over $X$. In particular $\mathcal{A}_0(X) = \K$ and $\mathcal{A}_1(X) = \vect_{\K}(X)$.
\end{definition}

\begin{definition}[Free Lie algebra] \label{def:free.lie}
For $X$ as above, $\mathcal{A}(X)$ is endowed with a natural structure of Lie algebra, the Lie bracket operation being defined by $[a,b] = ab - ba$. This operation satisfies $[a, a] = 0$ and the Jacobi identity $[a,[b,c]]+[c,[a,b]]+[b,[c,a]] = 0$. 
{We also write $[a,b]$ as $\ad_a(b)$ (respectively $\dad_b(a)$) which allows for iterated left (resp.\ right) bracketing}.
We consider $\mathcal{L}(X)$, the \emph{free Lie algebra} generated by $X$ over the field $\K$, which is defined as the Lie subalgebra generated by $X$ in $\mathcal{A}(X)$. It can be seen as the smallest linear subspace of $\mathcal{A}(X)$ containing all elements of $X$ and stable by the Lie bracket  (see also \cite[Theorem~0.4]{MR1231799}). $\mathcal{L}(X)$ is a graded Lie algebra: 
\begin{equation}
    \mathcal{L}(X) = \underset{n \in \N}{\bigoplus} \mathcal{L}_n(X), \qquad 
    [\mathcal{L}_m(X),\mathcal{L}_n(X)] \subset \mathcal{L}_{m+n}(X)
 \end{equation}
 where, for each $n \in \N$, we define $\mathcal{L}_n(X) := \mathcal{L}(X) \cap \mathcal{A}_n(X)$.
\end{definition}

\begin{definition}[Nilpotent Lie algebra] \label{Def:Nilpotent Lie algebra}
     Let $L$ be a Lie algebra. 
     We define recursively the following two-sided Lie ideals: $L^1 := L$ and, for $k \geq 1$, $L^{k+1} := [L,L^k]$ i.e.\ $L^{k+1}$ is the linear subspace of $L$ generated by brackets of the form $[a,b]$ with $a \in L$ and $b \in L^k$.
     Let $m\in\N^*$. We say that $L$ is a \emph{nilpotent Lie algebra of index $m$} when $L^m = \{0\}$ and $m$ is the smallest integer for which this property holds.
\end{definition}

\begin{definition}[Free nilpotent Lie algebra]
    Let $m \in \N^*$.
    The \emph{free $m$-nilpotent Lie algebra} over $X$ is the quotient $\mathcal{N}_{m}(X) := \mathcal{L}(X) / \mathcal{L}(X)^m$ (with the notation of \cref{Def:Nilpotent Lie algebra}). 
    Then the canonical surjection $\sigma_m : \mathcal{L}(X) \rightarrow \mathcal{N}_{m}(X)$ is an homomorphism {of Lie algebras}.
\end{definition}

The universal properties of the various free algebras constructed above allow to transport on algebras relations proved at the free level.

\begin{lemma} \label{Lem:identification_libre}
    The following universal properties hold.
    \begin{itemize}
        \item For each unital associative algebra $A$ and map $\Lambda : X \to A$, there exists a unique homomorphism of algebras $\mathcal{A}(X) \to A$ that extends $\Lambda$.
        \item For each Lie algebra $L$ and map $\Lambda : X \to L$, there exists a unique homomorphism of Lie algebras $\mathcal{L}(X) \to L$ that extends $\Lambda$.
        \item  Let $m\in\N^*$. For each nilpotent Lie algebra $L$ of index $m$ and map $\Lambda: X \to L$, there exists a unique homomorphism of Lie algebras $\mathcal{N}_m(X) \to L$ that extends $\Lambda$. 
    \end{itemize}
 \end{lemma}

\subsubsection{{Iterated brackets} and evaluation}
\label{s:iterated-brackets}

\begin{definition}[{Iterated brackets}]
 For $X$ as above, we consider $\Br(X)$ the \emph{set of {iterated brackets}} of elements of $X$. 
 This set can be defined by induction: for $X_i \in X$, $X_i \in \Br(X)$ and if $b_1, b_2 \in \Br(X)$, then the ordered pair {$(b_1,b_2)$} belongs to $\Br(X)$. More rigorously, one can define $\Br(X)$ as the free magma over $X$ or as the set of {rooted full} binary trees, with leaves labeled by $X$.
\end{definition}

For $b \in \Br(X)$, we will use the following notations:
\begin{itemize}
    \item $|b|$ will denote the length of $b$ (i.e.\ the number of leaves of the tree).
    \item If $|b| > 1$, there exists a unique {pair $b_1 \in \Br(X)$ and $b_2 \in \Br(X)$} such that {$b = (b_1, b_2)$} (left and right factors) which are denoted as $\lambda(b) = b_1$ and $\mu(b) = b_2$. We also write {$(b_1,b_2)$} as $\ad_{b_1}(b_2)$ {(respectively $\dad_{b_2}(b_1)$)} which allows iterated left {(resp.\ right)} bracketing.
    \item For $i \in I$, $n_i(b)$ denotes the number of occurrences of the indeterminate $X_i$ in $b$. When $I = \intset{0,q}$ we will also write $n(b) = n_1(b) + \dotsb + n_q(b) = |b| - n_0(b)$.
\end{itemize}

\begin{remark} \label{rk:eval}
    There is a natural \emph{evaluation} mapping $\eval$ from $\Br(X)$ to $\mathcal{L}(X)$ defined by induction by $\eval(X_i) := X_i$ for $X_i \in X$ and $\eval({(b_1, b_2)}) := [\eval(b_1), \eval(b_2)]$. 
    Through this mapping, $\Br(X)$ spans $\mathcal{L}(X)$ over $\K$, i.e.\ $\mathcal{L}(X) = \vect_\K \eval(\Br(X))$. 
    This mapping is however not injective: for example, ${(X_1, X_1)}$ and ${(X_2, (X_1, X_1))}$ are two different elements of $\Br(X)$, both evaluated to zero in $\mathcal{L}(X)$. 

    More precisely, the $\eval$ map extends to a surjective homomorphism {of algebras} from the nonassociative free algebra over $X$ (which is the free vector space over $\Br(X)$, whose elements are (finite) linear combinations of elements of $\Br(X)$, endowed with the natural product map induced by the product in $\Br(X)$).
    Moreover the kernel of the extended $\eval$ is precisely the ideal generated by the relations that define anticommutativity and the Jacobi identity in $\mathcal{L}(X)$.
    {This gives an alternative description of $\mathcal{L}(X)$ as a quotient of the free vector space over $\Br(X)$.}
\end{remark}

\begin{definition}[Subspaces of brackets] \label{def:SM}
    When $I = \intset{0,q}$ and {$M \in \N$, $S_M$} denotes the {vector subspace $\mathcal{L}(X)$} defined by
    \begin{equation}
        {S_M} := {\vect_\K} \left\{ {\eval(b)} ; \enskip b \in \Br(X), n(b) \leq M \right\}.
    \end{equation}
\end{definition}

\subsubsection{Formal power series, exponential and logarithms}

\begin{definition}[Formal power series]\label{-def:formalseries}
    We consider the (unital associative) algebra $\widehat{\mathcal{A}}(X)$ of \emph{formal power series} generated by $\mathcal{A}(X)$. 
    An element $a \in \widehat{\mathcal{A}}(X)$ is a sequence $a = (a_n)_{n\in\N}$ written $a = \sum_{n \in \N} a_n$, where $a_n \in \mathcal{A}_n(X)$ with, in particular, $a_0 \in \K$ being its constant term.
 We also define the Lie algebra of formal Lie series $\widehat{\mathcal{L}}(X)$ as the Lie algebra of formal power series $a \in \widehat{\mathcal{A}}(X)$ for which $a_n \in \mathcal{L}(X)$ for each $n \in \N$.
 For ${a} \in \widehat{\mathcal{A}}(X)$ and $\sigma \in I^*$, $\langle {a}, X_\sigma \rangle$ denotes the coefficient of $X_\sigma$ in ${a}$: ${a} =\sum_{\sigma \in I^*} \langle {a}, X_\sigma \rangle X_{\sigma}$.
\end{definition}
\begin{remark} \label{rk:topo-a-hat}
    The definition of $\widehat{\mathcal{A}}(X)$ can be made more rigorous by considering $\val : \mathcal{A}(X) \to \N \cup \{\infty\}$ {defined} by $\val(a) := \inf \{n \in \N; a \in \bigoplus_{k \geq n} {\mathcal{A}}_k(X)\} $. 
    Then ${\delta(a,b)} := e^{-\val(b-a)}$ is a distance on $\mathcal{A}(X)$, that induces the discrete topology on each $\mathcal{A}_n(X)$, and $\widehat{\mathcal{A}}(X)$ is defined as the completion of the metric space $\mathcal{A}(X)$, to which the operations on $\mathcal{A}(X)$ naturally extend as continuous operations, endowing it with a structure of topological algebra. 
    {This distance verifies a stronger triangular inequality: $\delta(a,b) \leq \max\{\delta(a,c),\delta(b,c)\}$ (usually referred to as the \emph{ultrametric inequality}).} 
    {This construction allows to write, for $a \in \widehat{\mathcal{A}}(X)$ with components $a_n \in \mathcal{A}_n(X)$,
    \begin{equation}
        a = \lim_{N \to +\infty} \sum_{n \leq N} a_n,
    \end{equation}
    where the convergence holds with respect to the topology described above.
    This justifies the notation $a = \sum_{n \in \N} a_n$ used in \cref{-def:formalseries}.}
    To avoid confusion {with convergence issues associated with the evaluation of formal power series when substituting the indeterminates by objects coming from analysis} we shall however not use the term \emph{convergence} in this context.
\end{remark}

If $a \in \widehat{\mathcal{A}}(X)$ has zero constant term, we define $\exp (a) \in \widehat{\mathcal{A}}(X)$ and $\log (1+a) \in \widehat{\mathcal{A}}(X)$ as
\begin{align}
 \label{def:exp}
 \exp (a) & := \sum_{m \geq 0} \frac{a^m}{m!}, \\
 \label{def:log}
 \log (1+a) & := \sum_{m \geq 1} \frac{(-1)^{m-1}}{m} a^m.
\end{align}
Since $a$ has zero constant term, one checks that the right-hand sides of \eqref{def:exp} and \eqref{def:log} indeed define formal power series of $\widehat{\mathcal{A}}(X)$ {(and the sums converge in the sense of the topology constructed in \cref{rk:topo-a-hat})}.
In particular, $\log(\exp(a)) = a$ and $\exp(\log(1+a)) = 1+a$.

\begin{lemma} \label{lemma:log_is_unique}
 Let $a, b \in \widehat{\mathcal{A}}(X)$ with zero constant term. Then $a = b$ if and only if $\exp(a) = \exp(b)$.
\end{lemma}

\begin{proof}
 The forward implication is obvious. Conversely, if $\exp(a) = \exp(b)$ in $\widehat{\mathcal{A}}(X)$, then, for every $r \geq 1$, their components in $\mathcal{A}_r$ are equal. Moreover, from \eqref{def:exp}, one has:
 \begin{equation}
  \left( \exp(a) \right)_r = \sum_{k = 1}^r \sum_{r_1 + \ldots r_k = r} \frac{a_{r_1} \dotsc a_{r_k}}{k!} = a_r + \Theta_r \left(a_1, \ldots a_{r-1}\right),
 \end{equation}
 for some function $\Theta_r$ depending only on the $a_{r'}$ for $r' < r$. Hence, we obtain by induction on $r \geq 1$ that $a_r = b_r$ from the equalities $(\exp(a))_r = (\exp(b))_r$. 
\end{proof}

\subsection{Formal differential equations and iterated integrals} 

Using the notations of \cref{sec:notations.algebra}, for $i \in I$, let $a_i \in L^1(\R_+;\K)$ and define $a$ by
\begin{equation} \label{eq:a.ai}
 a(t) := \sum_{i \in I} a_i(t) X_i.
\end{equation}
In this section, we consider the following formal ordinary differential equation set on $\widehat{A}(X)$, driven by $a$ and associated with some initial data $x^\star$,
\begin{equation} \label{formalODE}
 \left\{
 \begin{aligned}
  \dot{x}(t) & = x(t) a(t), \\
  x(0) & = x^\star,
 \end{aligned}
 \right.
\end{equation}
whose solutions are precisely defined in the following way.

\begin{definition}[Solution to a formal differential equation] \label{def:formalODE}
    Let $a_i \in L^1(\R_+;\K)$ for $i \in I$ and define $a$ by~\eqref{eq:a.ai}.
    Let $x^\star \in \widehat{\mathcal{A}}(X)$ with homogeneous components $x^\star_n \in\mathcal{A}_n(X)$. 
    The \emph{solution to the formal differential equation} \eqref{formalODE} is the formal-series valued function $x: \R_+ \to \widehat{\mathcal{A}}(X)$, whose homogeneous components $x_n : \R_+ \to \mathcal{A}_n(X)$ are the unique continuous functions that satisfy, for every $t \geq 0$, $x_0(t) = x_0^\star$ and, for every $n \in \N^*$,
 \begin{equation} \label{eq:xn.xn1}
  x_{n}(t) = x^\star_n + \int_0^t x_{n-1}(\tau) a(\tau) \dd\tau.
 \end{equation}
\end{definition}

{
\begin{definition}[Ordered simplex] \label{def:simplex}
    For $r \in \N^*$ and $t > 0$, we introduce
    \begin{equation}
        \Delta^r(t) := \{ (\tau_1, \dotsc, \tau_r) \in (0,t)^r ; \enskip 0 < \tau_1 < \dotsb < \tau_r < t \}.
    \end{equation}
\end{definition}
}

Iterating this integral formula yields the following power series expansion, which is the most direct way to compute the solution to \eqref{formalODE} and {was introduced in \cite{MR0073174,zbMATH03126609} and popularized in control theory by \cite{MR613847}.
In the ``chronological calculus'' terminology (not used in the present article), it is called ``(right) formal Volterra chronological series'' \cite[Section 1.5]{MR524203}.}

\begin{lemma}[{Chen series}] \label{cf.formula}
 In the context of \cref{def:formalODE}, the solution to~\eqref{formalODE} with initial data $x^\star = 1$ can be expanded as
 \begin{equation} \label{eq:chen-fliess}
  x(t) = \sum_{\sigma \in I^*} \left( \int_0^t a_\sigma \right) X_\sigma,
 \end{equation}
 where $\int_0^t a_\emptyset = 1$ by convention and, for $\sigma \in I^*$ with $|\sigma| \geq 1$, we introduce the notation
 \begin{equation} \label{a.sigma}
   \int_0^t a_\sigma :=
   \int_{{\Delta^n(t)}} a_{\sigma_1}(\tau_1) \dotsm a_{\sigma_n}(\tau_n) \dd \tau.
 \end{equation}
 \end{lemma}

\begin{proof}
 Expansion \eqref{eq:chen-fliess} is a direct consequence of the iterated application of \eqref{eq:xn.xn1} and of the definition of $X_\sigma$ in \cref{def:free.monoid} and can be proved by induction on the length of $\sigma$.
\end{proof}

\begin{remark}\label{rk:cf-drawbacks}
Despite its simplicity, the {Chen series} suffers from a major drawback: it involves non intrinsic quantities and is redundant. As an illustration, this has the following consequences:
\begin{itemize}
    \item The functionals $\int_0^t a_\sigma$ for $\sigma \in I^*$ are not algebraically independent. For example, for every solution to~\eqref{formalODE} and every $t \geq 0$, one has the identity
    \begin{equation}
     \langle x(t), X_1 X_2 \rangle + 
     \langle x(t), X_2 X_1 \rangle 
     - \langle x(t), X_1 \rangle \langle x(t), X_2 \rangle
     = 0
    \end{equation}
    \item In the context of nonlinear ordinary differential equations, the representation \eqref{eq:chen-fliess} can fail to converge for smooth vector fields despite strong structural assumptions (see \cref{sec:chen.nilpotent}).
    \item In the context of nonlinear ordinary differential equations, the representation \eqref{eq:chen-fliess} will not be invariant by diffeomorphism (see \cref{rk:chen.gb.fb}), which would be a desirable invariance.
\end{itemize}
This drawback motivates the search for more intrinsic representations of the solutions, which will turn out to involve Lie algebras. 
\end{remark}

{The Chen series give rise to \emph{Fliess operators} (stemming from \cite{MR613847,fliess1983realisation}) which can be defined, given some $c \in \widehat{\mathcal{A}}(X)$, as $a \mapsto \sum_{\sigma \in I^*} \langle c, X_\sigma \rangle \int_0^t a_\sigma$. Such operators are well-defined (converge) provided that the coefficients $\langle c, X_\sigma \rangle$ satisfy an appropriate asymptotic behavior.
Fliess operators can be used to model input-output systems and feedback groups.
For manipulations of such operators thanks to an underlying Hopf algebra structure, we refer to \cite{zbMATH05932536,zbMATH06372161,zbMATH07290581,gray2021additive}, which investigate the question of whether an interconnection of such operators remains a Fliess operator, and its convergence, both in scalar and multivariate settings.
See also \cite{gray2002fliess} for the investigation of global convergence issues, and realization of such formal operators on concrete systems.
}

\subsection{Logarithm of flows, coordinates of the first kind} \label{sec:formal.log.flows}

In the particular case where $a(t)$ is a constant element $a \in \mathcal{A}_1(X)$, evaluating the iterated integrals in \eqref{a.sigma} yields the elegant formula $x(t) = x^\star \exp (ta)$, with the notation of \eqref{def:exp}.
Of course, it is no longer valid for a time-varying dynamic (because the indeterminates do not commute), but one can wish to find an object of which the flow is the exponential, the \emph{logarithm of the flow}.

In this section, we recall and prove \cref{thm:formal.log}, which states that the logarithm of flows of formal linear differential equations is given by explicit Lie brackets. 
{The key argument is the structure result \cref{thm:log-lie}, which states that the logarithm of the flow is a Lie series, and of which we give an elementary proof based on the differential equation satisfied by the logarithm of the flow.}
We rely on well-known algebraic results, which we recall, for the sake of giving a self-contained presentation.

\subsubsection{A differential equation for the logarithm of the flow}

{
We start by deriving the formal differential equation \eqref{eq:EDO-Z} satisfied by the logarithm of the flow.
This equation is well-known (see e.g.\ \cite[Theorem~III]{MR0067873}, \cite[Theorem~4.1]{zbMATH03126609} or \cite[formula~(5.2)]{MR524203}).
We provide an elementary derivation (see \cref{rk:dz-joli}).
}

\begin{proposition} \label{prop:EDO-Z}
    The following statements hold.
    \begin{enumerate}
    \item Let $T>0$ and $z \in \CC^1([0,T]; {\widehat{\mathcal{A}}(X)})$. Then, {for every $t \in [0,T]$,}
    \begin{equation} \label{eq:dez-dt}
        \frac{\dd}{\dd t} \exp(z(t)) 
        = \exp(z(t)) \sum_{n=0}^{+\infty} \frac{(-1)^n}{(n+1)!} \ad^n_{z(t)} (\dot{z}(t)). 
    \end{equation}
    \item Let $a$ be given by (\ref{eq:a.ai}) and $x$ denote the solution to \eqref{formalODE} with initial data $x^\star = 1$.
    Then $z := \log x$ satisfies, for almost every $t \in \R_+$,
    \begin{gather}
        \label{eq:EDO-Z-reversed}
        \sum_{n=0}^{+\infty} \frac{(-1)^n}{(n+1)!} \ad^n_{z(t)} (\dot{z}(t)) = a(t), \\
        \label{eq:EDO-Z}
        \dot{z}(t) = \sum_{n=0}^{+\infty} \frac{(-1)^n B_n}{n!} \ad^n_{z(t)}(a(t)),
    \end{gather}
    where the Bernoulli numbers $(B_n)_{n\in\N}$ are defined in \eqref{eq:def:bernoulli}.
    \end{enumerate}
\end{proposition}

\begin{proof}
    We prove the two claims successively.
    \begin{enumerate}
    \item {The regularity assumption $z \in \CC^1([0,T]; \widehat{\mathcal{A}}(X))$ is to be understood component by component, i.e.\ means that for each $\sigma \in I^*$, $t \mapsto \langle z(t), X_\sigma \rangle$ belongs to $\CC^1([0,T];\K)$.}
    We have
    \begin{equation}
        \begin{split}
            \frac{\dd}{\dd t} \exp(z(t)) 
            & = \frac{\dd}{\dd t} \left( \sum_{k=0}^{+\infty} \frac{z^k(t)}{k!} \right) 
              = \sum_{k=0}^{+\infty} \frac{1}{(k+1)!} \sum_{j=0}^{k} z^j(t) \dot{z}(t) z^{k-j}(t)
            \\
            & = \exp(z(t)) \left( \sum_{l=0}^{+\infty} \frac{(-1)^l}{l!}z^l(t) \right)
            \left( \sum_{k=0}^{+\infty} \frac{1}{(k+1)!} \sum_{j=0}^{k} z^j(t) \dot{z}(t) z^{k-j}(t) \right).
        \end{split}
    \end{equation}
    Letting $n := k + l$ and $i := l + j$, we obtain that
    \begin{equation}
        \begin{split}
            \frac{\dd}{\dd t} \exp(z(t)) 
            & = \exp(z(t)) \sum_{n=0}^{+\infty} \frac{1}{(n+1)!} \sum_{i=0}^n z^{i}(t) \dot{z}(t) z^{n-i}(t)
            \sum_{l=0}^i (-1)^l \binom{n+1}{l} 
        \end{split}
    \end{equation}
    The following formulas, which can be proved by induction using Pascal's rule,
    \begin{align}
        \sum_{l=0}^i (-1)^l \binom{n+1}{l} & = (-1)^i \binom{n}{i}, \\
        \sum_{i=0}^n (-1)^i \binom{n}{i} z^i y z^{n-i} & = (-1)^n \ad^n_z(y) 
    \end{align}
    give the conclusion.
    {Of course, if $z \in W^{1,1}((0,T);\widehat{\mathcal{A}}(X))$ (i.e.\ absolutely continuous), equations \eqref{eq:dez-dt} remains true as an equality in $L^1((0,T);\widehat{\mathcal{A}}(X))$, i.e.\ holding for almost every $t \in (0,T)$.}
    
    \item {Since $z = \log x$ and $\dot{x} = x a$, \eqref{eq:EDO-Z-reversed} is an immediate consequence of \eqref{eq:dez-dt}, using the preceding comment since both $x$ and~$z$ have $W^{1,1}$ regularity in time when $a$ has $L^1$ regularity in time.
    
    Starting from \eqref{eq:EDO-Z-reversed} and applying $\sum_k (-1)^k B_k / k! \ad^k_{z(t)}$ to both sides yields \eqref{eq:EDO-Z} because
    \begin{equation}
        \sum_{k = 0}^{+\infty} \frac{(-1)^k B_k}{k!} \sum_{\ell = 0}^{+\infty} \frac{(-1)^\ell}{(\ell+1)!} \ad^{k+\ell}_{z(t)}(\dot{z}(t)) = \dot{z}(t).
    \end{equation}
    This follows from the change of index $n := k + \ell$ and the combinatorial relation \eqref{eq:bernoulli.2}.} \qedhere
    \end{enumerate}
\end{proof}

{
\begin{remark}
    \label{rk:dz-joli}
    The historical proofs of \cref{prop:EDO-Z} are written using the Poisson bracket notation $\{ \cdot, z^k \} := (-1)^k \ad_z^k(\cdot)$ which allows to write \eqref{eq:EDO-Z} as the nice equality
    \begin{equation}
        \dot{z} = \left\{ a, \frac{z}{e^z- 1} \right\},
    \end{equation}
    using the generating series \eqref{eq:def:bernoulli} of the Bernoulli numbers.
    This approach allows elegant computations, but requires some setup (see \cite[Section~III]{MR0067873} or \cite[Section~1]{zbMATH03126609}), which is why we prefer here the elementary computations used in the preceding proof.
\end{remark}}

\subsubsection{The logarithm of the flow is a Lie series}

{The fundamental result concerning the logarithm of the flow is that it is a Lie series.
We repeat here the proof given in \cite[Theorem~4.2]{zbMATH03126609} for the sake of completeness.
At least two other approaches can be used: one relying on shuffle relations and Ree's theorem (see \cref{sec:ree}) and another one relying on Friedrich's criterion (see \cref{sec:friedrichs}).

\begin{theorem} \label{thm:log-lie}
    Let $a$ be given by \eqref{eq:a.ai} and $x$ be the solution to \eqref{formalODE} with initial data $x^\star = 1$.
    Then, for every $t > 0$, $\log x(t) \in \widehat{\mathcal{L}}(X)$.
\end{theorem}

\begin{proof}
    The proof relies on an iterated integration of \eqref{eq:EDO-Z}, where $z = \log x$.
    More precisely, writing $z = \sum z_n$ where $z_n \in \mathcal{A}_n(X)$, we prove by induction on $n$ that, for every $t$, $z_n(t) \in \mathcal{L}(X)$.
    First, for every $t > 0$, since $x_0(t) = 1$, one has $z_0(t) = 0$ so $z_0(t) \in \mathcal{L}(X)$.
    Then, for every $n \geq 1$, by \eqref{eq:EDO-Z},
    \begin{equation} \label{eq:dzn/dt}
        \dot{z}_n(t) = \sum_{k = 0}^{n - 1} \frac{(-1)^k B_k}{k!} 
        \sum_{n_1 + \dotsb + n_k = n - 1} [z_{n_1}(t), [z_{n_2}(t), \dotsc [ z_{n_k}(t), a(t)]\dotsb]],
    \end{equation}
    where the sum ranges over indexes $n_i \geq 1$.
    Moreover, for every $T > 0$ and $y \in L^1((0,T);\mathcal{L}(X))$ one checks that, for every $t \in [0,T]$, $\int_0^t y \in \mathcal{L}(X)$.
    By the induction assumption, $z_{n_i}(t) \in \mathcal{L}(X)$ for each $n_i \leq n-1$ and every $t$.
    By the previous comment, this property is preserved by the time-integration of \eqref{eq:dzn/dt}, so $z_n(t) \in \mathcal{L}(X)$ for every $t$.
\end{proof}

}

\subsubsection{Notations}

We start with an abstract definition of the truncated logarithm of a time-dependent dynamic.

\begin{definition}
 For $m, r \in \N^*$, we define the set of ordered positive partitions of size $m$ of $r$, 
 \begin{equation} \label{def:Nmr}
  \N^m_r := \left\{ \mathbf{r}=(\mathbf{r}_1,\dotsc,\mathbf{r}_m) \in (\N^*)^m ; \enskip \mathbf{r}_1+\dotsb+\mathbf{r}_m=r \right\},
 \end{equation}
 where $\N^m_r = \emptyset$ when $r < m$. For each $\mathbf{r} \in \N^m_r$ and $t > 0$, we also define {the product of simplexes}
 \begin{equation} \label{def:Trpt}
    {\Delta^{\mathbf{r}}(t) :=
    \Delta^{\mathbf{r}_1}(t) \times \dotsb \times \Delta^{\mathbf{r}_m}(t)
    }.
 \end{equation}
\end{definition}

\begin{example}
    The sets {$\Delta^{\mathbf{r}}(t)$} will be used as integration domains.
    {As easy examples,} one has
\begin{align}
 {\Delta^{(3)}(t)} & = \{ \tau = (\tau_1, \tau_2,\tau_3) \in (0,t)^3 ; \enskip 0 < { \tau_1 < \tau_2 < \tau_3 } < t\},
 \\
 {\Delta^{(1,1,1)}(t)} & = \{ \tau = (\tau_1, \tau_2,\tau_3) \in (0,t)^3 \}.
\end{align}
A more complex example for $r = 4$, $m = 2$ and $\mathbf{r} = (2,2) \in \N^2_4$ is
\begin{equation}
  {\Delta^{(2,2)}(t)} = \{ \tau = (\tau_1, \tau_2,\tau_3,\tau_4) \in (0,t)^4 ; \enskip 0 < {\tau_1 < \tau_2} < t \enspace  \mathrm{and} \enspace  0 < {\tau_3 < \tau_4} < t \}.
\end{equation}
\end{example}

We now give a notation for the (truncated or complete) logarithm of a time-dependent dynamic. We will see in the sequel why this quantity indeed corresponds to a logarithm.

\begin{definition}[Abstract logarithm of a time-varying field] \label{def:LOGM}
 Let $M \in \N$ or $M = +\infty$, $t > 0$ and $F$ be a map from $[0,t]$ with values in some algebra. We introduce the notation
 \begin{equation} \label{eq:LOGM}
  \mathrm{Log}_M \{F\}(t) := \sum_{r=1}^M \frac{1}{r} \sum_{m=1}^r \frac{(-1)^{m-1}}{m}
  \sum_{\mathbf{r} \in \N^m_r} \int_{{\Delta^{\mathbf{r}}(t)}} 
  {[ \dotsb [ F(\tau_1), F(\tau_2) ], \dotsc  F(\tau_r) ]} \dd \tau.
 \end{equation}
\end{definition}

\begin{remark}
 In such an abstract setting, the right-hand side of \eqref{eq:LOGM} does not make sense since we are not able to define an integral over an abstract algebra (without topology on the algebra and without time-regularity on $F$). At this stage, we see \eqref{eq:LOGM} as an abstract formula or notation. We will check, each time we use it, that we can give a meaning to the integrals. 
\end{remark}

\subsubsection{A preliminary algebraic result}

{Since the monomials form a basis of $\mathcal{A}(X)$, one can define the following} linear map $\beta$ from $\mathcal{A}(X)$ to $\mathcal{L}(X)$ by setting its values on the monomials by $\beta(1) := 0$, $\beta(X_i) := X_i$ for $1 \leq i \leq q$, and, for $1 \leq i_1, \dotsc, i_k \leq q$ with $k \in \N^*$,
\begin{equation} \label{def:B}
 \beta(X_{i_1} X_{i_2} \dotsm X_{i_k}) := [ \dotsb [ X_{i_1}, X_{i_2} ], \dotsc, X_{i_k} ].
\end{equation}
This process defines a standard way, the ``left to right'' or ``left normed'' bracketing, to associate a Lie bracket to each monomial. The following important result, proved successively by Dynkin~\cite{MR0021940}, Specht~\cite{MR0028301} and Wever~\cite{MR0034397} states that, if a polynomial is a Lie element, then it is equal to its left normed bracketing.

\begin{lemma}[Dynkin's theorem]
 \label{thm:dynkin}
 For $a \in \mathcal{A}_n(X)$, 
 $a \in \mathcal{L}(X)$ if and only if $\beta(a) = n a$.
\end{lemma}

\begin{proof}
 This statement is contained in the equivalence between \emph{(i)} and \emph{(v)} of \cite[Theorem~1.4]{MR1231799}.
\end{proof}

\begin{example}
 The element $X_1 X_2$ does not belong to $\mathcal{L}(X)$. And indeed, $\beta(X_1 X_2) = X_1 X_2 - X_2 X_1 \neq 2 X_1 X_2$. On the contrary, the element $[X_1, X_2] = X_1 X_2 - X_2 X_1$ belongs to $\mathcal{L}(X)$. And indeed, $\beta([X_1, X_2]) = (X_1 X_2 - X_2 X_1) - (X_2 X_1 - X_1 X_2) = 2 [X_1, X_2]$.
\end{example}

\subsubsection{An explicit formula for the logarithm of the flow}

{
We now state an explicit expansion of the state as the exponential of the logarithm of the flow. 
It is the continuous analogue of the well-known CBHD formula (which we recall in \cref{sec:cbhd} as a corollary).
It was originally derived by Magnus in \cite[Theorem~III]{MR0067873} and is thus often referred to as the ``Magnus expansion''.
}

\begin{theorem} \label{thm:formal.log}
 For $t \in \R_+$ and $x^\star \in \widehat{\mathcal{A}}(X)$, the solution $x$ to \eqref{formalODE} satisfies
 \begin{equation} \label{eq:x.exp.log}
   x(t) = x^\star \exp \left( \mathrm{Log}_\infty\{a\}(t) \right),
 \end{equation}
 with the notation of \cref{def:LOGM}.
\end{theorem}

\begin{proof}
    First, by linearity, it suffices to prove \eqref{eq:x.exp.log} for $x^\star = 1$.
    Repeated integration of \eqref{eq:xn.xn1} yields {the following formula (which is a slightly different form of the Chen series of \cref{cf.formula})},
    \begin{equation}
     x(t) = 1 + \sum_{r \geq 1} \int_{{\Delta^r(t)}} {a(\tau_1) \dotsm a(\tau_r)} \dd \tau.
    \end{equation}
    Hence, recalling the definitions \eqref{def:Nmr} of $\N^m_r$ and \eqref{def:Trpt} of ${\Delta^{\mathbf{r}}(t)}$, one has
    \begin{equation} \label{eq:log.xta.before}
     \log (x(t)) = \sum_{r = 1}^{+\infty} \sum_{m=1}^r \sum_{\mathbf{r} \in \N^m_r} \frac{(-1)^{m-1}}{m} \int_{{\Delta^{\mathbf{r}}(t)}} {a(\tau_1) a(\tau_2) \dotsm a(\tau_r)} \dd\tau.
    \end{equation}
    {By \cref{thm:log-lie}, for each $t \geq 0$, $\log(x(t)) \in \widehat{\mathcal{L}}(X)$.
    Hence,} applying \cref{thm:dynkin} to each of its homogeneous components in $\mathcal{A}_r$ proves that
    \begin{equation} \label{eq:log.xta.after}
      \log (x(t)) 
      = \sum_{r = 1}^{+\infty} \frac{1}{r} \sum_{m=1}^r \sum_{\mathbf{r} \in \N^m_r} \frac{(-1)^{m-1}}{m} \int_{{\Delta^{\mathbf{r}}(t)}} {[ \dotsb [ a(\tau_1), a(\tau_2)], \dotsc a(\tau_r)]} \dd\tau.
    \end{equation}
    Recalling the notation \eqref{def:LOGM} and taking the exponential concludes the proof of~\eqref{eq:x.exp.log}.
\end{proof}

{Magnus expansions (also called BCH expansions) have been extended to more general structures than Lie algebras, for instance to pre-Lie (another name for ``chronological algebras'') and post-Lie algebras \cite{zbMATH07040544}, Rota-Baxter algebras \cite{al2021post,zbMATH06369956} and dendriform algebras \cite{zbMATH05577220,zbMATH06339545}.}

\subsubsection{Coordinates of the first kind}

Although the expansion \eqref{eq:log.xta.after} already has some interest by itself, it is not written on a basis of $\mathcal{L}(X)$, which has some drawbacks.
In this paragraph, we define canonical representations for this expansion, in appropriate bases of $\mathcal{L}(X)$.

\begin{definition}[Monomial basis] \label{Def:monomial_basis}
    Let $\basis \subset \mathcal{L}(X)$. We say that $\basis$ is a \emph{basis} of $\mathcal{L}(X)$ when each element $a \in \mathcal{L}(X)$ can be written as a unique finite linear combination of elements of $\basis$. 
    {
    We say that $\basis$ is a \emph{monomial} basis of $\mathcal{L}(X)$ when moreover $\basis \subset \eval(\Br(X))$.
    In particular, for such bases, if $b \in \basis$, one can define $|b|$, $n_i(b)$ for $i \in I$ and $n(b)$ as in \cref{s:iterated-brackets} by importing these notions from $\Br(X)$.}
    {In particular, for} $n\in\N^*$, we use the notations $\mathcal{B}_n:=\{b\in\mathcal{B};|b|=n\}$ and $\mathcal{B}_{\intset{1,n}}:=\{b\in\mathcal{B}; |b| \leq n\}$.
\end{definition}

\begin{proposition} \label{Prop:coord1.0_existence}
 Let $\mathcal{B}$ be a monomial basis of $\mathcal{L}(X)$.
 There exists a unique set of functionals $(\zeta_b)_{b \in \mathcal{B}}$, with $\zeta_b \in \CC^0\left(\R_+ \times L^1(\R_+;\K)^{|I|} ; \K\right)$, such that, for every $a_i \in L^1(\R_+;\K)$, $x^\star \in \widehat{\mathcal{A}}(X)$ and $t \geq 0$, the solution to \eqref{formalODE} satisfies
 \begin{equation} \label{eq:x.zeta.b}
  x(t) = x^\star \exp \left( \sum_{b \in \mathcal{B}} \zeta_b(t, a) b \right).
 \end{equation}
 Moreover, the functionals $\zeta_b$ are ``causal'' in the sense that, for every $t \geq 0$, $\zeta_b(t,a)$ only depends on the restrictions of the functions $a_i$ to $[0,t]$.
\end{proposition}

\begin{proof}
 For each $b \in \mathcal{B}$, since $\mathcal{B}$ is monomial, only a finite number of summands of the right-hand side of \eqref{eq:log.xta.after} have a non vanishing component along $b$ (indeed, only terms sharing the same homogeneity can be involved). 
 Hence, it is clear that the functionals thereby defined are continuous on $\R_+ \times L^1(\R_+;\K)^q$, due to their explicit expression. 
 The sum in \eqref{eq:x.zeta.b} is understood in the sense of a well-defined formal power series. Indeed, for each word $\sigma \in I^*$, only a finite number of elements $b \in \mathcal{B}$ have a non-vanishing component $\langle b, X_\sigma \rangle$.
\end{proof}

\begin{definition}[Coordinates of the first kind]
 \label{def:coord1}
 The functionals $\zeta_b$ are usually called \emph{coordinates of the first kind} associated to the (monomial) basis $\mathcal{B}$ of $\mathcal{L}(X)$.
\end{definition}

{The terminology \emph{coordinates of the first kind} or \emph{first species} and the opposition with \emph{coordinates of the second kind} (see \cref{sec:coord-second}) is classical, see e.g.\ \cite[III.4.3]{zbMATH03499968}.
See also \cref{sec:coordinates1} for references concerning the computation of such coordinates in the context of control theory.}

\begin{remark} \label{rk.homogeneous}
 Thanks to the monomial nature of the basis, one does not need to specify the full basis in order to define a given functional. For example, if $\lambda \in \N^I$ is a given homogeneity, let 
 \begin{equation}
   \Br_\lambda(X) := \{ b \in \Br(X); \enskip \forall i \in I, n_i(b) = \lambda_i \}.
 \end{equation}
 Then the coordinates of the first kind $\zeta_b$ for $b \in \mathcal{B} \cap \Br_\lambda(X)$ only depend on $\mathcal{B} \cap \Br_\lambda(X)$.
\end{remark}

\begin{remark}
 An important particular case for applications to control theory is the case $X = \{X_0,X_1\}$, with $a_0(t) = 1$ and $a_1(t) = u(t)$. This corresponds to formal scalar-input control-affine systems $\dot x(t) = x(t) (X_0 + u(t) X_1)$. One often writes $\zeta_b(t,u)$ (omitting the dependency on $a_0 \equiv 1$) to denote the coordinates of the first kind in this particular context.
\end{remark}

\subsubsection{Campbell Baker Hausdorff Dynkin formula}
\label{sec:cbhd}

{As a corollary, we obtain the classical finite CBHD formula.}

\begin{corollary} \label{Cor:CBH_formel}
    Let $X$ be a finite set, $n \in \N^*$ and $y_1,\dotsc,y_n \in \widehat{\mathcal{L}}(X)$ without constant term. 
    There exists a unique $w \in \widehat{\mathcal{L}}(X)$ such that 
 \begin{equation} \label{CBHD_formel}
 e^{y_1} \dotsm e^{y_n} = e^w.
 \end{equation}
 We will use the notation $w=\CBHD_\infty(y_1,\dotsc,y_n)$.
 Moreover, for each monomial basis $\mathcal{B}$ of $\mathcal{L}(\{Y_1,\dotsc,Y_n\})$, there exists a unique sequence $(\alpha_b)_{b\in\mathcal{B}} \subset \K^{\mathcal{B}}$ such that, for every finite set $X$ and $y_1,\dotsc,y_n \in \widehat{\mathcal{L}}(X)$ 
 \begin{equation} \label{CBH_infini_coord}
  \CBHD_\infty(y_1,\dotsc,y_n)=\sum_{b\in\mathcal{B}} \alpha_b y_b 
  \end{equation}
 where $y_b:=\Lambda(b)$ and $\Lambda:\mathcal{L}(\{Y_1,\dotsc,Y_n\}) \rightarrow \widehat{\mathcal{L}}(X)$ is the homomorphism of Lie {algebras} such that $\Lambda(Y_j)=y_j$ for $j \in \intset{1,n}$.
\end{corollary}

\begin{proof}
    We prove that \eqref{CBHD_formel} holds with
    \begin{equation}
        w: = \mathrm{Log}_\infty \left\{ \sum_{j=1}^n y_j 1_{[j-1,j]} \right\}(n)
    \end{equation}
    in the sense of \cref{def:LOGM}.

\medskip

\noindent \emph{Step 1: Proof when $X=\{X_1,\dotsc,X_n\}$ and $y_j=X_j$ for $j \in \intset{1,n}$.} 
The solution to \eqref{formalODE} with $a(t)=\sum_{j=1}^n X_j 1_{[j-1,j]}(t)$ is $x(t)=x^\star e^{X_1} \dotsm e^{X_n}$. 
By \cref{thm:formal.log}, $w$ {satisfies} (\ref{CBHD_formel}). By injectivity of the exponential (see \cref{lemma:log_is_unique}), it is the unique solution. 
By \cref{Prop:coord1.0_existence}, the equality (\ref{CBH_infini_coord}) holds with  $\alpha_b:=\zeta_b(n,1_{[0,1]},\dotsc,1_{[n-1,n]})$.

\medskip

\noindent \emph{Step 2: Proof in the general case.} 
Let $X$ be a finite set, $n \in \N^*$, $y_1,\dotsc,y_n \in \widehat{\mathcal{L}}(X)$. 
Let $Y:=\{Y_1,\dotsc,Y_n\}$ be another set of indeterminates. 

The map $\Lambda: Y \rightarrow \widehat{\mathcal{L}}(X)$ defined by $\Lambda(Y_j)=y_j$ for $j \in \intset{1,n}$ extends into an homomorphism {of algebras} $\widehat{\mathcal{A}}(Y) \rightarrow \widehat{\mathcal{A}}(X)$,
which is also an homomorphism {of Lie algebras} $\widehat{\mathcal{L}}(Y) \rightarrow \widehat{\mathcal{L}}(X)$, that we still denote~$\Lambda$. 
Indeed \cref{Lem:identification_libre} ensures the extension as an homomorphism {of algebras}  $\mathcal{A}(Y) \rightarrow \widehat{\mathcal{A}}(X)$ (resp.\ an homomorphism {of Lie algebras} $\mathcal{L}(Y) \rightarrow \widehat{\mathcal{L}}(X)$). The extension can be done on $\widehat{\mathcal{A}}(Y)$ (resp.\ $\widehat{\mathcal{L}}(Y)$) because $y_1,\dots,y_n$ do not have constant terms and the target space $\widehat{\mathcal{A}}(X)$ (resp.\ $\widehat{\mathcal{L}}(X)$) is a space of formal power series.

Let $W:=\mathrm{Log}_\infty \{ \sum_{j=1}^n Y_j 1_{[j-1,j]} \}(n) \in \widehat{\mathcal{L}}(Y)$. 
Then $\Lambda(W)=w$.
By applying the homomorphism {of algebras} $\Lambda$ to the relation 
$e^{Y_1} \dotsm e^{Y_n}=e^{W}$
we get (\ref{CBHD_formel}). 
By applying the homomorphism {of Lie algebras} $\Lambda$ to the relation $W=\sum_{b\in\mathcal{B}} \alpha_b b$
we get (\ref{CBH_infini_coord}).
\end{proof}

Despite the fact that the product $e^{y_1} \dotsb e^{y_n}$ is of course non-commutative, there is some structure and symmetry inside its logarithm, which we highlight for future use in the following result.

\begin{proposition}
    There exists a family of {elements $F_{q,h}(Y_1, \dotsc, Y_q) \in \mathcal{L}(\{ Y_1, \dotsc, Y_q \})$ for $q\in\N^*$ and $h=(h_1,\dots,h_q)\in(\N^*)^q$,} such that 
    \begin{itemize}
        \item {for each $i \in \intset{1,q}$, $F_{q,h}(Y_1, \dotsc, Y_q)$ is of homogeneity $h_i$ with respect to $Y_i$},
        \item for every $n \geq 2$, $y_1, \dotsc, y_n \in \widehat{\mathcal{L}}(X)$ {with zero constant term},
        \begin{equation} \label{eq:CBH-Fqh}
            \CBHD_{\infty}(y_1,\dotsc,y_n)=\sum_{\substack{q\in \intset{1,n}, h \in (\N^*)^q\\ j_1 < \dotsb < j_q \in \intset{1,n}}} F_{q,h}(y_{j_1},\dots,y_{j_q}),
        \end{equation}
        {where $F_{q,h}(y_{j_1}, \dotsc, y_{j_q})$ denotes the image of $F_{q,h}(Y_1, \dotsc, Y_q)$ by the homomorphism of algebras from $\mathcal{L}(\{ Y_1, \dotsc, Y_q \})$ to $\widehat{\mathcal{L}}(X)$ which sends $Y_i$ to $y_{j_i}$ for each $i \in \intset{1,q}$.}
    \end{itemize}
    For $q = 1$, $F_{1,(1)}{(Y)}={Y_1}$ and $F_{1,(h_1)} = 0$ for $h_1 \geq 2$. For $q = 2$ and $h_1 + h_2 \leq 4$,
    \begin{equation} \label{eq:CBH-2-low}
        \begin{aligned}
        F_{2,(1,1)}{(Y)} &= \frac{1}{2}{[Y_1,Y_2] }
        &  F_{2, (2,2)}{(Y)} &= -\frac{1}{24}{[Y_2,[Y_1,[Y_1, Y_2]]]} \\
        F_{2,(2,1)}{(Y)} & =\frac{1}{12}{[Y_1,[Y_1,Y_2]]} 
        & \quad \quad F_{2, (3,1)}{(Y)} &= 0 \\
        F_{2,(1,2)}{(Y)} &= \frac{1}{12}{[Y_2, [Y_2,Y_1]]}
        & F_{2,(1,3)}{(Y)} &= 0.
        \end{aligned}
    \end{equation}
    For higher order terms, {we state below a recursive formula.}
\end{proposition}

\begin{proof}
    Using the same Lie algebra homomorphism arguments as in the proof of \cref{Cor:CBH_formel}, it is sufficient to consider the case where $y_i = Y_i$ is an indeterminate.

    For $n = 2$, the statement is merely a rewriting of \eqref{CBH_infini_coord} where the terms are grouped by their homogeneity with respect to $y_1$ and $y_2$. 
    This defines the elements $F_{1,(1)}(Y_1) = Y_1$ and $F_{1,(h)}(Y_1) = 0$ for $h \geq 2$ and $F_{2,h}(Y_1,Y_2)$ for $h \in (\N^*)^2$ according to the usual two-variables formula, of which the well-known low-order terms are recalled in \eqref{eq:CBH-2-low}.
    
    We define by induction on $n \geq 3$ the functions $F_{n,h}$ by the relations
    \begin{equation} \label{eq:fqh-rec}
        F_{n,h}(Y_1, \dotsc, Y_n)
        := \sum_{m | h_1, \dotsc, h_{n-1}}
        F_{2,(m,h_n)}\big(F_{n-1, (\frac{h_1}{m}, \dotsc, \frac{h_{n-1}}{m})}(Y_1, \dotsc, Y_{n-1}), Y_n\big).
    \end{equation}
    We now prove the result by induction on $n$.  Let $n \geq 3$. By associativity of the product, the formula for two indeterminates and the induction hypothesis at step $n-1$, we obtain
    \begin{equation} \label{eq:cbh-w-yfqh-1}
        \begin{split}
            & \CBHD_\infty (Y_1, \dotsc, Y_n) \\
        & = \CBHD_\infty(\CBHD_\infty(Y_1, \dotsc, Y_{n-1}), Y_n) \\
        & = \CBHD_\infty(Y_1, \dotsc, Y_{n-1}) + Y_n + \sum_{g \in (\N^*)^2} F_{2,g}(\CBHD_\infty(Y_1, \dotsc, Y_{n-1}), Y_n) \\
        & = Y_n + \sum_{\substack{q\in \llbracket 1 , n-1 \rrbracket, h'\in (\N^*)^q\\ j_1 < \dotsb < j_q \in \llbracket 1 , n-1 \rrbracket}} \left( F_{q,h'}(Y_{j_1},\dots,Y_{j_q}) + 
        \sum_{g \in (\N^*)^2} F_{2,g}(F_{q,h'}(Y_{j_1},\dots,Y_{j_q}), Y_n)
        \right).
        \end{split}
    \end{equation}
    We now check that the right-hand side of \eqref{eq:cbh-w-yfqh-1} is the same as the right-hand side of \eqref{eq:CBH-Fqh}. 
    Since we are working in the free Lie algebra over $Y_1, \dotsc Y_n$, we can proceed by homogeneity.
    \begin{itemize}
        \item The terms not involving $Y_n$ are equal, since they have the same expression.
        \item The term involving only $Y_n$ on both sides is $Y_n$ itself, so they are equal.
        \item Now, let $q \in \intset{1,n-1}$, $j_1 < \dotsb < j_q \in \intset{1,n-1}$ and $h \in (\N^*)^{q+1}$. We look for the term involving $h_i$ times $Y_{j_i}$ for $i \in \intset{1,q}$ and $h_{q+1}$ times $Y_n$, which is $F_{q+1,h}(Y_{j_1}, \dotsc, Y_{j_q}, Y_n)$ in~\eqref{eq:CBH-Fqh}.
        In \eqref{eq:cbh-w-yfqh-1}, it is
        \begin{equation}
            \sum_{\substack{h'\in (\N^*)^q}}
            \sum_{g \in (\N^*)^2} F_{2,g}(F_{q,h'}(Y_{j_1},\dots,Y_{j_q}), Y_n),
        \end{equation}
        where the sum is restricted to $g_1 h'_i = h_i$ and $g_2 = h_{q+1}$.
        Hence, both terms are equal thanks to the definition \eqref{eq:fqh-rec}.
    \end{itemize}
    This concludes the proof and gives a way to compute the {elements} $F_{q,h}$ iteratively. 
\end{proof}

\begin{remark}
    In particular, the component of $\CBHD_{\infty}(y_1,\dots,y_n)$ homogeneous with degree $h=(h_1,\dots,h_q)$ with respect to $(y_{j_1},\dots,y_{j_q})$ is $F_{q,h}(y_{j_1},\dots,y_{j_q})$. 
    It depends neither on the total number $n$ of arguments in the initial product, nor on the selection of indexes $(j_1,\dots,j_q)$. 
    This is the natural symmetry that we wish to highlight.
\end{remark}

{Algorithms to compute iteratively the terms in the CBHD formula are investigated for instance in \cite[Section~4.a]{zbMATH01343124} or in \cite{MR2031789,MR2510918} within Hall bases, or in \cite{arnal2021note} for an expansion on right-nested brackets, which uses fewer terms.}

\subsubsection{Computation of some coordinates of the first kind}
\label{sec:coordinates1}

In this paragraph, we focus on the case $X = \{ X_0, X_1 \}$. 
Computing the coordinates of the first kind is of paramount interest for applications (see e.g.\ \cite{kawski2000calculating} where the first 14 such coordinates are computed,
and \cite{MR2510918,MR2271214} for efficient algorithms and explicit formulas obtained by an approach relying on rooted binary labeled trees).

\newcommand{\Sone}{{\dot{{S}}_{1}}}
\newcommand{\Stwo}{{{S}_{2}^+}}

Here, we calculate as an illustration (and because they will be used later) all coordinates of the first kind on a basis of
\begin{equation} \label{def:S1}
    \Sone := \vect_{\K} \left\{ \eval(b) ; \enskip b \in \Br(X), n_1(b) = 1 \right\} \subset \mathcal{L}(X),
\end{equation}
{where this notation is chosen so that $S_1 = \K X_0 \oplus \Sone$ (see \cref{def:SM}).
We define moreover} 
\begin{equation} \label{def:S2}
 \Stwo := \vect_{\K} \left\{ \eval(b) ; \enskip b \in \Br(X), n_1(b) \geq 2 \right\} \subset \mathcal{L}(X),
\end{equation}
thanks to which we can write the direct sum decomposition $\mathcal{L}(X) = \K X_0 \oplus \Sone \oplus \Stwo$.

\begin{lemma} \label{lem:S1-basis}
 The family $({\ad}_{X_0}^k (X_1))_{k\in\N}$ is a (monomial) basis of $\Sone$.
\end{lemma}

\begin{proof}
    From \eqref{def:S1}, $\Sone$ is spanned by the evaluations in $\mathcal{L}(X)$ of the {iterated brackets} $b \in \Br(X)$ involving $X_1$ exactly once. 
    Let $b \in \Br(X)$ be such \emph{an iterated  bracket}. 
    We assume $\eval(b)\neq 0$ in $\mathcal{L}(X)$ and $b \neq X_1$.
    {Then $\eval(b)=[\eval(\lambda(b)),\eval(\mu(b))]$ thus $\eval(\lambda(b))$ and $\eval(\mu(b))$ are non null in $\mathcal{L}(X)$}. 
    Moreover, either {$\lambda(b)$ or $\mu(b)$} does not involve $X_1$ and is thus equal to $X_0$. 
    Therefore $\eval(b) = \pm [X_0, \eval(\bar{b})]$ where $\bar{b} \in \Br(X)$ involves $X_1$ exactly once and $\eval(\bar{b}) \neq 0$. Working by induction on the number $k$ of occurrences of $X_0$ in $b$, we obtain $\eval(b) = \pm \ad_{X_0}^{k}(X_1)$.

    The previous argument proves that the given family spans $\Sone$. 
    Moreover, this family is linearly independent in $\mathcal{L}(X)$ because two different elements have different lengths.
\end{proof}

{
We now compute the coordinates of the first kind associated with these elements.
Up to our knowledge, the following explicit expression is new.
}

\begin{proposition} \label{prop:S1.coord1}
 Let $\mathcal{B}$ a monomial basis of $\mathcal{L}(X)$ containing $X_0$ and the family $(\ad^k_{X_0}(X_1))_{k\in\N}$. The associated coordinates of the first kind satisfy, for each $t > 0$, $a_0, a_1 \in L^1((0,t);\K)$ and $k \in \N$, 
 \begin{equation} \label{Coord_1.0_S1}
    \zeta_{\ad^{k}_{X_0}(X_1)} (t,a_0,a_1) = 
    {
    (-1)^k \sum_{\ell=0}^{k} A_0(t)^{k-\ell} \frac{B_{k-\ell}}{(k-\ell)!} \int_{\Delta^{\ell+1}(t)} a_1(\tau_1) a_0(\tau_2) \dotsb a_0(\tau_{\ell+1}) \dd\tau,}
\end{equation}
 where $A_0(t):=\int_0^t a_0$ {and the Bernoulli numbers $(B_n)_{n \in \N}$ are defined in \eqref{eq:def:bernoulli}.}
\end{proposition}

\begin{proof}
 First, the considered coordinates are well-defined independently on the exact choice of $\mathcal{B}$ (see \cref{rk.homogeneous}). Let $x$ be the solution to \eqref{formalODE} starting from $x^\star = 1$. To simplify the notations in this proof, we write $x(t)$, $\zeta_k(t)$ and $Z(t)$ instead of $x(t,a)$, $\zeta_{\ad_{X_0}^k(X_1)}(t,a_0,a_1)$ and $\mathrm{Log}_\infty\{a\}(t)$. From \eqref{eq:x.zeta.b}, 
 \begin{equation}
    Z(t) 
    = \sum_{b \in \mathcal{B}} \zeta_b(t,a) b
    = \zeta_{X_0}(t,a) X_0 + Z_1(t) + Z_2(t),
 \end{equation}
 where $Z_2(t) \in \Stwo$ and
 \begin{equation}
    Z_1(t) := \sum_{k=0}^{+\infty} \zeta_{k}(t) \ad_{X_0}^k(X_1).
 \end{equation}
    First, a straightforward identification in $\eqref{eq:LOGM}$ yields $\zeta_{X_0} = A_0$ and {$\zeta_{X_1}(t) = \int_0^t a_1$}. 
    Let $k\in\N^*$.
 The proof consists in computing $\langle x(t) , X_1 X_0^k \rangle$ in two ways: first by the differential equation \eqref{formalODE}, then by the formula $x(t)=e^{Z(t)}$. By definition of the solution to \eqref{formalODE}, we have, for every word $\sigma \in I^*$ and $t>0$
 \begin{equation}
 \langle x(t) , X_\sigma X_0 \rangle = \int_0^t  \langle x(\tau) , X_\sigma \rangle a_0(\tau) \dd\tau.
 \end{equation}
 Taking into account that {$\langle x(t) , X_1 \rangle = \int_0^t a_1$}, we obtain
 \begin{equation} \label{int3}
  \langle x(t) , X_1 X_0^k \rangle =
    {\int_{\Delta^{k+1}(t)} a_1(\tau_1) a_0(\tau_2) \dotsm a_0(\tau_{k+1}) \dd\tau.}
 \end{equation}
  On the other hand, we deduce from the expansion of $x(t)=e^{Z(t)}$ that
 \begin{equation} \label{int1}
 \left\langle x(t) , X_1 X_0^k \right\rangle = \left\langle Z(t) , X_1 X_0^k \right\rangle + \sum_{\ell=2}^{k+1} \frac{1}{\ell !} \left\langle Z(t)^\ell , X_1 X_0^k \right\rangle 
 \end{equation}
 because, for $\ell \geq (k+2)$, $Z(t)^\ell$ is a sum of words with length at least $(k+2)$.
 For $\ell \in \llbracket 2 , k+1 \rrbracket$,
 \begin{equation}
  Z(t)^\ell = \sum_{j=0}^{\ell-1} (A_0(t)X_0)^j Z_1(t) (A_0(t)X_0)^{\ell-1-j} + Z_{2,\ell}(t), \quad \text{ where } \quad  Z_{2,\ell}(t) \in \Stwo.
 \end{equation}
 Thus
 \begin{equation} \label{int2}
 \left\langle Z(t)^\ell , X_1 X_0^k \right\rangle =  \left\langle Z_1(t) (A_0(t)X_0)^{\ell-1} , X_1 X_0^k \right\rangle = A_0(t)^{\ell-1} (-1)^{k-\ell+1} \zeta_{k-\ell+1}(t), 
 \end{equation}
 because the word $X_1 X_0^{k-\ell+1}$ appears in the decomposition of $\ad_{X_0}^n(X_1)$ iff $k-\ell+1=n$ and then it appears with coefficient $(-1)^n$. We deduce from \eqref{int1} and \eqref{int2} that
 \begin{equation}
  \left\langle x(t) , X_1 X_0^k \right\rangle= (-1)^k \zeta_k(t)+\sum_{\ell=2}^{k+1} \frac{(-1)^{k+1-\ell}}{\ell !} A_0(t)^{\ell-1} \zeta_{k+1-\ell}(t).
 \end{equation}
 Using \eqref{int3} and the index change $j = k+1-\ell \in \llbracket 0,k-1 \rrbracket$, we obtain
 \begin{equation} \label{Coord_1.0_S1_implicit}
   {\int_{\Delta^{k+1}(t)} a_1(\tau_1) a_0(\tau_2) \dotsm a_0(\tau_{k+1}) \dd\tau}
   = (-1)^k \zeta_k(t) + \sum_{j=0}^{k-1} \frac{(-1)^j A_0(t)^{k-j}}{(k+1-j)!} \zeta_j(t),
\end{equation} 
 When $A_0(t) = 0$, this formula yields \eqref{Coord_1.0_S1} immediately. When $A_0(t) \neq 0$, let, for $j\in\N$,
 \begin{equation} \label{alphabeta}
  \alpha_j := \frac{\langle x(t) , X_1 X_0^j \rangle}{A_0(t)^{j+1}}
  \qquad \text{and} \qquad \beta_j :=  \frac{(-1)^j\zeta_j(t)}{A_0(t)^{j+1}}
  \end{equation}
we deduce from (\ref{Coord_1.0_S1_implicit})  that
\begin{equation}
  \alpha_k = \sum_{j=0}^k \frac{\beta_j}{(k+1-j)!}.
\end{equation}
We have
\begin{equation}
  z \left( \sum_{k \geq 0} \alpha_k z^k \right) 
  = \sum_{k \geq 0} \sum_{j=0}^k \beta_j z^j \frac{z^{k+1-j}}{(k+1-j)!}
  = \left( \sum_{j \geq 0} \beta_j z^j \right) (e^z-1)
\end{equation}
or equivalently
\begin{equation} \sum_{j \geq 0} \beta_j z^j =\frac{z}{e^z -1}\left( \sum_{k \geq 0} \alpha_k z^k \right) = \sum_{n \geq 0} \sum_{k \geq 0} B_n \frac{z^n}{n!} \alpha_k z^k.
\end{equation}
Thus, for every $j \in \N^*$
\begin{equation} \beta_j= \sum_{k=0}^j \frac{B_{j-k}}{(j-k)!} \alpha_{k}.\end{equation}
Finally (\ref{alphabeta}) and (\ref{int3}) give \eqref{Coord_1.0_S1}.
\end{proof}

{
\begin{remark}
    Formula \eqref{Coord_1.0_S1} bears a strong similarity with the differential equation \eqref{eq:EDO-Z} satisfied by $z(t)$, which also involves the Bernoulli numbers.
    Unfortunately, we have not been able to obtain a shorter proof using this equation.
\end{remark}
}

In particular, {using \cref{prop:S1.coord1},} we recover the following very classical formula for the partial coefficients of the CBHD formula (see e.g.\ \cite[equation~(2)]{MR0156878} or \cite[Corollary~3.24]{MR1231799}).

\begin{corollary} \label{thm:cbh.Z1}
    {One has} $e^{X_1} e^{X_0} = e^{Z}$ where $Z=X_0 + Z_1 + Z_2$, $Z_2 \in \Stwo$ {(see \eqref{def:S2})} and
  \begin{equation} \label{eq:cbh.Z1}
    Z_1 
    := \sum_{n=0}^{+\infty} \frac{B_n}{n!} \ad_{X_0}^n(X_1) 
    = X_1 - \frac{1}{2}[X_0,X_1] +\sum_{n=1}^{+\infty} \frac{B_{2n}}{(2n)!} \ad_{X_0}^{2n}(X_1).  
  \end{equation}
\end{corollary}

\begin{proof}
 We apply the previous result to the controls $a_0(t)=\mathbf{1}_{(1,2)}(t)$ and $a_1(t)=\mathbf{1}_{(0,1)}(t)$, for which the solution to \eqref{formalODE} with $x^\star = 1$ satisfies $x(2)=e^{X_1} e^{X_0}$.
 For $\ell \in \N^*$ and $0<{\tau_1<\dotsb<\tau_{\ell+1}}<2$, the real number {$a_1(\tau_1) a_0(\tau_1) \dotsm a_0(\tau_{\ell+1})$} does not vanish iff {$0 < \tau_1 < 1$ and $1 < \tau_2 < \dotsb < \tau_{\ell+1} < 2$} and then it equals $1$. 
 Thus, for every $k \geq 2$, using \eqref{Coord_1.0_S1} and \eqref{eq:bernoulli.1},
 \begin{equation}
  (-1)^k \zeta_k(2)=\sum_{\ell=0}^k  \frac{B_{k-\ell}}{(k-\ell)!} \frac{1}{\ell!} = \sum_{j=0}^k  \frac{B_{j}}{j ! (k-j)!} = \frac{B_k}{k!} \end{equation}
 We conclude by noticing, thanks to \eqref{Coord_1.0_S1}, that $\zeta_0(2)=1=B_0$ and $\zeta_1(2)=-\frac{1}{2}=B_1$.
\end{proof}

\begin{example} \label{ex:zeta.adx0x1.1t}
    As an example and for later use in the sequel, we compute the coordinates of the first kind for the particular choice $a_0(t) := 1$ and $a_1(t) := t$. 
    Let $k \in \N$. 
    Using formula \eqref{Coord_1.0_S1} of \cref{prop:S1.coord1} and {the identity \eqref{eq:bernoulli.4}} we obtain
 \begin{equation} \label{eq:zeta.adx0x1.1t}
  \begin{split}
   \zeta_{\ad^k_{X_0}(X_1)}(t,a) 
   & = (-1)^k \sum_{\ell = {0}}^k t^{k-\ell} \frac{B_{k-\ell}}{(k-\ell)!} \frac{t^{\ell + 2}}{(\ell + 2)!} \\
  & = (-1)^k t^{k+2} \sum_{\ell=0}^k \frac{B_{k-\ell}}{(k-\ell)!(\ell+2)!} \\
  & = (-1)^{k+1} t^{k+2} \frac{B_{k+1}}{(k+1)!}.
  \end{split}
 \end{equation}
\end{example}

\subsection{Interaction picture, coordinates of the pseudo-first kind}
\label{sec:format.1.1}

In quantum mechanics, the \emph{interaction picture} is an intermediate representation between the \emph{Schrödinger picture} (in which the state vectors are time-dependent and the operators are time-independent) and the \emph{Heisenberg picture} (in which the state vectors are time-independent and the operators are time-dependent).
{The interaction picture} is particularly useful when the dynamics can be written as the sum of a time-independent part, which can be solved exactly, and a time-dependent perturbation. In this section, we introduce and study a formal counterpart of this situation, that can be useful for applications.

\subsubsection{A new formal expansion}

In this paragraph, we therefore consider $I = \intset{0,q}$ to isolate the role of $X_0$. For some given $a_i \in L^1(\R_+;\K)$ for $i \in \intset{1,q}$, we assume that $a$ takes the form
\begin{equation} \label{eq:a.inter}
  a(t) = X_0 + \sum_{i=1}^q a_i(t) X_i.
\end{equation}

\begin{theorem} \label{thm:Magnus1.0_formel}
 For $t \in \R_+$, $x^\star \in \widehat{\mathcal{A}}(X)$ and $a$ of the form \eqref{eq:a.inter}, the solution $x$ to \eqref{formalODE} satisfies
 \begin{equation} \label{eq:11formal}
     x(t) = x^\star \exp (t X_0) \exp \left( \mathcal{Z}_\infty(t,X,a) \right),
 \end{equation}
 where $\mathcal{Z}_\infty(t,X,a):=\mathrm{Log}_\infty \{ b_t \} (t)$
 with the notation of \cref{def:LOGM} and 
 \begin{equation} \label{eq:bts}
  b_t(s) 
  := e^{-(t-s)X_0} \left(\sum_{i=1}^q a_i(s) X_i\right) e^{(t-s) X_0} 
  = \sum_{i=1}^q \sum_{k=0}^{+\infty} \frac{(-1)^k}{k!} (t-s)^k a_i(s) \ad^k_{X_0}(X_i)
 \end{equation}
i.e.\
\begin{equation} \label{Loginfini(bt)}
\begin{aligned}
\mathcal{Z}_\infty(t,X,a) = \sum \frac{(-1)^{m-1}}{mr}
 \int_{{\Delta^{\mathbf{r}}(t)}} 
  \frac{(\tau_1-t)^{k_1}}{k_1!} & \dotsm \frac{(\tau_r-t)^{k_r}}{k_r!} 
  a_{i_1}(\tau_1) \dotsm a_{i_r}(\tau_r) \dd\tau 
  \\ & 
  {
  [ \dotsb [ \ad_{X_0}^{k_1}(X_{i_1}), \ad_{X_0}^{k_2}(X_{i_2}) ], \dotsc  \ad_{X_0}^{k_r}(X_{i_r}) ],}
\end{aligned}
\end{equation}
where the sum is taken over $r \in \llbracket 1, \infty \rrbracket$, $m \in \intset{1,r}$, $\mathbf{r} \in \N^m_r$, $k_1,\dotsc,k_r \in \N$ and $i_1,\dotsc,i_r \in \intset{1,q}$.
\end{theorem}

\begin{proof}
    {First, note that the second equality in \eqref{eq:bts} stems from the fact that both functions $g_1(\tau) := e^{-\tau X_0} X_i e^{\tau X_0}$ and $g_2(\tau) := \sum_{k=0}^{+\infty} \frac{(-\tau)^k}{k!} \ad_{X_0}^k(X_i)$ solve the Cauchy problem $\dot{g}(\tau)=[g(\tau),X_0]$ and $g(0)=X_i$, so they are equal.}

 Let $t > 0$.
 A key point is to remark that all the definitions and results from the previous paragraphs which are stated for a finite set $I$ of indeterminates are still valid if $I$ is an infinite set. For mathematicians with a background in analysis, all equalities can be understood ``in the weak sense'' as equalities holding along each monomial. Therefore, for a set of {indeterminates} $\{ Y_{k,i} \}_{k\in\N,i\in\intset{1,q}}$, the solution to
 \begin{equation}
   \dot{z}(s) = z(s) \gamma_t(s) 
   \quad \text{where} \quad 
   \gamma_t(s) := \sum_{k,i} \frac{(-1)^k}{k!} (t-s)^k a_i(s) Y_{k,i},
 \end{equation}
 with initial data $z(0) = 1$ satisfies, thanks to \cref{thm:formal.log}, 
 \begin{equation}
   z(t) = \exp \left( \mathrm{Log}_\infty \{ \gamma_t \}(t) \right).
 \end{equation}
 Let $\Theta$ be the unique homomorphism {of algebras} from $\widehat{\mathcal{A}}(\{ Y_{k,i} \}_{k\in\N,i\in\intset{1,q}})$ to $\widehat{\mathcal{A}}(X)$ defined by
 \begin{equation}
   \Theta(Y_{k,i}) := \ad^k_{X_0}(X_i).
 \end{equation}
 Then $z_\Theta(s) = \Theta(z(s))$ satisfies on the one hand $z_\Theta(0) = 1$ and $\dot{z}_\Theta(s) = z_\Theta(s) b_t(s)$, and on the other hand $z_\Theta(t) = \exp \left( \mathrm{Log}_\infty \{ b_t \}(t) \right)$.
 
 We introduce the change of {variables} $y(s) := x(s) e^{(t-s)X_0}$. Then,
 \begin{equation}
  \dot{y}(s) = \dot{x}(s) e^{(t-s)X_0} - x(s) X_0 e^{(t-s)X_0}
  = x(s) \left( \sum_{i=1}^q a_i(s) X_i \right) e^{(t-s)X_0}
  = y(s) b_t(s).
 \end{equation}
 Hence 
 \begin{equation}
   x(t) = y(t) = y(0) z_\Theta(t)
   = x^\star e^{tX_0} \exp \left( \mathrm{Log}_\infty \{ b_t \}(t) \right),
 \end{equation}
 which concludes the proof of \eqref{eq:11formal}.
\end{proof}

{\begin{remark}
    In the above proof, $\mathcal{Z}_\infty(t,X,a)$ is defined by the logarithm of the product of two flows: the one associated with $-X_0$ and the one associated with $a(t)$. 
    It is a particular case of the construction of the chronological logarithm of the product of two flows associated with two non-autonomous vector fields, see \cite[Section~2.2]{MR579930} or \cite[p.~92]{kawski2011chronological}.
\end{remark}}

\begin{remark}
    In expansion \eqref{eq:11formal}, the choice to write $\exp (tX_0)$ to the left of the formal logarithm is arbitrary. 
    One could obtain a similar formula with $\exp (tX_0)$ to the right. 
    Depending on the application one has in mind, both choices can be helpful. 
\end{remark}

\subsubsection{Coordinates of the pseudo-first kind}

\begin{proposition} \label{Prop:coord1.1_existence}
 Let $q \in \N^*$, 
 $X = \{ X_0, X_1, \dotsc, X_q \}$ and
 $\mathcal{B}$ be a monomial basis of $\mathcal{L}(X)$.
 There exists a unique set of functionals $(\eta_b)_{b \in \mathcal{B}}$, with $\eta_b \in \CC^0\left(\R_+ \times L^1(\R_+;\K)^q ; \K\right)$, such that, 
 for every $a_i \in L^1(\R_+;\K)$ and $t \geq 0$
 \begin{equation} \label{tilde_Z_infty_decomp_coord1.1}
     \mathcal{Z}_\infty (t,X,a) = \sum_{b \in \mathcal{B}} \eta_b(t, a) b \quad \text{ in } \quad \widehat{\mathcal{L}}(X).
 \end{equation}
 Moreover, $\eta_{X_0}=0$ and the functionals $\eta_b$ are ``causal'' in the sense that, for every $t \geq 0$, $\eta_b(t,a)$ only depends on the restrictions of the functions $a_i$ to $[0,t]$.
\end{proposition}

\begin{proof}
For every $r\in \N^*$ and $\nu \in \N$ we introduce the finite sum of brackets \begin{equation} \label{def:Irnu}
\begin{aligned}
\mathcal{Z}_\infty^{r,\nu}(t,X,a)
=\sum \frac{(-1)^{m-1}}{mr}
 \int_{{\Delta^{\mathbf{r}}(t)}} 
  \frac{(\tau_1-t)^{k_1}}{k_1!} & \dotsm \frac{(\tau_r-t)^{k_r}}{k_r!} 
  a_{i_1}(\tau_1) \dotsm a_{i_r}(\tau_r) \dd\tau 
  \\ & 
  {
  [ \dotsb [ \ad_{X_0}^{k_1}(X_{i_1}), \ad_{X_0}^{k_2}(X_{i_2}) ], \dotsc  \ad_{X_0}^{k_r}(X_{i_r}) ],}
\end{aligned}
\end{equation}
where the sum is taken over $m \in \intset{1,r}$, $\mathbf{r} \in \N^m_r$, $k_1,\dotsc,k_r \in \N$ such that $k_1+\dotsb+k_r=\nu$ and $i_1,\dotsc,i_r \in \intset{1,q}$. 
For each term in this sum, the bracket 
 \begin{equation}
    {
 [ \dotsb [ \ad_{X_0}^{k_1}(X_{i_1}), \ad_{X_0}^{k_2}(X_{i_2}) ], \dotsc  \ad_{X_0}^{k_r}(X_{i_r}) ]}
 \end{equation} 
 has a unique expansion on the basis
$\mathcal{B}_{r,\nu}=\{b\in\mathcal{B}; \enskip n(b)=r \textrm{ and } n_0(b)=\nu \}$. By summing these expansions we obtain causal functions $(\eta_b)_{b\in\mathcal{B}_{r,\nu}}$ in $\CC^0\left(\R_+ \times L^1(\R_+;\K)^q ; \K\right)$ such that the following equality holds in  $\mathcal{L}(X)$
\begin{equation} \label{In=sum_coord_pseudo}
\mathcal{Z}_\infty^{r,\nu}(t,X,a)
=\sum_{b\in\mathcal{B}_{r,\nu}} \eta_b(t,a) b.
\end{equation}
By summing these relations, we get (\ref{tilde_Z_infty_decomp_coord1.1}).
\end{proof}

\begin{definition}[Coordinates of the pseudo-first kind]
 \label{def:coord1.1}
 We call the functionals $\eta_b$ \emph{coordinates of the pseudo-first kind} associated to the (monomial) basis $\mathcal{B}$ of $\mathcal{L}(X)$, by analogy with coordinates of the first kind.
\end{definition}

\subsubsection{Structure constants and estimates for the coordinates}
\label{sec:structure-constants}

At the formal level, series such as \eqref{tilde_Z_infty_decomp_coord1.1} make sense. 
However, in the sequel, we will need to give a meaning to such series where the indeterminates are replaced by true objects.
To make sure that the resulting series converge, it will be necessary to have estimates on the coordinates of the pseudo-first kind. 
In this paragraph, we suggest a criterion based on the structure constants of~$\mathcal{L}(X)$ relative to the underlying monomial basis to obtain such estimates.

\begin{definition}[Structure constants]
    Let $\mathcal{B}$ be a basis of $\mathcal{L}(X)$. For every $a, b \in \mathcal{B}$, since $[a, b] \in \mathcal{L}(X)$, it can be written as a finite linear combination of basis elements, say
    \begin{equation}
        [a, b] = \sum_{c \in \mathcal{B}} \gamma^c_{a,b} c,
    \end{equation}
    where the coefficients $\gamma^c_{a,b} \in \K$ and only a finite number of them are non-zero. The set of these coefficients are called the structure constants of $\mathcal{L}(X)$ relative to the basis $\mathcal{B}$. 
\end{definition}

\begin{definition}[Geometric growth]
    Let $X$ be a finite set and $\mathcal{B}$ be a monomial basis of $\mathcal{L}(X)$.
    We say that $\mathcal{B}$ has \emph{geometric growth} when there exists $\Gamma \geq 1$ such that, for every $b_1, b_2 \in \mathcal{B}$,
    \begin{equation}
        \sum_{c \in  \mathcal{B}} |\gamma_{b_1,b_2}^c| 
        \leq \Gamma^{|b_1|+|b_2|}.
    \end{equation}
\end{definition}

\begin{definition}[Asymmetric geometric growth]
    \label{def:asym-geom-growth}
    Let $q \in \N^*$, $X = \{ X_0, X_1, \dotsc, X_q \}$ and $\mathcal{B}$ be a monomial basis of $\mathcal{L}(X)$.
    We say that $\mathcal{B}$ has \emph{geometric growth with respect to $X_0$} when, for every $k \in \N$, there exists $\Gamma(k) \geq 1$ such that, for every $b_1, b_2 \in \mathcal{B}$ with $n(b_1) + n(b_2) \leq k$,
    \begin{equation} \label{eq:growth.asym}
        \sum_{c \in  \mathcal{B}} |\gamma_{b_1,b_2}^c| 
        \leq \Gamma(k)^{|b_1|+|b_2|}.
    \end{equation}
\end{definition}

Asymmetric geometric growth is a weaker notion than geometric growth (which can be seen as asymmetric geometric growth with a constant $\Gamma$ {independent of} $k$).
These definitions therefore lead to the following algebraic open problem:.

\begin{open}
    Which monomial bases $\mathcal{B}$ of $\mathcal{L}(X)$ have (asymmetric) geometric growth?
\end{open}

{
\begin{remark} \label{rk:A1}
    A family of examples of monomial bases of $\mathcal{L}(X)$ is given by Hall bases (see \cref{sub:lazard}, and in particular \cref{def:ghb}). 
    In the paper \cite{A1} dedicated to studying the growth of structure constants for Hall bases of $\mathcal{L}(X)$, we provide examples of Hall bases of $\mathcal{L}(X)$ that have geometric growth (in particular, the classical examples of length-compatible Hall bases and the Lyndon basis have geometric growth, see Theorem 1.5 and Theorem 1.6 in the cited work).
    More importantly, we show that every Hall basis has asymmetric geometric growth (see Theorem 1.9 in the cited work).
\end{remark}
}

For such bases, we can prove nice estimates for the coordinates of the pseudo-first kind. We start with an estimate concerning the decomposition of the Lie brackets involved in \eqref{def:Irnu}.

{
\begin{lemma} \label{thm:b1br.b}
    Let $q\in\N^*$, $X=\{X_0,X_1,\dotsc,X_q\}$, $\mathcal{B}$ be a monomial basis of $\mathcal{L}(X)$ with geometric growth with respect to $X_0$.
    For every $r \geq 1$, there exists $C(r) \geq 1$ such that, for every $\ell \geq 2$, $b_1,\dots,b_{\ell} \in \mathcal{B} \setminus\{X_0\}$ with $n(b_1)+\dots+n(b_{\ell})\leq r$ and $b \in \mathcal{B}$,
    \begin{equation}
        \left|
        \left\langle [ \dotsb [ b_1, b_2 ], \dotsc  b_{\ell} ] , b \right\rangle_{\mathcal{B}}
        \right| 
        \leq C(r)^{|b|},
    \end{equation}
    where the bra-ket denotes the component of the Lie bracket along $b$ in its decomposition on $\mathcal{B}$.
\end{lemma}
\begin{proof}
    Any $a \in \mathcal{L}(X)$ can be written as a linear combination of basis elements, say $a=\sum_{c\in\mathcal{B}} \alpha_c^a c$, where the coefficients $\alpha_c^a \in \mathbb{K}$ and only a finite number of them are non-zero. We endow $\mathcal{L}(X)$ with the norm $\|a\|_{\mathcal{B}}:=\sum_{c\in\mathcal{B}}|\alpha_c^a|$. Then, \cref{def:asym-geom-growth} gives, for every $b_1, b_2 \in \mathcal{B}$, $\|[b_1,b_2]\|_{\mathcal{B}} \leq \Gamma(n(b_1)+n(b_2))^{|b_1|+|b_2|}$.
    We prove by induction on $\ell \geq 2$ that, for every $b_1,\dots ,b_{\ell} \in \mathcal{B} \setminus\{X_0\}$, 
    \begin{equation}
        \| [ \dotsb [ b_1, b_2 ], \dotsc  b_{\ell} ] \|_{\mathcal{B}} \leq \Gamma\left(n(b_1)+\dots+n(b_{\ell})\right)^{(\ell-1)(|b_1|+\dots+|b_\ell|)}
    \end{equation}
    which implies \cref{thm:b1br.b} with $C(r)=\Gamma(r)^{r-1}$.
    The result for $\ell=2$ is already known. 
    Let $\ell \geq 2$ and $b_1,\dots,b_{\ell+1}\in\mathcal{B} \setminus\{X_0\}$. Then
    $[ \dotsb [ b_1, b_2 ], \dotsc  b_{\ell} ] = \sum_{d\in\mathcal{B}} \alpha_d d $ 
    where the sum is finite and $\sum_{d\in\mathcal{B}}|\alpha_d| \leq \Gamma(n(b_1)+\dots+n(b_{\ell}))^{(\ell-1)(|b_1|+\dots+|b_{\ell}|)}$. Then
    \begin{equation}
        [[ \dotsb [ b_1, b_2 ], \dotsc  b_{\ell}],b_{\ell+1}] = \sum_{d\in\mathcal{B}} \alpha_d [d,b_{\ell+1}]=\sum_{d\in\mathcal{B}} \alpha_d \sum_{c\in\mathcal{B}} \gamma_{d,b_{\ell+1}}^c c =\sum_{c\in\mathcal{B}} \left(  \sum_{d\in\mathcal{B}} \alpha_d\gamma_{d,b_{\ell+1}}^c\right) c,
    \end{equation}
    where the sums are finite and indexed by $d\in\mathcal{B}$ such that $n(d)+n(b_{\ell+1})=n(b_1)+\dots+n(b_{\ell+1})$ and $|d|+|b_{\ell+1}|=|b_1|+\dots+|b_{\ell+1}|$
    thus
    \begin{equation}
        \begin{split}
     \| [[ \dotsb & [ b_1, b_2 ], \dotsc b_{\ell}],b_{\ell+1}] \|_{\mathcal{B}} = \sum_{c\in\mathcal{B}} \left|  \sum_{d\in\mathcal{B}} \alpha_d\gamma_{d,b_{\ell+1}}^c\right|  \leq  \sum_{d\in\mathcal{B}} | \alpha_d | \sum_{c\in\mathcal{B}} |\gamma_{d,b_{\ell+1}}^c| \\ & \leq \Gamma(n(b_1)+\dots+n(b_{\ell}))^{(\ell-1)(|b_1|+\dots+|b_{\ell}|)} \Gamma(n(b_1)+\dots+n(b_{\ell+1}))^{|b_1|+\dots+|b_{\ell+1}|}
    \end{split}    
    \end{equation}
    which gives the conclusion.
\end{proof}
}

\begin{proposition} \label{thm:etab-geom}
    Let $q\in\N^*$, $X=\{X_0,X_1,\dotsc,X_q\}$, $\mathcal{B}$ be a monomial basis of $\mathcal{L}(X)$ with geometric growth with respect to $X_0$.
    Then, for every $M\in\N^*$, there exists $C_M>0$ such that, for every $T\geq 0$, $u\in L^1((0,T);\K^q)$, $b\in\mathcal{B}$ with $n(b)\leq M$ and $t \in [0,T]$,
    \begin{equation} \label{eta_b|b|!<geom}
        |\eta_b(t,u)| \leq \frac{C_M^{|b|}}{|b|!} t^{n_0(b)} \|u\|_{L^1(0,t)}^{n(b)}.
    \end{equation}
\end{proposition}

\begin{proof}
    We may assume that $(C(r))_{r\in\N^*}$ given by \cref{thm:b1br.b} is non-decreasing. Then, for every $i_1, \dotsc, i_r \in \intset{1,q}$ and $k_1, \dotsc, k_r \in \N$, for every $b \in \mathcal{B}$,
    \begin{equation} \label{bra-ket:adX0kXi}
        \left|
        \left\langle {[ \dotsb [ \ad_{X_0}^{k_1}(X_{i_1}), \ad_{X_0}^{k_2}(X_{i_2}) ], \dotsc  \ad_{X_0}^{k_r}(X_{i_r}) ]} , b \right\rangle_{\mathcal{B}}
        \right| 
        \leq C(r)^{|b|}.
    \end{equation}
    {
    Indeed, for each $j \in \intset{1,r}$, there exists $b_j \in \mathcal{B}$ with $n(b_j) = 1$ and $|b_j| = k_j + 1$ such that $\ad_{X_0}^{k_j}(X_{i_j}) = \pm b_j$ in $\mathcal{L}(X)$. 
    Indeed, the homogeneous part of $\mathcal{L}(X)$ containing $k_j$ times $X_0$ and $X_{i_j}$ once is of dimension one. 
    }
    
Let $M\in\N^*$ and $b\in\mathcal{B}$ be such that $n(b)\leq M$.
We deduce from (\ref{Loginfini(bt)}) that
\begin{equation}
\begin{split} \label{sum_etab}
\eta_b(t,u)=
\sum \frac{(-1)^{m-1}}{mr}
 \int_{{\Delta^{\mathbf{r}}(t)}} &
  \frac{(\tau_1-t)^{k_1}}{k_1!} \dotsm \frac{(\tau_r-t)^{k_r}}{k_r!} u_{i_1}(\tau_1) \dotsm u_{i_r}(\tau_r) \dd\tau 
  \\ & 
  \left\langle {[ \dotsb [ \ad_{X_0}^{k_1}(X_{i_1}), \ad_{X_0}^{k_2}(X_{i_2}) ], \dotsc  \ad_{X_0}^{k_r}(X_{i_r}) ]} , b \right\rangle
\end{split}
\end{equation}
where the sum is taken over $r \in \intset{1, \infty}$, $m \in \intset{1,r}$, $\mathbf{r} \in \N^m_r$, $k_1,\dotsc,k_r \in \N$ and $i_1,\dotsc,i_r \in \intset{1,q}$. 
If the summand bra-ket in (\ref{sum_etab}) does not vanish, then $r=n(b)$ and $k_1+\dotsb+k_r=n_0(b)$. Thus the sum in (\ref{sum_etab}) is taken over the finite set $r=n(b)$, 
$m \in \llbracket 1, n(b) \rrbracket$, 
$k_1,\dotsc,k_r \in \N$ such that $k_1+\dotsb+k_r=n_0(b)$
and $i_1,\dotsc, i_r \in \intset{1,q}$,  whose cardinal is bounded by $M  2^{|b|} q^{M}$. Moreover, for every $r,m,k_1,\dotsc,k_r,i_1,\dotsc,i_r$ in this set, the associated term in 
(\ref{sum_etab}) is bounded, thanks to (\ref{bra-ket:adX0kXi}), by
 \begin{equation} \frac{t^{k_1}}{k_1!} \dotsm \frac{t^{k_r}}{k_r!} \|u\|_{L^1(0,t)}^r C(r)^{|b|} \leq t^{n_0(b)} \|u\|_{L^1}^r \left(2^{r} C(r)\right)^{|b|} \frac{n(b)!}{|b|!}\end{equation}
thanks to (\ref{factorielle_multiple}). Thus
 \begin{equation}|\eta_b(t,u)|  \leq \frac{1}{|b|!} M! M  q^{M} \left(2^{M+1} C(M)\right)^{|b|} t^{n_0(b)} \|u\|_{L^1}^{n(b)}\end{equation}
which gives the conclusion with, for instance,
$C_M := M! M  q^{M} 2^{M+1} C(M)$.
\end{proof}

\subsection{Infinite product, coordinates of the second kind}
\label{Sec:Lazard}

In this section, we present an expansion for the formal power series $x(t)$ solution to (\ref{formalODE})
as a product of exponentials of the members of a \GHB\ of $\mathcal{L}(X)$, multiplied by coefficients that have simple expressions as iterated integrals, called \emph{coordinates of the second kind}. 
This infinite product is an extension, suggested in \cite{edd33d15c4b947f39991805f8c1d726f}, of Sussmann's infinite product on length-compatible Hall bases \cite{MR935387} to all \GHBS\ {(understood in the generalized sense of Viennot \cite[Theorem~1.2]{MR516004} or Shirshov \cite[Definition~1]{zbMATH03234480})}.

\subsubsection{Lazard sets, Hall sets and \GHBS}\label{sub:lazard}

{We start by defining Lazard sets and Hall sets, which are two equivalent notions, as proved by Viennot in \cite[Corollary 1.1]{MR516004}.
They lead to the essential notion of Hall bases (see \cref{def:ghb}).} 

\begin{definition}[Lazard set] \label{Def:Laz}
    A \emph{Lazard set} is a subset $\bset$ of $\Br(X)$, totally ordered by a relation $<$ and such that, for every $M \in\N^*$, the set $\bset_{\intset{1,M}}$ of elements of $\bset$ with length at most $M$, {labeled as} $\bset_{\intset{1,M}}=\{b_1, \dotsc , b_{k+1}\}$ with $k\in\N$ and $b_1 < \dotsb < b_{k+1}$ satisfies
    \begin{equation} \label{Laz.1} 
    \left\lbrace\begin{aligned}
& b_1 \in Y_0:=X, \\
& b_2 \in Y_1:=\{ \ad_{b_1}^j(v) ; j \in \N, v \in Y_0\setminus \{b_1\} \}, \\
& \dots \\
& b_{k+1} \in Y_k:=\{ \ad_{b_k}^j(v) ; j \in \N, v \in Y_{k-1} \setminus\{b_k\} \}
    \end{aligned}\right.
    \end{equation}
    and
    \begin{equation} \label{Laz.2}
    \bset_{\intset{1,M}} \cap Y_k =\{b_{k+1}\},
    \end{equation}
    where condition (\ref{Laz.2}) can equivalently be written
    \begin{equation} \label{Laz.3}
    \bset_{\intset{1,M}} \cap Y_{k+1} =\emptyset,
    \end{equation}
    where $Y_{k+1}:=\{ \ad_{b_{k+1}}^j(v) ; j \in \N, v \in Y_{k} \setminus\{b_{k+1}\} \}$.
\end{definition}

The elements $\ad_{b_\ell}^j(v)$ for $\ell\in\{0,\dotsc,k+1\}$, $j\in\N$ and $v\in Y_{\ell-1}\setminus\{b_\ell\}$ are all different in $\Br(X)$ (identify their left and right factors iteratively) and all belong to $\bset$.

\begin{definition}[Hall set] \label{Def2:Laz}
    A \emph{Hall set} is a subset $\bset$ of $\Br(X)$, totally ordered by a relation $<$ and such that
\begin{itemize}
    \item $X \subset \bset$,
    \item for $b = {(b_1,b_2)} \in \Br(X)$, $b \in \bset$ iff $b_1, b_2 \in \bset$, $b_1 < b_2$ and either $b_2 \in X$ or $\lambda(b_2) \leq b_1$, 
    \item for every $b_1, b_2 \in \bset$ such that ${(b_1,b_2)} \in \bset$ then $b_1 < {(b_1,b_2)}$.
\end{itemize}
\end{definition}

When $b={(b_1,(b_3,b_4))} \in \bset$ then $b_1$ is ``sandwiched'' in between $b_3$ and $b$, since $b_3 \leq b_1 < b$.

\begin{remark}
    \label{rk:hall-construction}
    A Hall set can be built by induction on the length. One starts with the set $X$ as well as an order on it. To find all Hall monomials with length $n$ given those of smaller length, one adds first all ${(b_1, b_2)}$ with $b_1 \in \bset$, $|b_1|=n-1$, $b_2 \in X$ and $b_1 < b_2 $. Then for each bracket $b_2 = {(b_2', b_2'')} \in \bset$ with length $|b_2| < n$ one adds all the ${(b_1, b_2)}$ with $b_1 \in \bset$ with $|b_1|=n-|b_2|$ and $b_2' \leq b_1 < b_2$. Finally, one inserts the newly generated monomials of degree $n$ into an ordering, maintaining the condition that $b_1 < {(b_1, b_2)}$. 
\end{remark}

Viennot proves in \cite[Corollary 1.1]{MR516004} that a subset $\bset$ of $\Br(X)$ is a Lazard set iff it is a Hall set.
He also proves in \cite[Proposition 1.1 and Theorem 1.1]{MR516004} that properties (\ref{Laz.1}) and (\ref{Laz.2}) ensure that $\eval(\bset)$ is a linearly independent and generating family of $\mathcal{L}(X)$. 
{This leads to the following definition.}

\begin{definition}[Hall basis] \label{def:ghb}
    {Given $\bset$ a Hall set (or equivalently a Lazard set), $\basis := \eval(\bset)$ is a basis of $\mathcal{L}(X)$.
    Such bases are call \emph{Hall bases}.}
\end{definition}

\begin{remark}
    Historically, {such} bases where introduced by Marshall Hall in \cite{zbMATH03059664}, based on ideas of Philip Hall in \cite{zbMATH03010343}. 
    In his historical narrower definition, the third condition in \cref{Def2:Laz} was replaced by the stronger condition: for every $b_1, b_2 \in \bset$, $b_1 < b_2 \Rightarrow |b_1| \leq |b_2|$.
    To avoid confusion with the generalized definition, we name them \emph{length-compatible Hall bases} in the sequel.
\end{remark}

{Given a Hall set $\bset$, the evaluation mapping $\eval$ is one to one from $\bset$ to the associated Hall basis~$\basis$, so that the elements of the basis (belonging to $\mathcal{L}(X)$) can be identified with the bracket of $b \in \bset \subset \Br(X)$ of which they are the evaluation. 
We will use this identification in the sequel when no confusion is possible.}

Two famous families of \GHBS\ of $\mathcal{L}(X)$ are the Chen-Fox-Lyndon basis (see \cite[Chapter~1]{MR516004}) and the historical length-compatible Hall bases, for which $b_1<b_2 \Rightarrow|b_1|\leq|b_2|$.

\begin{example}
    For instance, with $X=\{X_1,X_2\}$, the elements with length at most $4$ of each Hall set $\bset$ of $\mathcal{L}(X)$ with a length-compatible order $<$ such that $X_1<X_2$ are: $X_1$, $X_2$, {$(X_1,X_2)$, $\ad_{X_1}^2(X_2)$, $(X_2,(X_1,X_2))$, $\ad_{X_1}^3(X_2)$ and $(X_2,\ad_{X_1}^2(X_2))$, $\ad_{X_2}^2((X_1,X_2))$}. Note that, however, {$(X_1,(X_2,(X_1,X_2)))$} does not belong to $\bset$ because ${\lambda((X_2,(X_1,X_2)))}=X_2$ is not smaller than $X_1$, and the following equality holds in $\mathcal{L}(X)$
    \begin{equation}
        [X_1,[X_2,[X_1,X_2]]] =
        [[X_1,X_2],[X_1,X_2]]+[X_2,[X_1,[X_1,X_2]]]
        =[X_2,\ad_{X_1}^2(X_2)]
    \end{equation}
    This illustrates how \Cref{Def2:Laz} prevents elements from $\Br(X)$, whose evaluations in $\mathcal{L}(X)$ are linked by Jacobi relations, to appear simultaneously in $\bset$.
\end{example}

{\begin{remark}
    Let $X := \{ X_0, X_1 \}$ and $B \subset \Br(X)$ be a Hall set with an order such that $X_0 < X_1$.
    The definition of a Hall set implies that, for every $k \in \N$, $\ad^k_{X_0}(X_1) \in B$.
    Moreover, these are all the elements of $B$ containing $X_1$ exactly once.
    Since $\eval(B)$ is a basis of $\mathcal{L}(X)$, $\eval(B) \cap \dot{S}_1$ is a basis of $\dot{S}_1$ and this provides an alternative proof of \cref{lem:S1-basis}.
\end{remark}}

\subsubsection{Infinite product on a \GHB}

\begin{definition}[Infinite product]\label{Def:infinite_prod}
Let $J$ be a totally ordered set and $(S^j)_{j\in J}$ be a family of $\widehat{\mathcal{A}}(X)$ such that
\begin{itemize}
    \item for every $j\in J$, $\langle S^j , 1 \rangle = 1$
    \item for every $\sigma \in I^*$ with $\sigma \neq \emptyset$, the set $\{ j \in J ; \langle S^j , X_\sigma \rangle \neq 0 \}$ is finite.
\end{itemize}
The infinite product $\underset{j\in J}{\overset{\leftarrow}{\Pi}} S^j$ is the element of $\widehat{\mathcal{A}}(X)$ defined by
\begin{equation}
    \underset{j\in J}{\overset{\leftarrow}{\Pi}} S^j = \sum_{\sigma \in I^*} P_\sigma X_\sigma,
\end{equation}
where $P_\emptyset = 1$ and $P_\sigma$ is the finite sum
\begin{equation} \label{def:P_sigma}
 P_\sigma :=
 \sum_{n=0}^{|\sigma|}
 \sum_{\substack{\sigma_1,\dotsc,\sigma_n \in I^*, \\
X_{\sigma_1} \dotsb X_{\sigma_n} = X_\sigma
}}
\sum_{\substack{j_1,\dotsc, j_n \in J, \\  j_1>\dotsb>j_n}}
\langle S^{j_1} , X_{\sigma_1} \rangle \dotsm \langle S^{j_n}, X_{\sigma_n} \rangle.
\end{equation}
\end{definition}

The following lemma is the key point to generalize rigorously Sussmann's infinite product on length-compatible Hall bases, to {all} \GHBS.

\begin{lemma} \label{Prop:inf_prod_Laz}
 Let $\mathcal{B}$ be a \GHB\ and $(\alpha_b)_{b\in\mathcal{B}}$ be a family of $\K$. The infinite product
 $\underset{b \in \mathcal{B}}{\overset{\leftarrow}{\Pi}} e^{\alpha_b b}$
is well defined in $\widehat{\mathcal{A}}(X)$. Moreover, for every $\sigma \in I^*$,
\begin{equation} \label{Inf_prod_coeff}
\left\langle \underset{b \in \mathcal{B}}{\overset{\leftarrow}{\Pi}} e^{\alpha_b b}, X_{\sigma} \right\rangle = \left\langle \underset{b \in \mathcal{B}_{\llbracket 1, |\sigma|\rrbracket}}{\overset{\leftarrow}{\Pi}} e^{\alpha_b b}, X_{\sigma} \right\rangle
\end{equation}
where $\mathcal{B}_{\llbracket 1, |\sigma|\rrbracket}$ is ordered by the induced order of $\mathcal{B}$.
\end{lemma}

\begin{proof}
$\mathcal{B}$ is a totally ordered set and, for every $b\in\mathcal{B}$, $\langle e^{\alpha_b b} , 1 \rangle =1$. Let  $\sigma \in I^*$ with $|\sigma| \geq 1$.
For $\alpha \in \K$ and $b\in\mathcal{B}$, the property $\langle e^{\alpha b}, X_\sigma \rangle \neq 0$ requires $|b| \leq |\sigma|$. Indeed
\begin{equation}e^{\alpha b}-1=\sum_{k=1}^{+\infty} \frac{\alpha^k}{k!} b^k\end{equation}
has non vanishing coefficients only on monomials $X_{\sigma'}$ with length $|\sigma'| \geq |b|$. 
Thus the set $\{ b \in \mathcal{B} , \langle e^{\alpha_b b} , X_\sigma \rangle \neq 0 \}$ is finite. This proves that the infinite product 
is well defined in $\widehat{A}(X)$ and, by (\ref{def:P_sigma}), the formula  (\ref{Inf_prod_coeff}) holds.
\end{proof}

\subsubsection{Coordinates of the second kind}
\label{sec:coord-second}

\begin{definition}\label{Def:Coord2}
    Let $\mathcal{B}$ be a \GHB\ of $\mathcal{L}(X)$.
    The coordinates of the second kind associated to $\mathcal{B}$ is the unique family $(\xi_b)_{b\in\mathcal{B}}$ of functionals $\R_+ \times L^1_{\mathrm{loc}}(\R_+;\K^I) \rightarrow \K$ defined by induction in the following way: for every $t>0$ and $a \in L^1_{\mathrm{loc}}(\R_+;\K^I)$
\begin{itemize}
\item $\xi_{X_i}(t,a) := \int_0^t a_i$, for $i \in I$, 
\item for $b \in \mathcal{B} \setminus X$, there exists  a unique {pair} $(b_1,b_2)$ of elements of $\mathcal{B}$ such that $b_1<b_2$ and a unique maximal integer $m \in\N^*$ such that $b=\ad_{b_1}^m (b_2)$ and then
\begin{equation} \label{eq:def:coord2}
\xi_{b}(t,a):=\frac{1}{m!} \int_0^t  \xi_{b_1}^m(\tau,a) \dot{\xi}_{b_2}(\tau,a) \dd\tau.\end{equation}
\end{itemize}
\end{definition}

Formula \eqref{eq:def:coord2} indeed defines continuous functionals on $L^1$ and the following estimates hold.

\begin{lemma} \label{lm:xib-dotxib}
    Let $a_i \in L^1_{\mathrm{loc}}(\R_+;\K)$ for $i \in I$. For every $b \in \mathcal{B}$ and $t \geq 0$,
    \begin{align}
    	\label{eq:dotxib-u}
    	| \dot{\xi}_b(t,a) | & \leq |b| |a(t)| \|a\|_{L^1(0,t)}^{|b|-1}, \\
    	\label{eq:xib-u}
    	| \xi_b(t,a) | & \leq \|a\|_{L^1(0,t)}^{|b|}.
    \end{align}
\end{lemma}

\begin{proof}
    Estimate \eqref{eq:dotxib-u} is valid for $b \in X$ because $\dot{\xi}_{X_i}(t) = a_i(t)$ for $i \in I$ and propagated by induction on~$b$ using the recursive definition \eqref{eq:def:coord2}. 
    Estimate \eqref{eq:xib-u} is obtained by time-integration of \eqref{eq:dotxib-u} for each $b$.
\end{proof}

\subsubsection{Infinite product expansion of the solution to the formal differential equation}

{The following result is due to Sussmann in \cite{MR935387}.
The proof below follows Sussmann's argument.
It is recalled for sake of giving a self consistent presentation and also to treat the extension from length-compatible Hall bases to all Hall bases (which were not included in Sussmann's original statement).
}

\begin{theorem}
Let $\mathcal{B}$ be a \GHB\ of $\mathcal{L}(X)$.
Let $T > 0$ and $a_i \in L^1((0,T);\K)$ for $i \in I$.
For every $x^\star \in \widehat{\mathcal{A}}(X)$, the solution to the formal diff (\ref{formalODE})
satisfies, for every $t \in [0,T]$,
\begin{equation}
  x(t)= x^\star \underset{b \in \mathcal{B}}{\overset{\leftarrow}{\prod}} e^{\xi_b(t,a)b}.  
\end{equation}
\end{theorem}

\begin{proof}
It is sufficient to prove the formula with $x^\star=1$. 
To simplify the notations in this proof, we write $\xi_b(t)$ instead of $\xi_b(t,a)$.
By \Cref{Prop:inf_prod_Laz} it is sufficient to prove that, for every $t\in [0,T]$ and $\sigma \in I^*$
\begin{equation}\label{Comp_x(t)_Xsigma}
\left\langle x(t) , X_\sigma \right\rangle = \left\langle \underset{b \in \mathcal{B}_{\llbracket 1, |\sigma|\rrbracket}}{\overset{\leftarrow}{\Pi}} e^{\xi_b(t) b} , X_\sigma \right\rangle.
\end{equation}
Let $\sigma \in I^*$, $M:=|\sigma|$, $k\in\N$ and $b_1,\dotsc,b_{k+1}$ and $Y_0,\dotsc,Y_{k+1}$ be as in (\ref{Laz.1}). The equality (\ref{Comp_x(t)_Xsigma}) can equivalently we written
\begin{equation}\label{Comp_x(t)_Xsigma_BIS}
\left\langle x(t) , X_\sigma \right\rangle = \left\langle  e^{\xi_{b_{k+1}}(t) b_{k+1}} \dotsb e^{\xi_{b_1}(t) b_1} , X_\sigma \right\rangle.
\end{equation}
We define $x_0(t) := x(t)$ and, for $j \in \intset{1,k+1}$,
\begin{equation}
    x_j(t):= x(t)e^{-\xi_{b_1}(t) b_1} \dotsb e^{-\xi_{b_j}(t) b_j}.
\end{equation}
We prove by induction on $j\in \intset{0,k+1}$ that 
\begin{equation}\label{aux_j}
    \dot{x}_j(t) = x_j(t) \left(\sum_{b \in Y_{j}} \dot{\xi}_{b}(t) b\right) 
    \quad \textrm{and} \quad 
    x_j(0)=1.
\end{equation}
It is clear for $j=0$ because $x_0(t)=x(t)$, $Y_0=X$ and $\dot{\xi}_{X_i}(t)=a_i(t)$ for $i \in I$. 
Let $j\in\intset{1,k+1}$. 
We assume (\ref{aux_j}) holds at step $j-1$.
We deduce from the definition of $x_{j}(t)$ that
\begin{equation}
x_{j}(t)=x_{j-1}(t)e^{-\xi_{b_{j}}(t) b_{j}}.
\end{equation}
Since $\xi_{b_{j}}(0) = 0$, $x_{j}(0) = 1$. 
Moreover,
\begin{equation}
    \begin{split}
        \dot{x}_j(t) & = x_{j-1}(t)\left(\sum_{b \in Y_{j-1} } \dot{\xi}_{b}(t) b \right)e^{-\xi_{b_j}(t) b_j} - x_{j-1}(t) \dot{\xi}_{b_j}(t) b_je^{-\xi_{b_j}(t) b_j}
\\ & = 
x_j(t) e^{\xi_{b_j}(t) b_j} \left( \sum_{b \in Y_{j-1}\setminus\{b_j\} } \dot{\xi}_{b}(t) b  \right)e^{-\xi_{b_j}(t) b_j}
\\ & = 
x_j(t) 
\sum_{m \in \N} \sum_{b \in Y_{j-1}\setminus\{b_j\} } \frac{\xi_{b_j}^m(t)}{m!} \dot{\xi}_{b}(t) \ad_{b_j}^m(b) 
    \end{split}
\end{equation}
which ends the proof by induction.

We deduce from (\ref{Laz.3}) and (\ref{aux_j}) for $j=(k+1)$  that $x_{k+1}(t)-1$ has non vanishing coefficients only on monomials $X_{\sigma'}$ with $|\sigma'|>|\sigma|$. Therefore, by (\ref{def:P_sigma}),
 \begin{equation}
 \left\langle x(t) , X_\sigma \right\rangle = \left\langle  x_{k+1}(t) e^{\xi_{b_{k+1}}(t) b_{k+1}} \dotsb e^{\xi_{b_1}(t) b_1}  , X_\sigma \right\rangle = \left\langle  e^{\xi_{b_{k+1}}(t) b_{k+1}} \dotsb e^{\xi_{b_1}(t) b_1} , X_\sigma \right\rangle,
 \end{equation}
 which concludes the proof.
\end{proof}

\section{Technical tools about functions and vector fields} \label{sec:technical}

In this section, we state classical definitions and technical results about functions and vector fields. For the sake of completeness, the proofs, although classical, are provided.

Throughout the whole article, $d \in \N^*$ denotes the dimension of the state space for the considered ordinary differential equations. We work locally, in neighborhoods of the origin $0 \in \K^d$. For $\delta > 0$, $B_\delta$ denotes the closed ball of center $0$ and radius $\delta$ in the state space $\K^d$.

\subsection{Functional spaces for finite or analytic regularity}

\subsubsection{Conventions for multi-indexes}

For $a \in \N^*$ and a multi-index $\alpha=(\alpha^1,\dotsc,\alpha^a) \in \N^a$, we use the notations $|\alpha|:=\alpha^1+\dotsb+\alpha^a$,  $\partial^\alpha:={\partial_1^{\alpha^1}\dotsb\partial_a^{\alpha^a}}$ and $\alpha!:=\alpha^1!\dotsb\alpha^a!$.

\begin{lemma}
    The following estimates hold:
    \begin{align}
        \label{eq:encadrement.n!}
        \forall n \in \N, & \quad n^n e^{-n} e \leq n! \leq (n+1)^{n+1} e^{-(n+1)} e, \\
        \label{factorielle_multiple}
        \forall a \in \N^*, \forall \alpha=(\alpha^1,\dotsc,\alpha^a) \in \N^a, & \quad
    2^{-(a-1)|\alpha|} |\alpha|! \leq  \alpha! \leq |\alpha|!
    \end{align}
\end{lemma}

\begin{proof}
    The first inequality can be proved using classical series-integral comparison and the second by iterating $p! q! \geq 2^{-(p+q)} (p+q)!$ for every $p,q\in\N$.
\end{proof}

\subsubsection{{Finite-regularity norms}}

\begin{definition}[Regular functions] \label{def:ckb}
    Let $a, b \in \N^*$ and $K$ a compact subset of $\K^a$.
    Let $k \in \N$. We endow $\CC^k(K;\K^b)$, the space of functions whose real-derivatives are well-defined and continuous up to order $k$ on an open neighborhood of $K$ to $\K^b$ with the norm
    \begin{equation} \label{eq:def:Ck}
        \|f\|_{\CC^k} := \sum_{j=1}^b \sum_{|\alpha|\leq k} \frac{1}{\alpha!} \| \partial^\alpha f_j \|_{L^\infty(K)},
    \end{equation}
    where the sum ranges over multi-indexes $\alpha \in \N^a$ whose sum is at most $k$ and $f_1, \dotsc, f_b$ are the components of the vector-valued function $f$. We denote by $\CC^\infty(K;\K^b)$ the intersection of these spaces over $k \in \N$.
\end{definition}

\newcommand{\nck}[1]{\left\| #1 \right\|_{\CC^k}}

{
\begin{lemma}[Submultiplicativity] \label{p:mul-ck}
    Let $a \in \N^*$, $K$ a compact subset of $\K^a$ and $k \in \N$.
    For every $f, g \in \CC^{k}(K;\K)$, one has
	\begin{equation} \label{normck(fg)}
		\nck{f g} \leq \nck{f} \nck{g}.
	\end{equation}
\end{lemma}

\begin{proof}
	Using the multivariate Leibniz formula, one has	
	\begin{equation}
	    \begin{split}
	        \nck{f g}
	& = \sum_{|\alpha| \leq k} \frac{1}{\alpha!} \| \partial^\alpha (fg) \|_{L^\infty(K)} \\
	& \leq \sum_{|\alpha| \leq k} \frac{1}{\alpha!} \sum_{\beta \leq \alpha} \binom{\alpha}{\beta} \| \partial^\beta f \|_{L^\infty(K)} \| \partial^{\alpha-\beta} g \|_{L^\infty(K)} 
	\leq \nck{f} \nck{g},
	    \end{split}
	\end{equation}
	where the sum ranges over all multi-indexes $\beta \in \N^a$ such that $\beta_i \leq \alpha_i$ for each $i \in \intset{1,a}$.
\end{proof}

\begin{lemma}[Control of gradients] \label{p:grad-ck}
    Let $a \in \N^*$, $K$ a compact subset of $\K^a$ and $k \in \N$.
	For every $f \in \CC^{k+1}(K;\K)$ and $j \in \intset{1,a}$,
	\begin{equation}
	    \nck{\partial_j f} \leq (k+1) \| f \|_{\CC^{k+1}}.
	\end{equation}
\end{lemma}

\begin{proof}
    By \eqref{eq:def:Ck}, 
	\begin{equation}
	\begin{split}
	    \nck{\partial_j f} =
	    \sum_{|\alpha|\leq k} \frac{1}{\alpha!} \|\partial^{\alpha + e_j} f\|_{L^\infty(K)}
		& = \sum_{|\alpha|\leq k} \frac{\alpha_j+1}{(\alpha+e_j)!} \|\partial^{\alpha + e_j} f\|_{L^\infty(K)} 
		\leq (k+1)  \| f \|_{\CC^{k+1}}
	\end{split}
	\end{equation}
	since $\alpha_j \leq |\alpha| \leq k$.
\end{proof}
}

\subsubsection{Analytic norms}

\begin{definition}[Analytic norms]
    Let $a, b \in \N^*$ and $K$ a compact subset of $\K^a$. We define $\CC^\omega(K;\K^b)$ the space of real-analytic functions defined on an open neighborhood of $K$ to $\K^b$, as the union for $r > 0$ of the spaces $\CC^{\omega,r}(K;\K^b)$, which are the subsets of $\CC^\infty(K;\K^b)$ for which the following norm is finite
    \begin{equation} \label{eq:def.analytic.r}
        \opnorm{f}_r :=
        \sum_{i=1}^b \sum_{\alpha \in \N^a} \frac{r^{|\alpha|}}{\alpha!} \| \partial^\alpha f_i \|_{L^\infty(K)}.
    \end{equation}
\end{definition}

\begin{lemma}[{Submultiplicativity}]
    Let $a \in \N^*$, $K$ a compact subset of $\K^a$, $r > 0$. Then, for every $f, g \in \CC^{\omega,r}(K;\K)$, one has
	\begin{equation} \label{opnorm(fg)r}
		\opnorm{f g}_r \leq \opnorm{f}_r \opnorm{g}_r.
	\end{equation}
\end{lemma}

\begin{proof}
	Using the multivariate Leibniz formula, one has	
	\begin{equation}
	    \begin{split}
	        \opnorm{f g}_r
	& = \sum_{\alpha \in \N^a} \frac{r^{|\alpha|}}{\alpha!} \| \partial^\alpha (fg) \|_{L^\infty(K)} \\
	& \leq \sum_{\alpha \in \N^a} \frac{r^{|\alpha|}}{\alpha!} \sum_{\beta \leq \alpha} \binom{\alpha}{\beta} \| \partial^\beta f \|_{L^\infty(K)} \| \partial^{\alpha-\beta} g \|_{L^\infty(K)} = \opnorm{f}_r \opnorm{g}_r,
	    \end{split}
	\end{equation}
	where the sum ranges over all multi-indexes $\beta \in \N^a$ such that $\beta_i \leq \alpha_i$ for each $i \in \intset{1,a}$.
\end{proof}

\begin{lemma}[Control of gradients] \label{lemma:gradient.estimate}
    Let $a \in \N^*$, $K$ a compact subset of $\K^a$. 
	For every $r_2 > r_1 > 0$, $f \in \CC^{\omega,r_2}(K;\K)$ and $j \in \intset{1,a}$,
	\begin{equation} \label{eq:djf.rr'}
	    \opnorm{\partial_j f}_{r_1} \leq \frac{1}{r_1} \left(e \ln \frac{r_2}{r_1}\right)^{-1} \opnorm{f}_{r_2}.
	\end{equation}
	In particular, if $r_2 \leq e r_1$,
	\begin{equation} \label{eq:djf.r.dr}
	    \opnorm{\partial_j f}_{r_1} \leq \frac{1}{r_2 - r_1} \opnorm{f}_{r_2}.
	\end{equation}
\end{lemma}

\begin{proof}
	We start with the first estimate \eqref{eq:djf.rr'}. One has
	\begin{equation}
	\begin{split}
		\opnorm{\partial_j f}_{r_1}
		= \sum_{\alpha \in \N^a} \frac{r_1^{|\alpha|}}{\alpha!} \| \partial^{\alpha+e_j} f \|_{L^\infty(K)}
		& = \frac{1}{r_1} \sum_{\alpha \in \N^a} \frac{r_1^{|\alpha+e_j|}}{(\alpha+e_j)!} \frac{(\alpha+e_j)!}{\alpha!}\| \partial^{\alpha+e_j} f \|_{L^\infty(K)} \\
		& \leq \frac{1}{r_1} \opnorm{f}_{r_2} \sup_{\alpha \in \N^a} \left(\frac{r_1}{r_2}\right)^{|\alpha+e_j|} \frac{(\alpha+e_j)!}{\alpha!} \\
		& \leq \frac{1}{r_1} \opnorm{f}_{r_2} \sup_{n \geq 1} n \left(\frac{r_1}{r_2}\right)^n.
	\end{split}
	\end{equation}
	For $x \in (0,1)$, let $C(x) := \sup_{n \geq 1} n x^n = \sup_{n \geq 1} \exp (\ln n + n \ln x)$.
	{Considering that $x$ is a fixed parameter and} differentiating {the argument} with respect to $n \in [1,+\infty)$ yields
	\begin{equation}
		{\frac{\dd}{\dd n}} \left( \ln n + n \ln x \right) = \frac{1}{n} + \ln x.
	\end{equation}
	Since $x < 1$, {this} derivative is negative for $n$ large enough. 
	For $x \geq 1/e$, the global maximum is for $n= -1/\ln x$. So its value yields the bound
	\begin{equation} \label{eq:cx.elnx}
		C(x) \leq \left( - e \ln x \right)^{-1}.
	\end{equation}
	For $x \leq 1/e$, the supremum over $n$ is achieved for $n = 1$ and its value is $x$. Since $x \leq \left( - e \ln x \right)^{-1}$ for $x \in (0,1)$, the bound \eqref{eq:cx.elnx} is looser and valid for every $x \in (0,1)$. 
	
	\medskip
	
	The second inequality is a consequence of the estimate $\ln (1+\sigma) \geq \sigma / (e-1)$ for $\sigma \leq e-1$.	
\end{proof}

\begin{remark}
	The first estimate \eqref{eq:djf.rr'} is classical (see e.g.\ \cite{zbMATH06012919}). The second estimate~\eqref{eq:djf.rr'} is a simplified version, restricted to the case when the relative radius loss is small enough. This is the form under which we will use \cref{lemma:gradient.estimate} in the sequel since we consider small radius losses.
\end{remark}

\subsection{Estimates for differential operators and Lie brackets}

\subsubsection{Vector-valued functions, vector fields and differential operators}

{As is usual, we will identify each smooth vector-valued function with the associated first-order linear differential operator, and we will refer to both objects as a \emph{vector field}.
Let $K$ be a compact subset of $\K^d$.

\begin{definition}[Vector field]
    Given \emph{coordinates} $a_1, \dotsc, a_d \in \CC^\infty(K;\K)$, we define the associated \emph{vector field}~$f$ indifferently as the vector-valued function $f = (a_1, \dotsc, a_d)$ of $\CC^\infty(K;\K^d)$ (mapping each point of the state space to a vector of $\K^d$), or as the first-order linear differential operator $f = a_1 \partial_{1} + \dotsb + a_d \partial_{d}$ (acting on smooth functions $\phi \in \CC^\infty(K;\K)$).
\end{definition}

\begin{remark}[Composition of vector fields] \label{rk:no-nabla}
    Seen as linear operators, vector fields can be composed, yielding higher-order differential operators.
    Let $f, g \in \CC^\infty(K;\K^d)$ (denoting their coordinates by $a_1, \dotsc, a_d$ and $b_1, \dotsc b_d$) and $\phi \in \CC^\infty(K;\K)$.
    We will hence write
    \begin{equation}
        f \phi = \sum_{i=1}^{d} a_i \partial_{i} \phi
    \end{equation}
    and
    \begin{equation} \label{eq:fgphi}
        f g \phi = \sum_{i=1}^{d} \sum_{j=1}^d a_i \left(b_j \partial_{i} \partial_{j} \phi + (\partial_{i} b_j) (\partial_{j} \phi) \right).
    \end{equation}
    Similarly, for $k \in \N$, $f^k \phi$ will denote the composition of the linear differential operator $f$ with itself $k$ times, applied to $\phi$.
    Such formulas still make sense in a finite regularity setting, as long as all derivatives are well-defined.
\end{remark}}

\subsubsection{Estimates for products}

{
\begin{lemma} \label{p:prod-diff-ck}
    Let $k \in \N$, $n,b \in \N^*$, $\delta > 0$, $f_1, \dotsc f_n \in \CC^{k+n-1}(B_\delta;\K^d)$ and $\phi \in \CC^{k+n}(B_\delta;\K^b)$.
    With the notations of \cref{rk:no-nabla},
    \begin{equation} \label{eq:prod-diff-ck}
        \| f_n \dotsb f_1 \phi \|_{\CC^k} 
        \leq
        \frac{(k+n)!}{k!} \| f_1 \|_{\CC^{k+n-1}} \dotsb \| f_n \|_{\CC^{k+n-1}} \| \phi \|_{\CC^{k+n}}.
    \end{equation}
\end{lemma}

\begin{proof}
    For $n = 1$, it is a consequence of \eqref{eq:def:Ck}, \cref{p:mul-ck} and \cref{p:grad-ck}.
    For $n > 1$, the estimate follows by induction. 
\end{proof}
}

\begin{lemma} \label{thm:prod.analytic}
    Let $r_2 > 0$, $r_1 \in [r_2/e, r_2)$, $n{,b} \in \N^*$ and $\delta > 0$.
    {Let $f_1, \dotsc, f_n \in \CC^{\omega,r_2}(B_\delta;\K^d)$ and $\phi \in \CC^{\omega,r_2}(B_\delta;\K^{b})$.
    With the notations of \cref{rk:no-nabla},}
    \begin{equation} \label{eq:prod.analytic}
        \opnorm{{f_n \dotsb f_1 \phi}}_{r_1}
        \leq \frac{n!}{e} \left(\frac{e}{r_2-r_1}\right)^n \opnorm{f_n}_{r_2} \dotsb \opnorm{f_1}_{r_2} \opnorm{\phi}_{r_2}.
    \end{equation}
    In particular, under the same assumptions,
    \begin{equation} \label{eq:prod.analytic.C0}
        \| {f_n \dotsb f_1 \phi} \|_{\CC^0} 
        \leq n! \left(\frac{5}{r_2}\right)^n \opnorm{f_n}_{r_2} \dotsb \opnorm{f_1}_{r_2} \opnorm{\phi}_{r_2}.
    \end{equation}
\end{lemma}

\begin{proof}
	For $n = 1$, estimate \eqref{eq:prod.analytic} is a consequence of \eqref{eq:def.analytic.r}, \eqref{opnorm(fg)r} and \eqref{eq:djf.r.dr}. For $n > 1$, one applies the $n=1$ estimate $n$ times with a radius increment $(r_2-r_1) / n$ at each step. This yields more precisely
	\begin{equation}
		\begin{split}
		    \opnorm{{f_n \dotsb f_1 \phi}}_{r_1}
		    & \leq 
			\left(\frac{n}{r_2-r_1}\right) 
			\opnorm{f_n}_{r_1} \opnorm{{f_{n-1} \dotsb f_1 \phi}}_{r_1+\frac{r_2-r_1}{n}} \\
			& \leq 
			\left(\frac{n}{r_2-r_1}\right)^n \opnorm{\phi}_{r_2}
			\prod_{j=1}^n \opnorm{f_j}_{r_1+(n-j)\frac{r_2-r_1}{n}},
		\end{split}
	\end{equation}
	which concludes the proof because the norm \eqref{eq:def.analytic.r} is non-decreasing with respect to $r$, and we can bound $n^n$ using \eqref{eq:encadrement.n!}. Estimate \eqref{eq:prod.analytic.C0} is a direct consequence for the particular choice $r_1 = r_2/e$, because $e^2/(e-1) \leq 5$.
\end{proof}

\subsubsection{Lie brackets}

\begin{definition}[Lie bracket of vector fields]
    {We define the \emph{Lie bracket} of smooth vector fields $f$ and $g$ as the usual commutator of the associated linear differential operators: $[f,g] := fg - gf$ (with the notations of \cref{rk:no-nabla}).
    By \eqref{eq:fgphi} and Schwarz's theorem, one checks that $[f,g]$ is also a first-order differential operator, which, as a vector-valued function, can be computed as $[f,g] = (Dg) f - (Df) g$.}
\end{definition}

\begin{definition}[Evaluated Lie bracket] \label{Def:evaluated_Lie_bracket}
    Let $I$ be a finite set of indices,
    $X=\{X_i ; i\in I\}$ be indeterminates and
    $\{f_i ; i\in I\}$ be $\CC^\infty$ vector fields on a subset $\Omega$ of $\K^d$.
    For {an iterated bracket} $b \in \Br(X)$, we define $f_b:=\Lambda({\eval(b)})$, where $\Lambda:\mathcal{L}(X) \rightarrow \CC^\infty(\Omega;\K^d)$ is the unique homomorphism of Lie {algebras} such that $\Lambda(X_i)=f_i$ for every $i\in I$ (see {\cref{rk:eval} and} \cref{Lem:identification_libre}).
    
    The vector field $f_b$ is obtained by replacing the indeterminates $X_i$ with the corresponding vector fields $f_i$ in the {iterated bracket} $b$, for instance $f_{{(X_1,(X_2,X_3))}}=[f_1,[f_2,f_3]]$.
    
    The notation $f_b$ will sometimes denote the same vector field, {seen as a vector-valued function,} under weaker regularity assumptions, for instance $f_i \in \CC^{|b|-1}$ and then $f_b \in \CC^0$.
\end{definition}

\begin{lemma}[Finite regularity estimate]
    \label{thm:bracket.ck}
    Let $k \in \N$ and $\delta > 0$.
    Let $b \in \Br(X)$.
    For $i \in I$, let {$f_i \in \CC^{k+|b|-1}(B_\delta;\K^d)$}.
    Then,
    \begin{equation}
        \| f_b \|_{\CC^k} \leq 2^{|b|{-1}} \frac{(k+|b|-1)!}{k!} \prod_{i\in I} \| f_i \|_{\CC^{k+|b|-1}}^{n_i(b)}.
    \end{equation}
\end{lemma}

\begin{proof}
    {This follows from \eqref{eq:prod-diff-ck} because, as can be checked by induction on $|b|$, $f_b$ is a sum of at most $2^{|b|-1}$ terms of the form studied in \cref{p:prod-diff-ck}, where $\phi$ is one of the vector fields $f_i$.}
\end{proof}

\begin{lemma}[Analytic estimate] \label{thm:bracket.analytic}
    Let $r_2 > 0$, $r_1 \in [r_2/e, r_2)$ and $\delta > 0$.
    Let $b \in \Br(X)$.
    For $i \in I$, let {$f_i \in \CC^{\omega,r_2}(B_\delta;\K^d)$}.
    Then,
    \begin{equation} \label{eq:bracket.analytic}
        \opnorm{f_b}_{r_1} 
        \leq \frac{(|b|-1)!}{e} \left( \frac{2e}{r_2-r_1} \right)^{|b|-1}
        \prod_{i \in I} \opnorm{f_i}_{r_2}^{n_i(b)}.
    \end{equation}
    In particular, under the same assumptions,
    \begin{equation} \label{eq:bracket.analytic.C0}
        \| f_b \|_{\CC^0}
        \leq (|b|-1)! \left( \frac{9}{r_2} \right)^{|b|-1}
        \prod_{i \in I} \opnorm{f_i}_{r_2}^{n_i(b)}.
    \end{equation}
    \begin{equation} \label{eq:bracket.analytic.C1}
        \| f_b \|_{\CC^1}
        \leq \max \left\{ 1, \frac{1}{r_2} \right\} (|b|-1)! \left( \frac{9}{r_2} \right)^{|b|-1}
        \prod_{i \in I} \opnorm{f_i}_{r_2}^{n_i(b)}.
    \end{equation}
\end{lemma}

\begin{proof}
    Estimate \eqref{eq:bracket.analytic} stems from \eqref{eq:prod.analytic} because, as can be checked by induction on $|b|$, $f_b$ is a sum of at most $2^{|b|-1}$ terms of the form studied in \cref{thm:prod.analytic}, where $\phi$ is one of the vector fields $f_i$. Estimates \eqref{eq:bracket.analytic.C0} and \eqref{eq:bracket.analytic.C1} are direct consequences of \eqref{eq:bracket.analytic} for the particular choice $r_1 = r_2/e$ because $2e^2/(e-1) \leq 9$ and, for every $r_1 > 0$, $\| f_b \|_{\CC^1} \leq \max \{ 1, \frac{1}{r_1} \} \opnorm{f_b}_{r_1}$.
\end{proof}

\begin{remark} \label{rk:bracket.optimal}
    The fact that estimate  \eqref{eq:bracket.analytic} scales like the factorial of the length of the Lie bracket is optimal, as illustrated by the following vector fields.
    For $x \in \R^2$ with $|x| < 1$, define
    \begin{equation} \label{eq:def:f0f1.optimal}
     f_0(x) := e_1 \quad \textrm{and} \quad f_1(x) := \frac{1}{1-x_1} e_2.
    \end{equation}
    Using \eqref{eq:def.analytic.r}, one checks that these vector fields belong in particular to {$\CC^{\omega,r}(B_\delta;\R^2)$} for $r = \frac 1 4$ and $\delta = \frac 1 2$, with $\opnorm{f_0}_{r} = 1$ and $\opnorm{f_1}_{r} = 2$.
    For $k \in \N$, one has
    \begin{equation} \label{eq:adkf0f1.optimal}
      \ad^{k}_{f_0}(f_1)(x) = \partial_1^k \left( \frac{1}{1-x_1} \right) e_2 = \frac{k!}{(1-x_1)^{k+1}} e_2.
    \end{equation}
    Moreover, since $f_0$ is constant and $f_1$ depends only on $x_1$ but is {a multiple of} $e_2$, every Lie bracket involving $f_1$ at least twice vanishes identically.
    Since these analytic vector fields ``saturate'' the bounds and exhibit such a nice structure, we will use them repeatedly in our counter-examples.
\end{remark}

\subsection{Well-posedness of ordinary differential equations}
\label{Subsec:WP_ODE}

The nonlinear differential equations 
\begin{equation} \label{ODE:f}
    \dot{x}(t) = f(t,x(t))
    \quad \textrm{and} \quad
    x(0) = p
\end{equation}
will be studied in the following classical frameworks.

\begin{lemma} \label{Lem:WP_ODE:f}
    Let $\delta, T>0$ and {$f\in L^1((0,T);\CC^1(B_{2\delta};\K^d))$} such that $\|f\|_{L^1((0,T);\CC^0)}<\delta$.
    \begin{enumerate}
        \item For each $p \in B_\delta$, there exists a unique function $x(\cdot;f,p) \in \CC^0([0,T];B_{2\delta})$ such that
        \begin{equation} \label{ODE:f_faible}
            \forall t \in [0,T], \quad x(t;f,p)=p+\int_0^t f\left(\tau,x(\tau;f,p)\right) \dd\tau.
        \end{equation}
    
        \item If $f \in \CC^0([0,T]\times B_{2\delta};\K^d)$ then $x(\cdot;f,p) \in \CC^1([0,T];B_{2\delta})$ and satisfies (\ref{ODE:f}) pointwise.

        \item If $f \in \CC^\infty([0,T]\times B_{2\delta};\K^d)$, the map $p \in B_\delta \mapsto x(\cdot;f,p) \in \CC^0([0,T];B_{2\delta})$ is smooth.
        
        \item If $g$ satisfies the same assumptions as $f$, {then,} for each $p \in B_\delta$ and $t \in [0,T]$,
        \begin{equation}
        |x(t;f,p) - x(t;g,p)| \leq 
        \| f - g \|_{L^1((0,t);\CC^0)}
        \exp \left( \|f\|_{L^1((0,t);\CC^1)} \right).
        \end{equation}
    \end{enumerate}
\end{lemma}

\begin{proof}
    We proceed step by step. Let ${E} := \CC^0([0,T];B_{2\delta})$.  
\begin{enumerate}
    \item Define $\Theta: {E} \to {E}$ by  $\Theta(x)(t):=p+\int_0^t f(\tau,x(\tau))\dd\tau$ for $x \in {E}$. Thanks to the smallness assumption on $f$, $\Theta(x)(t) \in B_{2\delta}$.
    Let $n \in \N^*$ be such that $ \|f\|_{L^1((0,T);\CC^1)}^n / n! < 1$.
    By the Banach fixed-point theorem, $\Theta^n$ has a unique fixed point, which is also a fixed point of $\Theta$.
    \item If $f$ is continuous, then 
    $t \mapsto \Theta(x(t;f,p))$ belongs to $\CC^1([0,T];B_{2\delta})$ and its derivative at time~$t$  is $f(t,x(t;f,p))$. 
    \item If $f$ is smooth, let $\bar{p} \in B_\delta$, $\bar{x} := x(\cdot;f,\bar{p})$ and define
    $F: B_\delta \times {E} \to {E}$
    by
    \begin{equation}\forall t \in [0,T], \quad F(p,x)(t):=x(t)-p-\int_0^t f(\tau,x(\tau)) \dd\tau\end{equation}
    Then $F$ is of class $\CC^\infty$, vanishes at $(\bar{p},\bar{x})$ and ${\partial_x F}(\bar{p},\bar{x})$ is a bijection on ${E}$. 
    By the implicit function theorem, the map $p \mapsto x(\cdot;f,p)$ is $\CC^\infty$ on a neighborhood of $\bar{p}$.
    \item This follows from a standard Gr\"onwall's lemma argument. \qedhere
\end{enumerate}
\end{proof}

\begin{lemma} \label{Prop:Sol_analytic/control}
    Let $\delta,\delta_u > 0$, $q\in\N^*$ and $f \in \CC^\omega(B_{2\delta} \times B_{\K^q}(0,\delta_u) ; \K^d )$.
    Let $T:=\delta/\|f\|_{\CC^0}$.
    For each $p \in B_\delta$ and $u \in L^\infty((0,T);\K^q)$ with $\|u\|_{L^\infty} \leq \delta_u$, there exists a unique solution $x \in \CC^0([0,T];B_{2\delta})$ to
    \begin{equation}
        \left\{
        \begin{aligned}
            & \dot{x}(t) = f(x(t),u(t)), \\
            & x(0) = p,
        \end{aligned}
        \right.
    \end{equation}
    denoted $x(t;f,u,p)$. Moreover, the map $(u,p)  \mapsto x(\cdot;f,u,p) \in \CC^0([0,T];B_{2\delta})$ is real-analytic on $B_\delta \times B_{L^\infty(0,T)}(0,\delta_u)$.
\end{lemma}

\begin{proof}
    Existence stems from \cref{Lem:WP_ODE:f}. Analyticity is a consequence of the implicit function theorem, which yields the analyticity of the implicit function when the direct function is analytic (see e.g.\ \cite[Theorem~4.5.4]{zbMATH01889824}).
\end{proof}

\subsection{Flows, compositions and pushforwards}

{Here and in the sequel, when we manipulate \emph{flows} of vector fields, we always make sure that an appropriate smallness assumption ensures that the local flow is well-defined up to the time at which we evaluate it.}

\subsubsection{Definitions and approximations}

By applying \cref{Lem:WP_ODE:f} to a time-independent vector field we obtain the following object.

\begin{definition}[Flow of time-independent vector fields]
    Let $\delta > 0$. 
    Let {$f \in \CC^1(B_{2\delta};\K^d)$} such that $\| f \|_{\CC^0} < \delta$.
    We denote by $e^f$ the flow at time one of the vector field $f$,
    \begin{equation}
        e^f :
        \left\{
            \begin{aligned}
                B_\delta & \to B_{2\delta}, \\
                p & \mapsto x(1;f,p),
            \end{aligned}
        \right.
    \end{equation}
    with the notations of \cref{Subsec:WP_ODE}. 
    We write $e^f p$ instead of $e^f (p)$ to allow easier composition of flows.
    {When moreover $f \in \CC^\infty(B_{2\delta};\K^d)$}, $e^f$ can also be seen as the zero-order linear operator on $\CC^\infty(B_{2\delta};\K)$ defined by $e^f \phi : p \mapsto \phi(e^f p)$. 
\end{definition}

\begin{lemma} \label{lm:basic-flow}
    Let $\delta > 0$ and {$f \in \CC^1(B_{\delta};\K^d)$}.
    Assume that $\delta' := \delta - \| f \|_{{\CC^0(B_\delta)}} > 0$.
    For each $p \in B_{\delta'}$, $e^f p$ is well-defined and $e^f p \in B_\delta$.
    Moreover,
    \begin{equation} \label{eq:efp-p}
        |e^f p - p| \leq \|f\|_{{\CC^0(B_\delta)}},
    \end{equation}
    and
    \begin{equation} \label{eq:dpef}
        \| D (e^f) \|_{{\CC^0(B_{\delta'})}} \leq e^{\|D f\|_{{\CC^0(B_\delta)}}}
        \leq e^{\|f\|_{{\CC^1(B_\delta)}}}.
    \end{equation}
\end{lemma}

\begin{proof}
    The second estimate comes from the fact that $D (e^f)_{\rvert p} = R(1)$ where
    \begin{equation}
        \dot{R}(t) = Df ( e^{tf} p ) R(t) 
        \quad \textrm{and} \quad R(0)= \mathrm{Id}.
    \end{equation}
    Thus, by Gr\"onwall's lemma,
    \begin{equation} 
        \|R(1)\| \leq \|\mathrm{Id}\| e^{\int_0^1 \| D f ( e^{tf} p) \| \dd t} \leq
        e^{\|D f\|},
    \end{equation}
    which concludes the proof.
\end{proof}

The exponential notation is motivated by the possibility to approximate $e^f$ by partial sums of the exponential series of {the linear differential operator $f$}.
It is completely legitimate in the analytic setting, as underlined by the following result.

\begin{lemma}[Approximation of autonomous flows] \label{Lem:flow}
    Let $\delta > 0$, {$f \in \CC^1(B_{2\delta};\K^d)$} with  $\|f\|_{\CC^0} < \delta$.
    {Using the notations of \cref{rk:no-nabla}:}
    \begin{enumerate}
        \item For each $M \in \N$, if {$f \in \CC^M(B_{2\delta};\K^d)$} and $\phi \in \CC^{M+1}(B_{2\delta};\K)$, for each $p \in B_\delta$,
         \begin{equation} \label{estimate.ef.phi}
          \left| \left( e^f  - \sum_{k=0}^M \frac{{f^k}}{k!} \right) (\phi)(p) \right| 
          \leq
          \| f \|_{\CC^M}^{M+1} { \| \phi \|_{\CC^{M+1}}}.
         \end{equation}
        \item If {$f \in \CC^\omega(B_{2\delta};\K^d)$} and $\phi \in \CC^{\omega}(B_{2\delta};\K)$, for $t$ small enough, for each $p \in B_\delta$,
        \begin{equation} \label{estimate.ef.phi_analytic}
        e^{tf} (\phi)(p) =  \sum_{k=0}^{+\infty} \frac{t^k}{k!} {f^k} \phi(p)
        \end{equation}
        and the sum converges absolutely in the sense of analytic functions.
 \end{enumerate}
\end{lemma}

\begin{proof}
    \emph{First statement}. By the first point of \cref{Lem:WP_ODE:f}, $e^{tf}(p)$ is well defined for every $t \in [0,1]$ and takes values in $B_{2\delta}$. For $t\in[0,1]$ and $k \in \intset{0,M+1}$, we have
 \begin{equation}
   \frac{\dd^k}{\dd t^k} \left[ \phi( e^{t f}( p )) \right]
   = \big( {f^k} \phi \big) (e^{t f} (p)).
 \end{equation}
 Thus, the considered sum is the Taylor expansion of order $M$ of the map $t \mapsto \phi( e^{t f}(p) )$ at $t=0$ and
 \begin{equation} \label{efphi.1}
    \left( e^{f} - \sum_{k=0}^{M} \frac{{f^k}}{k!} \right)(\phi)(p)
     = \int_0^1 \frac{(1-s)^M}{M!} \big({f^{M+1}} \phi \big) ( e^{s f} (p) ) \dd s.  
 \end{equation}
 {This concludes the proof of \eqref{estimate.ef.phi} thanks to the integration in \eqref{efphi.1} and \cref{p:prod-diff-ck}.}
 
 \medskip
 
 \noindent \emph{Second statement}. 
 Let $r > 0$ be such that {$f \in \CC^{\omega,r}(B_{2\delta};\K^d)$} and $\phi \in \CC^{\omega,r}(B_{2\delta};\K)$. 
 Let $r' \in [r/e, r)$. By~\eqref{eq:prod.analytic}, for each $k \in \N$,
 \begin{equation} \label{eq:tkfkphi.analytic}
    \opnorm{\frac{t^k }{k!} {f^k} \phi}_{r'}
    \leq 
    \frac{|t|^k}{k!} \frac{k!}{e} \left(\frac{e}{r-r'}\right)^{k} \opnorm{f}_{r}^{k} \opnorm{\phi}_r,
 \end{equation}
 so that the sum converges absolutely in $\CC^{\omega,r'}$ when $|t| e \opnorm{f}_{r} < r-r'$. Moreover, by~\eqref{efphi.1} with $f \leftarrow tf$ and \eqref{eq:prod.analytic}, 
 \begin{equation}
    \left\| \left( e^{tf} - \sum_{k=0}^{M} \frac{t^k }{k!} {f^k} \right)(\phi) \right\|_{\CC^0}
    \leq 
    \frac{|t|^{M+1}}{(M+1)!} \| {f^{M+1}} \phi \|_{\CC^0},
 \end{equation}
 where, using \eqref{eq:tkfkphi.analytic}, the right-hand side tends to zero as $M \to +\infty$ under the same smallness condition; so that the sum converges towards $e^{tf} \phi$ in $\CC^{\omega,r'}$ when $|t| e \opnorm{f}_{r} < r-r'$.
\end{proof}

\subsubsection{Pushforwards of vector fields by diffeomorphisms}

\begin{definition}[Pushforward of a vector field by a diffeomorphism]
 Let $\Omega, \Omega'$ be open subsets of~$\K^d$.
 Let $\theta \in \CC^1(\Omega;\Omega')$ be a local diffeomorphism from $\Omega$ to $\Omega'$. 
 Let $f \in \CC^0(\Omega;\K^d)$ be a vector field. We define $\theta_* f \in \CC^0(\Omega';\K^d)$ \emph{the pushforward of $f$ by $\theta$} as
 \begin{equation}
   (\theta_* f)(q) := (D\theta)_{\rvert \theta^{-1}(q)} f (\theta^{-1}(q))
   = (D\theta^{-1})^{-1}_{\rvert q} f(\theta^{-1}(q)).
 \end{equation}
\end{definition}

\begin{lemma}[Chain rule for pushforwards] \label{thm:push.composition}
 Let $\Omega, \Omega', \Omega''$ be open subsets of~$\K^d$.
 Let $\theta \in \CC^1(\Omega;\Omega')$ be a local diffeomorphism from $\Omega$ to $\Omega'$ and $\theta' \in \CC^1(\Omega';\Omega'')$ be a local diffeomorphism from $\Omega'$ to~$\Omega''$. 
 Let $f \in \CC^0(\Omega;\K^d)$ be a vector field. Then, on $\Omega''$,
 \begin{equation}
   \theta'_* (\theta_* f) = (\theta' \circ \theta)_* f.
 \end{equation}
\end{lemma}

\begin{proof}
 This is a consequence of the chain rule for differentiation, see e.g.\ \cite[Problem 12-10]{zbMATH06034615}.
\end{proof}

\begin{lemma}[Lie brackets of pushforwards] \label{thm:push.lie}
 Let $\Omega, \Omega'$ be open subsets of~$\K^d$.
 Let $\theta \in \CC^2(\Omega;\Omega')$ be a local diffeomorphism from $\Omega$ to $\Omega'$.
 Let $f,g \in \CC^1(\Omega;\K^d)$ be two vector fields. Then, on $\Omega'$,
 \begin{equation}
   [\theta_* f, \theta_* g] = \theta_* [f, g].
 \end{equation}
\end{lemma}

\begin{proof}
 This is a consequence of the chain rule for differentiation, see e.g.\ \cite[Corollary 8.31]{zbMATH06034615}.
\end{proof}

\subsubsection{Composition of vector fields with flows}

\begin{lemma}\label{Lem:tool_serie_ad}
    Let $\delta > 0$, {$f_0 \in \CC^1(B_{2\delta};\K^d)$} and $t \in \R$ such that $|t| \| f_0 \|_{\CC^0} < \delta$.
    Denote by $\Phi_0(t,p) := e^{tf_0}(p)$ the associated flow for $p \in B_\delta$. 
\begin{enumerate}
    \item For each $M \in \N$, if {$f_0, f_1 \in \CC^{M+1}(B_{2\delta};\K^d)$}, then, for each $p \in B_\delta$,
    \begin{equation} \label{Lem:tool_serie_ad_formule1}
        \left| {\left(\partial_p \Phi_0 (t,p)\right)}^{-1} f_1\left( \Phi_0(t,p) \right) - \sum_{k=0}^{M-1} \frac{t^k}{k!} \ad_{f_0}^k(f_1)(p) \right| \leq  \frac{t^M}{M!} \left\| \ad_{f_0}^M(f_1) \right\|_{\CC^0}.
    \end{equation}
   \item For each $M \in \N$, if {$f_0, f_1 \in \CC^{M+1}(B_{2\delta};\K^d)$} and $\ad_{f_0}^M(f_1) \equiv 0$, then, for each $p \in B_\delta$,
   \begin{equation} \label{eq:push.nilpotent}
       (\Phi_0(-t)_* f_1) (p) 
        = {\left(\partial_p \Phi_0 (t,p)\right)}^{-1} f_1\left( \Phi_0(t,p) \right) 
        = \sum_{k=0}^{M-1} \frac{t^k}{k!} \ad_{f_0}^k(f_1)(p).
   \end{equation}
   This holds in particular when $\mathcal{L}(\{f_0,f_1\})$ is nilpotent with index {at most} $(M+1)$.
    \item If $r>0$, {$f_0, f_1 \in \CC^{\omega,r}(B_{2\delta};\K^d)$}, then, for $|t|<\frac{r}{9 \opnorm{f_0}_r}$, for each $p \in B_\delta$,
    \begin{equation}\label{tool_series}
        (\Phi_0(-t)_* f_1) (p) 
        = {\left(\partial_p \Phi_0 (t,p)\right)}^{-1} f_1\left( \Phi_0(t,p) \right) 
        = \sum_{k=0}^{+\infty} \frac{t^k}{k!} \ad_{f_0}^k(f_1)(p),
    \end{equation}
    where, for every $r' \in [r/e,r)$ the series converges in {$\CC^{\omega,r'}(B_{2\delta};\K^d)$} when $|t|<\frac{r-r'}{6 \opnorm{f_0}_r}$.
    \item Let $H_0, H_1 \in \mathcal{M}_d(\K^d)$ and $M \in \N^*$. Then
    \begin{equation}
        \left\| e^{H_0} H_1 e^{-H_0} - \sum_{k=0}^{M-1} \frac{1}{k!} \ad_{H_0}^k(H_1) \right\| \leq \frac{(2\|H_0\|)^M}{M!} \|H_1\|  e^{2\|H_0\|}
    \end{equation}
    and
    \begin{equation}
        e^{H_0} H_1 e^{-H_0}= \sum_{k=0}^{+\infty} \frac{1}{k!} \ad_{H_0}^k(H_1),
    \end{equation}
    where $\ad$ is the commutator of matrices $\ad_A(B):=[A,B]=AB-BA$ and $\|\cdot\|$ a sub-multiplicative norm on $\mathcal{M}_d(\K)$ such that $\|\mathrm{Id}_d\|=1$.
\end{enumerate}
\end{lemma}

\begin{proof}
    We proceed step by step.
\begin{enumerate}
    \item First, for each $\tau \in [0,t]$, $\Phi_0(\tau,p)$ is well-defined.
    Taking into account that
\begin{equation}
 \begin{split}
  \frac{\dd}{\dd \tau}\left[ {\left(\partial_p \Phi_0 (t,p)\right)}^{-1} \right] 
  & = - {\left(\partial_p \Phi_0 (t,p)\right)}^{-1} \frac{\dd}{\dd \tau}\left[ \partial_p \Phi_0 (\tau,p) \right] {\left(\partial_p \Phi_0 (t,p)\right)}^{-1} \\
  & = - {\left(\partial_p \Phi_0 (t,p)\right)}^{-1} {D f_0}_{\rvert \Phi_0(\tau,p)},
 \end{split}
\end{equation}
one obtains by induction on $k\in\intset{0,M+1}$ that
\begin{equation}
 \frac{\dd^k}{\dd \tau^k}\left[ {\left(\partial_p \Phi_0 (t,p)\right)}^{-1} f_1\left( \Phi_0(\tau,p) \right) \right] = {\left(\partial_p \Phi_0 (t,p)\right)}^{-1} \ad_{f_0}^k(f_1)\left( \Phi_0(\tau,p) \right).
\end{equation}
The Taylor formula
\begin{equation} \label{tool_serie_ad:Taylor}
 \begin{split}
   {\left(\partial_p \Phi_0 (t,p)\right)}^{-1} f_1\left( \Phi_0(t,p) \right) 
   & - \sum_{k=0}^{M-1} \frac{t^k}{k!} \ad_{f_0}^k(f_1)(p) \\
   & = \int_0^t \frac{(t-s)^{M-1}}{(M-1)!} {\left(\partial_p \Phi_0 (s,p)\right)}^{-1} \ad_{f_0}^{M}(f_1)\left( \Phi_0(s,p) \right) \dd s    
 \end{split}
\end{equation}
     proves the first statement.
     
    \item Equation \eqref{tool_serie_ad:Taylor} yields the conclusion.
    
    \item Let $r' \in [r/e,r)$. 
    Thanks to \eqref{eq:bracket.analytic},
    \begin{equation} \label{daadf0kf1(x)}
        \opnorm{\frac{t^k}{k!} \ad^k_{f_0}(f_1)}_{r'}
        \leq \frac{|t|^k}{k!} \frac{k!}{e} \left(\frac{2e}{r-r'}\right)^k \opnorm{f_0}_r^k \opnorm{f_1}_r,
    \end{equation}
    so the series converges absolutely in $\CC^{\omega,r'}$ when $2 e |t| \opnorm{f_0}_r < r - r'$, which is the case when $6 |t| \opnorm{f_0}_r < r - r'$ because $2e<6$. The weakest bound, for $r'=r/e$ is $2 e |t| \opnorm{f_0}_r < (1-1/e)r$ and it holds when $9 |t| \opnorm{f_0}_r < r$ because
$2e/(1-1/e) < 9$.
    
    Moreover, thanks to \eqref{tool_serie_ad:Taylor} and \eqref{daadf0kf1(x)},
    \begin{equation} \label{eq:thm.compo.proof.3}
        \begin{split}
            \left| (\Phi_0(-t)_* f_1)(p) -\sum_{k=0}^{M-1} \frac{t^k}{k!} \ad_{f_0}^k(f_1)(p) \right|
            & \leq \frac{|t|^{M}}{M!} \| \ad^M_{f_0}(f_1) \|_{\CC^0} \sup_{s\in [0,t]} \| (\partial_p \Phi_0(s,\cdot))^{-1} \|_{\CC^0} \\
            & \leq A_0 \opnorm{f_1}_r \left(\frac{2e|t|\opnorm{f_0}_{r}}{r-r'}\right)^M,
        \end{split}
    \end{equation}
    where $A_0$ denotes the supremum in the right-hand side of \eqref{eq:thm.compo.proof.3} which is finite. So the sum converges towards the pushforward under the same smallness assumption on time.

    \item The last statement is proved similarly, by considering the function $t \mapsto e^{tH_0} H_1 e^{-tH_0}$.  \qedhere
\end{enumerate}

\end{proof}

\subsubsection{Partial derivative of a flow with respect to a parameter}

In this paragraph, we compute the partial derivative of a flow with respect to a parameter on which the vector field depends, under a particular {nilpotency assumption}.

\begin{lemma} \label{Prop:d(flot)/d(param)}
    Let $J$ an interval of $\R$. Let $\delta > 0$ and $f \in \CC^\infty(J \times B_{4\delta} ; \K^d)$ such that $\| f \|_{\CC^0} < \delta$.
    Let $\lambda_0 \in J$, $M \in \N$ and assume that, for each $\lambda \in J$, $\ad^M_{f(\lambda_0)} (f(\lambda)) \equiv 0$. Then, for each $p \in B_\delta$,
    \begin{equation}
        {\frac{\dd}{\dd \lambda}} \left( e^{f(\lambda)} p \right)_{\rvert \lambda = \lambda_0}
        =
        \sum_{k=0}^{M-1} \frac{(-1)^{k}}{(k+1)!} \ad^k_{f(\lambda_0)}\left(\partial_\lambda f (\lambda_0) \right)\left( e^{f(\lambda_0)} p \right).
    \end{equation}
    This holds in particular when $\mathcal{L}(f(J))$ is nilpotent with index at most $M+1$.
\end{lemma}

\begin{proof}
    Let $\Theta \in \CC^\infty([0,1] \times J \times B_\delta)$ defined by $\Theta(t,\lambda,p) := e^{tf(\lambda)}(p)$. 
    Let $p_0 \in B_\delta$ and $\lambda_0 \in J$.
    Let $x_0(t) := e^{t f(\lambda_0)}(p_0)$ for $t \in [0,1]$. 
    Then, the desired derivative is $\partial_\lambda \Theta(1,\lambda_0,p_0) = z(1)$ where $z$ is the solution to $z(0) = 0$ and
    \begin{equation}
        \dot{z}(t) = \partial_x f(\lambda_0, x_0(t)) z(t) + \partial_\lambda f(\lambda_0, x_0(t)).
    \end{equation}
    Let $R:(t,s) \in [0,1]^2 \to \mathcal{M}_d(\K)$ be the resolvent associated with the linearized system at $p_0$, which is the solution to $R(s,s) = \mathrm{Id}$ and
    \begin{equation}
        \partial_t R(t,s) = \partial_x f(\lambda_0, x_0(t)) R(t,s),
    \end{equation}
    i.e.\ $R(t,s)= \partial_p \Theta (t-s,\lambda_0,x_0(s))$. 
    Then by Duhamel's principle
    \begin{equation}
        \begin{split}
            z(1) & = \int_0^1 R(\tau,1)^{-1} \partial_\lambda f(\lambda_0, x_0(\tau)) \dd \tau \\
            & = \int_0^1 {\left( \partial_p \Theta(\tau-1,\lambda_0, x_0(1))\right)}^{-1} \partial_\lambda f(\lambda_0, \Theta(\tau-1, \lambda_0, x_0(1))) \dd \tau.
        \end{split}
    \end{equation}
    By \eqref{eq:push.nilpotent} of \cref{Lem:tool_serie_ad} with $t \leftarrow \tau -1$, $f_0 \leftarrow f(\lambda_0, \cdot)$, $f_1 \leftarrow \partial_\lambda f(\lambda_0, \cdot)$ and $p \leftarrow x_0(1)$,
    \begin{equation}
        z(1)
        = \int_0^1 \sum_{k=0}^{M-1} \frac{(\tau-1)^k}{k!} \ad^k_{f(\lambda_0)}\left( \partial_\lambda f(\lambda_0) \right) ( x_0(1) ) \dd \tau,
    \end{equation}
    which gives the conclusion.
\end{proof}

\section{Error estimates in time for nonlinear vector fields}
\label{sec:estimates}

Using a classical linearization {principle (see \cref{subsec:linear_trick})}, we show that the formal expansions for linear equations of \cref{sec:formal} can yield approximate formulas in the context of nonlinear ordinary differential equations. 
We derive rigorous error bounds at every fixed order with respect to time, involving finite sums or products.

\subsection{Linearization principle for nonlinear vector fields}
\label{subsec:linear_trick}

We explain how, by {seeing} vector fields {as} first-order differential operators and points on the manifold {as} the operator of evaluation at this point, one {classically} recasts a nonlinear ODE driven by smooth vector fields to a linear equation set on a larger space of operators on smooth functions.
{This approach is notably used in \cite{MR524203, MR579930} (replacing nonlinear objects by infinite-dimensional linear ones is the foundation of the ``chronological calculus'') and in \cite{MR886816}.
More generally, the idea of replacing the study of a space by the study of the ring of functions on that space is reminiscent of the representation results of \cite{zbMATH03097323}.
For readers with a background in PDE analysis, this linearization principle can be seen as a ``reversed method of characteristics'': it transforms a nonlinear ODE into a linear transport PDE, considered at the level of evolution operators.
}

\subsubsection{Definition of an operator acting on smooth functions}

When $T>0$ and $f \in \CC^\infty_c([0,T] \times \K^d{;\K^d})$ satisfies $\|f\|_{{L^1((0,T);\CC^0)}} < 1$, we take the nonlinear ODE~\eqref{ODE:f} back to a linear framework by considering, for every $t \in [0,T]$ the linear operator $L(t)$ on $\CC^\infty_c(\K^d;\K)$ defined, for $\varphi \in \CC^\infty_c(\K^d;\K)$, by
\begin{equation}
    L(t) \varphi : p \mapsto \varphi\left( x(t;f,p) \right).
\end{equation}
$L(t)\varphi$ is of class $\CC^\infty$ as a composition of $\CC^\infty$ functions, by the third statement of \cref{Lem:WP_ODE:f}. $L(t)\varphi$ is compactly supported in $\K^d$ because $\varphi$ is and $|x(t;f,p) - p| \leq 1$ for every $p\in\K^d$, by the first statement of \cref{Lem:WP_ODE:f} (which is of course invariant by translation of the origin).
We don't specify the dependence of $L(t)$ with respect to $f$ to simplify the notations.

For every $p \in \K^d$, the map $t \in [0,T] \mapsto \big(L(t)\varphi\big)(p)$ belongs to $\CC^1([0,T];\K)$ and satisfies, for every $t\in [0,T]$, {using the notations of \cref{rk:no-nabla},}
 \begin{equation}
  \frac{\dd}{\dd t} \big( L(t)\varphi \big)(p) 
  = D\varphi\Big( x(t;f,p) \Big)  f \Big( t, x(t;f,p) \Big) 
  = \big( L(t) {f(t)} \varphi \big)(p).    
 \end{equation}
 Thus, $L$ {satisfies} the following linear equation
 \begin{equation} \label{L_equation}
   \frac{\dd}{\dd t} L(t) =  L(t) {f(t)}
 \end{equation}
 in the weak sense explicited above. For every fixed $t \in [0,T]$,
 \begin{equation}
  \forall \varphi \in \CC^\infty_c(\K^d;\K), 
  \forall p\in\K^d, 
  \quad
  \big( L(t)\varphi \big)(p) 
  = \varphi(p) + \int_0^t \big( L(\tau) {f(\tau)} \varphi \big)(p) \dd\tau, 
 \end{equation} 
 where the symbol $\int_0^t$ is the Lebesgue integral on $L^1((0,t);\K)$. We will use the following notation to refer to this property:
 \begin{equation} \label{L_integrale}
   L(t) = \text{Id} + \int_0^t L(\tau) {f(\tau)} \dd\tau.
 \end{equation}
 In the sequel, all integral equalities between operators on $\CC^\infty_c(\K^d;\K)$ should be understood in this weak sense (after evaluation on a test function and at a point). The right-hand side refers to the composition of two operators on $\CC^\infty_c(\K^d;\K)$: {$L(\tau)$ and $f(\tau)$, seen as} a first-order differential operator on smooth functions. 

Equation \eqref{L_equation} is now a linear differential equation satisfied by the object $L(t)$ (in a much larger space), so one can hope to apply the linear results of the previous sections.

\subsubsection{Approximating sequence}

In order to approximate the operator $L(t)$, it is natural to introduce the sequence $(L_j)_{j \in \N}$ of time-dependent operators on $\CC^\infty_c(\K^d;\K)$ defined, for every $t \in [0,T]$, by $L_0(t) := \mathrm{Id}$ and, for $j \in \N$,
 \begin{equation} \label{eq:def-Lj}
    L_{j+1}(t) := \int_0^t L_j(\tau) {f(\tau)} \dd\tau,
 \end{equation}
 where this definition should be understood in the weak sense.
 Hence, {recalling \cref{rk:no-nabla},}
 \begin{equation} \label{Lj_t}
    L_j(t) = \int_{{\Delta^j(t)}} 
    {f(\tau_1) \dotsb f(\tau_j)} \dd\tau,
\end{equation}
 where the integration domain is {the ordered simplex of} \cref{def:simplex}.
 Then, for every $j\in\N$, $L_j$ is ``of order $j$ with respect to $f$'', and a differential operator of order at most $j$ (with respect to~$x$) on $\CC^\infty_c(\K^d;\K)$. And this sequence indeed allows to approximate $L(t)$ {in the sense exposed in \cref{Prop:error_CF} below, in a finite regularity setting.}
 
\subsection{Chen-Fliess expansion}
\label{subsec:error_ID}

The approximating sequence for the operator $L(t)$ yields the following straight-forward estimate for the Chen-Fliess expansion of the state, {which can also be found in \cite[Proposition 2.3]{MR524203}.}

\begin{proposition} \label{Prop:error_CF}
    For every $M \in \N$, $\delta > 0$, $T > 0$, {$f \in L^1((0,T);\CC^{\max(M,1)}(B_{2\delta};\K^d))$}, with $\|f\|_{{L^1((0,T);\CC^0)}} < \delta$ and $\varphi \in \CC^{M+1}(B_{2\delta};\K)$, for each $t \in [0,T]$, {with the notations of \cref{rk:no-nabla},}
    \begin{equation} \label{error_CF_phi_x}
        \left| \varphi\left(x(t;f,p)\right) - \sum_{j=0}^{M} \int_{{\Delta^j(t)}} \big(
    {f(\tau_1) \dotsb f(\tau_j) \varphi} \big) (p) \dd\tau  \right| \leq (M+1)! \|f\|_{{L^1((0,t);\CC^{M})}}^{M+1} \| \varphi\|_{\CC^{M+1}}.
    \end{equation}
    In particular, for each $p \in B_\delta$,
    \begin{equation} \label{error_CF_x}
        \left| x(t;f,p) - \sum_{j=0}^{M} \int_{{\Delta^j(t)}} \big( { f(\tau_1)  \dotsb f(\tau_j) \mathrm{Id}_d } \big) (p) \dd\tau  \right|
     \leq 
        (M+1)! \|f\|_{{L^1((0,t);\CC^{M})}}^{M+1}.
    \end{equation}
    Hence, if $f \in L^\infty((0,T);\CC^M)$, both estimates correspond to a bound scaling like $t^{M+1}$.
\end{proposition}

\begin{proof}
    Let $p \in B_\delta$. Thanks to \cref{Lem:WP_ODE:f}, $x(\tau;f,p)$ is well-defined for $\tau \in [0,T]$ and $x(\cdot;f,p) \in \CC^1([0,T];\K^d)$. Thus, for each $\tau \in [0,T]$,
    \begin{equation}
        \varphi(x(\tau;f,p)) = \varphi(p) + \int_0^\tau \big( {f(\tau_1)} \varphi \big)(x(\tau_1;f,p)) \dd \tau_1.
    \end{equation}
    By iterating this formula, we obtain for $t \in [0,T]$,
    \begin{equation} \label{CF:reste_expl}
    \begin{aligned} 
    \varphi\left(x(t;f,p)\right) & - \varphi(p) - \sum_{j=1}^{M} \int_{{\Delta^j(t)}} \big(
    {f(\tau_1) \dotsb f(\tau_j) \varphi} \big) (p) \dd\tau 
    \\ & =
    \int_{{\Delta^{M+1}(t)}} \big( { f(\tau_1) \dotsb f(\tau_{M+1}) \varphi} \big) \left(x(\tau_{M+1};f,p)\right) \dd\tau,
    \end{aligned}
    \end{equation}
    which concludes the proof {of \eqref{error_CF_phi_x} using \cref{p:prod-diff-ck}.}
    {Then \eqref{error_CF_x} follows by applying \eqref{error_CF_phi_x} to coordinate functions.}
\end{proof}

\subsection{Magnus expansion in the usual setting}

In \cref{subsec:standarderror}, we state a precise estimate of the difference between the exact flow and the exponential of its truncated logarithm.
In \cref{sec:estimate-cbhd}, we show that this estimate implies a similar estimate for the CBHD formula.
\cref{subsec:products<-brackets} is devoted to a technical result used in the proof, which transposes to vector fields a formal integral identity.

\subsubsection{Standard error estimate in time}
\label{subsec:standarderror}

The following estimate can be viewed as a refined version of classical time-focused estimates (see e.g.\ \cite[Proposition 4.3]{MR793239}). It bears a lot of similarity with \cite[Theorem 1.32]{driver2018truncated}, but is both easier to state and to prove in our flat setting since \cite{driver2018truncated} is concerned with the truncated logarithm of flows in general Riemannian manifolds. We propose a proof for sake of completeness, and because this precise estimate is the founding principle of the new estimate, proved in the next section. {It relies on the usual arguments, used for instance in \cite{MR886816} and \cite[Proposition 4.1]{MR524203} (which states a slightly tighter estimate).}

\begin{proposition} \label{Prop:Magnus_1}
 For every $M\in\N$, there exists $\delta_M, C_M>0$ such that, 
 for every $\delta>0$, $T>0$, {$f \in L^1((0,T);\CC^{\max(M^2,1)}(B_{2\delta};\K^d))$} with $\|f\|_{{L^1((0,T);\CC^{M^2})}} \leq \delta_M \min \{1;\delta\}$, $p\in B_\delta$ and $t\in[0,T]$,
 \begin{equation} \label{Estim_Magnus}
  \left| x(t;f,p) - e^{Z_M(t,f)} p  \right| \leq C_M \|f\|^{M+1}_{{L^1((0,t);\CC^{M^2})}},
 \end{equation}
 where $Z_M(t,f):= \mathrm{Log}_M\{f\}(t)$ is the vector field introduced in \cref{def:LOGM}.
 
 Hence, if $f \in L^\infty((0,T);{\CC^{M^2}(B_{2\delta};\K^d)})$ this estimate corresponds to a bound scaling like $t^{M+1}$.
 
    Moreover, if $f(t,x)=\sum_{i \in I} u_i(t) f_i(x)$ with $u_i \in L^1((0,T);\K)$ and $f_i \in {\CC^{M^2}(B_{2\delta};\K^d)}$, then, for each monomial basis $\mathcal{B}$ of $\mathcal{L}(X)$,
\begin{equation} \label{ZM(f)=coord1.0}
Z_M(t,f)=\sum_{b\in\mathcal{B}_{\llbracket 1,M \rrbracket}} \zeta_b(t,u) f_b
\end{equation}
where the functionals $\zeta_b$ are the associated coordinates of the first kind and $f_b$ are the evaluated Lie brackets (see Definitions \ref{Def:monomial_basis}, \ref{def:coord1} and \ref{Def:evaluated_Lie_bracket}).
\end{proposition}

\begin{proof}
    For $M=0$, $Z_0(t,f)=0$ thus (\ref{Estim_Magnus}) holds with $C_0=1$ because $|x(t;f,p)-p| \leq \|f\|_{L^1(\CC^0)}$. 
    From now on $M\in\N^*$ is fixed. 
    By \cref{def:LOGM}, there exists $C_M'>0$ such that, for every $\delta>0$, $T>0$, $f\in L^1((0,T);{\CC^{M-1}})$ with $\|f\|_{L^1(\CC^{M-1})} \leq 1$ and $t\in [0,T]$,
    \begin{equation}
    \|\mathrm{Log}_M \{f\}(t)\|_{{\CC^0}} 
    \leq 
    C_M' \|f\|_{{L^1(\CC^{M-1})}}.
    \end{equation}
    In particular, for every $\delta>0$, $T>0$, $f\in L^1((0,T);{\CC^{M-1}})$ with $\|f\|_{L^1(\CC^{M-1})} \leq \min\{1;\delta;\delta/C_M'\}$, for every $p \in B_\delta$ and $t \in [0,T]$,
    \begin{itemize}
        \item $x(t;f,p)$ is well defined and belongs to $B_{2\delta}$,
        \item for every $s\in [0,1]$, $e^{s \mathrm{Log}_M \{f\}(t)} p$ is well defined belongs to $B_{2\delta}$.
    \end{itemize}
    This happens, in particular, when $\|f\|_{L^1(\CC^{M-1})} \leq \delta_M \min\{1;\delta\}$ with $\delta_M:=\min\{1;1/C_M'\}$. 

    \bigskip

    From now on, we fix $\delta, T>0$ and $f\in L^1((0,T);{\CC^{M^2}})$ with $\|f\|_{L^1(\CC^{M^2})} \leq \delta_M \min \{1;\delta\}$.

\medskip

In order to use the operators $L(t)$ defined in \cref{subsec:linear_trick}, we assume that $f \in \CC^\infty_c([0,T] \times \K^d{;\K^d})$. This is not restrictive because this space is dense in $L^1((0,T);\CC^{M^2}{(B_{2\delta};\K^d)})$ and both sides of~\eqref{Estim_Magnus} are continuous for the $L^1((0,T);\CC^{M^2})$ topology on $f$ {(see the fourth item of \cref{Lem:WP_ODE:f}).
Moreover, this regularization procedure is merely an heuristical convenience, since all the computations performed below make perfect sense even in our finite regularity setting.}

\medskip
 
 \textbf{Step 1:  Construction of the formal logarithm.} We introduce $Z_M(t,f)$ the finite sum of terms ``of order at most $M$ with respect to $f$'' in the following formal power series (recall the formal power series for $\log (1+x)$):
 \begin{equation} \label{logLT}
  \mathrm{log~} L(t) = \sum_{m\in\N^*} \frac{(-1)^{m-1}}{m} \left( \sum_{j\in\N^*} L_j(t) \right)^m,
 \end{equation}
 {with the notation of \eqref{eq:def-Lj}}, 
 i.e.\ we define
 \begin{equation} \label{ZM_produit}
  Z_M(t,f) := \sum_{r=1}^M \sum_{m=1}^{M} \frac{(-1)^{m-1}}{m} \sum_{\mathbf{r} \in \N^m_r} L_{\mathbf{r}_m}(t) \dotsb L_{\mathbf{r}_1}(t),
 \end{equation}
 where $\N^m_r$ is defined in \eqref{def:Nmr}. For instance,
 \begin{equation}
  Z_3 = L_1 
  + \left( L_2-\frac{1}{2} L_1^2 \right) 
  + \left( L_3-\frac{1}{2} \left( L_1 L_2 + L_2 L_1 \right) + \frac{1}{3} L_1^3 \right).
 \end{equation}
 Then, by \eqref{Lj_t},
 \begin{equation}
  Z_M(t,f) = 
  \sum_{r=1}^M 
  \sum_{m=1}^{M} \frac{(-1)^{m-1}}{m}
  \sum_{{\mathbf{r}} \in \N^m_r} 
  \int_{{\Delta^{\mathbf{r}}(t)}}
  {
  f(\tau_1) \dotsb f(\tau_r)} \dd\tau, 
 \end{equation}
 \emph{A priori}, $Z_M(t,f)$ is thus an inhomogeneous differential operator on $\CC^\infty_c(\K^d;\K)$, of order at most~$M$. Using \cref{thm:logarithm-vf} (see below in the next paragraph) and \cref{def:LOGM}, $Z_M(t,f) = \mathrm{Log}_M \{f\}(t)$ and satisfies (\ref{ZM(f)=coord1.0}). 
 Thus $Z_M(t,f)$ is {a smooth vector field, i.e.\ both a vector-valued function and a first-order differential operator}.

 \bigskip

\textbf{Step 2: Strategy for the proof of the estimate.}
The key observation is that it is sufficient to prove that there exists $C_M > 0$ (independent of $\delta, T, f$) such that, for every $p\in B_\delta$, $t\in [0,T]$ and $\varphi \in \CC^\infty_c(\K^d;\K)$,
\begin{equation} \label{Estimee_phi}
 \left|  \Big(L(t) - e^{Z_M(t,f)}\Big)(\varphi)(p) \right| \leq C_M  \|f\|_{L^1(\CC^{M^2})}^{M+1} \| \varphi \|_{\CC^{M^2+1}}.
\end{equation}
Then, the conclusion follows by considering an appropriate $\CC^\infty_c$ truncation of the coordinate functions $\varphi_j : x \in \K^d \mapsto x_j \in \K$. To prove \eqref{Estimee_phi}, we will decompose the difference in three terms
\begin{equation}
L-e^{Z_M} = \left(  L-\sum_{j=0}^{M} L_j \right)
+ \left( \sum_{j=0}^{M} L_j - \sum_{k=0}^{M} \frac{Z_M^k}{k!}\right) + \left( \sum_{k=0}^{M} \frac{Z_M^k}{k!} - e^{Z_M} \right),
\end{equation}
{with the notation of \eqref{eq:def-Lj}}.
The first term is estimated in \cref{Prop:error_CF}. 

\bigskip

\textbf{Step 3: Bound for $\sum L_j - \sum \frac{Z_M^k}{k!}$.} By \eqref{ZM_produit}, this operator is a (finite) linear combination of terms of the form $L_{j_1}(t) \dotsb L_{j_p}(t)$ where $p\in\N^*$, $j_1,\dotsc,j_p \in \intset{1,M}$ and $M+1 \leq j_1+ \dotsc +j_p \leq M^2$. Indeed, $Z_M(t,f)$ is also the finite sum of terms ``of order at most $M$ with respect to $f$'' in the formal power series~\eqref{logLT}. Thus, there exists $C_M''>0$ (independent of $\delta, T, f$) such that, for every $p\in B_\delta$, $t\in [0,T]$ and $\varphi \in \CC^\infty_c(\K^d;\K)$,
\begin{equation} \label{Estimee_phi.2}
 \left| \left(\sum_{j=0}^{M} L_j(t) - \sum_{k=0}^{M} \frac{Z_M(t,f)^k}{k!}\right) (\varphi)(p) \right|
 \leq C_M'' 
 \|f\|_{L^1(\CC^{M^2-1})}^{M+1} \| \varphi \|_{\CC^{M^2}}.  
\end{equation}

\bigskip

\textbf{Step 4: Bound for $\sum\frac{Z_M^k}{k!} - e^{Z_M}$.} Using \cref{Lem:flow} for the time-independent vector field $Z_M(t,f)$ (where $t \in [0,T]$ has been fixed), estimate \eqref{estimate.ef.phi} yields for every $p\in B_\delta$, $t\in [0,T]$ and $\varphi \in \CC^\infty_c(\K^d;\K)$,
\begin{equation}
 \left| \left( e^{Z_M(t,f)} - \sum_{k=0}^{M} \frac{Z_M(t,f)^k}{k!} \right) (\varphi)(p) \right|
 \leq \| Z_M(t,f) \|_{\CC^M}^{M+1} {\| \varphi \|_{\CC^{M+1}}}.
\end{equation}
We deduce from \eqref{eq:LOGM} the existence of $C_M'''>0$ (independent of $\delta,T, f$) such that for every $t\in[0,T]$
\begin{equation}
 \|Z_M(t,f)\|_{\CC^M} \leq C_M''' 
 \|f\|_{L^1((0,t);\CC^{2M-1})}.
\end{equation}
Hence, for every $p\in B_\delta$, $t\in [0,T]$ and $\varphi \in \CC^\infty_c(\K^d;\K)$
\begin{equation} \label{Estimee_phi.3}
 \left| \left( e^{Z_M(t,f)} - \sum_{k=0}^{M} \frac{Z_M(t,f)^k}{k!} \right) (\varphi)(p) \right|
 \leq (C_M''')^{M+1}
 \|f\|_{L^1(\CC^{2M-1})}^{M+1} {\| \varphi \|_{\CC^{M+1}}}.
\end{equation}
Gathering \eqref{error_CF_phi_x}, \eqref{Estimee_phi.2} and \eqref{Estimee_phi.3} concludes the proof of \eqref{Estim_Magnus}.
\end{proof}

\subsubsection{Campbell Baker Hausdorff Dynkin formula} 
\label{sec:estimate-cbhd}

We deduce from \cref{Prop:Magnus_1} the following estimate for the classical CBHD formula with $q$ time-independent vector fields.

\begin{corollary} \label{Prop:CBH}
    For every $M \in \N^*$, there exists $\delta_M, C_M>0$ such that, for every $\delta>0$, $q\in\N^*$, $f_1,\dotsc,f_q \in {\CC^{M^2}(B_{2\delta};\K^d)}$ with $\sum_{1\leq j\leq q} \|f_j\|_{\CC^{M^2}} \leq \delta_M\min\{1;\delta\}$,
 \begin{equation} \label{CBH_estimee_M}
  \left\| e^{f_q} \dotsb e^{f_1} - e^{\CBHD_M(f_1,\dotsc,f_q)} \right\|_{{\CC^0}} \leq C_{M}  \|f\|^{M+1}   
 \end{equation}
 where $\CBHD_M(f_1,\dotsc,f_q)=\mathrm{Log}_M \{f\}(q)$, where the time-dependent vector field $f$ is defined by
 $f:(t,x) \in [0,q] \times B_{2\delta} \mapsto \sum_{j=1}^q 1_{[j-1,j]}(t) f_j(x)$ and $\|f\| := \|f\|_{L^1(\CC^{M^2})}=\sum_{1\leq j\leq q} \|f_j\|_{\CC^{M^2}}$.
 
 Moreover, for each monomial basis $\mathcal{B}$ of $\mathcal{L}(\{X_1,\dotsc,X_q\})$
 \begin{equation} \label{eq:cbhdm_f}
 \CBHD_M(f_1,\dotsc,f_q) = \sum_{b \in \mathcal{B}_{\intset{1,M}}} \alpha_b f_b\end{equation}
 where $(\alpha_b)_{b\in\mathcal{B}} \subset \K^{\mathcal{B}}$ is given by \cref{Cor:CBH_formel}.
\end{corollary}

\begin{proof}
Because of the particular form of $f$, we have $x(t;f,p)=e^{f_q} \dotsm e^{f_1} p$. Thus the estimate (\ref{CBH_estimee_M}) is an application of \cref{Prop:Magnus_1}. Let $\Lambda:\mathcal{L}(\{X_1,\dotsc,X_q\}) \rightarrow \mathcal{L}(\{f_1,\dotsc,f_q\})$ be the homomorphism of Lie algebras such that $\Lambda(X_j)=f_j$. The map $\CBHD_M$ is defined by a finite sum  of Lie brackets, thus it commutes with $\Lambda$
\begin{equation}
\CBHD_M(f_1,\dotsc,f_q)=
\Lambda(\CBHD_M(X_1,\dotsc,X_q))=
\Lambda \left( \sum_{b \in \mathcal{B}_{\intset{1,M}}} \alpha_b b \right)=
\sum_{b \in \mathcal{B}_{\intset{1,M}}} \alpha_b \Lambda(b),
\end{equation}
which proves \eqref{eq:cbhdm_f}.
\end{proof}

\subsubsection{Replacing products with brackets in logarithm integrals}
\label{subsec:products<-brackets}

The goal of this section is to prove \cref{thm:logarithm-vf}, which is a key point in the proof of \cref{Prop:Magnus_1}, as it allows to replace products of differential operators with Lie brackets in the integrals involved in the computation of the logarithm of the flow.

We first state and prove a corollary of \cref{thm:formal.log} in algebras. Indeed, \cref{thm:formal.log} is a statement about formal differential equations, but it has consequences for concrete realizations, e.g.\ for systems governed by vector fields or matrices (this will be used in \cref{subsubsec:Magnus_CV_matrix}).

\begin{corollary} \label{thm:a.brackets}
 Let $A$ be a unital associative algebra over $\K$ and $A_1$ be a finite dimensional linear subspace of $A$. Then, for every $r \in \N^*$, $t > 0$ and $\mathrm{a} \in L^1((0,t); A_1)$, one has
 \begin{equation} \label{eq:a.brackets}
  \begin{split}
 \sum_{m=1}^{r} \sum_{\mathbf{r} \in \N^m_r}
\frac{(-1)^{m-1}}{m} & \int_{{\Delta^{\mathbf{r}}(t)}} { \mathrm{a}(\tau_1) \mathrm{a}(\tau_2) \dotsm \mathrm{a}(\tau_r)} \dd\tau
= \\
& \frac{1}{r} \sum_{m=1}^{r} \sum_{\mathbf{r} \in \N^m_r}
\frac{(-1)^{m-1}}{m} \int_{{\Delta^{\mathbf{r}}(t)}} {[\dotsb [\mathrm{a}(\tau_1) , \mathrm{a}(\tau_2)], \dotsc \mathrm{a}(\tau_r)]} \dd\tau,
  \end{split}
 \end{equation}
 where the equality should be seen as an equality between elements of a finite dimensional linear subspace of $A$ (generated by monomials of terms in $A_1$ of degree $r$), so that one can give a meaning to the integrals without introducing any topology on $A$.
 
 Moreover,  if $\mathrm{a}(\tau) = \sum_{i\in I} \alpha_i(\tau) \mathrm{y}_i$ with $\alpha_i \in L^1((0,t);\K)$ and $\mathrm{y}_i \in \mathrm{A}$ then, for each monomial basis $\mathcal{B}_r$ of $\mathcal{L}_r(X)$,
 \begin{equation} \label{Zrhom=Coord1.0_algebre}
  \frac{1}{r} \sum_{m=1}^{r} \sum_{\mathbf{r} \in \N^m_r}
\frac{(-1)^{m-1}}{m} \int_{{\Delta^{\mathbf{r}}(t)}} {[\dotsb [\mathrm{a}(\tau_1) , \mathrm{a}(\tau_2)], \dotsc \mathrm{a}(\tau_r)]} \dd \tau
= \sum_{b \in \mathcal{B}_r} \zeta_b(t,\alpha) \mathrm{y}_b,
 \end{equation}
 where the functionals $\zeta_b$ are the associated coordinates of the first kind and $\mathrm{y}_b=\Upsilon(b)$ where $\Upsilon: \mathcal{A}(X) \rightarrow A$ is the homomorphism of {algebras} such that $\Upsilon(X_i)=\mathrm{y}_i$ (see \cref{def:coord1} and \Cref{Lem:identification_libre}).
\end{corollary}

\begin{proof}
 Let $q \in \N^*$ be the dimension of $A_1$ (as a linear subspace) and $\mathrm{y}_1, \dotsc \mathrm{y}_q$ be a linear basis of~$A_1$. 
 Let $\alpha_i \in L^1((0,t);\K)$ denote the components of $\mathrm{a}(\cdot)$ in the basis $\mathrm{y}_1, \dotsc \mathrm{y}_q$, i.e.\ $\mathrm{a}(\tau) = \alpha_1(\tau) \mathrm{y}_1 + \dotsc + \alpha_q(\tau) \mathrm{y}_q$ for almost every $\tau \in [0,t]$.
 Then $\mathrm{a}(t) = \Upsilon(a(t))$ where $a(\tau) := \alpha_1(\tau) X_1 + \dotsc + \alpha_q(\tau) X_q \in \mathcal{A}_1(X)$.
 From \eqref{eq:log.xta.before} and \eqref{eq:log.xta.after}, one obtains that \eqref{eq:a.brackets} holds for $a(\cdot)$. Applying the homomorphism {$\Upsilon$ of algebras} to both sides proves~\eqref{eq:a.brackets} for $\mathrm{a}(\cdot)$. 
 The same strategy proves~\eqref{Zrhom=Coord1.0_algebre}.
\end{proof}

\begin{lemma} \label{thm:logarithm-vf}
 For every $r \in \N^*$, $t>0$ and {$f \in \CC^\infty_c([0,t]\times\K^d;\K^d)$}, 
 \begin{equation} \label{eq:f.brackets}
  \begin{split}
 \sum_{m=1}^{r} \sum_{\mathbf{r} \in \N^m_r}
\frac{(-1)^{m-1}}{m} & \int_{{\Delta^{\mathbf{r}}(t)}} { f(\tau_1) f(\tau_2) \dotsm f(\tau_r)} \dd\tau
= \\
& \frac{1}{r} \sum_{m=1}^{r} \sum_{\mathbf{r} \in \N^m_r}
\frac{(-1)^{m-1}}{m} \int_{{\Delta^{\mathbf{r}}(t)}} { [\dotsb [f(\tau_1), f(\tau_2)], \dotsc f(\tau_r)]} \dd\tau,
  \end{split}
 \end{equation}
 which should be seen as an equality between linear operators on $\CC^\infty_c(\K^d;\K)$, hence only valid after evaluation at a function $\varphi$ at a point $p$, so that the integrals are integrals of {scalar-valued functions}.
 
 Moreover,  if $f(\tau,x) = \sum_{i \in I} u_i(\tau) f_i(x)$ with $u_i \in L^1((0,t);\K)$ and $f_i \in \CC^\infty_c(\K^d;\K^d)$ then
  \begin{equation} \label{eq:f.brackets.univ}
    \frac{1}{r} \sum_{m=1}^{r} \sum_{\mathbf{r} \in \N^m_r}
\frac{(-1)^{m-1}}{m} \int_{{\Delta^{\mathbf{r}}(t)}} {[\dotsb [f(\tau_1), f(\tau_2)], \dotsc f(\tau_r)]} \dd\tau 
     = \sum_{b \in \mathcal{B}_r} \zeta_b(t,u)
   {f_b},
 \end{equation}
 where $\mathcal{B}_r$ is a monomial basis of $\mathcal{L}_r(X)$, the functionals $\zeta_b$ are the associated coordinates of the first kind and $f_b$ are the evaluated Lie brackets  (see Definitions \ref{def:free.lie}, \ref{def:coord1} and \ref{Def:evaluated_Lie_bracket}).
\end{lemma}

\begin{proof}
    Let $(f_n)_{n\in\N^*}$ be a sequence of functions in {$\CC^\infty_c([0,t]\times\K^d;\K^d)$} such that $f_n$ takes values in an at-most $n$-dimensional vector subspace $E_n$ of $\CC^\infty_c(\K^d;\K^d)$ and $\|f_n-f\|_{L^1((0,t);\CC^{r})} \rightarrow 0$ when $n\rightarrow \infty$. For example, one can choose an $n$-points trapezoidal approximation of $f$. For each fixed $n$, applying \cref{thm:a.brackets} with $A=\mathrm{Op}(\CC^\infty_c(\K^d;\K))$ and {$A_1 = E_n$} proves \eqref{eq:f.brackets} for $f_n$. 
    Let $\varphi \in \CC^\infty_c(\K^d;\K)$ and $p \in \K^d$. For each $n \in \N^*$, we deduce that
 \begin{equation}
  \begin{split}
    \sum_{m=1}^{r} \sum_{\mathbf{r} \in \N^m_r} &
\frac{(-1)^{m-1}}{m} \int_{{\Delta^{\mathbf{r}}(t)}} {\left( f_n(\tau_1) \dotsm f_n(\tau_r) \varphi \right)} ( p) \dd\tau = \\
    & \frac 1 r \sum_{m=1}^{r} \sum_{\mathbf{r} \in \N^m_r}
\frac{(-1)^{m-1}}{m} \int_{{\Delta^{\mathbf{r}}(t)}} {\left( [\dotsb [f_n(\tau_1), f_n(\tau_2)], \dotsc f_n(\tau_r)] \varphi \right)} (p) \dd\tau.
  \end{split}
 \end{equation}
 For each fixed $\varphi$ and $p$, both sides converge as $n \to +\infty$ towards the same quantities for $f$. This proves that \eqref{eq:f.brackets} holds as an equality between linear operators. Applying (\ref{Zrhom=Coord1.0_algebre}) gives (\ref{eq:f.brackets.univ}).
\end{proof}

\begin{remark}
 Although most algebraic results of \cref{sec:formal} remain valid for infinite alphabets (sets of indeterminates), there is a difficulty when one wishes to ``evaluate'' equalities in the free algebra over an infinite alphabet towards some target algebra (one must somehow introduce compatible topologies on both sides).
 Our approach to prove \cref{thm:logarithm-vf}, where $f$ is allowed to take values in the infinite-dimensional space $\CC^\infty_c$, therefore relies on a discretization scheme to return to a finite alphabet, and the convergence of the involved integrals in a weak sense. 
 Another approach, followed in \cite{MR1852645, MR2010037}, consists in introducing definitions allowing an infinite (continuous) number of generators and proving analogous algebraic results in such a setting.
\end{remark}

\subsection{Magnus expansion in the interaction picture}
\label{Sec:Magnus2}

\newcommand{\fp}{{f_\sharp}}

In this section, we consider the nonlinear ordinary differential equation
\begin{equation} \label{eq:EDO.interaction}
    \dot{x}(t)=f_0(x)+\fp(t,x)
\end{equation}
We show how the formal expansion introduced in \cref{sec:format.1.1} allows to obtain error bounds at every order in the size of the time-varying perturbation $\fp$, provided that the flow of $f_0$ is known.
Such estimates can be useful for example to design splitting methods in the case of a small perturbation (see e.g.\ \cite[Section 3.6]{zbMATH06057121} or \cite[Section 2]{zbMATH05831260}). 

{The results of this section can be seen as quite natural, but we are not aware of references containing the same statements.
We adopt the notations specified in the following definition.

\begin{definition} \label{Def:Phi0_gt_ZM}
    Let $M \in \mathbb{N}$, $\delta>0$, $T>0$ and $f_0 \in \CC^{M^2+1}(B_{5\delta};\K^d)$ such that $T \|f_0\|_{\CC^0}<\delta$.
    Let $\fp \in L^1((0,T);\CC^{M^2}(B_{5\delta};\K^d))$ and $t \in [0,T]$. We consider
    \begin{itemize}
        \item $\Phi_0 \in \CC^{M^2+1}([0,T]\times B_{4\delta};B_{5\delta})$ the flow associated with $f_0$ i.e.\ $\Phi_0(\tau;p)=e^{\tau f_0}(p)$,
        \item $g_t \in L^1((0,T);\CC^{M^2}(B_{4\delta};\K^d))$ defined by
\begin{equation} \label{def:gt_star} 
 g_t(\tau,y)
 := (\Phi_0(t-\tau)_* \fp(\tau)) (y) 
 = \left(\partial_p \Phi_0 (\tau-t,y)\right)^{-1} \fp \big( \tau , \Phi_0(\tau-t,y)\big),
\end{equation}
        \item $\mathcal{Z}_M(t,f_0,\fp) := \mathrm{Log}_M\{g_t\}(t) \in \CC^{M^2-M+1}(B_{4\delta};\K^d)$ in the sense of \cref{def:LOGM}.
    \end{itemize}
\end{definition}}

\subsubsection{Error bound}

\begin{proposition} \label{Prop:Magnus_2}
    Let $M, \delta, T, f_0, \fp$ as in \cref{Def:Phi0_gt_ZM}.
There exists $\gamma=\gamma(M,\delta,\|f_0\|_{\CC^{M^2+1}})>0$ such that, if
\begin{equation} \label{Magnus1.1:hyp_f1}
\|\fp\|_{{L^1((0,T);\CC^{M^2})}} \leq \gamma
\end{equation}
then, for every $p\in B_\delta$ and $t\in[0,\gamma]$,
\begin{equation} \label{Magnus_2}
\left| x(t;f_0+\fp,p)- e^{\mathcal{Z}_M(t,f_0,\fp)} e^{t f_0} p \right| \leq  C_M  \|g_t\|^{M+1}_{{L^1((0,t);\CC^{M^2})}}
\end{equation}
where $C_M>0$ is the constant of \cref{Prop:Magnus_1}.

Hence, if $\fp \in L^\infty((0,T);{\CC^{M^2}(B_{5\delta};\K^d)})$, estimate (\ref{Magnus_2}) scales like $t^{M+1}$.
\end{proposition}

\begin{proof}
    Let $\eta_M, C_M>0$ be as in \cref{Prop:Magnus_1}. 
   There exists $T^*=T^*(\delta,\|f_0\|_{\CC^{M^2+1}}) \in (0,T]$ such that, for every $\fp \in L^1((0,T);{\CC^{M^2}(B_{5\delta};\K^d)})$ and $t \in [0,T^*]$
        \begin{equation}\label{gt<f1}
        \|g_t\|_{{L^1((0,t),\CC^{M^2})}} \leq 2 \|\fp\|_{{L^1((0,t),\CC^{M^2})}}.
        \end{equation}
  Let $\gamma:=\min\{T^*,\delta,\frac{\delta_M}{2}\min\{1,\delta\}\}$.
  Let $\fp \in L^1((0,T);{\CC^{M^2}(B_{5\delta};\K^d)})$ with $\|\fp\|_{L^1(\CC^{M^2})} < \gamma$.
Then, for every $p\in B_\delta$ and $\tau \in[0,T]$, $x(\tau;f_0+\fp,p)$ is well defined and belongs to $B_{3\delta}$.
To simplify the notations in this proof, we write $x(\tau)$ instead of $x(\tau;f_0+\fp,p)$. Let $t \in [0,\gamma]$. The function $y:[0,t] \rightarrow \K^d$ defined by
\begin{equation}\label{x/y}
y(\tau):=\Phi_0 \big( t-\tau ; x(\tau) \big)
\end{equation}
takes values in $B_{4\delta}$ and satisfies, for every $\tau \in [0,t]$,
\begin{equation}
    \dot{y}(\tau)= g_t \left( \tau,y(\tau) \right).
\end{equation}
By (\ref{gt<f1}), $\|g_t\|_{{L^1(\CC^{M^2)}}} < 2 \eta \leq \delta_M \min\{1,\delta\}$ thus, by  \cref{Prop:Magnus_1}
\begin{equation}
| y(t)-e^{\mathcal{Z}_M(t,f_0,\fp)} y(0) | \leq  C_M  \| g_t \|_{L^1((0,t);\CC^{M^2})}^{M+1}
\end{equation}
which is exactly (\ref{Magnus_2}) because $y(t)=x(t)$ and $y(0)=e^{t f_0} p$.
\end{proof}

\subsubsection{Expansions of $\mathcal{Z}_M$}

\begin{proposition} \label{Prop:expansion_ZM_1.1_analytic}
Let $M, \delta, T, f_0, \fp$ as in \cref{Def:Phi0_gt_ZM}.
Let $r>0$.
If $f_0 \in {\CC^{\omega,r}(B_{5\delta};\K^d)}$ and $\fp\in \CC^0([0,T];{\CC^{\omega,r}(B_{5\delta};\K^d)})$ then, for $0 \leq \tau \leq t \leq \min\{ T ; \frac{r}{9 \opnorm{f_0}_r} \}$
\begin{equation} \label{gt_series}
    g_t(\tau,\cdot) = e^{(\tau-t) \ad_{f_0}}(\fp(\tau)) =  \sum_{k=0}^{+\infty} \frac{(\tau-t)^k}{k!}\ad_{f_0}^k(\fp(\tau))
\end{equation}
and
\begin{equation} \label{expansion_ZMtilde}
    \begin{aligned}
    \mathcal{Z}_M(t,f_0,\fp) =  \sum 
    \frac{(-1)^{m-1}}{r m}
    & \int_{{\Delta^{\mathbf{r}}(t)}} 
    \frac{(\tau_1 -t)^{k_1}}{k_1!} \dotsm \frac{(\tau_r -t)^{k_r}}{k_r!}
    \\ & 
    {
    \left[ \dotsb \left[ \ad_{f_0}^{k_1}(\fp(\tau_1)) ,\ad_{f_0}^{k_2}(\fp(\tau_2))\right],\dotsc,\ad_{f_0}^{k_r}(\fp(\tau_r))\right]} \dd\tau,
    \end{aligned}
\end{equation}
where the sum is taken over $r \in \intset{1,M}$, 
$m\in \intset{1,r}$, $\mathbf{r} \in \N^m_r$,
and $k_1,\dotsc,k_r\in\N$. Moreover, for every $r' \in [r/e,r)$ and $0 \leq \tau \leq t \leq \min\{T; \frac{r-r'}{6 \opnorm{f_0}_{r}}\}$, the series (\ref{gt_series}) and (\ref{expansion_ZMtilde}) converge absolutely in ${\CC^{\omega,r'}(B_{5\delta};\K^d)}$.
\end{proposition}

\begin{proof}
We apply the third statement of \cref{Lem:tool_serie_ad} to $f_0$ and $\fp(\tau)$ to get (\ref{gt_series}). The absolute convergence in this series allows to interchange the sums and the integrals.
\end{proof}

When the perturbation $\fp(t,x)$ is affine, i.e.\ of the form $\sum_{i=1}^q u_i(t) f_i(x)$, by analogy with \cref{thm:Magnus1.0_formel}, we use the notation
$\mathcal{Z}_M(t,f,u)$ instead of $\mathcal{Z}_M \left(t,f_0,\sum_{i=1}^q u_i f_i\right)$, with $f=(f_0,f_1,\dotsc,f_q)$ and $u=(u_1,\dotsc,u_q)$. In this context, we have the following result, that emphasizes that $\mathcal{Z}_M$ is a truncated version of $\mathcal{Z}_\infty$.

\begin{proposition} \label{Prop_Magnus2_Coord1.1_paquet}
    Let $M, \delta, T, f_0, \fp$ as in \cref{Def:Phi0_gt_ZM}.
    Let $r>0$. 
    If $f_0 \in {\CC^{\omega,r}(B_{5\delta};\K^d)}$ and $\fp(t,x)=\sum_{i=1}^q u_i(t) f_i(x)$ where $u_i \in L^1(0,T)$ and $f_i \in {\CC^{\omega,r}(B_{5\delta};\K^d)}$. 
    Then
    \begin{equation} \label{chp_coord1.1_paquets}
    \mathcal{Z}_M(t,f,u) = \lim_{N \rightarrow \infty}
    \sum_{\substack{b \in \mathcal{B} \\ n(b) \leq M \\ n_0(b) \leq N}}
    \eta_b(t,u) f_b  
    \end{equation}
    where, for every $r' \in [r/e,r]$ the limit holds in ${\CC^{\omega,r'}(B_{5\delta};\K^d)}$ when $0\leq t \leq \min\{T; \frac{r-r'}{6 \opnorm{f_0}_{r}}\}$.
\end{proposition}

\begin{proof}
    Let $X = \{ X_0, X_1, \dotsc X_q \}$ and $\Lambda:\mathcal{L}(X) \rightarrow {\CC^{\omega,r}(B_{5\delta};\K^d)}$ be the homomorphism of Lie {algebras} such that $\Lambda(X_i)=f_i$ for $i \in \intset{0,q}$ (see \Cref{Lem:identification_libre}). By applying $\Lambda$ to each term in the equality (\ref{In=sum_coord_pseudo}) (where $\mathcal{Z}_\infty^{r,\nu}(t,X,a)$ is the finite sum defined in (\ref{def:Irnu})), we obtain for every $r\in\N^*$ and $\nu \in \N$
    \begin{equation}
    \mathcal{Z}_\infty^{r,\nu}(t,f,u)= \sum_{b\in\mathcal{B}_{r,\nu}} \eta_b(t,u) f_b.
    \end{equation}
    By \Cref{Prop:expansion_ZM_1.1_analytic}
    \begin{equation}
    \mathcal{Z}_M(t,f,u)
    = \lim_{N \rightarrow \infty} \sum_{\nu=0}^N  \sum_{r=1}^M 
    \mathcal{Z}_\infty^{r,\nu}(t,f,u) 
    \end{equation}
    where for every $r' \in [r/e,r]$ the limit holds in ${\CC^{\omega,r'}(B_{5\delta};\K^d)}$ when $0\leq t \leq \min\{T; \frac{r-r'}{6 \opnorm{f_0}_{r}}\}$.
    This proves~(\ref{chp_coord1.1_paquets}).
\end{proof}

\begin{remark}
    Although the family $\eta_b(t,u) f_b$ for {$b \in \mathcal{B} \cap S_M = \{ b \in \mathcal{B} ; n(b) \leq M \}$ (using \cref{def:SM})} is not proved to be absolutely summable, equality (\ref{chp_coord1.1_paquets}) gives a sense to the expression
    \begin{equation}
    \mathcal{Z}_M(t,f,u)
    = \sum_{b\in\mathcal{B} \cap S_M} \eta_b(t,u) f_b.
    \end{equation}
    Indeed, the proof above justifies the absolute summability of appropriate packages $\mathcal{Z}_\infty^{r,\nu}(t,f,u)$ for $r \in \intset{1,M}$ and $\nu \in \N$ of this family.
    The full absolute summability over $\mathcal{B} \cap S_M$ is investigated in the next subsection.
\end{remark}

\subsubsection{Absolute convergence for coordinates of the pseudo-first kind}

Continuing the discussion started in \cref{sec:structure-constants} we state a criterion on the  basis $\mathcal{B}$ which entails the absolute summability for analytic vector fields of the family $\eta_b(t,u) f_b$ for {$b \in \mathcal{B} \cap S_M = \{ b \in \mathcal{B} ; n(b) \leq M \}$ (using \cref{def:SM})}.

\begin{proposition}
     Let $q\in\N^*$, $X = \{ X_0, X_1, \dotsc, X_q \}$ and $\mathcal{B}$ a {Hall basis of $\mathcal{L}(X)$; or more generally a monomial basis of $\mathcal{L}(X)$ with geometric growth with respect to $X_0$ (see \cref{def:asym-geom-growth} and \cref{rk:A1})}. 
     
    Let $M, \delta, T, f_0, \fp$ as in \cref{Def:Phi0_gt_ZM}.
    Let $r>0$. 
    We assume $f_0 \in {\CC^{\omega,r}(B_{5\delta};\K^d)}$ and $\fp(t,x)=\sum_{i=1}^q u_i(t) f_i(x)$ where $u_i \in L^1(0,T)$ and $f_i \in {\CC^{\omega,r}(B_{5\delta};\K^d)}$.

    Let $r' \in [r/e,r)$.
    There exists $T^*=T^*(M,{q,}r,r'{,\opnorm{f_0}_r}) >0$ such that, for every $t \in (0,T^*)$ and $u\in L^1((0,t),\mathbb{K}^q)$ 
    \begin{equation} \label{Z_M=sum(coord1.1)_chpvect}
    \mathcal{Z}_M\left(t,f,u\right)=\sum_{b \in \mathcal{B} \cap S_M} \eta_b(t,u) f_b
    \end{equation}
    where the series converges absolutely in ${\CC^{\omega,r'}(B_{\delta};\K^d)}$.
\end{proposition}

\begin{proof}
    By (\ref{eta_b|b|!<geom}) of \cref{thm:etab-geom} and (\ref{eq:bracket.analytic}), for every $b\in\mathcal{B} \cap S_M$ and $t \in [0,T]$
\begin{equation}
|\eta_b(t,u)| \opnorm{f_b}_{r'} \leq 
\frac{r-r'}{2e^2} \left(\frac{2eC_M t{\opnorm{f_0}_r}}{r-r'}   \right)^{n_0(b)} \left(\frac{2eC_M}{r-r'}  \|u\|_{L^1(0,t)}\opnorm{f}_r \right)^{n(b)}
\end{equation}
where $\opnorm{f}_r := \max\{ \opnorm{f_j}_{r} ; j \in \intset{0,q} \}$. In particular, if $|t|<T^*(M,r,r'):=\frac{r-r'}{4(q+1)eC_M{\opnorm{f_0}_r}} $ then the series $\sum \eta_b(t,a) f_b$ converges absolutely in $\CC^{\omega,r'}$ because
 \begin{equation}\sum_{b \in \mathcal{B} \cap S_M} (2(q+1))^{-n_0(b)} \leq \sum_{n=1}^M \sum_{n_0=0}^{+\infty} (q+1)^{n_0+n} (2(q+1))^{-n_0} \leq M (q+1)^M.
 \end{equation}
\end{proof}

\subsection{Sussmann's infinite product expansion}

Let $T > 0$. In this section, we consider affine systems of the form
\begin{equation} \label{affine_syst_q}
\dot{x}(t) = \sum_{i \in I} u_i(t) f_i(x(t))
\quad \textrm{and} \quad 
x(0) = p,
\end{equation}
where, for $i \in I$, $f_i$ is a vector field and $u_i \in L^1((0,T);\K)$.
When well-defined, its solution is denoted $x(t;f,u,p)$.
For every norm $\| \cdot \|$ on vector fields, $\|f\|$ denotes $\sum_{i \in I} \|f_i\|$.

\begin{proposition} \label{Prop:Error_Sussman}
    Let $\mathcal{B}$ be a \GHB\ of $\mathcal{L}(X)$ and
$(\xi_b)_{b\in\mathcal{B}}$ be the associated coordinates of the second kind.
	For every $M \in \N^*$, there exist $C_M, \eta_M>0$ such that the following property holds.
	Let $T, \delta > 0$, $f_i \in {\CC^{2M}(B_{3\delta};\K^d)}$ and $u_i \in L^1((0,T);\K)$ for $i \in I$.
	Assume that
	\begin{equation} \label{eq:sussman.small1}
	    \| u \|_{{L^1(0,T)}} \| f \|_{\CC^M} \leq \eta_M \min \{ 1, \delta \}. 
	\end{equation}
	Then, for each $t \in [0,T]$ and $p \in B_\delta$,
    \begin{equation} \label{Prod_Suss_fini_erreur}
    \left| x(t;f,u,p) - \underset{b \in \mathcal{B}_{\intset{1,M} }}{\overset{\rightarrow}{\Pi}} e^{\xi_b(t,u) f_b } p \right| \leq C_{M} \|u\|_{{L^1(0,t)}}^{M+1} \|f\|_{\CC^{2M}}^{M+1} \left(1+\|f\|_{\CC^{2M}}^{M-1}\right),
    \end{equation}
    where the arrow above the product symbol designates the order for the product, i.e.\ with the notations of \cref{Def:Laz}
    \begin{equation}
     \underset{b \in \mathcal{B}_{\intset{1,M}}}{\overset{\rightarrow}{\Pi}} e^{\xi_b(t,u) f_b } = e^{\xi_{b_1}(t,u)f_{b_1}} \dotsm e^{\xi_{b_{k+1}}(t,u) f_{b_{k+1}}}.
     \end{equation}
\end{proposition}

\begin{proof}
	Let $M \in \N^*$. 
	We adopt the notations $b_1,\dotsc,b_{k+1}$ and $Y_0, \dotsc, Y_{k+1}$ of \cref{Def:Laz}. 
	For $j \in \intset{1,k+1}$, we denote by $\Phi_j$ the flow associated with $f_{b_j}$, i.e.\ $\Phi_j(t,p):=e^{t f_{b_j}}(p)$.
	To simplify the notations in this proof, we write $x(t)$ and $\xi_b(t)$ instead of $x(t;f,u,p)$ and $\xi_b(t,u)$. 
	Let $\eta_M := 1 / (4 |I| M!)$. 
	For brevity, we use the shorthand notation $F := \| f \|_{\CC^{2M-1}}$.

\medskip

\noindent \emph{Step 1: Well-definition of the flows.}
Using \eqref{eq:sussman.small1},
\begin{equation}
    \left\| \sum_{i\in I} u_i f_i \right\|_{{L^1((0,T);\CC^0)}} 
    \leq \eta_M \min \{ 1, \delta \} \leq \delta.
\end{equation}
Thus, for $t \in [0,T]$, $x(t)$ is well-defined and $x(t) \in B_{2\delta}$.
For $b \in \mathcal{B}$, using \eqref{eq:xib-u} and \cref{thm:bracket.ck}, we obtain, for each $t \in [0,T]$,
{
\begin{equation}
    \left\| \xi_b(t) f_b \right\|_{\CC^1} \leq \ell! 2^{\ell-1} \|u\|_{L^1(0,t)}^\ell \|f\|_{\CC^\ell}^\ell,
\end{equation}
where $\ell := |b|$.
Hence, using}
the crude estimate $|\mathcal{B}_\ell| \leq |I|^{\ell}$, we obtain, for each $t \in [0,T]$,{
\begin{equation} \label{borne_xib*fb_C1}
    \begin{split}
        \sum_{b \in \mathcal{B}_{\intset{1,M}}} \left\| \xi_b(t) f_b \right\|_{\CC^1}
        & \leq \sum_{\ell = 1}^M |\mathcal{B}_\ell| \ell! 2^{\ell-1} \|u\|_{L^1}^\ell \|f\|_{\CC^\ell}^\ell \\
        & \leq M! \sum_{\ell = 1}^{+\infty} \left(2 |I| \|u\|_{L^1} \|f\|_{\CC^\ell}\right)^\ell \\
        & \leq \frac{2 M! |I| \|u\|_{L^1} \|f\|_{\CC^M}}{1-2|I|\|u\|_{L^1}\|f\|_{\CC^M}} \leq \min \{ 1, \delta \}.
    \end{split}
 \end{equation}}Thus, for every $j \in \intset{1,k+1}$, 
\begin{equation}
x_j(t):=e^{-\xi_{b_j}(t)f_{b_j}} \dotsb e^{-\xi_{b_1}(t) f_{b_1}} (x(t))
\end{equation}
is well-defined and belongs to $B_{3\delta}$.

\medskip \noindent
\emph{Step 2: Estimates along a Lazard elimination.}
We prove by induction on $j \in \intset{0,k+1}$ the existence of a numerical constant $C_j>0$ such that
\begin{equation} \label{EDO_Lazard_j}
    (\mathcal{H}_j): 
    \begin{cases}
        \dot{x}_j(t) = \sum_{b \in \mathcal{B}_{\intset{1,M}} \cap Y_j} \dot{\xi}_b(t) f_b( x_j(t) ) + \varepsilon_j(t), \\
        x_j(0) = p,
    \end{cases}
\end{equation}
where
\begin{equation} \label{epsj<CjtM}
|\varepsilon_j(t)| \leq C_j |u(t)| \|u\|_{L^1}^M F^{M+1} (1+F^{M-1}).
\end{equation}
First, letting $x_0(t) := x(t)$ by convention, $(\mathcal{H}_0)$ holds with $\varepsilon_0=0$, $C_0=0$ because $\dot{\xi}_{X_i}(t)=u_i(t)$ for $i \in I$.
Let $j \in \intset{1,k+1}$ and assume that $(\mathcal{H}_{j-1})$ holds.
We deduce from the definition of $x_{j}(t)$ that
 \begin{equation}x_{j}(t)= e^{-\xi_{b_j}(t)f_{b_j}} (x_{j-1}(t)) = \Phi_j \left( -\xi_{b_j}(t),x_{j-1}(t) \right) \end{equation}
and, using $(\mathcal{H}_{j-1})$, that
\begin{equation}
        \dot{x}_{j}(t)
        = -\dot{\xi}_{b_j}(t) f_{b_j} \left( x_{j}(t) \right)  +
\sum_{b \in \mathcal{B}_{\intset{1,M}} \cap Y_{j-1}} \dot{\xi}_b(t)
\partial_p \Phi_j \left(-\xi_{b_j}(t),x_{j-1}(t)\right) 
 f_b( x_{j-1}(t) ) + \widetilde{\varepsilon}_j(t) 
\end{equation}
where
$\widetilde{\varepsilon}_{j-1}(t) := \partial_p \Phi_j\left(-\xi_{b_j}(t),x_{j-1}(t)\right) \varepsilon_{j-1}(t)$.
By (\ref{borne_xib*fb_C1}), $\| \xi_{b_j}(t) f_{b_j} \|_{\CC^1} \leq 1$, so, using \eqref{eq:dpef},
\begin{equation} \label{borne_varepsilon_tilde}
    |\widetilde{\varepsilon}_{j-1}(t)| \leq e |\varepsilon_{j-1}(t)|.
\end{equation}
Moreover, for each $b \in \mathcal{B}$
\begin{equation}
\partial_p \Phi_j \left(-\xi_{b_j}(t),x_{j-1}(t)\right) 
 f_b( x_{j-1}(t) ) =
 \left( \Phi_j \left(-\xi_{b_j}(t) \right)_*  f_b \right)(x_{j}(t)),
\end{equation}
thus,
\begin{equation}
        \dot{x}_{j}(t) = \sum_{b \in \mathcal{B}_{\intset{1,M}} \cap Y_{j-1} \setminus\{b_j\}} \dot{\xi}_b(t)
 \left( \Phi_j \left(-\xi_{b_j}(t) \right)_*  f_b \right)(x_{j}(t)) + \widetilde{\varepsilon}_j(t).
\end{equation}
For $b \in \mathcal{B}_{\intset{1,M}} \cap Y_j \setminus \{b_j\}$, we introduce the maximal integer $h(b)\in\N^*$ such that
\begin{equation} \label{def:mu(b)_longueur}
|b|+(h(b)-1)|b_j| \leq M.
\end{equation}
Then, by the first statement of \cref{Lem:tool_serie_ad} and  \cref{Def:Coord2} 
\begin{equation}
\begin{split}
\dot{\xi}_b(t) \left( \Phi_j \left(-\xi_{b_j}(t) \right)_*  f_b \right)(x_{j}(t))
&
= \sum_{m=1}^{h(b)-1} \frac{\xi^m_{b_j}(t)}{m!} \dot{\xi}_b(t) f_{\ad_{b_j}^m(b)}(x_{j}(t)) + \overline{\varepsilon}^j_b(t)
\\ & =
\sum_{m=1}^{h(b)-1} \dot{\xi}_{\ad_{b_j}^m(b)}
(t) f_{\ad_{b_j}^m(b)}(x_j(t)) + \overline{\varepsilon}^j_b(t)
\end{split}
\end{equation}
  where
  \begin{equation}
  	|\overline{\varepsilon}^j_b(t)| \leq 
  	|\dot{\xi}_b(t)| \frac{|\xi_{b_j}(t)|^{h(b)}}{h(b)!}  
  	\| f_{\ad_{b_j}^{h(b)}(b)} \|_{\CC^0}.
  \end{equation}
  By definition of $h(b)$ we have $M+1 \leq |b|+h(b)|b_j|\leq M + |b_j| \leq 2M$. Thus, using \cref{thm:bracket.ck}, \eqref{eq:dotxib-u} and \eqref{eq:xib-u}, we get
  \begin{equation}
  	\begin{split}
  	| \overline{\varepsilon}^j_b(t) | 
  	& \leq |u(t)| \| u \|_{L^1}^{|b|+h(b)|b_j|-1} \frac{|b|}{h(b)!} 2^{2M} (2M-1)! F^{M+1}(1+F^{M-1}) \\
  	& \leq |u(t)| \| u \|_{L^1}^{M} M 2^{2M} (2M-1)! F^{M+1}(1+F^{M-1}).
  	\end{split}
  \end{equation}
  By definition of $Y_{j}$ in \cref{Def:Laz}, we obtain $(\mathcal{H}_{j})$ with 
  \begin{equation} \label{def:varepsilon(j)}
  \varepsilon_{j}(t):= \widetilde{\varepsilon}_{j-1}(t) + \sum_{b \in \mathcal{B}_{\intset{1,M}} \cap Y_{j-1} \setminus\{b_j\}} \overline{\varepsilon}^j_b(t).
  \end{equation}
  that satisfies (\ref{epsj<CjtM}) with, for instance $C_{j+1}:=e C_j + |I|^{M+1} M 2^{2M} (2M-1)!$.
 
\medskip \noindent  
  \emph{Step 3: Conclusion.} Taking into account that $\mathcal{B}_{\intset{1,M}} \cap Y_{k+1}=\{0\}$, we get $\dot{x}_{k+1}(t)=\varepsilon_{k+1}(t)$ thus $|x_{k+1}(t)-p| \leq C_{k+1} \|u\|_{L^1}^{M+1} F^{M+1}(1+F^{M-1})$, i.e.\ 
   \begin{equation}  \left| \underset{b \in \mathcal{B}_{\intset{1,M}}}{\overset{\leftarrow}{\Pi}} e^{-\xi_b(t,u) f_b } x(t) -p \right| \leq C_{k+1} \|u\|_{L^1}^{M+1} F^{M+1}(1+F^{M-1}). \end{equation}
   Applying the locally Lipschitz map $e^{\xi_{b_1}(t,u)f_{b_1}} \dotsm e^{\xi_{b_{k+1}}(t,u) f_{b_{k+1}}}$ to the two terms in the left-hand side, we get another constant $C_M>0$ such that (\ref{Prod_Suss_fini_erreur}) holds. Note that (\ref{borne_xib*fb_C1}) and \eqref{eq:dpef} ensure that $C_M \leq e C_{k+1}$, so that $C_M$ depends indeed only on $M$.
\end{proof}

\section{Convergence results and issues}
\label{sec:convergence}

The formal expansions of \cref{sec:formal} generally exhibit poor convergence properties for smooth vector fields without any additional assumption. Nevertheless, one can hope to obtain convergence results in the following particular contexts:
\begin{itemize}
    \item \textbf{Nilpotent Lie algebras}. Here, one assumes that the Lie algebra generated by the set of smooth vector fields $\{ f(t, \cdot); \enskip t \in [0,T] \}$ is nilpotent (see \Cref{Def:Nilpotent Lie algebra}). This structural assumption turns most of the involved infinite expansions into finite ones, and it is thus reasonable to expect convergence properties.
    \item \textbf{Banach algebras}. Here, one assumes that the vector fields are actually linear in the space variable, e.g. that $f(t,x) = A(t) x$ for some $A(t) \in \mathcal{M}_d(\K)$. This assumption yields better estimates for Lie brackets (since products of matrices behave more nicely than differentiation of nonlinear vector fields) and it is thus reasonable to expect convergence properties. In this section, we give statements for matrices for consistence, but similar results can be obtained for linear operators in a Banach algebra.
    \item \textbf{Analytic vector fields}. Here, one assumes that the vector fields are locally real-analytic, i.e.\ than their $k$-th derivative grows roughly no more that $k!$. This bound is compatible with the $1/k!$ factors which come out of the corresponding time integrals, and it is thus reasonable to expect convergence properties.
\end{itemize}
In the following paragraphs, we investigate the convergence properties of each expansion in each of these three reasonable contexts and encounter some surprises.
We summarize the results in the following table.

\begin{center}
\begin{tabular}{ |c|c|c|c|c| }
    \hline
    Expansion & Lie-Nilpotent & Banach & Analytic \\
 \hline
 \makecell{Chen-Fliess} 
 & \makecell{No \\ (\cref{sec:chen.nilpotent})}
 & \makecell{Global \\ (\cref{sec:conv-cf-matrix})}
 & \makecell{Yes \\ (\cref{sec:conv-cf-analytic})}
 \\ 
 \hline
 \makecell{Magnus in the \\ usual setting}
 & \makecell{Yes for $\CC^\infty$ \\ (\cref{sec:conv-magnus-nilpotent})}
 & \makecell{Small time \\ (\cref{subsubsec:Magnus_CV_matrix})}
 & \makecell{No \\ (\cref{sec:conv-magnus-analytic})}
 \\
 \hline
 \makecell{Magnus in the \\ interaction picture}
 & \makecell{Yes for $\CC^\omega$ \\ (\cref{sec:conv-interaction-nilpotent})}
 & \makecell{Small perturbation \\ (\cref{sec:conv-interaction-linear})}
 & \makecell{No \\ (\cref{sec:conv-interaction-analytic})}
 \\
 \hline
 \makecell{Sussmann's \\ infinite product}
 & \makecell{Yes for $\CC^\infty$ \\ (\cref{sec:conv-sussmann-nilpotent})}
 & \makecell{Small time \\ (\cref{sec:conv-sussmann-banach})}
 & \makecell{Open problem \\ (\cref{subsec:algo})}
 \\
 \hline
\end{tabular}
\end{center}

\subsection{Chen-Fliess expansion} \label{sec:conv-chen}

\subsubsection{Counter-example for nilpotent vector fields}
\label{sec:chen.nilpotent}

As already discussed in \cref{rk:cf-drawbacks}, the Chen-Fliess expansion is not an intrinsic representation of the flow and involves quantities which are not Lie brackets of the dynamics. Therefore, this expansion is not expected to converge under a Lie-{nilpotency assumption}. The following counter-example (where the dynamic does not depend on time, thereby obviously generating a nilpotent Lie algebra of order 2) proves that this expansion indeed relies on quantities which are not Lie brackets.

\begin{proposition} \label{prop:cf.counter.nilp}
 There exists $f_0 \in \CC^\infty(\R;\R)$ such that, for every $t \in (0,1]$, the solution $x(t;f,0)$ to~\eqref{ODE:f} with $f(t,x) := f_0(x)$ satisfies{, with the notations of \cref{rk:no-nabla},}
 \begin{equation} \label{eq:prop.cf.counter.nilp}
  \lim_{N \to + \infty} \left| x(t;f,0) - \sum_{n=0}^N \frac{t^n}{n!}
  \left( {f_0^n} \mathrm{Id}_1 \right) (0) \right| = + \infty.      
 \end{equation}
\end{proposition}

\begin{proof}
  For every sequence $(\alpha_n)_{n\in\N} \in \R^\N$, there exists $f_\alpha \in \CC^\infty(\R;\R) \cap L^\infty(\R;\R)$ with $f_\alpha(0) = 1$ such that
  \begin{equation} \label{eq:borel.alpha}
   \forall n \geq 2, \quad \left( {f_\alpha^n} \mathrm{Id}_1 \right) (0) = \alpha_n.
  \end{equation}
  This is an easy consequence of Borel's lemma. Indeed, for $n \geq 2$ and $f_\alpha(0) = 1$,
  \begin{equation}
    \left( {f_\alpha^n} \mathrm{Id}_1 \right) (0) = f_\alpha^{(n-1)}(0) + P_n\left(f_\alpha(0), \dotsc, f_\alpha^{(n-2)}(0)\right),
  \end{equation}
  for some polynomial $P_n$. Thus, given a sequence $(\alpha_n)_{n\in\N}$, one can prescribe an appropriate value for $f_\alpha^{(n-1)}$ and recursively ensure \eqref{eq:borel.alpha}. Let $f_0$ be a vector field constructed following this process for $\alpha_n := n!^2$. On the one hand, since $f_0 \in L^\infty(\R;\R)$, $x(t;f,0)$ is bounded for $t \in [0,1]$. On the other hand, thanks to \eqref{eq:borel.alpha}, for each $t > 0$
  \begin{equation}
   \sum_{n=0}^N \frac{t^n}{n!}
  \left( {f_0^n} \mathrm{Id}_1 \right) (0)
  = \sum_{n=0}^N n! t^n \to +\infty,
  \end{equation}
  which proves \eqref{eq:prop.cf.counter.nilp}.
\end{proof}

\begin{remark}
    In this counter-example, the local change of coordinates which transforms $f_0(x) e_1$ into the constant vector field $e_1$ allows to transform the ODE on $x$ to a new ODE for which the Chen-Fliess expansion is finite (and thus convergent). 
    It would be even more interesting to construct a counter-example, probably in dimension $d \geq 2$, for which no local change of coordinates can restore the convergence of the Chen-Fliess expansion.
\end{remark}

\subsubsection{Global convergence for matrices}
\label{sec:conv-cf-matrix}

Let $T > 0$. In this paragraph, we study linear systems of the form
\begin{equation} \label{EDO_lineaire}
    \dot{x}(t) = A(t) x(t) 
    \quad \textrm{and} \quad
    x(0) = p,
\end{equation}
where $A \in L^1((0,T);\mathcal{M}_d(\K))$. The solution is denoted $x(t;A,p)$.

\begin{proposition} \label{Prop:CV_CF_global_matrix}
 Let $T>0$ and $A \in L^1((0,T); \mathcal{M}_d(\K))$. For each $t \in [0,T]$ and $p \in \K^d$,
 \begin{equation} \label{CF_macro_value_matrices}
   x(t;A,p) = p + \sum_{j=1}^{+\infty} 
   \int_{{\Delta^j(t)}} 
   A(\tau_j) \dotsb A(\tau_1) p \dd\tau,
 \end{equation}
 where the series converges absolutely.
\end{proposition}

\begin{proof}
To simplify the notations, we write $x(t)$ instead of $x(t;A,p)$.
By Gr\"onwall's lemma, we have $|x(\tau)| \leq |p| e^{\|A\|_{L^1(0,\tau)}}$ for every $\tau \in [0,T]$.
By iterating the formula
 \begin{equation}x(\tau)=p+\int_0^{\tau} A(\tau')x(\tau') \dd\tau'\end{equation}
we obtain, for every $M \in\N^*$
\begin{equation}
    \begin{split}
        & \left| x(t) - p - \sum_{j=1}^{M-1}  \int_{{\Delta^j(t)}}  A(\tau_j) \dotsb A(\tau_1) p \dd\tau \right|
   = \left|  \int_{{\Delta^M(t)}}  A(\tau_M) \dotsb A(\tau_1) x(\tau) \dd\tau \right|
   \\  \leq & \int_{{\Delta^M(t)}}  \|A(\tau_M)\| \dotsb \|A(\tau_1)\|  \dd\tau |p| e^{\|A\|_{L^1(0,t)}}
   = \frac{\|A\|_{L^1(0,t)}^M}{M!} |p| e^{\|A\|_{L^1(0,t)}},
    \end{split}
\end{equation}
which proves the convergence. Similar estimates prove the absolute convergence.
\end{proof}

\subsubsection{Local convergence for analytic vector fields}
\label{sec:conv-cf-analytic}

For analytic vector fields, it is well known that the Chen-Fliess series {(also called ``(right) chronological exponential'' in \cite[Section 2.1]{MR524203})} converges locally in time (see e.g.\ {\cite[Proposition 2.1]{MR524203}, or} \cite[Proposition 4.3]{MR710995} for slightly different assumptions).
{The analyticity assumption is necessary, as highlighted by the counter-example of \cref{sec:chen.nilpotent}.}

\begin{proposition} \label{Prop:CV_CF_analytiq}
    Let $T, \delta, r > 0$. 
    There exists $\eta > 0$ such that, for $f \in L^1([0,T];{\CC^{\omega,r}(B_{2\delta};\K^d)})$ with $\|f\|_{L^1(\CC^{\omega,r})} \leq \eta$, $\varphi \in {\CC^{\omega,r}(B_{2\delta};\K^d)}$, $t \in [0,T]$ and $p \in B_\delta$,
 \begin{equation} \label{CF_macro}
   \varphi\left(x(t;f,p)\right) = \varphi(p) + \sum_{j=1}^{+\infty}  \int_{{\Delta^j(t)}} { \left(  f(\tau_1) \dotsb f(\tau_j)\right)} ( \varphi ) (p) 
\dd\tau,
 \end{equation}
 where the sum converges absolutely. In particular,
 \begin{equation} \label{CF_macro_value}
   x(t;f,p) = p + \sum_{j=1}^{+\infty}  \int_{{\Delta^j(t)}} { \left(  f(\tau_1) \dotsb f(\tau_j)\right)} ( \mathrm{Id}_d ) (p)
\dd\tau.
 \end{equation}
\end{proposition}

\begin{proof}
 Let $\eta := \min \{ \delta/2, r/10 \}$.
 By \cref{Lem:WP_ODE:f}, $x(t;f,p)$ is well defined for $t \in [0,T]$, $p \in B_\delta$ and belongs to $B_{2\delta}$.
 Moreover, by \cref{thm:prod.analytic}, we have, for every $j\in\N^*$
 \begin{equation}
   \int_{{\Delta^j(t)}} \left| {\left(  f(\tau_1) \dotsb f(\tau_j)\right)} ( \varphi ) (p) \right| \dd\tau
   \leq j! \left(\frac{5}{r}\right)^j \cdot \frac{\|f\|^j}{j!}
   \opnorm{\varphi}_r,
 \end{equation}
 where $\|f\|:=\|f\|_{{L^1((0,t);\CC^{\omega,r})}}$, which proves the absolute convergence because the right-hand side is bounded by $2^{-j} \opnorm{\varphi}_r$. Eventually, we deduce from (\ref{CF:reste_expl}) and \cref{thm:prod.analytic} that
 \begin{equation}
    \left| \varphi\left(x(t;f,p)\right) - \sum_{j=0}^{M-1} \int_{{\Delta^j(t)}}  {\left( f(\tau_1)  \dotsb f(\tau_j) \right)} ( \varphi ) (p) \dd\tau \right|
 \leq 2^{-M} \opnorm{\varphi}_r,
 \end{equation}
 which proves (\ref{CF_macro}).
\end{proof}

\subsection{Magnus expansion in the usual setting}
\label{sec:conv-magnus}

\subsubsection{Equality for nilpotent systems}
\label{sec:conv-magnus-nilpotent}

The goal of this section is to prove that the Magnus expansion is an exact expansion for regular vector fields generating a nilpotent Lie algebra (see \cref{thm:magnus1.nilpotent}).

If the vector fields are analytic in space, a simple proof can be given (see e.g.\ \cite[Remark A.1]{zbMATH06323245} for the case of the CBHD formula), with the following steps. First, by density, one can assume that the dynamic depends analytically on time. Then, the maps $t \mapsto x(t)$ and $t \mapsto e^{Z_M}(t)$ are analytic. Because of the {nilpotency assumption}, $Z_M = Z_{M'}$ for every $M' \geq M$ and estimate \eqref{Estim_Magnus} proves that both functions have the same Taylor expansion at $t = 0$, and are thus equal.

For non-analytic vector fields, the proof is much more intricate. A sketch of proof is briefly suggested in \cite[Proposition 2.4]{MR579930}. In this paragraph, we write the proof completely. The difficulty is to formulate the question in the nilpotent Lie algebra generated by the vector fields, in order to conclude with the universal property of free nilpotent Lie algebras (\cref{Lem:identification_libre}). 

To that end, we start with the following {statement}.

\begin{lemma} \label{Lem:Magnus_nilp_libre}
    Let $a$ be given by (\ref{eq:a.ai}), $M \in \N^*$, $Z_M(t):=\mathrm{Log}_M\{a\}(t)$ with the notation of \cref{def:LOGM}. Then for every $t\in\R$, the following equality holds in $\mathcal{N}_{M+1}(X)$
    \begin{equation} \label{d/dt_exp(z(t))_N_M+1(X)}
    \sum_{n=0}^{M-1} \frac{(-1)^n}{(n+1)!} \ad^n_{Z_M(t)} (\dot{Z}_M(t)) = a(t),
    \end{equation}
    where $Z_M(t)$ belongs to the space $\underset{r \in \intset{1,M}}{\oplus}  \mathcal{L}(X)^r$ which is {identified with} $\mathcal{N}_{M+1}(X)$ as a vector space.
\end{lemma}

\begin{proof}
    The canonical surjection $\sigma_{M+1}:\mathcal{L}(X) \rightarrow \mathcal{N}_{M+1} (X)$ is an homomorphism {of Lie algebras}. 
    {Applying this homomorphism to \eqref{eq:EDO-Z-reversed}, where $z(t) = \mathrm{Log}_\infty\{a\}(t)$ thanks to \cref{thm:formal.log}, proves~\eqref{d/dt_exp(z(t))_N_M+1(X)}}.
\end{proof}

\begin{proposition} \label{thm:magnus1.nilpotent}
    Let $M \in \N^*$. 
    There exists $\eta_M>0$ such that, for every $T,\delta>0$ and {every time-varying vector field} $f : [0,T] \to {\CC^\infty(B_{4\delta};\K^d)}$ such that $\mathcal{L}(f([0,T]))$ is nilpotent with index at most $M+1$ and $f \in L^1((0,T);{\CC^{M}(B_{4\delta};\K^d)})$ with
    \begin{equation} \label{eq:magnus1.nilp.size}
        \|f\|_{{L^1((0,T);\CC^M)}}+\|f\|_{{L^1((0,T);\CC^M)}}^{M} \leq \eta_M \delta,    
    \end{equation}
    then, for each $p\in B_\delta$ and $t \in [0,T]$, one has $x(t;f,p)= e^{Z_M(t,f)}(p)$ where $Z_M(t,f):= \mathrm{Log}_M \{ f \}(t)$ is the vector field defined in \cref{def:LOGM}.
\end{proposition}

\begin{proof}
    Let $M\in\N^*$. 
    By \cref{def:LOGM} {and \cref{thm:bracket.ck}}, there exists $\eta_M>0$ such that,
    \begin{equation} \label{borne_ZM_t}
        \|Z_M(t,f)\|_{{\CC^1}} \leq \frac{1}{\eta_M} \left( \|f\|_{L^1(\CC^M)} + \|f\|_{L^1(\CC^M)}^M \right).
    \end{equation}
    {In particular, for every $p \in B_\delta$, $e^{Z_M(t,f)}(p)$ is well-defined thanks to \cref{eq:magnus1.nilp.size}.}
    
    \medskip \noindent \emph{Step 1: Proof for $f(t,x)=\sum_{j=1}^q a_j(t) f_j(x)$ with $q \in \N^*$, $a_j \in {\CC^\infty([0,T];\K)}$ and $f_j \in \CC^\infty_c(\K^d;\K^d)$.} 
    By uniqueness in the Cauchy-Lipschitz theorem, it is sufficient to prove that for every $t \in [0,T]$ and $p\in B_\delta$,
    \begin{equation}\frac{\dd}{\dd t}( e^{Z_M(t,f)}(p)) = f \left( t , e^{Z_M(t,f)}(p) \right).
    \end{equation}
    By \cref{def:LOGM}, the map $(t,p) \mapsto Z_M(t,f)(p)$ belongs to $\CC^\infty([0,T] \times B_{4\delta};\K^d)$.
    Thanks to the {nilpotency assumption}, $\ad_{Z_M(t,f)}^M( Z_M(\tau,f))=0$ on $B_{4\delta}$ for every $t,\tau \in [0,T]$.
    Thus \cref{Prop:d(flot)/d(param)} yields
    \begin{equation} \frac{\dd}{\dd t}\left( e^{Z_M(t,f)} (p) \right) = \sum_{k=0}^{M} \frac{(-1)^{k}}{(k+1)!} \ad^k_{Z_M(t,f)}\left( \dot{Z}_M(t,f) \right)\left( e^{Z_M(t,f)} (p) \right).
    \end{equation}
    Let $\Lambda:\mathcal{N}_{M+1}(X) \rightarrow \mathcal{L}(f_1,\dotsc,f_q)$ be the homomorphism of nilpotent Lie {algebras} such that $\Lambda(X_j)=f_j$ for $j=1,\dotsc,q$. By applying $\Lambda$ to the equality (\ref{d/dt_exp(z(t))_N_M+1(X)}), we obtain that the right-hand side of the above equality is $f(t, e^{Z_M(t,f)}(p))$.

    \medskip \noindent \emph{Step 2: Proof for a general time-dependent vector field $f$.} 
    We apply Step 1 to a sequence $f_n$ of simple functions, taking values in $f([0,T])$, {converging towards $f$ in $L^1((0,T);\CC^{M}(B_{4\delta};\K^d))$}.
    We get the conclusion by passing to the limit in both sides{, using the fourth item of \cref{Lem:WP_ODE:f}.}
\end{proof}

\subsubsection{Convergence for linear systems}
\label{subsubsec:Magnus_CV_matrix}

In this paragraph, we consider linear systems of the form (\ref{EDO_lineaire}).
Since the Magnus expansion was designed for linear systems, its convergence in this context has received much attention.
Depending on the exact convergence notion that one considers and on the way one groups terms, different sufficient conditions for the convergence can be derived. 
In \cite{MR886816}, $T \| A \|_{L^\infty(0,T)} \leq 1$ is shown to be a sufficient condition for convergence on $[0,T]$ thanks to a careful estimate of the combinatorial terms.
In \cite{MR2413145}, $\|A\|_{L^1(0,T)} < \pi$ is shown to be a sufficient condition for convergence using complex analysis.

We give below a short elementary proof with a sub-optimal constant, for the sake of completeness and because it will be {used in the sequel}.
Let $\|\cdot\|$ be a sub-multiplicative norm on $\mathcal{M}_d(\K)$.

\begin{proposition} \label{Prop:CV_Magnus1.0_matrix}
    Let $T > 0$ and $A \in L^1((0,T);\mathcal{M}_d(\K))$ such that
 $\|A\|_{{L^1(0,T)}} < \frac{1}{4}$. For each $t \in [0,T]$,
   \begin{equation} \label{Zinfty_matrix}
   Z_\infty(t)
   := \sum_{r=1}^{+\infty} \frac{1}{r} \sum_{m=1}^r \frac{(-1)^{m-1}}{m}
  \sum_{\mathbf{r} \in \N^m_r} \int_{{\Delta^{\mathbf{r}}(t)}} {
  [ \dotsb [ A(\tau_1), A(\tau_2) ], \dotsc  A(\tau_r) ]} \dd \tau  
   \end{equation}
is well defined in $\mathcal{M}_d(\K)$ and, for every $p\in\K^d$, $x(t;A,p)=e^{-Z_\infty(t)} p$,
where the brackets refer to commutators of matrices, i.e.\ $[A,B] = AB-BA$.
\end{proposition}

\begin{proof}
\emph{Step 1: Absolute convergence of $Z_\infty(t)$.}
Let $r \in \N^*$. For every $m \in \intset{1,r}$ and $\mathbf{r} \in \N^m_r$, 
\begin{equation}
    \begin{split}
        \int_{{\Delta^{\mathbf{r}}(t)}} & 
  \left\| {[ \dotsb [ A(\tau_1), A(\tau_2) ], \dotsc  A(\tau_r) ]} \right\| \dd \tau
 \\ \leq & \int_{{\Delta^{\mathbf{r}}(t)}} 2^r {\|A(\tau_1)\| \dotsc \|A(\tau_r)\|} \dd\tau
  \leq 2^r 
  \left( \int_0^t \|A(\tau)\| \dd\tau \right)^r.
    \end{split}
\end{equation}
Moreover, recalling the definition of \eqref{def:Nmr}, $|\N^m_r| = \binom{r-1}{m-1}$ and
$\sum_{m=1}^r \binom{r-1}{m-1} = 2^{r-1}$.
Thus,
 \begin{equation} \sum_{r=1}^{+\infty} \frac{1}{r} \sum_{m=1}^r \frac{1}{m}  \sum_{\mathbf{r} \in \N^m_r} \int_{{\Delta^{\mathbf{r}}(t)}} \| { [ \dotsb [ A(\tau_1), A(\tau_2) ], \dotsc  A(\tau_r) ]} \| \dd \tau  \leq \sum_{r=1}^{+\infty} \left( 4 \|A\|_{L^1} \right)^r < \infty.\end{equation}
 
\medskip 

\noindent \emph{Step 2: Formula for the solution $L \in \CC^1([0,t];\mathcal{M}_d(\K))$ of}
\begin{equation}
    \begin{cases}
        L'(\tau)= L(\tau) A(\tau)\\
        L(0) = \mathrm{Id}_d.
    \end{cases}
\end{equation}
By working as in the proof of \cref{Prop:CV_CF_global_matrix}, we obtain
\begin{equation}
    L(t)=\mathrm{Id}_d + \sum_{r=1}^{+\infty} \int_{{\Delta^r(t)}} {A(\tau_1)\dotsm A(\tau_r)} \dd\tau
\end{equation}
where the series converges absolutely. Moreover,
we have
 \begin{equation} \left\| \sum_{r=1}^{+\infty} \int_{{\Delta^r(t)}} {A(\tau_1)\dotsm A(\tau_r)} \dd\tau \right\|  \leq  \sum_{r=1}^{+\infty} \frac{\|A\|_{L^1}^r}{r!} < e^{\frac{1}{4}}-1< 1.
 \end{equation}
Thus
 \begin{equation}\log \left( L(t)\right)  =  \sum_{m=1}^{+\infty} \frac{(-1)^m}{m} \left(\sum_{r=1}^{+\infty} \int_{{\Delta^r(t)}} {A(\tau_1)\dotsm A(\tau_r)} \dd\tau \right)^m\end{equation}
is well defined in $\mathcal{M}_d(\K)$ and $L(t)=e^{\log(L(t))}$. By applying \Cref{thm:a.brackets} with $A=A_1=\mathcal{M}_d(\K)$, we get $\log(L(t))=Z_\infty(t)$.

\medskip

\noindent \emph{Step 3: Conclusion.} The resolvent $R(\tau)$ associated to the linear system $\dot{x}=A(\tau) x$ with initial condition at $\tau=0$ is $R(\tau)=L(\tau)^{-1}$. Thus $x(t)=R(t)p=e^{-Z_\infty(t)}p$.
\end{proof}

\begin{remark}
    For $X, Y \in \mathcal{M}_d(\K)$ such that $\|X\|+\|Y\|<\frac{1}{8}$, the previous statement implies the convergence of the CBHD formula, yielding  a matrix $Z_\infty$ such that $e^X e^Y = e^{Z_\infty}$. 
    {Some authors have investigated the} optimal convergence domain in different contexts for the CBHD formula.
    Such a domain sometimes depends on the summation process (i.e.\  the way terms are grouped together) and the exact question one asks (existence of a logarithm, absolute summability of the series, convergence of the remainder, etc.).
    Better sufficient conditions than ours can be found for instance in \cite{MR2031789}, for instance, $\|X\|+\|Y\|< \frac{\ln 2}{2}$.
    We refer to \cite{biagi2018baker} for a nice survey of the convergence questions regarding the CBHD formula.
\end{remark}

\begin{remark}
    The smallness assumption (on time or on the matrices) is in general necessary, both for the CBHD formula (see \cite[Example~2.3]{biagi2018baker} or \cite[Section~II]{MR0156878}) and for the Magnus expansion (see \cite{MR2413145}, where the authors also prove that, although the condition $\|A\|_{L^1(0,T)} < \pi$ is not necessary for convergence, there exists $A$ with $\|A\|_{L^1(0,T)}=\pi$ for which the Magnus series at time $\pi$ does not converge).
\end{remark}

\subsubsection{Divergence for arbitrarily small analytic vector fields} \label{sec:conv-magnus-analytic}

The convergence of Magnus expansions is deeply linked with the convergence of the CBHD series. For analytic vector fields, it is expected that both series diverge (see e.g.\ \cite[p.1671]{MR579930} or \cite[p.335]{MR886816} for statements without examples). Some authors nevertheless suggested that, despite the divergence of the series, the flows could converge for analytic vector fields (see \cite[p.335]{MR886816} and \cite[p.241]{MR1972790}). 

In this paragraph, we give explicit counter-examples to the convergence, even in the weak sense of the flows, for arbitrarily small analytic vector fields, of both the CBHD series and the Magnus expansion. Similarly to counter-examples concerning the convergence of the CBHD series for large matrices (see e.g.\ \cite[Theorem~2.5]{biagi2018baker}), our construction relies on the choice of generators for which many brackets vanish thanks to their particular structure, and the remaining non-vanishing brackets are associated with coordinates of the first kind involving Bernoulli numbers.

\begin{proposition} \label{prop:cbh.c-ex}
    There exists $\delta > 0$ and $f_0, f_1 \in {\CC^{\omega}(B_{\delta};\R^2)}$ such that,
    \begin{equation} \label{eq:cbh.c-ex.1}
      \forall M \in \N,
      \exists C_M, \varepsilon_M > 0,
      \forall \varepsilon \in [0,\varepsilon_M],
      \quad
      \left| e^{\varepsilon f_0} e^{\varepsilon f_1} (0) - e^{\CBHD_M(\varepsilon f_1, \varepsilon f_0)} (0) \right|
      \leq C_M \varepsilon^{M+1},
    \end{equation}
    where $\CBHD_M(\varepsilon f_1, \varepsilon f_0)$ is defined in \cref{Prop:CBH}, but, simultaneously, for every $\varepsilon > 0$,
    \begin{equation} \label{eq:cbh.c-ex.2}
      \lim_{M \to + \infty} | \CBHD_M(\varepsilon f_1, \varepsilon f_0)(0) | = + \infty  
    \end{equation}
    and
    \begin{equation} \label{eq:cbh.c-ex.3}
      \lim_{M \to + \infty} \left| e^{\varepsilon f_0} e^{\varepsilon f_1} (0) - e^{\CBHD_M(\varepsilon f_1, \varepsilon f_0)}(0) \right| = + \infty.
    \end{equation}
\end{proposition}
    
\begin{proof}
    Let $f_0, f_1$ as in \cref{rk:bracket.optimal}.
    For these vector fields, estimate \eqref{eq:cbh.c-ex.1} comes from \cref{Prop:CBH}.
    Due to their structure, the only non vanishing brackets are those containing $f_1$ at most once.
    Therefore, formula \eqref{eq:cbh.Z1} of \cref{thm:cbh.Z1} yields, for $M \geq 1$,
 \begin{equation}
  \CBHD_M(\varepsilon f_1, \varepsilon f_0) = \varepsilon f_0 + \sum_{k = 0}^{M-1} \frac{B_k}{k!} \varepsilon^{k+1} \ad^k_{f_0}(f_1).
 \end{equation}
 Hence, using \eqref{eq:adkf0f1.optimal},
 \begin{equation}
  \CBHD_M(\varepsilon f_1, \varepsilon f_0)(x) = \varepsilon e_1 + \varepsilon \Theta^\varepsilon_{M} (x_1) e_2,
 \end{equation}
 where we introduce, for $q \in \R$,
 \begin{equation}
   \Theta^\varepsilon_M(q) := 
   \sum_{k=0}^{M-1} B_k \varepsilon^{k} (1-q)^{-k-1}.
 \end{equation}
 In particular,
 \begin{equation} \label{eq:cbh.zm.f1f0.0}
   \left| \CBHD_M(\varepsilon f_1, \varepsilon f_0)(0) \right| \geq  \left| \varepsilon \Theta^\varepsilon_M(0) \right|.
 \end{equation}
 Since the odd Bernoulli numbers except $B_1$ are zero, when $M = 2M'+2$ with $M' \geq 1$, $\Theta^\varepsilon_{2M'+2} = \Theta^{\varepsilon}_{2M'+1}$. Then,
 \begin{equation} \label{eq:ser-theta-div1}
   \Theta^\varepsilon_{2M'+1}(q) = \frac{1}{1-q}
   - \frac{\varepsilon}{2(1-q)^2}
   + \sum_{k=1}^{M'} B_{2k} \varepsilon^{2k} (1-q)^{-2k-1}.
 \end{equation}
 In particular, using \eqref{eq:bernoulli.3},
 \begin{equation} \label{eq:ser-theta-div2}
   \Theta^\varepsilon_{2M'+1}(0) 
   = 1 - \frac{\varepsilon}{2} + \sum_{k=1}^{M'} B_{2k} \varepsilon^{2k}
   = 1 - \frac{\varepsilon}{2} + \sum_{k=1}^{M'} (-1)^{k+1} \frac{2 (2k)!}{(2\pi)^{2k}} \zeta(2k) \varepsilon^{2k}.
 \end{equation}
 Thus, for every fixed $\varepsilon > 0$, $|\Theta^\varepsilon_{M}(0)| \to + \infty$ when $M \rightarrow +\infty$, because it involves a sum of the form $\sum_{k=1}^{M'} a_k$ where $|a_{k+1}|/|a_k| \rightarrow +\infty$ when $k \rightarrow +\infty$. Using \eqref{eq:cbh.zm.f1f0.0}, this proves~\eqref{eq:cbh.c-ex.2}.
 
 For $p \in \R^2$ close enough to the origin, one can also compute the flow $e^{\CBHD_M(\varepsilon f_1, \varepsilon f_0)}(p)$, which is $y(1)$ where $y$ is the solution to the ODE $y(0) = p$ and
 \begin{equation}
   \dot{y}_1(s) = \varepsilon 
   \quad \textrm{and} \quad 
   \dot{y}_2(s) = \varepsilon \Theta^\varepsilon_M \left( y_1(s) \right).
 \end{equation}
 Solving successively for $y_1$ then $y_2$ yields $y_1(s) = p_1 + s \varepsilon$ and 
 \begin{equation}
   y_2(s) = p_2 + \int_{y_1(0)}^{y_1(s)} \Theta^\varepsilon_M (h) \dd h.
 \end{equation}
 Thus,
 \begin{equation}
   e^{\CBHD_M(\varepsilon f_1, \varepsilon f_0)}(p)
   = (p_1 + \varepsilon) e_1 
   + \left( p_2 + \int_{p_1}^{p_1+\varepsilon}
   \Theta^\varepsilon_M (h) \dd h \right) e_2.
 \end{equation}
 In particular,
 \begin{equation}
   e^{\CBHD_M(\varepsilon f_1, \varepsilon f_0)}(0)
   = \varepsilon e_1 
   + \left( \int_{0}^{\varepsilon}
   \Theta^\varepsilon_M (h) \dd h \right) e_2 .
 \end{equation}
 When $M=2M'+2$ with $M'\geq 1$, using (\ref{eq:ser-theta-div1}), we get
 \begin{equation} \label{CBHM_infty}
   e^{\CBHD_M(\varepsilon f_1, \varepsilon f_0)}(0)
   = \varepsilon e_1 
   + \left(
   - \ln (1-\varepsilon)
   - \frac{\varepsilon}{2} \left( \frac{1}{1-\varepsilon} - 1 \right)
   + \sum_{k=1}^{M'} \frac{B_{2k}}{2k} \varepsilon^{2k} \left( \frac{1}{(1-\varepsilon)^{2k}} - 1 \right)
   \right) e_2.
 \end{equation}
 Hence, for the same reason as above, the flow satisfies $|e^{\CBHD_M(\varepsilon f_1, \varepsilon f_0)}(0)| \to +\infty$ when $M\rightarrow +\infty$, which proves \eqref{eq:cbh.c-ex.3}.
\end{proof}    

\begin{remark}
 If one sees $(x_1,x_2)$ as $(q,p)$ in an Hamiltonian setting, one checks that the vector fields defined in \eqref{eq:def:f0f1.optimal} and used in this counter-example are associated with the Hamiltonians $\mathcal{H}_0(q,p) := p$ and $\mathcal{H}_1(q,p) := \ln (1 - q)$. Therefore, assuming an Hamiltonian structure on the considered vector fields does not provide enough structure to yield convergence.
\end{remark}

One could wonder if assuming even more structure on the dynamics, for example assuming that it is time-reversible, prevents the construction of such counter-examples.

\begin{open}
    {Do there} exist Hamiltonians $\mathcal{H}_0$ and $\mathcal{H}_1$ on $\R^{2d}$, which are time-reversible (i.e.\ satisfy $\mathcal{H}_i(q,p) = \mathcal{H}_i(q,-p)$ for every $q, p \in \R^d$), locally real-analytic near zero and for which the convergence of the CBHD series fails as in \cref{prop:cbh.c-ex}?
\end{open}

The counter-example of \cref{prop:cbh.c-ex} for the convergence of the CBHD series allows to build counter-examples to the convergence of the Magnus expansion which blow up instantly, despite analytic regularity in both time and space.

\begin{proposition}
    There exist $T, \delta > 0$ and $f \in \CC^{\omega}([0,T]\times B_{\delta}{;\R^2})$ such that, for every $\varepsilon > 0$ and $t \in (0,T]$, 
    \begin{equation}
        \lim_{M\to +\infty} | Z_M(t,\varepsilon f)(0) | = +\infty
    \end{equation}
    and
    \begin{equation}
        \lim_{M\to +\infty} \left| x(t;\varepsilon f,0) - e^{Z_M(t,\varepsilon f)}(0) \right| = + \infty,
    \end{equation}
    where $x$ is the solution to $\dot{x}(t) = \varepsilon f(t,x(t))$ with $x(0) = 0$ and $Z_M(t,\varepsilon f) = \mathrm{Log}_M \{ \varepsilon f \}(t)$.
\end{proposition}

\begin{proof}
    Let $T = 1$. 
    We define $f(t,x) := f_0(x) + t f_1(x)$, where $f_0$ and $f_1$ are defined in \cref{rk:bracket.optimal}.
    Similarly as for the previous construction, only Lie brackets involving $f_1$ at most once are non-vanishing.
    Moreover, the coordinates of the first kind associated with the controls $a_0(t) = 1$ and $a_1(t) = t$ have been computed in \cref{ex:zeta.adx0x1.1t}.
    Hence, recalling \eqref{eq:adkf0f1.optimal}, we have
    \begin{equation}
        Z_M(t,\varepsilon f) = \varepsilon t e_1 + \sum_{k=0}^{M-1} \varepsilon^{k+1} \frac{(-1)^{k+1} t^{k+2}}{(k+1)!} B_{k+1} \frac{k!}{(1-x_1)^{k+1}} e_2.
    \end{equation}
    Proceeding along the same lines as in the proof of \cref{prop:cbh.c-ex} allows to conclude that both $Z_M(t,\varepsilon f)(0)$ and $e^{Z_M(t,\varepsilon)}(0)$ diverge when $M\rightarrow +\infty$.
\end{proof}

\subsection{Magnus expansion in the interaction picture}

\subsubsection{Nilpotent systems}
\label{sec:conv-interaction-nilpotent}

For ODEs of the form \eqref{eq:EDO.interaction}, the starting point of the interaction picture is to factorize the flow of $f_0$. 
Hence, the roles of $f_0$ and $f_1$ are asymmetric.
One can expect that, under the assumption that Lie brackets of $f_0$ and $f_1$ containing at least $M+1$ times $f_1$ identically vanish, the Magnus expansion in the interaction picture should yield an equality of the form
\begin{equation} \label{eq:equal.magnus2.nilpotent}
   x(t;f_0+f_1,p) = e^{\mathcal{Z}_{M}(t,f_0,f_1)} e^{tf_0} p,
\end{equation}
 where $\mathcal{Z}_{M}(t,f_0,f_1)$ is defined in \cref{Prop:Magnus_2}. 
 We prove in this paragraph that it is indeed the case, when $f_0$ and $f_1$ are analytic.
 However, contrary to the case of the usual Magnus expansion (see \cref{sec:conv-magnus-nilpotent}), we give examples highlighting the fact that the {analyticity} assumption cannot be removed, which is quite surprising but stems from the mixing induced by pushforwards.
 
 We therefore start with the following definition.
 
\begin{definition}[Semi-nilpotent family of vector fields]
 Let $\Omega$ be an open subset of $\K^d$. Let $\mathcal{F} \subset \CC^\infty(\Omega;\K^d)$, $f_0 \in \CC^\infty(\Omega;\K^d)$ and $M \in \N^*$. We say that the family of vector fields $\mathcal{F}$ is \emph{semi-nilpotent of index $M$ with respect to $f_0$} if every bracket of elements of $\mathcal{F}\cup\{f_0\}$ involving $M$ elements of $\mathcal{F}$ vanishes identically on $\Omega$ and $M$ is the smallest positive integer for which this property holds.
\end{definition}

{\begin{remark}
    Some authors (see e.g.\ \cite[Section 3]{hermes1976local}) refer to this situation by saying that $\mathcal{L}(\mathcal{S}_1)$ is nilpotent of index $M$, where $\mathcal{S}_1 := \{ \ad^k_{f_0}(f) ; f \in \mathcal{F}, k \in \N \}$.
    Both definitions are equivalent, thanks to the Jacobi identity for Lie brackets.
\end{remark}}

\begin{proposition} \label{thm:magnus2.flow-nilpotent}
    Let $T, \delta > 0$. Let $M \in \N$.
    Let $f_0 \in {\CC^\infty(B_{4\delta};\K^d)}$ with $T \| f_0 \|_{\CC^0} \leq \delta$.
 There exists $\eta > 0$ such that, for every $f_1 : [0,T] \to {\CC^\infty(B_{4\delta};\K^d)}$ with $f_1 \in L^1([0,T];{\CC^{M+1}(B_{4\delta};\K^d)})$ and $\| f_1 \|_{L^1(\CC^{M})} \leq \eta$, the following family is well-defined
 \begin{equation} \label{eq:def.G.family.f1}
   \mathcal{G} := \{ \Phi_0(-t)_* f_1(t) ; \enskip t \in [0,T] \} \subset {\CC^\infty(B_{\delta};\K^d)}.
 \end{equation}
 and, assuming moreover that $\mathcal{G}$ is nilpotent of index $M + 1$, then, for each $t \in [0,T]$ and $p \in B_\delta$, the solution to \eqref{eq:EDO.interaction} satisfies \eqref{eq:equal.magnus2.nilpotent}.
\end{proposition}

\begin{proof}
 Let $t > 0$.
 As in the proof of \cref{Prop:Magnus_2}, we introduce the new {variable} $y(s) := \Phi_0(t-s,x(s))$. 
 Then $\dot{y}(s) = g_t(s,y(s))$, where $g_t$ is defined in \eqref{def:gt_star}. 
 Thanks to \cref{thm:push.composition}, $g_t(s) = \Phi_0(t)_* \Phi_0(-s)_* f_1(s)$. 
 Thanks to the assumption and to \cref{thm:push.lie}, the family $\{ g_t(s); \enskip s \in [0,t] \}$ is nilpotent of index $M+1$.
 Thus, by \cref{thm:magnus1.nilpotent}, $y(t) = e^{\mathcal{Z}_M(t,f_0,f_1)} y(0)$.
 Since $x(t) = y(t)$ and $y(0) = \Phi_0(t,p)$, this concludes the proof of \eqref{eq:equal.magnus2.nilpotent}.
\end{proof}

\begin{lemma} \label{thm:analytic+brackets-nilpotent=>flow-nilpotent}
 Let $T, \delta > 0$, $\mathcal{F} \subset {\CC^\infty(B_{4\delta};\K^d)}$, $f_0 \in {\CC^\infty(B_{4\delta};\K^d)}$ such that $T \| f_0 \|_{\CC^0} \leq \delta$. 
 The following family is well-defined
 \begin{equation} \label{eq:def.G.family}
   \mathcal{G} := \{ \Phi_0(-t)_* f ; \enskip t \in [0,T], f \in \mathcal{F} \} \subset {\CC^\infty(B_{\delta};\K^d)}.
 \end{equation}
 Assume that the family $\mathcal{F}$ is semi-nilpotent of index $M$ with respect to $f_0$ and that there exists $r > 0$ such that $\mathcal{F} \cup \{f_0\} \subset {\CC^{\omega,r}(B_\delta;\K^d)}$. 
 Then $\mathcal{G}$ is nilpotent of index $M$.
\end{lemma}

\begin{proof}
 For $t \in [0,T]$ and $f \in \mathcal{F}$, equation \eqref{tool_series} of \cref{Lem:tool_serie_ad} implies that
 \begin{equation}
   \Phi_0(-t)_* f = \sum_{k=0}^{+\infty} \frac{t^k}{k!} \ad^k_{f_0}(f)
 \end{equation}
 and that the series converges absolutely in ${\CC^M(B_{\delta};\K^d)}$ (in particular). Hence, if $t_1, \dotsc, t_M \in [0,T]$ and $f_1, \dotsc f_M \in \mathcal{F}$, the bracket 
 \begin{equation}
  \begin{split}
     [ \Phi_0(-t_M)_* f_M, & [ \dotsb [ \Phi_0(-t_2)_* f_2 , \Phi_0(-t_1)_* f_1 ] \dotsb ] ] \\ 
      & = \sum_{k_1, \dotsc k_M \in \N} \frac{t_1^{k_1}\dotsb t_M^{k_M}}{k_1! \dotsb k_M!}
      [ \ad^{k_M}_{f_0}(f_M), [ \dotsb [ \ad^{k_2}_{f_0}(f_2), \ad^{k_1}_{f_0}(f_1)] \dotsb]]
  \end{split}
 \end{equation}
 vanishes thanks to the assumption and the absolute convergence of the sums. The same is true for every other bracket structure, which proves that $\mathcal{G}$ is nilpotent of index $M$.
\end{proof}

\begin{corollary} \label{thm:analytic+brackets-nilpotent=>magnus11-equality}
    Let $T, \delta, r > 0$.
    Let $f_0 \in {\CC^{\omega,r}(B_{4\delta};\K^d)}$ such that $T \| f_0 \|_{\CC^0} \leq \delta$ and $f_1 \in L^1([0,T];{\CC^{\omega,r}(B_{4\delta};\K^d)})$.
 Assume moreover that $\mathcal{F} := \{ f_1(t,\cdot); \enskip t \in [0,T] \}$ is semi-nilpotent of index $M + 1$ with respect to $f_0$. Then, for each $t \in [0,T]$ and $p \in B_\delta$, the solution to \eqref{eq:EDO.interaction} satisfies \eqref{eq:equal.magnus2.nilpotent}, where $\mathcal{Z}_{M}(t,f_0,f_1)$ is defined in \cref{Prop:Magnus_2}.
\end{corollary}

\begin{proof}
 This corollary is a direct consequence of \cref{thm:magnus2.flow-nilpotent} and \cref{thm:analytic+brackets-nilpotent=>flow-nilpotent}.
\end{proof}

The analyticity assumption in  \cref{thm:analytic+brackets-nilpotent=>flow-nilpotent} is necessary, as illustrated by the following counter-example for smooth functions.

\begin{example}
 We consider smooth vector fields on $\R^3$.
 Let $\chi \in \CC^\infty(\R;\R)$ with $\chi \equiv 0$ on $\R_-$ and $\chi(x) > 0$ for $x > 0$.
 Let $f_0$ and $\mathcal{F} := \{ f_1, f_2 \}$ where 
 \begin{align}
    f_0(x) & := e_2, \\
    f_1(x) & := \chi(x_2) x_1 e_3, \\
    f_2(x) & := \chi(-x_2) e_1.
 \end{align}
 Heuristically, $f_1$ and $f_2$ commute because they have disjoint (touching) supports, but the flow of $f_0$ involved in \eqref{eq:def.G.family} mixes these supports for every positive time. This is possible only because $\chi$ is not analytic.
 
 First, we check that $\mathcal{F}$ is semi-nilpotent of order $2$ with respect to $f_0$. Indeed, for every $j \in \N$,
 \begin{align}
    \ad^j_{f_0}(f_1)(x) & = \chi^{(j)}(x_2) x_1 e_3, \\
    \ad^j_{f_0}(f_2)(x) & = (-1)^j \chi^{(j)}(-x_2) e_1.
 \end{align}
 Thus, for $j, k \in \N$, $[\ad^j_{f_0}(f_1), \ad^{k}_{f_0}(f_1)]$ (resp.\ $[\ad^j_{f_0}(f_2), \ad^k_{f_0}(f_2)]$) vanishes because both vector fields are {multiples of} $e_3$ but {independent of} $x_3$ (resp.\ {multiples of} $e_1$ but {independent of} $x_1$). Moreover,
 \begin{equation}
   [\ad^j_{f_0}(f_1), \ad^{k}_{f_0}(f_2)](x) = - (-1)^k \chi^{(k)}(-x_2) \chi^{(j)}(x_2) e_3 = 0,
 \end{equation}
 because the supports of $\chi(\cdot)$ and $\chi(- \cdot)$ only touch at $x_2 = 0$ where all derivatives vanish.
 
 Second, let us check however that the family $\mathcal{G}$ defined in \eqref{eq:def.G.family} is not nilpotent of index $2$. Indeed, for $t \geq 0$ and $x \in \R^3$, $\Phi_0(t)(x) = x + t e_2$. Thus, for $f \in \CC^\infty(\R^3;\R^3)$, $(\Phi_0(-t)_*f)(x) = f(x+te_2)$. 
 Therefore, for every $T > 0$, $\mathcal{G}$ is well-defined on $\R^3$. Moreover,
 \begin{equation}
   [ f_2, (\Phi_0(-t)_*f_1) ](x)
   = \chi(-x_2) \chi(x_2+t) e_3.
 \end{equation}
 In particular, for every $\varepsilon > 0$, $[ f_2, (\Phi_0(-2\varepsilon)_*f_1) ](-\varepsilon e_2) = \chi(\varepsilon)^2 e_3 \neq 0$, which prevents the family $\mathcal{G}$ from being nilpotent of index $2$ (even locally in time and space).
\end{example}

The analyticity assumption in \cref{thm:analytic+brackets-nilpotent=>magnus11-equality} is also necessary, as illustrated by the following counter-example for smooth functions, inspired by the previous one.

\begin{example}
 We consider smooth vector fields on $\R^3$. 
 Let $\chi \in \CC^\infty(\R;\R)$ with $\chi \equiv 0$ on $\R_-$ and $\chi(x) > 0$ for $x > 0$.
 Let $f_0(x) := e_2$ and $f_1(t,x) := f_1(x)$ ({independent of} time) with
\begin{equation}
  f_1(x) := 2 {\chi^{(1)}}(x_2) x_1 e_3 + {\chi^{(1)}} (-x_2) e_1.
\end{equation}
For $j \in \N$, one has
\begin{equation}
  \ad^j_{f_0}(f_1) (x) = \partial_2^j f_1(x)
  = 2 \chi^{(j+1)}(x_2) x_1 e_3 + (-1)^j \chi^{(j+1)}(-x_2) e_1.
\end{equation}
Thus, for every $j_1, j_2 \in \N$,
\begin{equation}
 \begin{split}
   [\ad^{j_1}_{f_0}(f_1), \ad^{j_2}_{f_0}(f_1)] (x) & = 
   2 (-1)^{j_1} \chi^{(j_1+1)}(-x_2) \chi^{(j_2+1)}(x_2) e_3
   \\ & \quad - 2 (-1)^{j_2} \chi^{(j_2+1)}(-x_2) \chi^{(j_1+1)}(x_2) e_3 = 0
 \end{split}
\end{equation}
because the supports of $\chi(\cdot)$ and $\chi(- \cdot)$ only touch at $x_2 = 0$, where all derivatives vanish. 
Hence each bracket of $f_0$ and $f_1$ involving $f_1$ at least twice vanishes identically on $\R^3$.
Thus, for every $T > 0$, the family $\mathcal{F} := \{ f_1(t,\cdot); \enskip t \in [0,T] \} = \{ f_1 \}$ is semi-nilpotent of index $2$ with respect to $f_0$.
Let us prove that, despite this property, equality \eqref{eq:equal.magnus2.nilpotent} with $M = 1$ fails.

\bigskip

\noindent \textbf{Computation of the state.} We solve $\dot{x} = f_0(x) + f_1(x)$ for some initial data $p$. Solving the ODE successively for $x_2$, $x_1$ and $x_3$, we obtain 
\begin{align}
    x_1(t) & = p_1 + \chi(-p_2) - \chi(-p_2-t), \\
    x_2(t) & = p_2 + t, \\
    x_3(t) & = p_3 + 2 \left( \chi(p_2+t) - \chi(p_2) \right) (p_1 + \chi(-p_2)).
\end{align}
In particular, with $t=2\varepsilon$ and $p=-\varepsilon e_2$, $x(2\varepsilon;f_0+f_1,-\varepsilon e_2)=(\chi(\varepsilon),\varepsilon,2 \chi(\varepsilon)^2)$.

\bigskip

\noindent \textbf{Computation of the flow.} We compute $e^{\mathcal{Z}_1(t,f_0,f_1)} e^{t f_0}(p)$ for some initial data $p$. One has $\Phi_0(\tau,q) = q + \tau e_2$. Hence, in particular $(\Phi_0(\tau)_* f_1)(q) = f_1(q - \tau e_2)$.
Moreover $\mathcal{Z}_1(t,f_0,f_1)(q) = \int_0^t g_t(s,q) \dd s$ where $g_t(s,q) = (\Phi_0(t-s)_* f_1)(q)$. Hence $g_t(s,q) = f_1(q-(t-s)e_2)$ and
\begin{equation}
 \begin{split}
   \mathcal{Z}_1(t,f_0,f_1)(q) & = \int_0^t f_1(q+(s-t)e_2) \dd s \\
 & = 2 q_1 (\chi(q_2)-\chi(q_2-t)) e_3 + (\chi(-q_2+t) - \chi(-q_2))e_1.
 \end{split}
\end{equation}
Then $e^{\mathcal{Z}_1(t,f_0,f_1)} e^{tf_0} p = e^{\mathcal{Z}_1(t,f_0,f_1)}  (p+te_2)$ is $y(1)$ where $y$ is the solution to $y(0) = p+te_2$ and $\dot{y}(s) = \mathcal{Z}_1(t,f_0,f_1)(y(s))$.
Solving the ODE successively for $y_2$, $y_1$ and $y_3$, we obtain 
\begin{align}
    y_1(s) & = p_1 + s (\chi(-p_2) - \chi(-p_2-t)), \\
    y_2(s) & = p_2 + t, \\
    y_3(s) & = p_3 +  (\chi(p_2+t)-\chi(p_2))\left( 2 p_1 s + s^2 (\chi(-p_2) - \chi(-p_2-t)) \right).
\end{align} 
In particular, with with $t=2\varepsilon$ and $p=-\varepsilon e_2$, $e^{\mathcal{Z}_1(2\varepsilon,f_0,f_1)} e^{2\varepsilon f_0} (-\varepsilon e_2) = (\chi(\varepsilon) , \varepsilon,\chi(\varepsilon)^2 )$.
Thus, for every $\varepsilon > 0$,
\begin{equation}
 \left| x(2\varepsilon,f_0+f_1,-\varepsilon e_2) -
 e^{\mathcal{Z}_1(2\varepsilon,f_0,f_1)} e^{2\varepsilon f_0} (-\varepsilon e_2) \right|
 = \chi^2(\varepsilon) > 0.
\end{equation}
\end{example}

\subsubsection{Convergence for linear systems}
\label{sec:conv-interaction-linear}

Let $T > 0$. In this paragraph, we study linear systems of the form
\begin{equation} \label{eq:xH0H1}
    \dot{x}(t)= \left( H_0 + H_1(t) \right) x(t) 
    \quad \textrm{and} \quad x(0)=p,
\end{equation}
where $H_0 \in \mathcal{M}_d(\K)$ and $H_1 \in L^1((0,T);\mathcal{M}_d(\K))$.
Let $\|\cdot\|$ be a sub-multiplicative norm on $\mathcal{M}_d(\K)$.

\begin{proposition} \label{Prop:CV_Magnus1.1_matrix}
    Let $T>0$, $H_0 \in \mathcal{M}_d(\K)$ and $H_1 \in L^1((0,T);\mathcal{M}_d(\K))$ such that $\|H_1\|_{{L^1(0,T)}} < \frac{e^{-2T\|H_0\|}}{8}$. 
    Then, for each $t \in [0,T]$ and $p\in\K^d$ the solution to \eqref{eq:xH0H1} satisfies $x(t)=e^{-Z_\infty(t)} e^{t H_0} p$ where
$Z_\infty(t)$ is defined by (\ref{Zinfty_matrix}) with
 \begin{equation}A_t(\tau)= e^{(t-\tau)H_0} H_1(\tau) e^{(\tau-t)H_0}=\sum_{k=0}^{+\infty} \frac{(t-\tau)^k}{k!} \ad_{H_0}^k(H_1).\end{equation}
\end{proposition}

\begin{proof}
The function $y:\tau \in [0,t] \mapsto e^{(t-\tau)H_0} x(\tau)$ {satisfies} $y'(\tau)=A(\tau)$, $y(0)=e^{tH_0} p$. Thus, by \cref{Prop:CV_Magnus1.0_matrix},
$y(t)=e^{-Z_\infty(t)} e^{t H_0} p$, which gives the conclusion because $y(t)=x(t)$.
\end{proof}

\begin{remark}
    The Magnus expansion in the usual setting (\cref{Prop:CV_Magnus1.0_matrix}), when applied directly to $A(t)=H_0+H_1(t)$ requires a smallness assumption on $T \|H_0\|$ (through the condition $\|A\|_{{L^1(0,T)}} < \frac 1 8$), even for small perturbations $H_1$. 
    On the contrary, the Magnus expansion in the interaction picture (\cref{Prop:CV_Magnus1.1_matrix}) holds even when $T \| H_0 \|$ is large, provided that the perturbation $H_1$ is small enough.
\end{remark}

{\begin{remark}
    More generally, in \cite[p.~1671]{MR579930}, the authors consider the formal power series expressing the chronological logarithm of two flows, associated to two non-autonomous vector fields. They explain that, when the vector fields take values in a Banach algebra, and one of them is small enough, then this series converges. \cref{Prop:CV_Magnus1.1_matrix} is an illustration.
\end{remark}}

\subsubsection{Divergence for arbitrary small analytic vector fields}
\label{sec:conv-interaction-analytic}

Generally speaking, since, as illustrated in \cref{sec:conv-magnus-analytic}, the Magnus expansion does not converge for analytic vector fields, one cannot expect that the Magnus expansion in interaction picture converges for analytic vector fields.

For instance, if $f_0 = 0$, or if, for some $a \in \intset{1,d}$, $f_0(x)$ {is a linear combination of $e_1, \dotsc, e_a$ with coefficients depending only on $x_1, \dotsc, x_a$} and $f_1(t,x)$ {is a linear combination of $e_{a+1}, \dotsc e_d$, with coefficients depending only on $x_{a+1}, \dotsc x_{d}$}, then the vector field $g_t(\tau) = \Phi_0(t-\tau)_* f_1(\tau)$ defined in~\eqref{def:gt_star} and involved in the Magnus in the interaction picture formula satisfies $g_t(\tau) = f_1(\tau)$.

Hence, each counter-example to the convergence of the usual Magnus expansion also yields counter-examples to the convergence of the Magnus expansion in the interaction picture.

{More generally, in \cite[p.~1671]{MR579930}, the authors consider the formal power series expressing the chronological logarithm of two flows, associated to two non-autonomous vector fields. They claim that, even for analytic vector fields, this series does not converge in general. The counter examples of the present article illustrate this assertion.}

\subsection{Sussmann's infinite product expansion}

\subsubsection{Equality for nilpotent systems}
\label{sec:conv-sussmann-nilpotent}

In this section, we study affine systems of the form (\ref{affine_syst_q}).

\begin{proposition} \label{Prop:CV_Sussman_nilpontent}
    Let $\mathcal{B}$ be a \GHB\ of $\mathcal{L}(X)$ and
$(\xi_b)_{b\in\mathcal{B}}$ be the associated coordinates of the second kind.
	For every $M \in \N^*$, there exist $\eta_M>0$ such that the following property holds.
	Let $T, \delta > 0$, $f_i \in {\CC^{\infty}(B_{3\delta};\K^d)}$ and $u_i \in L^1((0,T);\K)$ for $i \in I$.
	Assume that the Lie algebra generated by the $f_i$ for $i \in I$ is nilpotent of index at most $M+1$. 
	Then, under the smallness assumption~\eqref{eq:sussman.small1}, for each $t \in [0,T]$ and $p \in B_\delta$,
\begin{equation} 
x(t;f,u,p) = \underset{b \in \mathcal{B}_{\intset{1,M}}}{\overset{\rightarrow}{\Pi}} e^{\xi_b(t,u) f_b } p.
\end{equation}
\end{proposition}

\begin{proof}
    The proof strategy is the same as for \cref{Prop:Error_Sussman}.
    We apply the second statement of \cref{Lem:tool_serie_ad} instead of the first one, which gives $\varepsilon_j=0$ for each $j \in \intset{0,k+1}$.
    The smallness assumption guarantees that all flows are well-defined.
\end{proof}

\subsubsection{{Bilinear systems}}
\label{sec:conv-sussmann-banach}

Let $T > 0$. 
In this paragraph, we study {the convergence of Sussmann's infinite product expansion for bilinear systems} of the form
\begin{equation} \label{affine_lin_syst_q}
\dot{x}(t)= \left(\sum_{i \in I} u_i(t) A_i\right) x(t)
\quad \textrm{and} \quad
x(0) = p
\end{equation}
where $A_i \in \mathcal{M}_d(\K)$ {are time-invariant} and $u_i \in L^1((0,T);\K)$. When well-defined, its solution is denoted $x(t;A,u,p)$.
{Local convergence is proved in \cref{Prop:CV_Prod_Matrix} while an example illustrating the lack of global convergence is proposed in \cref{Prop:contre-ex-matrix_CV_Prod}.}

\paragraph{{Local convergence}.}

The main goal of this paragraph is to prove \cref{Prop:CV_Prod_Matrix} which asserts that Sussmann's infinite product expansion for system \eqref{affine_lin_syst_q} converges locally (i.e.\ for small matrices, small controls or small time).

Before proving this result, we need a definition for an ordered infinite product (given in \cref{def:conv.lazard.prod.matrices} below) and a sufficient condition for its convergence (given in \cref{Lem:CS-CV-prod-matrix} below).

\bigskip

Defining the ordered product of a family of matrices indexed by a length-compatible Hall basis is straightforward, because there exists an indexation of the family by $\N$ which is compatible with the order induced by the Hall basis (since it does not involve infinite segments). Hence, one is brought back to the classical case of a sequence of products and usual definitions and convergence criteria can be used.

For {arbitrary }\GHBS~{(in the generalized sense of \cref{Def2:Laz})}, the situation is more intricate, due to the potential infinite segments which can prevent the order of the basis from being compatible with the order of natural integers. 
{For example, in the Lyndon basis of $\mathcal{L}(X)$ for $X = \{ X_0, X_1 \}$ with $X_0 < X_1$, $\ad^k_{X_0}(X_1) < (X_0,X_1)$ for all $k \geq 2$, so there is an infinite segment before $(X_0,X_1)$.}
This problem already appears for a product which would be indexed by $\N^2$ with the lexicographic order
\begin{equation}
	(0,0)<(0,1)<(0,2)<\dotsb<(1,0)<(1,1)<(1,2)<\dotsb <(2,0)<\dotsb
\end{equation}

We therefore propose a natural definition and a basic sufficient condition for convergence based on absolute convergence. In what follows, $\|\cdot\|$ is a submultiplicative norm on $\mathcal{M}_d(\K)$ such that $\|\mathrm{Id}\|=1$, for instance a subordinated norm.

\begin{definition} \label{def:conv.lazard.prod.matrices}
	Let $J$ be a totally ordered set and $(A_j)_{j \in J}$ matrices of $\mathcal{M}_d(\K)$. We say that the ordered product of the $e^{A_j}$ over $J$ \emph{converges} when there exists $M \in \mathcal{M}_d(\K)$ such that, for every $\varepsilon > 0$, there exists a finite subset $J_0$ of $J$ such that, for every finite subset $J_1$ of $J$ containing $J_0$, one has
	\begin{equation} \label{eq.def:M.j1}
	 \left\| M - \underset{j\in J_1}{\overset{\leftarrow}{\Pi}} e^{A_j} \right\| \leq \varepsilon.
	\end{equation}
	When such an $M$ exists, it is unique and we write
	\begin{equation}
		M = \underset{j\in J}{\overset{\leftarrow}{\Pi}} e^{A_j}.
	\end{equation}
\end{definition}

{The following natural convergence criteria also appears in \cite{karasev1976infinite}.}

\begin{lemma} \label{Lem:CS-CV-prod-matrix}
Let $J$ be a totally ordered set and $(A_j)_{j \in J}$ matrices of $\mathcal{M}_d(\K)$ such that
	\begin{equation} \label{sum.aj.a}
		\sum_{j \in J} \| A_j \| < + \infty.
	\end{equation}
	Then the ordered product of the $e^{A_j}$ over $J$ converges in the sense of \cref{def:conv.lazard.prod.matrices}.
\end{lemma}

\begin{proof}
Let $\alpha$ be the left-hand side of (\ref{sum.aj.a}).

\medskip

\noindent	\emph{Step 1: Basic claims.} We start with straightforward claims. First, for every $j \in J$, one has
	\begin{equation}
		\| e^{A_j} - \mathrm{Id} \| \leq e^{\|A_j\|} - 1 \leq \| A_j \| e^{\| A_j \|}
		\leq \| A_j \| e^\alpha.
	\end{equation}
	Second, for every finite part $J' \subset J$, one has
	\begin{equation}
	 	\left\| \underset{j\in J'}{\overset{\leftarrow}{\Pi}} e^{A_j} \right\|
	 	\leq \underset{j\in J'}{\Pi} e^{\|A_j\|}
	 	\leq e^\alpha.
	\end{equation}
	Third, for every finite parts $J_0 \subset J_1 \subset J$, one has
	\begin{equation} \label{eq:gap.j0.j1}
		\left\| \underset{j\in J_1}{\overset{\leftarrow}{\Pi}} e^{A_j}
		- \underset{j\in J_0}{\overset{\leftarrow}{\Pi}} e^{A_j} \right\|
		\leq e^{3\alpha} \sum_{j \in J_1 \setminus J_0} \| A_j \|.
	\end{equation}
	Indeed, writing $J_1 \setminus J_0 = \{ j_1 > \dotsb > j_n \}$, we have the following telescopic decomposition
	\begin{equation}
	  \underset{j\in J_1}{\overset{\leftarrow}{\Pi}} e^{A_j}
	  - \underset{j\in J_0}{\overset{\leftarrow}{\Pi}} e^{A_j} = 
	  \sum_{k=1}^n \underset{\substack{j\in J_0 \\ j > j_k}}{\overset{\leftarrow}{\Pi}} e^{A_j}
	  \left( e^{A_{j_k}} - \mathrm{Id} \right)
	  \underset{\substack{j \in J_1 \\ j < j_k}}{\overset{\leftarrow}{\Pi}} e^{A_j},
	\end{equation}
	which, together with the two first claims, proves estimate~\eqref{eq:gap.j0.j1}.
	
	\medskip
	
	\noindent \emph{Step 2: Construction of a limit.} We construct a possible limit. For each $n \geq 2$, let
	\begin{equation}
		J_n := \left\{ j \in J, \enskip \| A_j \| > \frac{1}{n} \right\}.
	\end{equation}
	Thanks to assumption \eqref{sum.aj.a}, the sets $J_n$ are finite and, moreover,
	\begin{equation}
		\varepsilon_n := \sum_{j \in J \setminus J_n} \|A_j\| \to 0.
	\end{equation}
	Now, for each $n \geq 2$, we define the matrix
	\begin{equation}
		M_n := \underset{j\in J_n}{\overset{\leftarrow}{\Pi}} e^{A_j}.
	\end{equation}
	This defines a Cauchy sequence in the complete space $\mathcal{M}_d(\K)$. Indeed, for every $n < p$, thanks to estimate~\eqref{eq:gap.j0.j1}, one has
	\begin{equation}
	 \| M_n - M_p \| \leq e^{3\alpha} \varepsilon_n.
	\end{equation}
	Hence, there exists $M \in \mathcal{M}_d(\K)$ towards which the sequence $(M_n)_{n\geq 2}$ converges.
	By letting $[p\rightarrow \infty]$ in the previous inequality we obtain, for every $n\geq 2$
	\begin{equation} \label{Mn-M}
	 \| M_n - M \| \leq e^{3\alpha} \varepsilon_n.
	\end{equation}
	
	\medskip
	
	\noindent \emph{Step 3: Proof of convergence.} We now prove that the ordered product of the $e^{A_j}$ over $J$ converges to~$M$ in the sense of \cref{def:conv.lazard.prod.matrices}. Let $\varepsilon > 0$. Let $n \geq 2$ large enough such that $e^{3\alpha} \varepsilon_n < \varepsilon/2$. For every finite set $J_1$ containing $J_n$, condition \eqref{eq.def:M.j1} holds thanks to \eqref{Mn-M} and \eqref{eq:gap.j0.j1}.
\end{proof}

\begin{proposition} \label{Prop:CV_Prod_Matrix}
    Let $\mathcal{B}$ be a \GHB\ of $\mathcal{L}(X)$, $(\xi_b)_{b\in\mathcal{B}}$ be the coordinates of the second kind associated to $\mathcal{B}$.
    There exists $\eta > 0$ such that the following property holds.
    Let $A_i \in \mathcal{M}_d(\K)$ for $i \in I$.
    For $b\in\mathcal{B}$, we define the matrix $A_b:=\Lambda(b)$ where $\Lambda:\mathcal{L}(X) \rightarrow \mathcal{M}_n(\K)$ is the homomorphism of Lie {algebras} such that $\Lambda(X_i)=A_i$ for $i \in I$ (see \Cref{Lem:identification_libre}).
    Let $T > 0$ and $u_i \in L^1((0,T);\K)$ for $i \in I$. 
    Assume that
    \begin{equation} \label{eq:smallness.sussmann.matrix}
        \|u\|_{{L^1(0,T)}} \| A \| \leq \eta.
    \end{equation}
    Then, for each $t \in [0,T]$ and $p \in \K^d$, the ordered product of the $e^{\xi_b(t,u) A_b}$ over $b\in\mathcal{B}$ converges. Moreover, for every $p\in\K^d$, 
    \begin{equation} \label{Suss_prod_matrix}
    x(t;A,u,p) = \underset{b \in \mathcal{B}}{\overset{\rightarrow}{\Pi}} e^{\xi_b(t,u) A_b } p. 
    \end{equation}
\end{proposition}

\begin{proof}
    Let $\eta := 1 / (8|I|^2)$.
    Let $T > 0$. Below, the variable $t$ implicitly belongs to $[0,T]$. To simplify the notations we write $\xi_b(t)$ instead of $\xi_b(t,u)$.

\medskip

\noindent \emph{Step 1: Convergence of the ordered product of the $e^{\xi_b(t) A_b}$ over $b\in\mathcal{B}$.}
One obtains, by induction on $|b|$, that for every $b \in \mathcal{B}$,
$\|A_b\| \leq (2 \|A\|)^{|b|}$.
Thus, recalling \eqref{eq:xib-u},
\begin{equation} \label{xi_b*A_b<}
    \| \xi_b(t) A_b \| \leq \left(2\|A\| \|u\|_{L^1(0,t)}\right)^{|b|}.
\end{equation}
Taking into account that $|\mathcal{B}_\ell| \leq |I|^{\ell}$, we obtain, using \eqref{eq:smallness.sussmann.matrix},
\begin{equation} \label{sum_xibAb}
\sum_{b\in\mathcal{B}} \| \xi_b(t) A_b\| 
\leq
\sum_{\ell=1}^{+\infty} \left(2|I|\|A\| \|u\|_{L^1(0,t)}\right)^\ell
\leq 1
 \end{equation}
and \cref{Lem:CS-CV-prod-matrix} gives the conclusion.

\medskip

\noindent \emph{Step 2: Estimates along a Lazard elimination in $\mathcal{B}_{\intset{1,M}}$.} Let $M \in \N^*$. We adopt the notations $b_1,\dotsc, b_{k+1}$ and $Y_0,\dotsc,Y_{k+1}$ of \cref{Def:Laz} and we define $x_0(t) := x(t)$ and, for $j \in \intset{1,k+1}$ 
 \begin{equation}
 x_j(t):=e^{-\xi_{b_j}(t,u) A_{b_j}} \dotsb e^{-\xi_{b_1}(t,u) A_{b_1}} x(t).
 \end{equation}
We prove by induction on $j \in \intset{0,k+1}$ that
\begin{equation} \label{Rec_Laz_matrix}
(\mathcal{H}_j): \begin{cases}
\dot{x}_j(t) = \left( \sum_{b \in \mathcal{B}_{\intset{1,M}} \cap Y_j} \dot{\xi}_{b}(t) A_b + \varepsilon_j(t) \right) x_j(t), \\
x_j(0)=p,
\end{cases}
\end{equation}
where $\varepsilon_0=0$ and
\begin{equation} \label{varepsj<}
\|\varepsilon_j(t)\| \leq \left( M |I| \|A\| |u(t)| (4|I|\|A\|\|u\|_{{L^1(0,t)}})^M
+ \|\varepsilon_{j-1}(t)\| \right) e^{2 \| \xi_{b_j}(t) A_{b_j}\|}.
\end{equation}

First, $(\mathcal{H}_0)$ holds with $\varepsilon_0=0$ because $x_0(t)=x(t)$ and $\dot{\xi}_{X_i}(t)=u_i(t)$ for $i \in I$. Let $j\in\intset{1,k+1}$ and assume that $(\mathcal{H}_{j-1})$ holds. We deduce from the definition of $x_j$ that
 \begin{equation}x_j(t)=e^{-\xi_{b_j}(t)A_{b_j}} x_{j-1}(t) \end{equation}
and from $(\mathcal{H}_{j-1})$ that
\begin{equation}
    \begin{split}
        \dot{x}_j(t)
& = -\dot{\xi}_{b_j}(t) A_{b_j} x_j(t) + e^{-\xi_{b_j}(t)A_{b_j}}
\left( \sum_{b \in \mathcal{B}_{\intset{1,M}} \cap Y_{j-1}} \dot{\xi}_{b}(t) A_b + \varepsilon_{j-1}(t) \right) e^{\xi_{b_j}(t)A_{b_j}} x_{j}(t)
\\ & =
\left( \sum_{b \in \mathcal{B}_{\intset{1,M}} \cap Y_{j-1} \setminus \{b_j\}} \dot{\xi}_{b}(t) e^{-\xi_{b_j}(t)A_{b_j}} A_b e^{\xi_{b_j}(t)A_{b_j}} + \widetilde{\varepsilon}_{j-1}(t) \right)  x_{j}(t)
    \end{split}
\end{equation}
where $\widetilde{\varepsilon}_{j-1}(t):=e^{-\xi_{b_j}(t)A_{b_j}} \varepsilon_{j-1}(t)e^{\xi_{b_j}(t)A_{b_j}}$ satisfies,
\begin{equation} \label{tildepsj_matrix}
\|\widetilde{\varepsilon}_{j-1}(t) \| \leq \|\varepsilon_{j-1}(t)\| e^{2 \| \xi_{b_j}(t) A_{b_j} \|}.
\end{equation}
For $b \in \mathcal{B}_{\intset{1,M}} \cap Y_{j-1} \setminus \{b_j\}$, let $h(b) \in \N^*$ be the maximal integer such that (\ref{def:mu(b)_longueur}) holds
and 
 \begin{equation}
    \overline{\varepsilon}_b^j(t):= \dot{\xi}_{b}(t) e^{-\xi_{b_j}(t)A_{b_j}} A_b e^{\xi_{b_j}(t)A_{b_j}} - \sum_{m=0}^{h(b)-1} \dot{\xi}_{b}(t) \frac{\xi_{b_j}^m(t)}{m!} A_{\ad_{b_j}^m(b)}.\end{equation}
Then, by definition of $Y_j$, $(\mathcal{H}_j)$ holds with $\varepsilon_j$ defined by (\ref{def:varepsilon(j)}).
Using the fourth statement of \cref{Lem:tool_serie_ad}, \eqref{eq:dotxib-u}, \eqref{xi_b*A_b<}, we obtain
\begin{equation} \label{epsilonb<_matrix}
    \begin{split}
        \|\overline{\varepsilon}_b^j(t)\| 
        & \leq |\dot{\xi}_{b}(t)| \frac{(2\|\xi_{b_j}(t)A_{b_j}\|)^{h(b)}}{h(b)!} \|A_b\| e^{2 \| \xi_{b_j}(t)A_{b_j} \|} 
        \\ & \leq 
        |b| |u(t)| \|u\|_{L^1}^{|b|-1}
        \cdot \frac{(2\|A\|\|u\|_{L^1})^{h(b)|b_j|}2^{h(b)}}{h(b)!} \cdot (2\|A\|)^{|b|}
        \\  & \leq 
        M |u(t)| (4\|A\|\|u\|_{L^1})^{M} \|A\|,
    \end{split}
\end{equation}
taking into account $M+1 \leq |b|+h(b)|b_j| \leq 2M$ and $\|A\|\|u\|_{L^1} \leq 1$.

We deduce from (\ref{def:varepsilon(j)}), (\ref{tildepsj_matrix}), (\ref{epsilonb<_matrix}) and the relation $|\mathcal{B}_{\intset{1,M}}| \leq |I|^{M+1}$ that (\ref{varepsj<}) holds.

\medskip

\emph{Step 4: Proof of an estimate on the ordered product of the $e^{\xi_b(t)A_b}$ over $\mathcal{B}_{\intset{1,M}}$.} We deduce from (\ref{varepsj<}),
(\ref{sum_xibAb}) and the relation $k+1 = |\mathcal{B}_{\intset{1,M}}| \leq |I|^{M+1}$ that
\begin{equation}
    \| \varepsilon_{k+1}(t) \| 
    \leq
    e M \|A\| |I|^2 |u(t)| (4 |I|^2 \|A\|\|u\|_{{L^1(0,t)}})^M.
\end{equation}
Hence, using \eqref{eq:smallness.sussmann.matrix},
\begin{equation}
    \| \varepsilon_{k+1} \|_{{L^1(0,t)}}
    \leq \frac{e}{4} (4 |I|^2 \|A\| \|u\|_{{L^1(0,t)}})^{M+1}
    \leq 2^{-M}.
\end{equation}
We deduce from $(\mathcal{H}_{k+1})$, (\ref{Laz.3}) and Gr\"onwall's lemma that
 \begin{equation}\left|\underset{b \in \mathcal{B}_{\intset{1,M}}}{\overset{\leftarrow}{\Pi}} e^{-\xi_b(t,u) A_b }x(t)-p\right| = \left| x_{k+1}(t)-p \right| \leq \int_0^t |\varepsilon_{k+1}(\tau) x_{k+1}(\tau)| d\tau \leq 2^{-M} e |p| \end{equation}
Multiplying both sides by the finite product $\underset{b \in \mathcal{B}_{\intset{1,M}}}{\overset{\rightarrow}{\Pi}} e^{\xi_b(t,u) A_b }$ gives
 \begin{equation}\left|x(t)-\underset{b \in \mathcal{B}_{\intset{1,M}}}{\overset{\rightarrow}{\Pi}} e^{\xi_b(t,u) A_b }p\right| \leq e^2 2^{-M} |p| \end{equation}
Passing to the limit $[M\rightarrow \infty]$ in the previous estimate gives (\ref{Suss_prod_matrix}).
\end{proof}

\paragraph{{Lack of global convergence}.} 

The goal of this paragraph is to illustrate that the smallness assumption \eqref{eq:smallness.sussmann.matrix} in \cref{Prop:CV_Prod_Matrix} is necessary because the equality does not hold globally.

\begin{proposition} \label{Prop:contre-ex-matrix_CV_Prod}
{Consider the constant control} $\overline{u}: t \in \R_+ \mapsto (1,1)\in\R^2$.
\begin{enumerate}
    \item  There exist a Hall basis $\mathcal{B}$ of $\mathcal{L}(\{X_1,X_2\})$ and a subsequence  $(b_k)_{k\in\N}$ of $\mathcal{B}$ such that
\begin{equation} \label{ss_sussm}
\exists \gamma>0, \forall k\in\N, t\geq 0, \quad \xi_{b_k}(t,\overline{u}) \geq \left( \frac{t}{\gamma} \right)^{|b_k|}
\end{equation}
\item There exists $A_1, A_2 \in \mathcal{M}_3(\C)$ and $t>\gamma$ such that 
$(e^{\xi_{b_k}(t,\overline{u}) A_{b_k}})_{k\in\N}$ does not converge to $\mathrm{Id}_3$ in $\mathcal{M}_3(\C)$. Thus, the ordered product of the $e^{\xi_b(t,\overline{u}) A_b}$ over $\mathcal{B}$ does not converge in $\mathcal{M}_3(\C)$.
\end{enumerate}
\end{proposition}

\begin{proof}
For the first point we adapt an argument due to Sussmann in \cite[pages 333-335]{MR935387}.
We define by induction two sequences $(b_k^1)_{k\in\N}$ and $(b_k^2)_{k\in\N}$ of $\Br(\{X_1,X_2\})$ by
\begin{equation}
b_0^1=X_1,\qquad  b_0^2=X_2, \qquad b_{k+1}^1:=[b_k^2,[b_k^1,b_k^2]], \qquad b_{k+1}^2:=[b_k^1,[b_k^1,b_k^2]].
\end{equation}
There exists a Hall basis of $\mathcal{L}(\{X_1,X_2\})$, whose order, denoted $<$, is compatible with length and such that, for every $k\in\N$, $b_k^1, b_k^2 \in \mathcal{B}$ and $b_k^1<b_k^2$. It suffices to choose, on the brackets with length $3^k$, some order such that $b_k^1 < b_k^2$. Then, automatically, $[b_k^1,b_k^2] \in \mathcal{B}$ and thus $b_{k+1}^1, b_{k+1}^2 \in \mathcal{B}$.
Such a process indeed allows to construct a Hall basis (see \cref{rk:hall-construction}), provided that one chooses an arbitrary length-compatible order on all other brackets.
\bigskip

To lighten the notations, we write $\xi_b(t)$, instead of $\xi_b(t,\overline{u})$.
We have $\xi_{X_1}(t)=\xi_{X_2}(t)=t$. An easy induction shows that, for every $b\in\mathcal{B}$, $\dot{\xi}_b(t)=t^{|b|-1} / \alpha_b$, where $\alpha_b \in \N^*$. 
The constants $\alpha_b$ can be computed recursively: $\alpha_{X_1} = \alpha_{X_2} = 1$ and, if $b=\ad_{b_1}^m(b_2)$ with $m\in\N^*$, $b_1< b_2$ and $\lambda(b_2) < b_1$ then $\alpha_b = \alpha_{b_1}^m |b_1|^m m! \alpha_{b_2}$. 
In particular, for every $k\in\N$,
 \begin{equation}\alpha_{b_{k+1}^1}=  \alpha_{b_k^1} \alpha_{b_k^2}^2 |b_k^2| |b_k^1| =  \alpha_{b_k^1} \alpha_{b_k^2}^2 3^{2k} , \qquad \alpha_{b_{k+1}^2}= 2 \alpha_{b_k^1}^2 \alpha_{b_k^2} |b_k^1|^2  =   2 \alpha_{b_k^1}^2 \alpha_{b_k^2} 3^{2k}.\end{equation}
Let $\beta_k=\max\{\alpha_{b_k^1},\alpha_{b_k^2}\}$. Then, $\beta_0 = 1$ and, by the previous relations,
\begin{equation}
\beta_{k+1} \leq 3^{2k+1} \beta_k^3.
\end{equation}
Thus $\theta_k:=3^{-k} \ln(\beta_k)$ satisfies $\theta_0 = 0$ and
\begin{equation}
\theta_{k+1} \leq \theta_k + (2k+1) 3^{-(k+1)} \ln(3),
\end{equation}
which leads to $\theta_k \leq \eta:=\sum_{j=1}^{+\infty}(2j+1) 3^{-(j+1)} \ln(3)$ i.e.\ $\beta_k \leq (\gamma')^{3^k}$ where $\gamma' =e^\eta$. 
Therefore, for every $k \in \N$ and $j \in \{1,2\}$ we have
 \begin{equation}\left| \xi_{b_k^j}(t) \right| \geq \frac{1}{|b_k^j|} \left(\frac{t}{\gamma'}\right)^{|b_k^j|}. \end{equation}
Let $\gamma > \gamma'$ be such that, for every $k \in \N$,
$\frac{1}{3^k} \left( \frac{\gamma}{\gamma'} \right)^{3^k} \geq 1$. Then (\ref{ss_sussm}) holds, for instance with $b_k = b_k^1$.

\bigskip

For the second point, let, for $j\in\{1,2,3\}$, $F_j \in \mathcal{M}_3(\R)$ be the matrix of the linear map $x\in\R^3 \mapsto e_j \wedge x$. Then $[F_1,F_2]=F_3$, $[F_2,F_3]=F_1$ and $[F_3,F_1]=F_2$. In particular
\begin{equation}
[F_2,[F_1,F_2]]=F_1, \quad [F_1,[F_1,F_2]]=-F_2.
\end{equation}
We consider $A_1=e^{i\frac{\pi}{6}} F_1$ and $A_2=e^{i\frac{\pi}{6}} F_2$ in $\mathcal{M}_3(\C)$. One easily proves by induction on $k\in\N^*$ that
$A_{b_k^1}=(-1)^{k+1} i F_1$ and $A_{b_k^2}= - i F_2$.
We have, for every $k\in\N$ and $t\in\R$
\begin{equation} e^{\xi_{b_k}(t) A_{b_k}} = \begin{pmatrix}
1 & 0 & 0 \\ 0 & \cosh(\xi_{b_k}(t)) & i(-1)^k \sinh(\xi_{b_k}(t)) \\ 0 & i(-1)^{k+1} \sinh(\xi_{b_k}(t)) & \cosh(\xi_{b_k}(t)) \end{pmatrix} \end{equation}
By (\ref{ss_sussm}), this sequence of matrices diverges for every $t>\gamma$.
\end{proof}

\subsubsection{Investigation for analytic vector fields} \label{subsec:algo}

In this paragraph, we study affine systems (\ref{affine_syst_q}). Our goal is to explain the difficulty of the convergence question for Sussmann's infinite product for arbitrary analytic vector fields. 
First, we state a definition (\cref{def:conv.lazard.prod.champs}) and a sufficient condition for the convergence (\cref{Lem_CV_prod_chpvect/CVA}), in the same spirit as for matrices. 
Then we show that they do not provide convergence for general analytic vector fields and we formulate an open problem.

\begin{definition} \label{def:conv.lazard.prod.champs}
	Let $J$ be a totally ordered set, $\delta>0$ and $(f_j)_{j \in J}$ a family of ${\CC^1(B_{2\delta};\K^d)}$.
	We say that the ordered product of the $e^{f_j}$ over $J$ converges uniformly on $B_\delta$
	if there exists $g \in {\CC^0(B_\delta;\K^d)}$ such that, for every $\varepsilon > 0$, there exists a finite subset $J_0$ of $J$ such that, for every finite subset $J_1$ of $J$ containing $J_0$, and for every $p\in B_\delta$ one has
	\begin{equation} \label{def:prod_CV_chpvect}
	 \left\| g(p) - \underset{j\in J_1}{\overset{\leftarrow}{\Pi}} e^{f_j} p \right\| \leq \varepsilon.
	\end{equation}
	When such a $g$ exists, it is unique and we write
	\begin{equation}
		g = \underset{j\in J}{\overset{\leftarrow}{\Pi}} e^{f_j}.
	\end{equation}
\end{definition}

\begin{lemma} \label{Lem_CV_prod_chpvect/CVA}
Let $J$ be a totally ordered set, $\delta>0$ and $(f_j)_{j \in J}$ a family of ${\CC^1(B_{2\delta};\K^d)}$ such that
\begin{equation} \label{sum.fj.a}
\sum_{j\in J} \|f_j\|_{\CC^0} < \delta \quad \text{ and } \quad \alpha := \sum_{j\in J} \|f_j\|_{\CC^1} < \infty.
\end{equation}
Then the ordered product of the $e^{f_j}$ over $J$ converges uniformly on $B_\delta$ and is $e^\alpha$-Lipschitz.
\end{lemma}

\begin{proof}
    We proceed as in the proof of \cref{Lem:CS-CV-prod-matrix}.
    \medskip

\noindent	\emph{Step 1: Basic claims.} First, for every finite subset $J' \subset J$ and $p \in \K^d$ with $|p| \leq 2\delta-\sum_{j\in J'} \|f_j\|_{\CC^0}$, then
	\begin{equation} \label{dp(produit)}
	\underset{j\in J'}{\overset{\leftarrow}{\Pi}} e^{f_j} p \in B_{2\delta}
	\qquad \text{ and } \qquad
	 	\left\| \partial_p \left[ \underset{j\in J'}{\overset{\leftarrow}{\Pi}} e^{f_j} p \right] \right\|
	 	\leq \underset{j\in J'}{\Pi} e^{\|f_j'\|_{\CC^0}}
	 	\leq e^\alpha
	\end{equation}
	because of \cref{lm:basic-flow} and the chain rule.
	
	Second, for every finite parts $J_0 \subset J_1 \subset J$
	and $p \in \K^d$ with $|p| \leq 2\delta-\sum_{j \in J_1} \|f_j\|_{\CC^0}$
	one has
	\begin{equation} \label{eq:gap.j0.j1_chvect}
		\left\| \underset{j\in J_1}{\overset{\leftarrow}{\Pi}} e^{f_j} p
		- \underset{j\in J_0}{\overset{\leftarrow}{\Pi}} e^{f_j} p \right\|
		\leq e^{\alpha} \sum_{j \in J_1 \setminus J_0} \| f_j \|_{\CC^0}.
	\end{equation}
	Indeed, writing $J_1 \setminus J_0 = \{ j_1 > \dotsb > j_n \}$, we have the following telescopic decomposition
	\begin{equation} \label{Prod_chpvect_telescop}
	  \underset{j\in J_1}{\overset{\leftarrow}{\Pi}} e^{f_j} p
	  - \underset{j\in J_0}{\overset{\leftarrow}{\Pi}} e^{f_j} p = 
	  \sum_{k=1}^n \left\{
	  \left( \underset{\substack{j\in J_0 \\ j > j_k}}{\overset{\leftarrow}{\Pi}} e^{f_j} \right) e^{f_{j_k}}
	  \left( \underset{\substack{j \in J_1 \\ j < j_k}}{\overset{\leftarrow}{\Pi}} e^{f_j} \right) p
	  -
	  \left( \underset{\substack{j\in J_0 \\ j > j_k}}{\overset{\leftarrow}{\Pi}} e^{f_j} \right)
\left(\underset{\substack{j \in J_1 \\ j < j_k}}{\overset{\leftarrow}{\Pi}} e^{f_j}\right) p \right\}.
	\end{equation}
	For $k \in \intset{1,n}$, let $x_k:= \underset{\substack{j \in J_1 \\ j < j_k}}{\overset{\leftarrow}{\Pi}} e^{f_j} p$ which is a point in $B_{2\delta-\|f_{j_k}\|_{\CC^0}}$. By (\ref{dp(produit)}) and (\ref{eq:efp-p}),
		the term with index $k$ in the previous sum is bounded by
	 \begin{equation}\left| \left( \underset{\substack{j\in J_0 \\ j > j_k}}{\overset{\leftarrow}{\Pi}} e^{f_j} \right) e^{f_{j_k}}  	x_k 	-\left( \underset{\substack{j\in J_0 \\ j > j_k}}{\overset{\leftarrow}{\Pi}} e^{f_j} \right) x_k \right| 	\leq e^\alpha \left|e^{f_{j_k}}  	x_k - x_k \right| \leq e^\alpha \|f_{j_k}\|_{\CC^0}.	\end{equation}
	which, together with (\ref{Prod_chpvect_telescop}) proves (\ref{eq:gap.j0.j1_chvect}).
	
	\medskip
	
	\noindent \emph{Step 2: Construction of a limit.} We construct a possible limit. For each $n \geq 2$, let
	\begin{equation}
		J_n := \left\{ j \in J, \enskip \| f_j \|_{\CC^1} > \frac{1}{n} \right\}.
	\end{equation}
	Thanks to assumption \eqref{sum.fj.a}, the sets $J_n$ are finite and, moreover,
	\begin{equation}
		\varepsilon_n := \sum_{j \in J \setminus J_n} \|f_j\|_{\CC^1} \to 0.
	\end{equation}
	Now, for each $n \geq 2$, we define $g_n \in {\CC^0(B_\delta;\K^d)}$ by
	\begin{equation}
		g_n(p) := \underset{j\in J_n}{\overset{\leftarrow}{\Pi}} e^{f_j} p.
	\end{equation}
	This defines a Cauchy sequence in the complete space ${\CC^0(B_{\delta};\K^d)}$. Indeed, for every $n < n'$ and $p\in B_\delta$, thanks to estimate~\eqref{eq:gap.j0.j1_chvect}, one has
	\begin{equation}
	 | g_n(p) - g_{n'}(p) | \leq e^{\alpha} \varepsilon_n.
	\end{equation}
	Hence, there exists $g \in {\CC^0(B_{\delta};\K^d)}$ towards which the sequence $(g_n)_{n\geq 2}$ uniformly converges on $B_\delta$. By (\ref{dp(produit)}), $g_n$ is $e^\alpha$-Lipschitz on $B_\delta$ for every $n\in\N$, thus so is $g$.
	By letting $[n' \rightarrow \infty]$ in the previous inequality we obtain, for every $n\geq 2$ and $p \in B_\delta$
	\begin{equation} \label{xn-z}
	 | g_n(p) - g(p) | \leq e^{\alpha} \varepsilon_n.
	\end{equation}
	
	\medskip
	
	\noindent \emph{Step 3: Proof of convergence.} We now prove that the ordered product of the $e^{f_j}$ over $J$ converges uniformly to~$g$ on $B_\delta$ in the sense of \cref{def:conv.lazard.prod.champs}. Let $\varepsilon > 0$. Let $n \geq 2$ large enough such that $e^{\alpha} \varepsilon_n < \varepsilon/2$. For every finite set $J_1$ containing $J_n$, condition \eqref{def:prod_CV_chpvect} holds thanks to \eqref{xn-z} and \eqref{eq:gap.j0.j1_chvect}. 
\end{proof}

\bigskip

Now, let us emphasize that, by using estimates on $\xi_b(t,u)$ and $f_b$ depending only on the length of the Lie bracket $b$, it is not possible to prove the convergence of $\sum |\xi_b(t,u)|  \|f_b\|_{\CC^1}$, where the sum ranges over $b \in \mathcal{B}$, an arbitrary \GHB\ of $\mathcal{L}(X)$.

On the one hand, one easily proves by induction on $|b|$ that, for every $b\in\mathcal{B}$ and $u \in L^\infty$ with $\|u\|_{L^\infty} \leq 1$, there holds $|\xi_b(t,u)| \leq t^{|b|}$. However, by the first statement of \cref{Prop:contre-ex-matrix_CV_Prod}, when $X$ contains at least two indeterminates, there are Hall bases (even compatible with length) for which one may not expect an upper bound, function of $|b|$ alone, that behaves better than geometrically. Hence, we should consider the $t^{|b|}$ bound to be sharp, when one restricts to bounds depending only on $|b|$.

On the other hand, if the vector fields are locally analytic, there exists $r, \delta > 0$ such that $f_i \in {\CC^{\omega,r}(B_\delta;\K^d)}$ for $i \in I$. By \eqref{eq:bracket.analytic.C1} with $r_1 \leftarrow r$ and $r_2 \leftarrow r/e$ for every $b \in \mathcal{B}$,
\begin{equation}
    \| f_b \|_{\CC^1} \leq \left(1+\frac{e}{r}\right) (|b|-1)! \left(\frac{9}{r}\right)^{|b|-1} F^{|b|},
\end{equation}
where $F := \max_{i\in I} \opnorm{f_i}_r$. However, by \cref{rk:bracket.optimal}, the dependence in $(|b|-1)!$ is optimal (again, if one restricts to bounds involving only $|b|$).

We deduce from the previous estimates that there exists $C > 0$ such that
\begin{equation}
    |\xi_b(t,u)| \|f_b\|_{\CC^1} \leq (Ct)^{|b|} |b|!.
\end{equation}
This bound does not provide the convergence of the considered series. 
Indeed, for every $t > 0$, $(Ct)^{|b|}|b|! \to +\infty$ as $|b| \to +\infty$, so an argument depending on $|b|$ alone doesn't even prove that the general term tends to zero.

To prove the convergence of Sussmann's infinite product expansion, one therefore either needs a better sufficient condition than \cref{Lem_CV_prod_chpvect/CVA} or one needs to prove estimates on $\xi_b$ and $f_b$ that take into account the structure of the bracket $b$, and not only its length.

\begin{open} \label{open:sussmann}
    Does Sussmann's infinite product converge for analytic vector fields?
\end{open}

In \cref{Subsec:error/u_Prod}, we prove the convergence (for analytic vector fields) of some infinite subproducts, by applying \cref{Lem_CV_prod_chpvect/CVA} with estimates on $\xi_b$ that depend on the structure of $b$.

\section{Error estimates for control systems}
\label{sec:estimates-control}

In this section, we consider control-affine systems with drift, i.e.\ of the form
\begin{equation} \label{eq:affine}
    \dot{x}(t) = f_0(x(t)) + \sum_{i=1}^q u_i(t) f_i(x(t))
    \quad \text{and} \quad
    x(0) = p,
\end{equation}
where $f_0, \dotsc, f_q$ are vector fields and $u=(u_1, \dotsc, u_q)\in L^1(\R;\K^q)$. When well-defined, the solution is denoted $x(t;f,u,p)$ where $f=(f_0,\dotsc,f_q)$ and $u=(u_1,\dotsc,u_q)$.

We prove error formulas at every order in $\|u\|_{L^1}$ for the Chen-Fliess expansion, the Magnus expansion in the interaction picture and for Sussmann's infinite product expansion. In each case, the error formula involves an infinite sum or an infinite product which turns out to be well-defined. We also propose a counter-example for the validity of such error estimates for the usual Magnus expansion, for which the infinite sum involved is not well-defined.

\subsection{Chen-Fliess expansion}

The convergence of the Chen-Fliess series, for control affine systems (\ref{eq:affine}) with analytic vector fields, under a smallness assumption on $t$ and a uniform bound on $u$, is a classical result, see for instance
\cite[Proposition~3.37]{zbMATH05150528} or \cite[Proposition~4.3]{MR710995}.
In this section we prove the convergence of the Chen-Fliess expansion,  (\cref{Prop:CF_CV_affine}) under a smallness assumption on $\|u\|_{L^1}$. We also generalize the Chen-Fliess expansion to nonlinear systems (not necessarily affine) with scalar input (\cref{Prop:CF_CV_NL_scalar}), because this fact will be used in \cref{subsec:CF_new}. 

\medskip

In the following statement $q\in\N^*$, $I=\intset{0,q}$. For a word $\sigma=\sigma_1  \dotsm \sigma_{\ell} \in I^*$, with $\ell \in \N^*$, $\sigma_1,\dotsc,\sigma_\ell \in I$, and vector fields $f_0,f_1,\dotsc, f_q$, we denote by $f_{\sigma}$ the differential operator {
$f_{\sigma_1} \dotsm f_{\sigma_\ell}$ (with the notations of \cref{rk:no-nabla})}.
For $t>0$ and $u=(u_1,\dotsc,u_q)\in L^1(0,t)$, the quantity
$\int_0^t u_\sigma$
is defined in~\eqref{a.sigma}, with $u_0 =1$.

\begin{proposition} \label{Prop:CF_CV_affine}
    Let $\delta, r>0$ and $f_0, f_1, \dotsc,f_q \in {\CC^{\omega,r}(B_{2\delta};\K^d)}$. There exists $\eta>0$ such that, for every 
    $\varphi \in \CC^{\omega,r}(B_{2\delta};\K)$,
    $t \in [0,\eta]$ and 
    $u \in L^1((0,t);\K^q)$ such that $\|u\|_{L^1} \leq \eta$
    and $p\in B_\delta$,
    then
    \begin{equation} \label{CF_affine_serie}
    \varphi(x(t;f,u,p))=\sum_{\sigma \in I^*} \left(\int_0^t u_\sigma\right) \left(f_{\sigma} \varphi\right)(p) 
    \end{equation}
    where the sum converges absolutely, uniformly with respect to $(t, u, p)$. Moreover, for every $\varphi \in \CC^{\omega,r}(B_{2\delta};\K)$, there exists $C>0$ such that, for every $M\in\N$,
    $p\in B_\delta$,
    $t\in [0,\eta]$ and 
    $u \in L^1((0,t);\K^q)$ such that $\|u\|_{L^1} \leq \eta$,
    then
    \begin{equation} \label{CF_affine_errorM}
    \left| \varphi(x(t;f,u,p)) - \sum_{n(\sigma)\leq M} \left(\int_0^t u_\sigma\right) \left(f_{\sigma} \varphi\right)(p) \right| \leq \left( C\|u\|_{L^1} \right)^{M+1},
    \end{equation}
    where the sum ranges over words $\sigma \in I^*$ such that the number of non-zero letters is at most $M$.
\end{proposition}

\begin{proof}
For $\sigma=\sigma_1\dotsm\sigma_\ell \in I^*$, let $n(\sigma)$ be the number of non zero letters in $\sigma$, i.e.\ $n(\sigma)=|\{i\in\intset{1,\ell} ; \sigma_i \neq 0 \}|$ and $n_0(\sigma)$  be the number of occurrences of the letter zero in $\sigma$, i.e.\ $n_0(\sigma)=|\{i\in\intset{1,\ell} ; \sigma_i=0\}|$. Then $\ell=n(\sigma)+n_0(\sigma)$. One proves by induction on the length $\ell$ of $\sigma \in I^*$ the following estimate, for every $t>0$ and $u\in L^1((0,t);\K^q)$,
\begin{equation} \label{majo_int_u_sigma}
\left| \left(\int_0^t u_\sigma\right) \right| \leq  \frac{\|u\|_{L^1(0,t)}^{n(\sigma)}}{n(\sigma)!} \frac{t^{n_0(\sigma)}}{n_0(\sigma)!}.
\end{equation}

Let $\|f\| =\sum_{i=0}^q \opnorm{f_i}_r$,
$\eta = r/(10 \|f\|)$,
$\varphi \in \CC^{\omega,r}(B_{2\delta};\K)$,
$t \in [0,\eta]$ and 
$u \in L^1((0,t);\K^q)$ such that $\|u\|_{L^1(0,t)}=\sum_{i=1}^{q}\|u_i\|_{L^1(0,t)}\leq\eta$
and $p\in B_\delta$.
Using (\ref{majo_int_u_sigma}) and (\ref{eq:prod.analytic.C0}), we get
\begin{equation}
\left| \left(\int_0^t u_\sigma\right) \left( f_{\sigma} \varphi \right)(p) \right| \leq \|u\|_{L^1(0,t)}^{n(\sigma)}  t^{n_0(\sigma)}  \left( \frac{10}{r} \|f\| \right)^\ell \opnorm{\varphi}_r
\end{equation}
which proves the absolute convergence of the sum in (\ref{CF_affine_serie}), uniformly with respect to $(t,u,p)$

The proof of the equality in (\ref{CF_affine_serie}) consists in applying (\ref{CF_macro}) to $f(t,x) = f_0(x)+\sum_{i=1}^q u_i(t) f_i(x)$.
In particular the sum involved in (\ref{CF_affine_errorM}) is the Taylor expansion of order $M$ of $u \mapsto \varphi(x(t;f,u,p))$ at $u=0$.
By adapting \cref{Prop:Sol_analytic/control} to affine systems with $L^1$ controls, we get the real-analyticity of the map
    $u  \mapsto \varphi(x(t;f,u,p))$ on $B_{L^1(0,t)}(0,\eta)$ uniformly with respect to $(t,p)\in [0,\eta] \times B_\delta$
    which ends the proof of (\ref{CF_affine_errorM}).
\end{proof}

\medskip

The last statement of this section focuses on nonlinear control systems with scalar input
\begin{equation}
    \dot{x} = f(x,u)
\end{equation}
where $f:\K^d \times \K \rightarrow \K^d$,  When well-defined, the solution of this ODE, with initial condition $x(0)=p$ is denoted $x(t;f,u,p)$. We introduce the notation
\begin{equation} \label{eq:def-intuk}
    \int_0^t u^k := \int_{{\Delta^n(t)}} 
u(\tau_1)^{k_1} \dotsm u(\tau_n)^{k_n} 
 \dd \tau
\end{equation}
for every $t>0$, $u\in L^1((0,t);\K)$, and every multi-index $k=(k_1,\dotsc,k_n)\in \N^n$ with $n\in\N^*$.

\begin{proposition} \label{Prop:CF_CV_NL_scalar}
Let $r, \delta, \delta_u>0$, $f \in \CC^{\omega,r}( B_{2\delta} \times [-\delta_u,\delta_u];\K^d)$ and
$f_k :=\frac{1}{k!} \partial_u^k f(\cdot,0)$ for every $k \in \N$.
There exists $T^*,\eta>0$ such that, for every 
$\varphi \in \CC^{\omega,r}(B_{2\delta};\K)$,
$t\in[0,T^*]$, 
$u\in L^\infty((0,t);\K)$ with $\|u\|_{L^\infty} \leq \eta$ and 
$p\in B_\delta${, with the notations of \cref{rk:no-nabla},}
\begin{equation} \label{CF_expanded}
\varphi\left(x(t;f,u,p)\right) = 
\sum_{\substack{n \in \N \\ k \in \N^n}}
\left(\int_0^t u^k\right)
    {\left( f_{k_1}  \dotsm f_{k_n} \right)}( \varphi)(p),
\end{equation}
where the sum converges absolutely, uniformly with respect to $(t, u, p)$. Moreover, for every $\varphi \in \CC^{\omega,r}(B_{2\delta};\K)$,
there exists $C>0$ such that, for every $M\in\N$, 
$t\in[0,T^*]$, $u\in L^\infty((0,t);\K)$ with $\|u\|_{L^\infty} \leq \eta$ and $p\in B_\delta$
\begin{equation} \label{Error_CF_x'=f(x,u)}
  \left| \varphi\left(x(t;f,u,p)\right) - 
\sum_{\substack{n\in\N  \\ k\in \N^n, |k| \leq M}}
\left(\int_0^t u^k\right)
 {\left( f_{k_1}  \dotsm f_{k_n} \right)}(\varphi)(p)\right|
    \leq \left( C \|u\|_{L^\infty} \right)^{M+1}
\end{equation}
where the sum is taken over $n\in\N$ and $k=(k_1,\dotsc,k_n) \in \N^n$ such that $k_1 + \dotsb + k_n \leq M$.
\end{proposition}

\begin{proof}
We define $r'=r/e$,
\begin{equation}
    T^*:=\min\left\{ \frac{r'}{10 \opnorm{f}_{r}} , \frac{\delta}{\|f\|_{\CC^0}}\right\}, \qquad 
\eta:=\min\left\{ \delta_u , \frac{r}{10} \right\}.
\end{equation}
Let $\varphi \in \CC^{\omega,r}(B_{2\delta};\K)$,
$t\in[0,T^*]$, 
$u\in L^\infty((0,t);\K^q)$ with $\|u\|_{L^\infty} \leq \eta$ and 
$p\in B_\delta$.
Then $x(t;f,u,p) \in B_{2\delta}$.

\medskip

\noindent \emph{Step 1: Uniform absolute convergence of the sum in (\ref{CF_expanded}).}
Using the iterated version of (\ref{eq:djf.r.dr}) and \eqref{eq:encadrement.n!}, we get, for every $k\in\N$,
\begin{equation} \label{opnorm_fk}
\opnorm{f_{k}}_{r'} \leq \frac{1}{k!} \left( \frac{k}{r-r'}\right)^k\opnorm{f}_{r} \leq 
\left( \frac{e}{r-r'}\right)^k\opnorm{f}_{r} \leq
\left( \frac{5}{r}\right)^k\opnorm{f}_{r}.
\end{equation}
For every $n\in\N^*$ and $k_1,\dotsc,k_n \in \N$, we have, using (\ref{eq:prod.analytic.C0}) and (\ref{opnorm_fk})
\begin{equation}
    \begin{split}
        \left| {\left( f_{k_1}  \dotsm f_{k_n} \right)}(\varphi)(p) \right|
& \leq n! \left( \frac{5}{r'} \right)^n \opnorm{f_{k_n}}_{r'} \dotsm \opnorm{f_{k_1}}_{r'}     \opnorm{\varphi}_{r'}
\\ & \leq  n! \left( \frac{5}{r'} \right)^n
\left( \frac{5}{r}\right)^{k_1+\dotsb+k_n} \opnorm{f}_{r}^n 
 \opnorm{\varphi}_{r'}
    \end{split}
\end{equation}
and
\begin{equation} \label{CH_u_majo_fonct}
\left|\int_0^t u^k\right| =
\left|\int_{{\Delta^n(t)}} 
u(\tau_1)^{k_1} \dotsm u(\tau_n)^{k_n} 
 \dd \tau\right| \leq \frac{t^n}{n!} \|u\|_{L^\infty}^{k_1+\dotsb+k_n}. 
\end{equation}
By definition of $T^*$ and $\eta$ we have
$\frac{5t}{r'} \opnorm{f}_r \leq \frac{1}{2}$ and $\frac{5}{r} \|u\|_{L^\infty} \leq \frac{1}{2}$, which gives the conclusion.

\medskip

\noindent \emph{Step 2: Equality in (\ref{CF_expanded}) and error formula (\ref{Error_CF_x'=f(x,u)}).} 
We have $f(\cdot,u)=\sum_{j=0}^{+\infty} u^j f_j$ with convergence in ${\CC^{\omega,r'}(B_{2\delta};\K^d)}$ uniformly with respect to $u \in B_{\K^q}(0,\eta)$.
 Thus, the equality (\ref{CF_expanded}) is a consequence of Fubini theorem and (\ref{CF_macro}) applied to $(t,x) \mapsto f(x,u(t))$. In particular the finite sum involved in (\ref{Error_CF_x'=f(x,u)}) is the Talyor expansion of order $M$ of $u \mapsto \varphi(x(t;f,u,p))$ at $u=0$.
By \cref{Prop:Sol_analytic/control}
    $u  \mapsto \varphi(x(t;f,u,p))$ is analytic on $B_{L^\infty(0,T^*)}(0,\eta) $ uniformly with respect to $(t,p)\in [0,T^*] \times B_\delta$,
    which ends the proof of (\ref{Error_CF_x'=f(x,u)}).
\end{proof}

\subsection{Magnus expansion in the usual setting: a counter-example}

Contrary to other expansions, the usual Magnus expansion does not yield, in general, error estimates involving the size of the control. Indeed, the infinite segments which would need to be summed do not converge, even for analytic vector fields, arbitrarily small times and even when the drift vector field vanishes at the origin. The following statement illustrates that even the series defining the terms which are linear with respect to the control does not converge.

\begin{proposition}
  Let $d := 2$. There exists $T, \delta > 0$, $f_0, f_1 \in {\CC^{\omega,\delta}(B_\delta;\K^d)}$ with $f_0(0) = 0$ and a control $u \in \CC^\infty([0,T])$, such that, if one defines, for $t \in (0,T)$, the sequence of vector fields
  \begin{equation}
    F_n(t) := \sum_{k = 0}^n \zeta_{\ad^k_{X_0}(X_1)}(t,u) \ad^k_{f_0}(f_1),
  \end{equation}
  then, for each $\delta^* \in (0,\delta)$ and $t\in(0,T)$, $F_n(t) \in {\CC^\infty(B_{\delta};\K^d)}$ does not converge in ${\CC^0(B_{\delta^*};\K^d)}$.
\end{proposition}

\begin{proof}
 We define the following vector fields for $x \in \R^2$ with $|x| < 1$,
 \begin{equation}
   f_0(x) := x_2 e_1
   \quad \textrm{and} \quad
   f_1(x) := \frac{1}{1-x_1} e_2.
 \end{equation}
 Then,
 \begin{equation}
   \ad^k_{f_0}(f_1)(x) = x_2^k \partial_1^k \left( \frac{1}{1-x_1} \right) e_2
   = \frac{k! x_2^k}{(1-x_1)^{k+1}} e_2.
 \end{equation}
 We now choose the particular control $u(t) := t$ for $t \in (0,T)$ with $T = 1$ (the simpler choice, $u(t) := 1$, would not produce a diverging counter-example).
 Using the expression \eqref{eq:zeta.adx0x1.1t} from \cref{ex:zeta.adx0x1.1t} for the coordinates of the first kind along the brackets $\ad^k_{X_0}(X_1)$ for this particular control, we obtain, for $t \in (0,T)$,
 \begin{equation}
   F_n(t)(x) = \sum_{k=0}^n (-1)^{k+1} t^{k+2} \frac{B_{k+1}}{k+1} \frac{x_2^k}{(1-x_1)^{k+1}}.
 \end{equation}
 Thus, for each $t,\delta^* > 0$, the sequence of vector fields $F_n(t)$ does not converge in ${\CC^0(B_{\delta^*};\K^d)}$, since for every $x_2 \neq 0$, the general term of the series does not tend to zero because of the asymptotic \eqref{eq:bernoulli.3} for Bernoulli numbers.
\end{proof}

\subsection{Magnus expansion in the interaction picture}

The following statement is an immediate consequence of \Cref{Prop:Magnus_2}.
{It illustrates that, contrary to the classical Magnus expansion, our ``Magnus in the interaction picture'' expansion allows to obtain error estimates involving the size of the control, at any order.}

\begin{proposition} \label{thm:Magnus_2_u}
Let $M \in \N$, $\delta>0$, $T>0$, $f_0 \in {\CC^{M^2+1}(B_{5\delta};\K^d)}$ with $T\|f_0\|_{\CC^0}\leq\delta$ and $f_1, \dotsc, f_q \in {\CC^{M^2}(B_{5\delta};\K^d)}$.
There exists $\gamma,C>0$ such that,
for every 
$u=(u_1,\dotsc, u_q) \in L^1((0,T);\K)$ with 
\begin{equation} \label{Magnus1.1:hyp_u}
\|u\|_{L^1} \leq \gamma
\end{equation}$
p\in B_\delta$ and $t\in[0,\gamma]$ then
\begin{equation} \label{Magnus_2_u}
\left| x(t;f,u,p)- e^{\mathcal{Z}_M(t,f,u)} e^{t f_0} p \right| \leq C \|u\|^{M+1}_{L^1(0,t)}.
\end{equation}
\end{proposition}

{In \eqref{Magnus_2_u}, $\mathcal{Z}_M(t,f,u)$ (where implicitly $f=(f_0,f_1,\dotsc,f_q)$ and $u=(u_1,\dotsc,u_q)$) is a notation for the vector field $\mathcal{Z}_M \left(t,f_0,\sum_{i=1}^q u_i f_i\right)$, defined in \cref{Def:Phi0_gt_ZM} for the affine perturbation $\fp(t,x)= \sum_{i=1}^q u_i(t) f_i(x)$.
This notation is chosen by analogy with  \cref{thm:Magnus1.0_formel}.}

\subsection{Sussmann's infinite product expansion}
\label{Subsec:error/u_Prod}

The goal of this section is to prove \cref{Prop:error_u_Prod} which states that, despite the difficulties mentioned in \cref{subsec:algo} concerning the full convergence of Sussmann's infinite product expansion, some (infinite) subproducts of it do converge and yield error estimates at every order in the size of the control for control-affine systems with drift of the form \eqref{eq:affine}.

We start with an elementary remark (\cref{lm:malabar}) on the structure of brackets of a {Hall set} which allows to prove nice asymmetric estimates on the associated coordinates of the second kind (see \cref{Lem:Coord2.0:borne_Sk}).
{The following result proves that, when one tries to factorize the lateral $X_0$ factors outside of a bracket of a Hall set $\bset \subset \Br(X)$ with $X_0 \in X$, these $X_0$ factors cannot alternate sides more than once.}

\begin{lemma} \label{lm:malabar}
    Let $q \in \N^*$, $X = \{ X_0, X_1, \dotsc, X_q \}$ and {$\bset \subset \Br(X)$ be a Hall set}. 
    For each $b \in \bset$, there exist $m, \overline{m} \in \N$ such that
    \begin{equation} \label{eq:b-malabar}
        b = \ad^m_{X_0} \dad_{X_0}^{\overline m} (b^*),
    \end{equation}
    where $\dad_{X_0}^{\overline m}$ denotes the iterated right bracketing ${\overline m}$ times by $X_0$ and $b^* \in \bset$ is such that either $b^* \in X$ or {$b = (b_1, b_2)$} with $b_1 \neq X_0$ and $b_2 \neq X_0$.
\end{lemma}

\begin{proof}
    The key point is that, by the third condition in \cref{Def2:Laz}, for each $b \in \bset \setminus X$, $\lambda(b) < b$. Let $b \in \bset$. 
    We disjunct cases.
    \begin{itemize}
        \item If $b \in X$ or ($\lambda(b) \neq X_0$ and $\mu(b) \neq X_0$), then \eqref{eq:b-malabar} holds with $m = \overline{m} = 0$ and $b^* = b$.
        \item If $\lambda(b) = X_0$, there exists a unique $m \in \N^*$ and $\tilde{b}\in \mathcal{B}$ such that $b = \ad^m_{X_0}(\tilde{b})$ where $\tilde{b} \in X$ or $\lambda(\tilde{b}) \neq X_0$.
        \begin{itemize}
            \item If $\tilde{b} \in X$ or $\mu(\tilde{b}) \neq X_0$, \eqref{eq:b-malabar} holds with $\overline{m} = 0$ and $b^* = \tilde{b}$.
            \item Otherwise, there exists a unique $\overline{m} \in \N^*$ and $b^* \in \bset$ such that $\tilde{b} = \dad^{\overline m}_{X_0}(b^*)$ where $b^* \in X$ or $\mu(b^*) \neq X_0$.
            \begin{itemize}
                \item If $b^* \in X$, \eqref{eq:b-malabar} holds.
                \item Else $\mu(b^*) \neq X_0$. one has $\lambda(b^*) < b^*$ as recalled. Moreover, since $\overline{m} \geq 1$, ${(b^*, X_0) \in \bset}$ so $b^* < X_0$ (by the second point of \cref{Def2:Laz}). 
                Hence $\lambda(b^*) < X_0$. 
                So we also have $\lambda(b^*) \neq X_0$ and \eqref{eq:b-malabar} holds.
            \end{itemize}
        \end{itemize}
        \item If $\mu(b) = X_0$, there exists a unique $\overline{m} \in \N^*$ and $\tilde{b} \in \bset$ such that $b = \dad^{\overline m}_{X_0}(\tilde b)$ where $\tilde b \in X$ or $\mu(\tilde b) \neq X_0$.
            \begin{itemize}
                \item If $\tilde{b} \in X$, \eqref{eq:b-malabar} holds with $m = 0$ and $b^* = \tilde{b}$.
                \item Else $\mu(\tilde b) \neq X_0$. Since $\overline{m} \geq 1$, ${(\tilde b, X_0) \in \bset}$, so $\tilde b < X_0$. 
                Since $\lambda(\tilde b) < \tilde b$, this proves $\lambda(\tilde{b}) \neq X_0$. 
                So \eqref{eq:b-malabar} holds with $m = 0$ and $b^* = \tilde b$.
            \end{itemize}
    \end{itemize}
    Hence, the decomposition \eqref{eq:b-malabar} always holds.
\end{proof}

We now turn to asymmetric estimates for the coordinates of the second kind, which, contrary to \cref{lm:xib-dotxib}, isolate the role of $X_0$ associated with the implicit control $u_0 = 1$. 

\begin{lemma} \label{Lem:Coord2.0:borne_Sk}
    Let $q\in\N^*$, $X=\{X_0,X_1,\dotsc, X_q\}$, {$\bset \subset \Br(X)$ a Hall set} and $(\xi_b)_{b\in\bset}$ the associated coordinates of the second kind. 
    For every $k \in \N^*$, there exists $c_k \geq 1$ such that, for each $b\in\bset$ with $n(b)=k$, $T > 0$, $u \in L^1((0,T);\K^q)$ and $t \in [0,T]$, 
\begin{equation} \label{Coord2.0:borne_Sk_xi}
     \left| \xi_b(t,1,u) \right| \leq  \|u\|_{{L^1(0,t)}}^{k}
     \frac{(c_k t)^{n_0(b)}}{n_0(b)!}
\end{equation}
    and
  \begin{equation} \label{Coord2.0:borne_Sk_xi'}
    | \dot{\xi}_b(t;1,u) | \leq 
    \begin{cases}
        k |u(t)| \|u\|_{{L^1(0,t)}}^{k-1} & \mathrm{when~} n_0(b) = 0, \\
        \|u\|_{{L^1(0,t)}}^{k-1} \Big(k t|u(t)| + n_0(b) \|u\|_{L^1(0,t)} \Big) \frac{c_k (c_k t)^{n_0(b)-1}}{n_0(b)!} & \mathrm{when~} n_0(b) > 0.
    \end{cases}
    \end{equation}  
\end{lemma}

\begin{proof}
    In this proof, we write $\xi_b(t)$ instead of $\xi_b(t,1,u)$ by concision for the value at time $t \in [0,T]$ of the coordinate of the second kind associated with the control $u_0 = 1$ and $u_i$ for $i \in \intset{1,q}$.
    First, when \eqref{Coord2.0:borne_Sk_xi'} holds on $[0,T]$, then so does \eqref{Coord2.0:borne_Sk_xi} by time-integration (with the same constant).
    Hence, we only need to prove the bound on the time derivative of the coordinates.
    
    \medskip

\noindent \emph{Step 1: Persistence of the estimates by right bracketing with $X_0$.}
Let $k\in\N^*$ and $b\in\bset$ such that $n(b)=k$. We assume that 
\eqref{Coord2.0:borne_Sk_xi} holds and we prove that $\tilde{b} := {(b,X_0)}$ satisfies both estimates with the same constant. Since $\dot{\xi}_{X_0}(t) = 1$, we have
\begin{equation} \label{xi'_X0}
    |\dot{\xi}_{\tilde{b}}(t)| = |\xi_b(t) \dot{\xi}_{X_0}(t) | \leq  \|u\|_{L^1(0,t)}^{k} \frac{(c_k t)^{n_0(b)}}{n_0(b)!}.
\end{equation}
Hence $\tilde{b}$ satisfies \eqref{Coord2.0:borne_Sk_xi'} (and \eqref{Coord2.0:borne_Sk_xi} by integration) because $c_k \geq 1$  and $n_0(\tilde b) > 0$.

\medskip
    
\noindent \emph{Step 2: Persistence of the estimates by arbitrary long left bracketing with $X_0$, up to $c_k \leftarrow 2 c_k$.} Let $k\in\N^*$ and $b\in\bset$ with $n(b)=k$. We assume that \eqref{Coord2.0:borne_Sk_xi'} holds and we prove that, for every $m \in \N^*$, $\tilde{b}:=\ad_{X_0}^m(b)$ satisfies both estimates with a constant $c_k \leftarrow 2 c_k$. If $n_0(b) = 0$, it is straightforward to check that $\tilde{b}$ satisfies \eqref{Coord2.0:borne_Sk_xi'} with $c_k \leftarrow 1$. If $n_0(b) = 1$, we have
\begin{equation} \label{xibtile'}
\begin{split}
    |\dot{\xi}_{\tilde{b}}(t)|
& =\frac{1}{m!} | \xi^m_{X_0}(t) \dot{\xi}_{b}(t) | 
\\ & \leq \frac{t^m}{m!}  \|u\|_{L^1}^{k-1} \Big(kt|u(t)| + n_0(b) \|u\|_{L^1} \Big) \frac{c_k (c_k t)^{n_0(b)-1}}{n_0(b)!}
\\ & \leq \|u\|_{L^1}^{k-1} \Big(kt|u(t)| + (m+n_0(b))\|u\|_{L^1} \Big)2^{m+n_0(b)} c_k^{n_0(b)}  \frac{t^{m+n_0(b)-1}}{(m+n_0(b))!}
\\ & \leq \|u\|_{L^1}^{k-1} \Big(kt|u(t)| + n_0(\tilde{b}) \|u\|_{L^1} \Big)(2c_k)^{n_0(\tilde{b})} \frac{t^{n_0(\tilde{b})-1}}{n_0(\tilde{b})!} 
\end{split}
\end{equation}
because $n_0(\tilde{b})=m+n_0(b)$ and $c_k \geq 1$. So $\tilde{b}$ satisfies \eqref{Coord2.0:borne_Sk_xi'} with a constant $c_k \leftarrow 2 c_k$.

\medskip

\noindent \emph{Step 3: Proof of the estimates by induction on $k \in \N^*$.}

\medskip

\emph{Initialization for $k=1$.} For $i \in \intset{1,q}$, $\dot{\xi}_{X_i}(t) = u_i(t)$ so both estimates are satisfied with constant $1$ when $b \in \{ X_1, \dotsc, X_q \}$. 
By \cref{lm:malabar}, Step 1 and Step 2, we deduce that  (\ref{Coord2.0:borne_Sk_xi}) and (\ref{Coord2.0:borne_Sk_xi'}) hold for $k=1$ with $c_1=2$.

\medskip

\emph{Induction $(k-1)\rightarrow k$.} Let $k \geq 2$ and let us assume that the estimates are proved for every $b \in \mathcal{B}$ with $n(b)\leq (k-1)$. Let $b \in \mathcal{B}$ with $n(b)=k$. By \cref{lm:malabar}, Step 1 and Step 2, we can assume that $b = \ad^m_{b_1}(b_2)$ with $b_1, b_2 \in \mathcal{B}$, $b_1 \neq X_0$ and ($b_2 \in X$ or $\lambda(b_2) < b_1$) and ($b_2 \neq X_0$ or $m > 1$).
Assume that $b_2 \neq X_0$.
Then the induction assumption applies to both $b_1$ and $b_2$. 
Let $k_1:=n(b_1)$ and $k_2:=n(b_2)$. Then $k=mk_1+k_2$, $n_0(b)= m n_0(b_1)+n_0(b_2) \geq n_0(b_2)$. Using the induction assumption and
(\ref{factorielle_multiple}) with $a\leftarrow (m+1)$, we obtain, when $n_0(b_2) > 0$,
\begin{equation}
    \begin{split}
    \left| \dot{\xi}_b(t) \right|
    & = \left| \frac{1}{m!} \xi_{b_1}^m(t) \dot{\xi}_{b_2}(t) \right| 
    \\ & \leq  \frac{1}{m!} \left( 
    \|u\|_{L^1}^{k_1} \frac{(c_{k_1} t)^{n_0(b_1)}}{n_0(b_1)!}
    \right)^{m} \|u\|_{L^1}^{k_2-1}
    \Big( k_2 t|u(t)| + n_0(b_2)\|u\|_{L^1} \Big)
    \frac{c_{k_2}(c_{k_2}t)^{n_0(b_2)-1}}{n_0(b_2)!} 
    \\ & \leq 
    \|u\|_{L^1}^{k-1} \Big(k t|u(t)| + n_0(b) \|u\|_{L^1} \Big)
    2^{m n_0(b)} c_{k_1}^{mn_0(b_1)} c_{k_2}^{n_0(b_2)}
\frac{t^{n_0(b)-1}}{n_0(b)!}.
    \end{split}
\end{equation}
Since $m \leq k$, we have the two desired estimates with $c_k:=2 \cdot 2^{k} \max\{ c_j ; j \in \llbracket 1,k-1 \rrbracket \}$, where the first factor $2$ comes from Step 2. When $n_0(b_2) = 0$, the proof is similar and easier.
When $b_2 = X_0$, the induction hypothesis applies because $m > 1$ so $n(b_1) < n(b)$ and the proof is straightforward.
\end{proof}

{
\begin{remark}
    The ``persistence'' of the estimates with respect to left or right bracketing by $X_0$, as mentioned and derived in Steps 1 and 2 of the proof of \cref{Lem:Coord2.0:borne_Sk} might be linked with sufficient conditions for small-time local controllability which ``ignore'' the number of leading (or trailing) $X_0$ factors (see \cite{bianchini1986sufficient},    \cite[Theorem~6]{hermes1987local} or \cite[Theorem~3.7]{kawski2017high}).
\end{remark}
}

These estimates allow to prove the main result of this section.

\begin{proposition} \label{Prop:error_u_Prod}
    Let $q\in\N^*$, $X=\{X_0,X_1,\dotsc, X_q\}$, $\mathcal{B}$ a \GHB\ of $\mathcal{L}(X)$ and $(\xi_b)_{b\in\mathcal{B}}$ the associated coordinates of the second kind. 
    Let $M \in \N$, $r,\delta>0$, $f_0,\dotsc,f_q \in {\CC^{\omega,r}(B_{4\delta};\K^d)}$.
    There exists $\eta, C_M > 0$ such that, for every $u \in L^1((0,T);\K^q)$ with $T \leq \eta$ and $\|u\|_{L^1(0,T)} \leq \eta$, the ordered product of the $e^{\xi_b(t,1,u) f_b}$ over the infinite set {$\mathcal{B} \cap S_M = \{ b \in \mathcal{B} ; n(b) \leq M \}$ (using \cref{def:SM})} converges uniformly on $B_\delta$ and, for each $t \in [0,T]$ and $p\in B_\delta$,
 	\begin{equation} \label{eq:Prop:error_u_Prod}
 		\left| x(t;f,u,p) - \underset{b \in \mathcal{B} \cap S_M}{\overset{\rightarrow}{\Pi}} e^{\xi_b(t,1,u) f_b } p \right| \leq C_M \|u\|_{L^1(0,t)}^{M+1}. 
 	\end{equation}
\end{proposition}

\begin{proof}
	In this proof, to simplify the notations, we write $x(t)$, $\xi_b(t)$ and $\|u\|$ instead of $x(t;f,u,p)$, $\xi_b(t,1,u)$ and $\|u\|_{L^1(0,t)}$.
	Let $(c_k)_{k\in\N^*}$ be the increasing sequence of constants of \cref{Lem:Coord2.0:borne_Sk}. We define
	\begin{align}
		\label{eq:def-c*}
		C_* & := \frac{18 \opnorm{f}_r}{r} \max_{k \in \intset{1,2M}} c_k, \\
		\label{eq:def-eta_ell}
		\eta & := \min \left\{ \frac{\delta}{2\|f\|_{\CC^1}},
		\frac{\min \{1, \delta\}}{2 C_* (q+1) M!(1+r)}
		 \right\} \\
		C_M & := e^{2\delta} (1+r) (2M)! (q+1)^{M+1} C_*^{M+1}.
	\end{align}
	For $t \in [0,T]$ and $u \in L^1((0,T);\K^q)$ with $T \leq \eta$ and $\|u\| \leq \eta$, using \eqref{eq:def-eta_ell},
	\begin{equation}
	   t \|f_0\|_{\CC^0} + \sum_{i=1}^q \|u_i\|_{L^1(0,t)} \|f_i\|_{\CC^0} \leq \eta \|f\|_{\CC^0} \leq \delta.
	\end{equation}
	Hence, for each $p \in B_{\delta}$, $x(t;f,u,p) \in B_{2\delta}$.

\bigskip

\noindent \emph{Strategy.} Since the product involved in \eqref{eq:Prop:error_u_Prod} is indexed by the infinite set $\mathcal{B} \cap S_M$, the proof strategy consists in considering the sequence of finite products $\mathcal{B}_{\intset{1,L}} \cap S_M$ for $L \in \N^*$ and let $L \to +\infty$. 
The error between the true solution and the finite product contains both a term scaling like $\|u\|^{M+1}$ which will persist in the limit and a transitory error term which vanishes as $L \to +\infty$.
Each bracket in $b \in \mathcal{B}$ is either, not involved at all in the process, involved in the final error, involved in the transitory error term, or involved in the finite product, depending on $L, M, n(b)$ and $n_0(b)$ as pictured in \cref{fig:prod_u_sussmann}.
In Steps 2, 3 and 4, $L \geq M+1$ is fixed. In Step 5, we take the limit $L \to +\infty$.

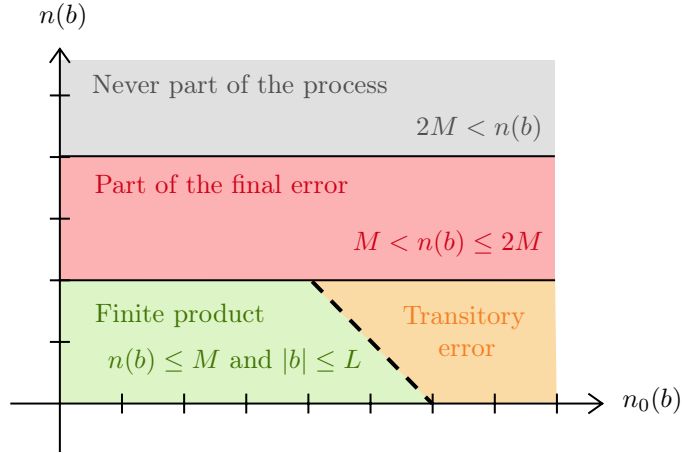
\begin{figure}[ht!]
	\centering
\tikzset{every picture/.style={line width=0.75pt}} 
\begin{tikzpicture}[x=0.75pt,y=0.75pt,yscale=-1,xscale=1]

\draw  [draw opacity=0][fill={rgb, 255:red, 245; green, 166; blue, 35 }  ,fill opacity=0.41 ] (288,146.8) -- (288.8,209) -- (226.6,209.6) -- (164.5,147.1) -- cycle ;
\draw  [draw opacity=0][fill={rgb, 255:red, 184; green, 233; blue, 134 }  ,fill opacity=0.47 ] (164.6,147.2) -- (226.6,209.6) -- (40.92,209.62) -- (41.2,147.6) -- cycle ;
\draw  [draw opacity=0][fill={rgb, 255:red, 248; green, 28; blue, 28 }  ,fill opacity=0.34 ][dash pattern={on 0.84pt off 2.51pt}] (41.2,85.6) -- (287.9,85.6) -- (287.9,147.6) -- (41.2,147.6) -- cycle ;
\draw  [draw opacity=0][fill={rgb, 255:red, 175; green, 175; blue, 175 }  ,fill opacity=0.38 ][dash pattern={on 0.84pt off 2.51pt}] (41.2,37) -- (287.9,37) -- (287.9,85.2) -- (41.2,85.2) -- cycle ;
\draw [color={rgb, 255:red, 0; green, 0; blue, 0 }  ,draw opacity=1 ][line width=0.75]  (16.2,209.62) -- (311.6,209.62)(40.92,31.2) -- (40.92,234) (304.6,204.62) -- (311.6,209.62) -- (304.6,214.62) (35.92,38.2) -- (40.92,31.2) -- (45.92,38.2) (71.92,204.62) -- (71.92,214.62)(102.92,204.62) -- (102.92,214.62)(133.92,204.62) -- (133.92,214.62)(164.92,204.62) -- (164.92,214.62)(195.92,204.62) -- (195.92,214.62)(226.92,204.62) -- (226.92,214.62)(257.92,204.62) -- (257.92,214.62)(288.92,204.62) -- (288.92,214.62)(35.92,178.62) -- (45.92,178.62)(35.92,147.62) -- (45.92,147.62)(35.92,116.62) -- (45.92,116.62)(35.92,85.62) -- (45.92,85.62)(35.92,54.62) -- (45.92,54.62) ;
\draw   ;
\draw [color={rgb, 255:red, 0; green, 0; blue, 0 }  ,draw opacity=1 ][line width=1.5]  [dash pattern={on 5.63pt off 4.5pt}]  (166.8,148.2) -- (226.6,209.6) ;
\draw [color={rgb, 255:red, 0; green, 0; blue, 0 }  ,draw opacity=1 ][line width=0.75]    (41.2,147.6) -- (288,147.6) ;
\draw [color={rgb, 255:red, 0; green, 0; blue, 0 }  ,draw opacity=1 ][line width=0.75]    (41.2,85.2) -- (287.9,85.2) ;

\draw (336.6,208.2) node  [color={rgb, 255:red, 0; green, 0; blue, 0 }  ,opacity=1 ]  {$n_{0}( b)$};
\draw (42.6,15) node  [color={rgb, 255:red, 0; green, 0; blue, 0 }  ,opacity=1 ]  {$n( b)$};
\draw (234.6,129.2) node  [color={rgb, 255:red, 208; green, 2; blue, 27 }  ,opacity=1 ]  {$M < n( b) \leq 2M $};
\draw (250.2,71.2) node  [color={rgb, 255:red, 80; green, 80; blue, 80 }  ,opacity=1 ]  {$2M < n( b)$};
\draw (128.2,188.4) node  [color={rgb, 255:red, 65; green, 117; blue, 5 }  ,opacity=1 ]  {$n(b) \leq M \textrm{ and } |b|\leq L$};
\draw (55.6,42.6) node [anchor=north west][inner sep=0.75pt]  [color={rgb, 255:red, 80; green, 80; blue, 80 }  ,opacity=1 ] [align=left] {Never part of the process};
\draw (56.8,93) node [anchor=north west][inner sep=0.75pt]  [color={rgb, 255:red, 208; green, 2; blue, 27 }  ,opacity=1 ] [align=left] {Part of the final error};
\draw (57.07,157.73) node [anchor=north west][inner sep=0.75pt]  [color={rgb, 255:red, 65; green, 117; blue, 5 }  ,opacity=1 ] [align=left] {Finite product};
\draw (210.47,158.8) node [anchor=north west][inner sep=0.75pt]  [color={rgb, 255:red, 245; green, 127; blue, 35 }  ,opacity=1 ] [align=left] {Transitory \\ \ \ \ \ error};

\end{tikzpicture}
	\caption{Decomposition of $\mathcal{B}$ along the Lazard elimination process for the product on $\mathcal{B} \cap S_M$.}
	\label{fig:prod_u_sussmann}
\end{figure}

\bigskip

\noindent \emph{Step 0: Preliminary estimates}. First, using estimate \eqref{Coord2.0:borne_Sk_xi'} from \cref{Lem:Coord2.0:borne_Sk}, for each $b \in \mathcal{B}$ with $n(b) = k$, one has in particular
\begin{equation}
	\| \dot{\xi}_b \|_{L^1} \leq \|u\|^k \frac{(c_k t)^{n_0(b)}}{n_0(b)!}.
\end{equation}
Taking into account that for every $m\in\N^*$, $|\mathcal{B}_m| \leq (q+1)^{m}$ and using the analytic estimate \eqref{eq:bracket.analytic.C1}, we obtain the following estimate for the terms which can be part of the final error
\begin{equation} \label{eq:final-error}
\begin{split}
\sum_{b \in \mathcal{B} \cap (S_{2M} \setminus S_M)} &
\|\dot{\xi}_b\|_{L^1} \|f_b\|_{\CC^1} \\
& \leq 
\sum_{k=M+1}^{2M} \sum_{n_0=0}^{+\infty} | \mathcal{B}_{k+n_0} |
\|u\|^k \frac{(c_k t)^{n_0}}{n_0!}
(1+r) \left(\frac{9\opnorm{f}_r}{r}\right)^{k+n_0}(k+n_0-1)! \\
& \leq 
(1+r) (2M-1)!
\sum_{k=M+1}^{2M} ((q+1)C_*\|u\|)^k \sum_{n_0=0}^{+\infty} ((q+1)C_* T)^{n_0} \\
& \leq (1+r) (2M)! (q+1)^{M+1} C_*^{M+1} \|u\|^{M+1},
\end{split}
\end{equation}
because $\|u\| \leq \eta$, $T\leq\eta$ and $(q+1)C_*\eta \leq \frac{1}{2}$. 
For the terms which can be part of the finite product or of the transitory error, there holds similarly
\begin{equation} \label{eq:prod-transit}
\begin{split}
\sum_{b \in \mathcal{B} \cap S_M} &
\|\dot{\xi}_b\|_{L^1} \|f_b\|_{\CC^1} \\
& \leq T \|f_0\|_{\CC^1} +
\sum_{k=1}^{M} \sum_{n_0=0}^{+\infty} | \mathcal{B}_{k+n_0} |
\|u\|^k \frac{(c_k t)^{n_0}}{n_0!}
(1+r) \left(\frac{9\opnorm{f}_r}{r}\right)^{k+n_0}(k+n_0-1)! \\
& \leq T \|f_0\|_{\CC^1} +
(1+r) (M-1)!
\sum_{k=1}^{M} ((q+1)C_*\|u\|)^k \sum_{n_0=0}^{+\infty} ((q+1)C_* T)^{n_0} \\
& \leq T \|f_0\|_{\CC^1} + (1+r) M! (q+1) C_* \|u\| 
\leq \delta.
\end{split}
\end{equation}

\bigskip

\noindent \emph{Step 1: Convergence of the ordered product of the $e^{\xi_b(t) f_b}$ over $\mathcal{B} \cap S_M$, uniformly on $B_\delta$, towards a Lipschitz map.} Thanks to \eqref{eq:prod-transit}, we have
\begin{equation} \label{rester_dans_la_boule}
    \sum_{b \in \mathcal{B} \cap S_M}
\left| \xi_b(t) \right|  \|f_b\|_{\CC^1}  
\leq 
\sum_{b \in \mathcal{B} \cap S_M}
\|\dot{\xi}_b\|_{L^1} \|f_b\|_{\CC^1}  
 \leq 
\delta
\end{equation}
and \cref{Lem_CV_prod_chpvect/CVA} gives the conclusion of Step 1.

\bigskip

\noindent \emph{Step 2: Lazard structure on $\mathcal{B}_{\intset{1,L}} \cap S_M$.} We use the notations of \cref{Def:Laz} to describe $\mathcal{B}_{\intset{1,L}}$. There exists $m\in\N$ and an extraction $\phi$ such that
 \begin{equation}\mathcal{B}_{\intset{1,L}} \cap S_M  = \{ b_{\phi(1)} < \dotsb < b_{\phi(m+1)} \}.\end{equation}
Let $i \in \llbracket 1, m+1 \rrbracket$ and $n=\phi(i)$. By \cref{Def:Laz}, there exists a unique factorization 
 \begin{equation}b_{\phi(i)} = b_{n} = \ad_{b_{n-1}}^{j_{n-1}} \dotsm \ad_{b_1}^{j_1}(b_0) \end{equation}
where $b_0 \in X$, $j_1,\dotsc,j_{n-1} \in \N$ (one just identifies left and right factors in $\Br(X)$).
For every $j \in \llbracket 1, n-1 \rrbracket \setminus \phi(\llbracket 1, i-1 \rrbracket )$, $b_j$ contains at least $(L+1)$ occurrences of the variables $X_1,\dotsc,X_q$, thus it cannot be involved in the factorization of $b_n$. This proves that
\begin{equation} \label{Laz.1_B_S} 
    \left\lbrace\begin{aligned}
& b_{\phi(1)} \in \widetilde{Y}_0:=X, \\
& b_{\phi(2)} \in \widetilde{Y}_1:=\{ \ad_{b_{\phi(1)}}^j(v) ; j \in \N, v \in \widetilde{Y}_0\setminus \{b_{\phi(1)}\} \}, \\
& \dots \\
& b_{\phi(m+1)} \in \widetilde{Y}_m:=\{ \ad_{b_{\phi(m)}}^j(v) ; j \in \N, v \in \widetilde{Y}_{m-1} \setminus\{b_{\phi(m)}\} \},        
    \end{aligned}\right.
    \end{equation}
    \begin{equation} \label{Laz.3_Sl}
    \mathcal{B}_{\intset{1,L}} \cap S_M \cap \widetilde{Y}_{m+1} =\emptyset,
    \end{equation}
    where $\widetilde{Y}_{m+1}:=\{ \ad_{b_{\phi(m+1)}}^j(v) ; j \in \N, v \in \widetilde{Y}_{m} \setminus\{b_{\phi(m+1)}\} \}$.

\bigskip

\noindent \emph{Step 3: Proof of estimates along the Lazard elimination on $\mathcal{B}_{\intset{1,L}} \cap S_M$.} To simplify the notations, from now on, we write $\mathcal{B}_{\intset{1,L}} \cap S_M = \{ b_1 < \dotsb < b_{m+1} \}$ and we use (\ref{Laz.1_B_S}) and (\ref{Laz.3_Sl}) with $\phi=\mathrm{Id}$.
Let $x_0(t) := x(t)$. By (\ref{rester_dans_la_boule}), for every $j \in \intset{1,m+1}$, 
\begin{equation}
x_j(t):=e^{-\xi_{b_j}(t)f_{b_j}} \dotsm e^{-\xi_{b_1}(t) f_{b_1}} x(t)
\end{equation}
is well-defined and belongs to $B_{3\delta}$.
The goal of Step 3 is to prove by induction on $j \in \intset{0,m+1}$ that
\begin{equation}  \label{eq:hj-sussmann_u}
(\mathcal{H}_j): \begin{cases}
\dot{x}_j(t)=\sum_{b \in \mathcal{B}_{\intset{1,L}} \cap S_M  \cap \widetilde{Y}_j} \dot{\xi}_b(t) f_b( x_j(t) ) + \varepsilon_j(t), \\
x_j(0)=p,
\end{cases}
\end{equation}
where
\begin{equation} \label{eq:epsj-xifb}
	\| \varepsilon_j \|_{L^1} \leq 
	e^{|\xi_{b_j}(t)| \|f_{b_j}\|_{\CC^1}} \|\varepsilon_{j-1}\|_{L^1}
	+ \sum_{\tilde{b} \in Z_j} \| \dot{\xi}_{\tilde{b}} \|_{L^1} \| f_{\tilde{b}} \|_{\CC^0},
\end{equation}
where $Z_j \subset (\mathcal{B} \cap S_{2M}) \setminus (\mathcal{B}_{\intset{1,L}} \cap S_M)$ is defined in \eqref{eq:def-zj-b}.

First $(\mathcal{H}_0)$ holds with $\varepsilon_0=0$ because $\dot{\xi}_{X_0}(t)=1$ and $\dot{\xi}_{X_i}(t)=u_i(t)$ for $i \in \intset{1,q}$.
Now, let $j \in \intset{1,m+1}$ and assume that $(\mathcal{H}_{j-1})$ holds.
We deduce from the definition of $x_j$ that
\begin{equation}
	x_{j}(t)= e^{-\xi_{b_j}(t)f_{b_j}} (x_{j-1}(t)) = \Phi_j \left( -\xi_{b_j}(t),x_{j-1}(t) \right)
\end{equation}
and thus that
\begin{equation}
	\dot{x}_{j}(t) =
 \sum_{b \in \mathcal{B}_{\intset{1,L}} \cap S_M \cap \widetilde{Y}_{j-1} \setminus\{b_j\}} \dot{\xi}_b(t)
 \left( \Phi_j \left( - \xi_{b_j}(t) \right)_* f_b \right)(x_{j}(t))
 + \widetilde{\varepsilon}_{j-1}(t),
\end{equation}
where
$\widetilde{\varepsilon}_{j-1}(t)= \partial_p \Phi_j\left(-\xi_{b_j}(t),x_{j-1}(t)\right) \varepsilon_{j-1}(t)$.
We get $(\mathcal{H}_{j})$ with
 \begin{equation} \label{def:epsilonj_rec}
 \varepsilon_j(t) := \sum_{b \in \mathcal{B}_{\intset{1,L}} \cap S_M \cap \widetilde{Y}_{j-1} \setminus\{b_j\}} \overline{\varepsilon}_b^j(t) + \widetilde{\varepsilon}_{j-1}(t)
 \end{equation}
 where, for every $b \in \mathcal{B}_{\intset{1,L}} \cap S_M \cap \widetilde{Y}_{j-1} \setminus\{b_j\}$,
 \begin{equation} \label{eq:sussmann-epsilonbarb-j-h}
 \overline{\varepsilon}_b^j(t):= \dot{\xi}_b(t) 
\left( \Phi_j \left(-\xi_{b_j}(t)\right)_* f_b\right)(x_j(t))
 - \sum_{k=0}^{h(b)-1} \dot{\xi}_b(t) \frac{\xi_{b_j}^k(t)}{k!} f_{\ad_{b_j}^k(b)}(x_{j}(t) ) 
 \end{equation}
where $h(b) \in \N^*$ is the maximal integer such that
\begin{equation} \label{def:m(b,bj)}
n(b)+\left(h(b)-1\right)n(b_j) \leq M \qquad \text{ and } \qquad |b|+(h(b)-1)|b_j| \leq L.
\end{equation}
By (\ref{Lem:tool_serie_ad_formule1}),
 \begin{equation} |\overline{\varepsilon}_b^j(t)|  \leq |\dot{\xi}_b(t)| \frac{|\xi_{b_j}(t)|^{ h(b) }}{h(b)!} \| f_{\ad_{b_j}^{h(b)}(b)} \|_{\CC^0} 
 = |\dot{\xi}_{\tilde{b}}(t)| \|f_{\tilde b}\|_{\CC^0},
\end{equation}
for $\tilde{b} := \ad_{b_j}^{h(b)}(b)$.
Hence, \eqref{eq:epsj-xifb} holds with
\begin{equation} \label{eq:def-zj-b}
	Z_j := \{ \ad_{b_j}^{h(b)}(b) ; \enskip b \in \mathcal{B}_{\intset{1,L} } \cap S_M \cap \widetilde{Y}_{j-1} \setminus \{b_j\} \}.
\end{equation}
This yields $Z_j \subset (\mathcal{B} \cap S_{2M}) \setminus (\mathcal{B}_{\intset{1,L}} \cap S_M)$ thanks to \eqref{def:m(b,bj)}.

 \bigskip
 
 \noindent \emph{Step 4: Proof of an estimate on the finite product over $\mathcal{B}_{\intset{1,L}} \cap S_M$.}
 By \eqref{eq:epsj-xifb}, \eqref{rester_dans_la_boule} and \eqref{eq:final-error}, we have
 \begin{equation}
 	\begin{split}
 	\| \varepsilon_{m+1} \|_{L^1}
 	& \leq e^{\delta} \sum_{b \in \mathcal{B} \cap (S_{2M} \setminus S_M)} \| \dot{\xi}_{b} \|_{L^1} \| f_{b} \|_{\CC^0}
 	+ e^{\delta} \sum_{b \in (\mathcal{B} \cap S_M) \setminus \mathcal{B}_{\intset{1,L}}} \| \dot{\xi}_{b} \|_{L^1} \| f_{b} \|_{\CC^0} \\
 	& \leq e^{-\delta} C_M \|u\|^{M+1} + o_{L \to + \infty}(1),
 \end{split}
\end{equation}
because the series in \eqref{eq:prod-transit} converges.
We deduce from (\ref{eq:hj-sussmann_u}) and (\ref{Laz.3_Sl}) that
\begin{equation} \label{Step3_ccl}
\left| \underset{b \in \mathcal{B}_{\intset{1,L}} \cap S_M}{\overset{\leftarrow}{\Pi}} e^{-\xi_b(t,1,u) f_b } x(t) -  p \right| 
= |x_{m+1}(t)-p| \leq e^{-\delta} C_M \|u\|^{M+1} + o_{L \to + \infty}(1)
\end{equation}
By (\ref{rester_dans_la_boule}), the map $\underset{b \in \mathcal{B}_{\intset{1,L}} \cap S_M}{\overset{\rightarrow}{\Pi}} e^{-\xi_b(t,u) f_b }$ is $e^{\delta}$ Lipschitz on $B_{3\delta}$. Then, by (\ref{Step3_ccl}),
\begin{equation} \label{Step3_ccl_2}
\left| x(t) -  \underset{b \in \mathcal{B}_{\intset{1,L}} \cap S_M}{\overset{\rightarrow}{\Pi}} e^{-\xi_b(t,1,u) f_b } p \right| 
 \leq 
C_M \|u\|^{M+1} + o_{L \to + \infty}(1)
\end{equation}

\bigskip
\noindent \emph{Step 5: Infinite subproduct limit}.
By Step 1, the infinite product over $\mathcal{B} \cap S_M$ is well-defined.
By letting $L \rightarrow +\infty$ in estimate \eqref{Step3_ccl_2}, we obtain the conclusion of \cref{Prop:error_u_Prod}.
\end{proof}

\section{Refined error estimates for scalar-input affine systems}
\label{sec:u1}

In this section, we consider scalar-input affine systems with drift, i.e.\ of the form
\begin{equation} \label{Scalar_affine_syst}
    \dot{x}(t) = f_0(x(t)) + u(t) f_1(x(t))
    \quad \text{and} \quad 
    x(0) = p,
\end{equation}
where $f_0, f_1$ are vector fields on $\K^d$ and $u\in L^1((0,T);\K)$.
When well-defined, its solution is denoted $x(t;f,u,p)$.
Such systems have been extensively studied in control theory, as toy models for more complex situations.

The goal of this section is to improve, in this particular framework, the error estimates of the previous section: the new bound is not expressed in terms of $\|u\|_{L^1}$ but in terms of the $L^\infty$ norm of the time-primitive of the input, which heuristically corresponds to the $W^{-1,\infty}$ norm of $u$. 

This refined estimate is somehow optimal in the scale of Sobolev spaces (as shown by the one dimensional system $\dot{x}(t)=u(t)$) and specific to the scalar-input case (see \cref{sec:u1.multi}).

Lowering the Sobolev regularity required on the input is of paramount interest for applications in control theory (see e.g.\ \cite{MR3741402}) and might also be useful for applications to stochastic differential equations where the input is a noise with low regularity (see e.g.\ \cite{MR981567}).

\begin{definition}[Integrated input]
    Let $T > 0$ and $u \in L^1((0,T);\K)$. 
    In this section, $U$ always denotes the time-primitive of $u$ vanishing at zero, i.e.\ defined by $U(t) := \int_0^t u(s) \dd s$ for $t \in [0,T]$.
\end{definition}

\subsection{Auxiliary system trick} \label{sec:trick-x1}

Enhancing the estimates relies on the following trick which factorizes the dependence of the input and introduces an auxiliary system involving the time-primitive $U$ of the input (and not $u$ itself).

\begin{proposition} \label{Prop:trick}
Let $\delta>0$, $f_0, f_1 \in {\CC^\omega(B_{3\delta};\K^d)}$ and $\eta^*>0$ small enough so that the two following maps are well defined and (globally) analytic
\begin{equation}
    \Phi_1 :
    \begin{cases}
         [-\eta^*,\eta^*] \times B_{2\delta} & \rightarrow B_{3\delta} \\
        (\tau ,  q) & \mapsto e^{\tau f_1}(q)
    \end{cases}
    \quad \textrm{and} \quad 
    F : 
    \begin{cases}
        B_{2\delta} \times [-\eta^*,\eta^*] & \rightarrow \K^d \\
        (q,\tau) & \mapsto (\Phi_1(-\tau)_* f_0) (q).
    \end{cases}
\end{equation}
Let $T>0$ be such that $T \|F\|_{\CC^0} \leq \delta$.
\begin{enumerate}
    \item For every $p\in B_{\delta}$ and $U \in \CC^0([0,T];\K)$ with $\|U\|_{L^\infty}\leq\eta^*$, there exists a unique solution
    $x_1 \in \CC^1([0,T];\K^d)$ to
    \begin{equation} \label{eq:def-x1}
        \begin{cases}
             \dot{x}_1(t) = F(x_1(t), U(t)), \\
            x_1(0) = p,
        \end{cases}
    \end{equation}
    denoted $x_1(t;F,U,p)$. It takes values in $B_{2\delta}$. 
    Moreover, the map $(p,U) \mapsto x_1(\cdot;F,U,p)$ is analytic from $B_\delta \times B_{\CC^0[0,T]}(0,\eta^*)$ to $\CC^1([0,T];\K^d)$.
    \item For every $p\in B_{\delta}$, $t\in [0,T]$ and $u \in L^1((0,T);\K)$ such that $\|U\|_{L^\infty}\leq\eta^*$,
     \begin{equation} \label{def:x1_auxiliaire}
    x(t;f,u,p)=\Phi_1\big( U(t) ; x_1(t;F,U,p) \big).
    \end{equation}
    \end{enumerate}
\end{proposition}

\begin{proof}
The existence of $\eta^*$ such that $\Phi_1$ and $F$ are well defined and globally analytic results from the third statement of \cref{Lem:tool_serie_ad}.
The analytic dependence of $x_1$ with respect to $(p,U)$ is given by \cref{Prop:Sol_analytic/control}. By definition of $x_1$, the right-hand side of (\ref{def:x1_auxiliaire}) {satisfies} the same Cauchy problem as $x$ thus the two functions are equal.
\end{proof}

\subsection{A new formulation of the Chen-Fliess expansion}\label{subsec:CF_new}

The goal of this section is to derive of a new formulation of the Chen-Fliess expansion for scalar-input affine systems (\ref{Scalar_affine_syst}).

\begin{proposition}
    Let $\delta,r>0$ and $f_0,f_1 \in {\CC^{\omega,r}(B_{3\delta};\K^d)}$.
There exists $\eta>0$ such that for every 
$\varphi \in \CC^{\omega,r}(B_{3\delta};\K)$, 
$t \in [0,\eta]$,
$u \in L^1((0,t);\K)$ such that $\|U\|_{L^\infty}\leq\eta$ and 
$p\in B_\delta$, {with the notations of \cref{rk:no-nabla},}
\begin{equation} \label{CF_scalaire_U_serie}
\varphi( x(t;f,u,p) ) =
 \sum_{\substack{\ell\in\N, n\in\N\\  k \in\N^n }}
 \frac{U(t)^\ell}{\ell! k!}  \left(\int_0^t U^k\right)
 {\Big( f_1^\ell (\ad_{f_1}^{k_1}(f_0)) \dotsm (\ad_{f_1}^{k_n}(f_0)) \Big)} (\varphi) (p)
\end{equation}
with the notation \eqref{eq:def-intuk},
where the sum converges absolutely, uniformly with respect to $(t,u,p)$. Moreover, for every $\varphi \in \CC^{\omega,r}(B_{3\delta};\K)$, there exists $C>0$ such that, for every $M\in\N^*$, $t \in [0,\eta]$,
$u \in L^1((0,t);\K)$ such that $\|U\|_{L^\infty}\leq\eta$ and $p\in B_\delta$,
\begin{equation} \label{CF_scalaire_U_errorM}
\begin{split}
  &  \left|
    \varphi( x(t;f,u,p) ) -
 \sum_{\substack{\ell\in\N, n\in\N\\  \ell+|k| \leq M}}
 \frac{U(t)^\ell}{\ell! k!}  \left(\int_0^t U^k\right)
 {\Big( f_1^\ell (\ad_{f_1}^{k_1}(f_0)) \dotsm (\ad_{f_1}^{k_n}(f_0)) \Big)} (\varphi) (p)
    \right|
\\  & \leq C^{M+1} \left( |U(t)|^{M+1} + \int_0^t |U|^{M+1}
\right) 
\end{split} 
\end{equation}
where the sum is taken over $\ell \in \N$, $n \in \N$ and $k=(k_1,\dotsc,k_n) \in \N^n$ such that $\ell + k_1+\dotsb+k_n \leq M$.
\end{proposition}

\begin{proof}
Let $\eta^*, T, x_1$ be as in \cref{Prop:trick}, $\|f\|:=\opnorm{f_0}_r+\opnorm{f_1}_r$ and
\begin{equation} \label{CF_scalaire_def:eta}
\eta := \min \left\{ T , \eta^* ,  \frac{\delta}{\|f\|} , \frac{r}{28 \|f\|}  \right\}.
\end{equation}
Let 
$\varphi \in \CC^{\omega,r}(B_{3\delta};\K)$, 
$t \in [0,\eta]$,
$u \in L^1((0,t);\K)$ such that $\|U\|_{L^\infty} \leq \eta$ and 
$p\in B_\delta$.
Then $x_1(t;F,U,p) \in B_{2\delta}$ and, by (\ref{def:x1_auxiliaire}) and \eqref{CF_scalaire_def:eta}, $x(t;f,u,p) \in B_{3 \delta}$. 

\medskip

\noindent \emph{Step 1: Proof of the absolute convergence in (\ref{CF_scalaire_U_serie}) uniformly with respect to $p\in B_\delta$.}
Let $r' := r/e$. Then, by \cref{thm:bracket.analytic}, for every $k\in \N$, $\ad_{f_1}^k(f_0) \in {\CC^{\omega,r'}(B_{3\delta};\K^d)}$ and
\begin{equation}
    \opnorm{\ad_{f_1}^k(f_0)}_{r'} \leq 
\frac{k!}{e} \left(\frac{9}{r}\right)^k \|f\|^{k+1}.
\end{equation}
Thus, by (\ref{eq:prod.analytic.C0}),
\begin{equation} \label{ing1}
    \begin{split}
        \Big|{\Big( f_1^\ell(\ad_{f_1}^{k_1}(f_0))} & {\dotsm (\ad_{f_1}^{k_n}(f_0)) \Big)} (\varphi) (p)\Big|
\\ \leq & (n+\ell)! \left(\frac{5}{r'}\right)^{n+\ell} \opnorm{f_1}_{r'}^\ell \opnorm{\ad_{f_1}^{k_1}(f_0)}_{r'} \dotsm \opnorm{\ad_{f_1}^{k_n}(f_0)}_{r'} 
\\ \leq & (n+\ell)! \left(\frac{14}{r}\right)^{n+\ell} \|f\|^\ell
\frac{k_1!}{e} \left(\frac{9}{r}\right)^{k_1} \|f\|^{k_1+1}
\dots
\frac{k_n!}{e} \left(\frac{9}{r}\right)^{k_n} \|f\|^{k_n+1}
\\ \leq & e^{-n} (n+\ell)! k_1! \dotsm k_n!  \left(\frac{14 \|f\|}{r}\right)^{n + \ell + k_1+\dotsb + k_n}.
    \end{split}
\end{equation}
Moreover, recalling notation \eqref{eq:def-intuk},
\begin{equation}
\left| \frac{U(t)^\ell}{\ell! k!}   \int_0^t U^k \right| =
\left| \frac{U(t)^\ell}{\ell!} 
\int_{{\Delta^n(t)}} 
\frac{U(\tau_1)^{k_1} \dotsm U(\tau_n)^{k_n}}{k_1! \dotsm k_n!} 
 \dd \tau  \right|
\leq  \|U\|_{L^\infty}^{\ell+k_1+\dotsb+k_n} \frac{t^n}{n!} \frac{1}{\ell! k_1! \dotsm k_n!}.    
\end{equation}
Thus it is sufficient to prove the summability over $\ell\in\N, n\in\N^*, k_1,\dotsc,k_n\in\N$ of the following quantity
\begin{equation}
    \begin{split}
       \left(\frac{t}{e}\right)^{n}  \frac{(n+\ell)!}{n! \ell !} &  \left(\frac{14 \|f\|}{r}\right)^{n + \ell + k_1+\dotsb+ k_n}
\|U\|_{L^\infty}^{\ell+k_1+\dotsb+k_n}  
\\ \leq & \left(\frac{t}{e}\right)^{n} 2^{n+\ell}   \left(\frac{14 \|f\|}{r}\right)^{n + \ell + k_1+\dotsb+k_n}
\|U\|_{L^\infty}^{\ell+k_1+\dotsb+k_n}
\\ \leq &
\left(\frac{28 t \|f\|}{e r}\right)^{n}
\left( \frac{28  \|f\|}{r} \|U\|_{L^\infty} \right)^{\ell+k_1+\dotsb+k_n} 
    \end{split}
\end{equation}
which is ensured by (\ref{CF_scalaire_def:eta}).

\medskip

\noindent \emph{Step 2: Proof of (\ref{CF_scalaire_U_serie}) and (\ref{CF_scalaire_U_errorM}).}
Applying \cref{Lem:flow} and \cref{Prop:CF_CV_NL_scalar} we get
\begin{equation}
\begin{split}
\varphi & ( x(t;f,u,p) ) 
= \varphi \left( e^{U(t) f_1} x_1(t;F,U,p) \right) \\
 & = \sum_{\ell=0}^{+\infty} \frac{U(t)^\ell}{\ell!} {f_1^\ell} \varphi ( x_1(t;F,U,p) ) 
 \\ & = \sum_{\ell=0}^{+\infty} \frac{U(t)^\ell}{\ell!} {f_1^\ell}
\sum_{\substack{n\in\N \\ k \in \N^n}}
\frac{1}{k!} \left( \int_0^t U^k\right)
 \Big({(\ad_{f_1}^{k_1}(f_0)) \dotsm (\ad_{f_1}^{k_n}(f_0))}\Big) \Big( \varphi \Big)(p)
\end{split}
 \end{equation}
The bound proved in Step 1 allows to exchange the differential operator ${f_1^\ell}$ and the second sum, which proves (\ref{CF_scalaire_U_serie}). 
To prove (\ref{CF_scalaire_U_errorM}), one bounds 
the queue of the series thanks to (\ref{ing1}) and the following consequence of H\"older's inequality, valid when $\ell+|k| \geq (M+1)$
\begin{equation}
    \left| U(t)^\ell 
\int_{{\Delta^n(t)}} 
U(\tau_1)^{k_1} \dotsm U(\tau_n)^{k_n} 
 \dd \tau  \right|
 \leq
 C(\eta) \left(|U(t)|^{M+1} + \int_0^t |U|^{M+1} \right).
\end{equation}
\end{proof}

{One of the ingredients of the above proof is the Chen-Fliess expansion of the auxiliary system $x_1(t;F,U,p)$, which appears in \cite[Section 3]{MR524203} under the denomination ``representation of the perturbation flow''.}

\begin{remark}
    The bound \eqref{CF_scalaire_U_errorM}  between the exact solution and the truncated Chen-Fliess series (in its' original formulation) is used by Stefani in \cite[Lemma~3.1 and Corollary~3.1]{MR935375}. 
    Our proof is both different and shorter. 
\end{remark}

\begin{remark}
    Equality \eqref{CF_scalaire_U_serie} where the sum converges absolutely proves that appropriate packages of the Chen-Fliess expansion are absolutely summable under a smallness assumption on $\|U\|_{L^\infty}$, which is weaker than the smallness assumption on $\|u\|_{L^1}$ which is used in \cref{Prop:CF_CV_affine} for multi-input systems.
\end{remark}

\subsection{Magnus expansion in the interaction picture}

In this section, we prove the following enhanced error estimate for the magnus expansion in the interaction picture with scalar input. 
Our proof relies on an appropriate approximation for the auxiliary system $x_1$ introduced in \cref{sec:trick-x1}.

\begin{proposition} \label{Magnus3_Cor1}
    Let $\delta > 0$ and $f_0, f_1 \in {\CC^{\omega}(B_{3\delta};\K^d)}$. 
    For every $M\in \N$, there exist $\eta_M, C_M>0$ such that, for every $T \in [0,\eta_M]$, $u\in L^1((0,T);\K)$ such that $\|U\|_{L^\infty}\leq\eta_M$,
    $t \in [0,T]$ and $p\in B_\delta$,
	\begin{equation} \label{Magnus_2_Cor}
	\left| x(t;f,u,p) -   e^{\mathcal{Z}_M(t,f,u)}  e^{t f_0} p  \right|
	\leq C_M \left( |U(t)|^{M+1} + \int_0^t |U|^{M+1} \right).
	\end{equation}
\end{proposition}

\begin{proof}
	In \cref{subsec:ym}, we introduce a vector field $\mathcal{Y}_M(t,f,U)$ such that $e^{\mathcal{Y}_M(t,f,U)} e^{tf_0}(p)$ is a good approximation of the auxiliary state $x_1$ defined in \eqref{eq:def-x1}.
	Since, by \eqref{def:x1_auxiliaire}, $x(t) = e^{U(t)f_1}(x_1(t))$, the desired estimate then relies on the following decomposition
	\begin{equation}
		\begin{split}
			x(t;f,u,p) - e^{\mathcal{Z}_M(t,f,u)}  e^{t f_0} p & =
			x(t;f,u,p) - e^{U(t) f_1} e^{\mathcal{Y}_M(t,f,U)}  e^{t f_0} p 
			\\ & \quad
			+ e^{U(t) f_1} e^{\mathcal{Y}_M(t,f,U)}  e^{t f_0} p 
			- e^{\mathcal{Z}_M(t,f,u)}  e^{t f_0} p.
		\end{split}
	\end{equation}
	Using \cref{Prop:Magnus_3} and \cref{thm:uym-zm-cbh-f} (see further) for the first and second lines, we get
	\begin{equation}
			\Big|x(t;f,u,p) -   e^{\mathcal{Z}_M(t,f,u)}  e^{t f_0} p\Big| \leq  C_M \left(
			\|U\|_{L^1}^{M+1} + |U(t)|^{M+1} +  \|U\|_{L^{M+1}}^{M+1}
			\right)
	\end{equation}
	which gives the conclusion since $\|U\|_{{L^1(0,t)}} \leq t^{\frac{M}{M+1}} \|U\|_{{L^{M+1}(0,t)}}$.
\end{proof}

In \cref{subsec:ym}, we define $\mathcal{Y}_M(t,f,U)$ and prove in \cref{Prop:Magnus_3} that it indeed provides a good approximation of the auxiliary state.
In \cref{subsec:identification}, we explain the link between $e^{U(t)X_1} e^{\mathcal{Y}_M(t,X,U)}$ and $e^{\mathcal{Z}_M(t,X,u)}$ at the formal level. 
In \cref{subsec:ym-cbh-error}, we show in \cref{thm:uym-zm-cbh-f} that this formal link entails that $e^{U(t)f_1} e^{\mathcal{Y}_M(t,f,U)}$ is close to $e^{\mathcal{Z}_M(t,f,u)}$.

\subsubsection{An approximation of the auxiliary state}
\label{subsec:ym}

We use the error formula of \cref{Prop:Magnus_2} for the Magnus expansion in the interaction picture to obtain an approximation of the auxiliary state.

\begin{proposition} \label{Prop:Magnus_3}
	Let $\delta,\rho>0$, $f_0,f_1 \in {\CC^{\omega,\rho}(B_{3\delta};\K^d)}$. 
	For every $M\in \N$, there exist $\eta_M, C_M>0$ such that, for every
	$p\in B_\delta$,
	$t\in [0,\eta_M]$,
	$u\in L^1((0,t);\K)$ such that $\|U\|_{L^\infty}\leq\eta_M$,
	\begin{equation} \label{Magnus_3_new}
		\left| x(t;f,u,p) -   e^{U(t) f_1} e^{\mathcal{Y}_M(t,f,U)}  e^{t f_0} p  \right|
		\leq C_M \|U\|_{L^{1}(0,t)}^{M+1}
	\end{equation}
	where $\mathcal{Y}_M(t,f,U):=\mathrm{Log}_M\{G_t\}(t)$,
	and $G_t:[0,t]\times B_{3\delta} \rightarrow \K^d$ is defined by
	\begin{equation} \label{G_t(s)}
		G_t(s,y):=
		\sum_{\substack{k \in\N^* \\ \ell \in \N}} \frac{(s-t)^\ell}{\ell!} \frac{U(s)^k}{k!} \ad_{f_0}^\ell \ad_{f_1}^k(f_0)(y)
	\end{equation}
	and this sum converges absolutely in ${\CC^{\omega,\rho'}(B_{3\delta};\K^d)}$ with $\rho'=\rho/e$. Moreover,
	\begin{equation} \label{def:ZMtilde}
		\begin{aligned}
			\mathcal{Y}_M(t,f,U)= \sum 
			\frac{(-1)^{m-1}}{r m}
			&
			\int_{{\Delta^{\mathbf{r}}(t)}} 
			{
			\frac{(\tau_1-t)^{\ell_1}}{\ell_1!} \frac{U(\tau_1)^{k_1}}{k_1!}\dotsm \frac{(\tau_r-t)^{\ell_r}}{\ell_r!} \frac{U(\tau_r)^{k_r}}{k_r!}} \dd\tau
			\\ & 
			{
			\left[ \dotsb  \left[ \ad_{f_0}^{\ell_1}(\ad_{f_1}^{k_1}(f_0)) ,\ad_{f_0}^{\ell_2}(\ad_{f_1}^{k_2}(f_0))\right],\dotsc,\ad_{f_0}^{\ell_r}(\ad_{f_1}^{k_r}(f_0))\right]},
		\end{aligned}
	\end{equation}
	where the sum is taken over 
	$r \in \intset{1,M}$,
	$m \in \intset{1, r}$,
	$\mathbf{r}\in\N^m_r$,
	$\ell_1,\dotsc,\ell_r \in \N$,
	$k_1,\dotsc,k_r \in \N^*$
	and the sum converges absolutely in ${\CC^{\omega,\rho'}(B_{3\delta};\K^d)}$. 
\end{proposition}

\begin{proof}
	\noindent \emph{Step 1: Convergence in (\ref{G_t(s)}) and (\ref{def:ZMtilde}).}
	By (\ref{eq:bracket.analytic}), for every $s\in[0,t]$,
	\begin{equation}
		\opnorm{\frac{(s-t)^\ell}{\ell!} \frac{U(s)^k}{k!} \ad_{f_0}^\ell \ad_{f_1}^k(f_0) }_{\rho'}  \leq t^\ell \|U\|_{L^\infty}^k \frac{(k+\ell)!}{\ell! k!} 
		\left( \frac{9}{\rho} \right)^{k+\ell} \opnorm{f}_\rho^{k+\ell+1}
	\end{equation}
	thus the sum in (\ref{G_t(s)}) converges absolutely in ${\CC^{\omega,\rho'}(B_{3\delta};\K^d)}$ when $t$ and $\|U\|_{L^\infty}$ are {smaller than} $\frac{\rho}{18 \opnorm{f}_\rho}$.

	For every $r \in \intset{1,M}$,
	$m \in \intset{1, r}$,
	$\mathbf{r}\in\N^m_r$,
	$\ell_1,\dotsc,\ell_r \in \N$,
	$k_1,\dotsc,k_r \in \N^*$,
	using (\ref{factorielle_multiple}) 
	and the non-decreasing of $q \in \intset{1, \infty} \mapsto \|\cdot\|_{{L^q(0,t)}}$ for $t \in [0,1]$, we get
	\begin{equation} \label{CVA_YM(t,f,U)}
		\begin{split}
			\bigg| \int_{{\Delta^{\mathbf{r}}(t)}} & {
			\frac{(\tau_1-t)^{\ell_1}}{\ell_1!} \frac{U(\tau_1)^{k_1}}{k_1!}\dotsm \frac{(\tau_r-t)^{\ell_r}}{\ell_r!} \frac{U(\tau_r)^{k_r}}{k_r!}} \dd\tau \bigg| (r+|\ell|+|k|-1)! \left(\frac{9\|f\|}{\rho} \right)^{r+|\ell|+|k|-1}
			\\ & \leq (2^{r-1} t)^{|\ell|}  
			(2^{r-1} \|U\|_{{L^{|k|}(0,t)}})^{|k|} 
			\left(\frac{36\|f\|}{\rho} \right)^{r+|\ell|+|k|-1}
			(r-1)!
		\end{split}
	\end{equation}
	Thus, by (\ref{eq:bracket.analytic}), the sum in (\ref{def:ZMtilde}) converges absolutely in ${\CC^{\omega,\rho'}(B_{3\delta};\K^d)}$ when $t$ and $\|U\|_{L^\infty}$ are {smaller than} $\frac{\rho 2^{-M}}{18 \opnorm{f}_\rho}$.
	
	\medskip
	
	\emph{Step 2: Proof of (\ref{Magnus_3_new}).}
	Let $T, \eta^*$ and $F$ as in \cref{Prop:trick}. We introduce the function $F_1:[0,T] \times B_{2\delta} \rightarrow \K^d$ defined by
	\begin{equation}\label{def:F1_Magnus1.2}
		F_1(t,y):=F(y,U(t))-f_0(y)=\sum_{j=1}^{+\infty} \frac{U(t)^j}{j!} \ad_{f_1}^j(f_0)(y)
	\end{equation}
	where the sum converges in ${\CC^{\omega,\rho'}(B_{2\delta};\K^d)}$ when $\|U\|_{L^\infty} < \frac{\rho}{9 \opnorm{f}_\rho}$.
	Let $M\in\N$. There exists $C>0$ such that,
	for every $t \in [0,T]$, $U \in \CC^0([0,T];\K)$ with $\|U\|_{L^\infty}\leq\eta^*$,
	the function $F_1$ defined by (\ref{def:F1_Magnus1.2}) satisfies
	\begin{equation}
		\|F_1\|_{L^1((0,t);\CC^{M^2})} \leq C \|U\|_{L^1(0,t)} \leq C T \|U\|_{L^\infty(0,t)}.
	\end{equation}
	Let $C_M>0$ and $\gamma = \gamma(M,\delta,\|f_0\|_{\CC^{M^2+1}})>0$ be as in \cref{Prop:Magnus_2} and 
	\begin{equation} \label{def:etaM_refined}
		\eta_M := 
		\min \left\{ 1 , \eta^* , \frac{\rho 2^{-M}}{36 \opnorm{f}_\rho} ,  \frac{\gamma}{C T } \right\}.\end{equation}
	Let
	$p\in B_\delta$,
	$t\in [0,\eta_M]$ and
	$u\in L^1((0,t);\K)$ such that $\|U\|_{L^\infty}\leq \eta_M$. 
	Then, the convergences of Step~1 hold and
	$\|F_1\|_{L^1((0,t);\CC^{M^2})} \leq \gamma$ thus we can apply \cref{Prop:Magnus_2} and \cref{Prop:expansion_ZM_1.1_analytic} to the equation $\dot{x}_1=f_0(x_1)+F_1(t,x_1)$
	\begin{equation} \label{Magnus1.1_sur_x1}
		\left|x_1(t;F,U,p)- e^{\mathcal{Y}_M(t,f,U)}  e^{t f_0} p  \right|
		\leq C_M \|G_t\|_{L^{1}((0,t);\CC^{M^2})}^{M+1}. 
	\end{equation}
	Moreover, there exists $C'$ (depending only on $\eta^*, f_0,f_1$) such that
	\begin{equation}\|G_t\|_{L^{1}((0,t);\CC^{M^2})} \leq C' \|U\|_{L^1(0,t)}.\end{equation}
	Thus, we get (\ref{Magnus_3_new}) by applying the $e^{\eta^* \|f_1\|_{\CC^1}}$-Lipschitz map $e^{U(t) f_1}$ to (\ref{Magnus1.1_sur_x1}).
\end{proof}

In the next paragraphs, we will use the following technical result about $\mathcal{Y}_M(t,f,U)$ and its decomposition in homogeneous components with respect to $U$.

\begin{lemma} \label{Prop:YMj}
	Let $\delta,\rho>0$, $f_0,f_1 \in {\CC^{\omega,\rho}(B_{3\delta};\K^d)}$. 
	For every $M\in\N^*$, there exists $\eta_M, C_M>0$ such that, for every $j\in\N^*$, $t \in [0,\eta_M]$, $u \in L^1((0,t),\K)$ such that $\|U\|_{L^\infty} \leq \eta_M$,
	the sum in the right-hand side of (\ref{def:ZMtilde}) taken over 
	$r \in \intset{1,M}$,
	$m \in \intset{1, r}$,
	$\mathbf{r}\in\N^m_r$,
	$\ell_1,\dotsc,\ell_r \in \N$ and
	$k_1,\dotsc,k_r \in \N^*$ such that $k_1+\dotsc+k_r=j$, converges absolutely
	in ${\CC^{\omega,\rho'}(B_{3\delta};\K^d)}$ and its sum, denoted $\mathcal{Y}_M^{j}(t,f,U)$, satisfies
	\begin{equation} \label{opnorm:YMj}
		\opnorm{ \mathcal{Y}_M^{j}(t,f,U) }_{\rho'}  \leq C_M  \left( \frac{\|U\|_{{L^j(0,t)}}}{2\eta_M} \right)^j
	\end{equation}
	where $\rho'=\rho/e$. Moreover, $\mathcal{Y}_M(t,f,U)=\sum_{j\in\N^*}\mathcal{Y}_M^{j}(t,f,U)$
	where the sum converges absolutely in ${\CC^{\omega,\rho'}(B_{3\delta};\K^d)}$. 
\end{lemma}

\begin{proof}
	Let $\eta_M>0$ be as in \cref{Prop:Magnus_3}, $t \in [0,\eta_M]$ and $u \in L^1((0,t);\K)$ such that $\|U\|_{L^\infty}<\eta_M$.
	The sum involved in $\mathcal{Y}_M^{j}(t,f,U)$ converges absolutely in ${\CC^{\omega,\rho'}(B_{3\delta};\K^d)}$ because it is a subfamily of the one considered in \cref{Prop:Magnus_3}. By (\ref{CVA_YM(t,f,U)}), there exists $C_M>0$ (independent of $t$ and $U$)  such that, for every $j\in\N^*$, (\ref{opnorm:YMj}) holds. The non-decreasing of $q \in \intset{1, \infty} \mapsto \|\cdot\|_{{L^q(0,t)}}$ (since $t \leq 1$) gives the last conclusion. 
\end{proof}

\subsubsection{Identification procedure at the formal level}
\label{subsec:identification}

In this paragraph, we highlight at the formal level the link between $e^{U(t)X_1} e^{\mathcal{Y}_M(t,X,U)}$ and $e^{\mathcal{Z}_M(t,X,u)}$ in $\widehat{\mathcal{L}}(X)$.
We start with a new formal factorization, well adapted to estimates with respect to the primitive of the scalar input.

\begin{proposition} \label{Magnus_1.2_formal}
	Let $X = \{ X_0, X_1 \}$ and $u \in L^1(\R_+;\K)$.
For every $x^\star \in \widehat{\mathcal{A}}(X)$, the solution $x$ to the formal differential equation
\begin{equation}
    \begin{cases}
         \dot{x}(t) = x(t) (X_0 + u(t) X_1), \\
         x(0) = x^\star,
    \end{cases}
\end{equation}
satisfies, for every $t\in\R_+$,
\begin{equation} \label{eq:X0-Yinf-UX1}
    x(t)=x^\star \exp \left( t X_0 \right) \exp \left( \mathcal{Y}_\infty(t,X,U) \right) \exp \left( U(t) X_1 \right)
\end{equation}
where $\mathcal{Y}_\infty(t,X,U)\in\widehat{\mathcal{L}}(X)$ is defined by
$\mathcal{Y}_\infty(t,X,U) = \mathrm{Log}_\infty \{\beta_t\} (t)$ and $\beta_t:[0,t] \rightarrow \widehat{\mathcal{L}}(X)$ is defined by
\begin{equation}
    \beta_t(s)=e^{-(t-s)X_0} \left(e^{U(s)X_1} X_0 e^{-U(s) X_1} - X_0 \right) e^{(t-s) X_0}
    = \sum_{\substack{k\in\N^*\\ \ell \in \N }}\frac{(s-t)^\ell}{\ell!} \frac{U(s)^k}{k!}  \ad_{X_0}^\ell \ad_{X_1}^k (X_0)
\end{equation}
    i.e.\
\begin{equation}
    \begin{split}
\mathcal{Y}_\infty(t,X,U)=
 \sum 
\frac{(-1)^{m-1}}{r m}
&
\int_{{\Delta^{\mathbf{r}}(t)}} 
{\frac{(\tau_1-t)^{\ell_1}}{\ell_1!} \frac{U(\tau_1)^{k_1}}{k_1!}\dotsm \frac{(\tau_r-t)^{\ell_r}}{\ell_r!} \frac{U(\tau_r)^{k_r}}{k_r!}} \dd\tau
\\ & 
{\left[ \dotsb \left[ \ad_{X_0}^{\ell_1}(\ad_{X_1}^{k_1}(X_0)) ,\ad_{X_0}^{\ell_2}(\ad_{X_1}^{k_2}(X_0))\right],\dotsc,\ad_{X_0}^{\ell_r}(\ad_{X_1}^{k_r}(X_0))\right]}
    \end{split}
\end{equation}
where the sum is taken over 
$r \in \N^*$,
$m \in \intset{1,r}$,
$\mathbf{r}\in\N^m_r$,
$\ell_1,\dotsc,\ell_r \in \N$,
$k_1,\dotsc,k_r \in \N^*$.
\end{proposition}

\begin{proof}
    First, in the same way as \cref{thm:formal.log} has been generalized to an infinite alphabet in the proof of \cref{thm:Magnus1.0_formel}, it is possible to generalize \cref{thm:Magnus1.0_formel} to an infinite alphabet.

    The function $x_1:[0,T] \rightarrow \widehat{\mathcal{A}}(X)$ defined by
$x_1(t):=x(t)e^{-U(t)X_1}$ satisfies $x_1(0)=x^\star$ and
\begin{equation}
    \dot{x}_1= x_1(t) e^{U(t) X_1} X_0 e^{-U(t) X_1} = x_1(t) \left( X_0 + \sum_{k\in\N^*} \frac{U(t)^k}{k!} \ad_{X_1}^k(X_0) \right).
\end{equation}
This equation is of the form $\dot{x}_1(t) = x_1(t) (X_0 + \sum_{k \in \N^*} a_k(t) Y_k)$ for some indeterminates $Y_k$. Thus, \cref{thm:Magnus1.0_formel} (adapted to an infinite alphabet) and the homomorphism {of algebras} sending $Y_k$ to $\ad^k_{X_1}(X_0)$ prove that 
\begin{equation}x_1(t)=x^\star\exp(t X_0) \exp(\mathcal{Y}_\infty(t,X,U)).\end{equation}
which gives the conclusion.
\end{proof}

We now use the formal expansion \eqref{eq:X0-Yinf-UX1} to obtain an alternative formula for $\mathcal{Z}_\infty(t,X,u)$ defined by \cref{thm:Magnus1.0_formel}, in terms of the primitive of the scalar input. For $r,\nu \in \N$, we introduce the finite dimensional subspace of $\mathcal{L}(X)$ 
\begin{equation}
	\mathcal{L}_{r,\nu} (X):=\vect\{\eval(b); \enskip b\in \Br(X), n_0(b)=\nu, n_1(b)=r\}
\end{equation} 
and $P_{r,\nu}: \widehat{\mathcal{L}}(X) \rightarrow \mathcal{L}_{r,\nu}(X)$ the associated canonical projection. 

\begin{proposition} \label{Prop:Magnus1.1/1.2/CBH_formel}
	Let $X = \{ X_0, X_1 \}$, $T>0$, $u \in L^1((0,T);\K)$, $t\in [0,T]$, $\mathcal{Y}_\infty(t,X,U)$ defined by \cref{Magnus_1.2_formal} and $\mathcal{Z}_\infty(t,X,u)$  defined by \cref{thm:Magnus1.0_formel}. 
	Then, in $\widehat{\mathcal{L}}(X)$,
	\begin{equation} \label{Zinfty=CBH}
		\mathcal{Z}_\infty(t,X,u)=\CBHD_\infty\left(\mathcal{Y}_\infty(t,X,U) , U(t) X_1 \right).
	\end{equation}
	In particular, for every $M \in \N^*$, $r \in \intset{1,M}$ and $\nu \in \N$, 
	\begin{equation} \label{Zinfty=CBH_rnu}
		P_{r,\nu} \mathcal{Z}_{M}(t,X,u)=P_{r,\nu} \CBHD_M\left(\mathcal{Y}_M(t,X,U) , U(t) X_1 \right).
	\end{equation}
	In this statement, $\CBHD_\infty$ is defined in \cref{Cor:CBH_formel}, $\CBHD_M$ is its truncation used in \cref{Prop:CBH} and $\mathcal{Z}_M(t,X,u)$ is defined in  \cref{thm:Magnus1.0_formel} and used in \cref{Prop:Magnus_2}.
\end{proposition}

\begin{proof}
	We deduce from \cref{Magnus_1.2_formal} and \cref{thm:Magnus1.0_formel} that 
	\begin{equation}\exp(\mathcal{Z}_\infty(t,X,u))=\exp(\mathcal{Y}_\infty(t,X,U)) \exp(U(t) X_1).\end{equation}
	Thus \cref{Cor:CBH_formel} proves (\ref{Zinfty=CBH}). 
	Let  $M \in \N^*$, $r \in \intset{1,M}$, $\nu \in \N$. We deduce from (\ref{Zinfty=CBH}) that
	\begin{equation}P_{r,\nu} \mathcal{Z}_\infty(t,X,u)= P_{r,\nu} \CBHD_\infty\left(\mathcal{Y}_\infty(t,X,U) , U(t) X_1 \right).\end{equation}
	By definition, $\mathcal{Z}_\infty(t,X,u)-\mathcal{Z}_{M}(t,X,u)$ is a linear combination of brackets all involving at least $(M+1)$ occurrences of $X_1$, thus $P_{r,\nu} \mathcal{Z}_\infty(t,X,u)=P_{r,\nu} \mathcal{Z}_{M}(t,X,u)$. By definition, $\mathcal{Y}_\infty(t,X,U)$ is a sum of brackets involving all at least one occurrence of $X_1$, thus
	\begin{equation}
		P_{r,\nu} \CBHD_\infty\left(\mathcal{Y}_\infty(t,X,U) , U(t) X_1 \right)
		=P_{r,\nu} \CBHD_M\left(\mathcal{Y}_\infty(t,X,U) , U(t) X_1 \right).
	\end{equation}
	Moreover $\mathcal{Y}_\infty(t,X,U)-\mathcal{Y}_M(t,X,U)$ is a linear combination of brackets involving all at least $(M+1)$ occurrences of $X_1$ thus
	\begin{equation}P_{r,\nu} \CBHD_M\left(\mathcal{Y}_\infty(t,X,U) , U(t) X_1 \right)=P_{r,\nu} \CBHD_M\left(\mathcal{Y}_M(t,X,U) , U(t) X_1 \right),\end{equation}
	which ends the proof of (\ref{Zinfty=CBH_rnu}).
\end{proof}

 \subsubsection{Error formula for analytic vector fields}
 \label{subsec:ym-cbh-error}
 
We prove in \cref{thm:uym-zm-cbh-f} an error bound between $e^{U(t)f_1} e^{\mathcal{Y}_M(t,f,U)}$ and $e^{\mathcal{Z}_M(t,f,u)}$.

\begin{proposition} \label{thm:uym-zm-cbh-f}
	Let $\delta, \rho > 0$, $f_0, f_1 \in {\CC^{\omega,\rho}(B_{3\delta};\K^d)}$.
	For every $M\in \N$, there exist $\eta_M, C_M>0$ such that, for every
	$t\in [0,\eta_M]$, $p \in B_\delta$ and
	$u\in L^1((0,t);\K)$ such that $\|U\|_{L^\infty}\leq\eta_M$,
	\begin{equation}
		\left| e^{U(t)f_1}e^{\mathcal{Y}_M(t,f,U)}e^{tf_0} p
		- e^{\mathcal{Z}_M(t,f,u)} e^{tf_0} p 
			\right|
			\leq C_M \left(|U(t)|^{M+1}+ \int_0^t |U|^{M+1} \right).
	\end{equation} 
\end{proposition}

\begin{proof}
	We split the difference as 
	\begin{equation}
		\begin{split}
			e^{U(t) f_1} & e^{\mathcal{Y}_M(t,f,U)}  e^{t f_0} p -
			e^{\CBHD_M(\mathcal{Y}_M(t,f,U) ,U(t) f_1 )} e^{t f_0} p
			\\ &  
			+e^{\CBHD_M(\mathcal{Y}_M(t,f,U) ,U(t) f_1 )} e^{t f_0} p
			- e^{\mathcal{Z}_M(t,f,u)}  e^{t f_0} p.
		\end{split}
	\end{equation}
	Taking into account that $\|\mathcal{Y}_M(t,f,U)\|_{\CC^{M^2}} \leq C \|U\|_{L^1(0,t)}$, the first line is bounded by \cref{Prop:CBH}.
	Using Gr\"onwall's lemma and \cref{Prop:Magnus1.1/1.2/chpvect} bounds the second line.
\end{proof}
 
\begin{proposition} \label{Prop:Magnus1.1/1.2/chpvect}
    Let $\delta, \rho > 0$, $f_0, f_1 \in {\CC^{\omega,\rho}(B_{3\delta};\K^d)}$ and $\rho':=\rho/e$.
    For every $\rho'' \in (0,\rho')$, $M\in \N$, there exist $\eta_M, C_M>0$ such that, for every
$t\in [0,\eta_M]$,
$u\in L^1((0,t);\K)$ such that $\|U\|_{L^\infty}\leq\eta_M$,
\begin{equation} \label{Majo:ZM-CBHDM}
\opnorm{ \mathcal{Z}_{M}(t,f,u) - \CBHD_M\left(\mathcal{Y}_M(t,f,U) , U(t) f_1 \right) }_{\rho''} \leq C_M \left(|U(t)|^{M+1}+ \int_0^t |U|^{M+1} \right).
\end{equation} 
In particular, $\mathcal{Z}_{M}(t,f,u)$ is the sum of the terms homogeneous with degree at most $M$ with respect to $U$ in $\CBHD_M\left(\mathcal{Y}_M(t,f,U) , U(t) f_1 \right)$.
\end{proposition}

\begin{proof}
	\emph{Step 1: Finite approximation of $\mathcal{Y}_M(t,f,U)$.}
	First, by \cref{Prop:YMj}, one can write
	\begin{equation}
		\mathcal{Y}_M(t,f,U) = \sum_{j=1}^{M} \mathcal{Y}^j_M(t,f,U)  + \sum_{j>M} \mathcal{Y}^j_M(t,f,U)
		=: \overline{\mathcal{Y}}_M(t,f,U)  + \mathcal{R}_M(t,f,U),
	\end{equation}
	where the remainder satisfies $\opnorm{\mathcal{R}_M(t,f,U)}_{\rho'} \leq C \| U \|_{{L^{M+1}(0,t)}}^{M+1}$.
	By the triangular and Young inequalities, it is therefore sufficient to prove \eqref{Majo:ZM-CBHDM} with $\mathcal{Y}_M$ replaced by the finite truncation $\overline{\mathcal{Y}}_M(t,f,U)$.
	
	\medskip
		
\noindent \emph{Step 2: Identification at the free level.}
Let $\Lambda:\mathcal{L}(X) \rightarrow {\CC^{\omega}(B_{3\delta};\K^d)}$ be the homomorphism of Lie {algebras} such that $\Lambda(X_i)=f_i$. The relation (\ref{Zinfty=CBH_rnu}) is made of finite linear combinations of brackets of $X_0$ and $X_1$. Let $M\in\N$. By applying $\Lambda$ to this equality, we get, for every $r \in \intset{1,M}$, $\nu \in \N$
\begin{equation}P_{r,\nu} \mathcal{Z}_{M}(t,f,u)=P_{r,\nu} \CBHD_M\left(\overline{\mathcal{Y}}_M(t,f,U) , U(t) f_1 \right).
\end{equation}
By definition
\begin{equation}\mathcal{Z}_{M}(t,f,u) = \sum_{\nu\in\N} \sum_{r=1}^M P_{r,\nu} \mathcal{Z}_{M}(t,f,u) 
\end{equation}
where the sum converges in ${\CC^{\omega,\rho'}(B_{3\delta};\K^d)}$ for appropriate $\rho'\in(0,\rho)$, by \cref{Prop:expansion_ZM_1.1_analytic}.
Thus, with the notations of \eqref{eq:CBH-Fqh},
\begin{equation} \label{eq:zmcbh-jh1h2}
	\mathcal{Z}_{M}(t,f,u) - \CBHD_M\left(\overline{\mathcal{Y}}_M(t,f,U) , U(t) f_1 \right)
	= \sum_{j h_1 + h_2 > M} F_{2,h}(\mathcal{Y}^j_M(t,f,U), U(t) f_1),
\end{equation}
where the sum is taken over $j, h_1, h_2 \in \intset{1,M}$.

\medskip

\noindent \emph{Step 3: Proof of (\ref{Majo:ZM-CBHDM}).}
From now on, $\eta_M>0$ is given by \cref{Prop:Magnus_3} and \cref{Prop:YMj}, $t \in [0,\eta_M]$, $u \in L^1((0,t);\K)$ is such that $\|U\|_{L^\infty}<\eta_M$ and $\rho'' \in (0,\rho')$.
For each term in the finite sum \eqref{eq:zmcbh-jh1h2}, one has, thanks to \cref{Prop:YMj},
\begin{equation} 
	\begin{split}
		\opnorm{F_{2,h}(\mathcal{Y}^j_M(t,f,U), U(t) f_1)}_{\rho''}
		& \leq 
		C \opnorm{\mathcal{Y}^j_M(t,f,U)}_{\rho'}^{h_1} 
		\opnorm{U(t)f_1}_{\rho'}^{h_2} \\
		& \leq C' \| U \|_{{L^j(0,t)}}^{jh_1} |U(t)|^{h_2}
		\leq C' \| U \|_{{L^{M+1}(0,t)}}^{jh_1} t^{h_1 - \frac{jh_1}{M+1}} |U(t)|^{h_2}
	\end{split}
\end{equation}
which concludes the proof thanks to Young's inequality since $jh_1+h_2 \geq M+1$.
\end{proof}

\subsection{Sussmann's infinite product expansion}

When the input is scalar, the estimates of the coordinates obtained in \cref{Lem:Coord2.0:borne_Sk} can be enhanced to involve only the primitive of the input, at least for \GHBS\ where $X_1$ is minimal, which in turn improves the estimate of \cref{Prop:error_u_Prod} (see \cref{Prop:error_u1_Prod} below).
The hypothesis that $X_1$ is the minimal element can be seen as the formal counterpart of the auxiliary system trick of \cref{sec:trick-x1}.

\begin{lemma} \label{Lem:Coord2.0:borne_Sk_u1}
    Let $X=\{X_0,X_1\}$, $\mathcal{B}$ be a \GHB\ of $\mathcal{L}(X)$ for which $X_1$ is the minimal element and $(\xi_b)_{b\in\mathcal{B}}$ the associated coordinates of the second kind.
    For every $k \geq 1$, there exists $c_k \geq 1$ such that, for each $b\in\mathcal{B}\setminus X$ with $n(b)=k$, $T>0$, $u \in L^1((0,T);\K)$ and $t \in [0,T]$,
\begin{equation} \label{Coord2.0:borne_Sk_xi_u1}
     | \xi_b(t,1,u) | \leq c_k \|U\|_{{L^k(0,t)}}^{k}
     \frac{(c_k t)^{n_0(b)-1}}{(n_0(b)-1)!}
\end{equation}
and
\begin{equation}
    \label{Coord2.0:borne_Sk_xi'_u1}
    | \dot{\xi}_b(t;1,u) | \leq
    \begin{cases}
        c_k |U(t)|^k & \textrm{when } n_0(b) = 1, \\
        c_k |U(t)|^k \frac{(c_k t)^{n_0(b)-1}}{(n_0(b)-1)!}
        + c_k^2 \|U\|_{{L^k(0,t)}}^{k} \frac{(c_k t)^{n_0(b)-2}}{(n_0(b)-2)!} & \textrm{when } n_0(b) \geq 2.
    \end{cases} 
\end{equation}
\end{lemma}

\begin{proof}
    As for \cref{Lem:Coord2.0:borne_Sk}, estimate \eqref{Coord2.0:borne_Sk_xi_u1} is obtained, for each $b$, by time integration of \eqref{Coord2.0:borne_Sk_xi'_u1}.
    Moreover, still as in \cref{Lem:Coord2.0:borne_Sk}, both estimates are invariant by right-bracketing with $X_0$, and also by arbitrary long left-bracketing with $X_0$, up to $c_k \leftarrow 2 c_k$. Let us prove (\ref{Coord2.0:borne_Sk_xi_u1}) and (\ref{Coord2.0:borne_Sk_xi'_u1}) by induction on $k$.

\medskip

\noindent \emph{Initialization for $k=1$.} We have $\xi_{X_1}(t)=U(t)$ and $\dot{\xi}_{[X_1,X_0]}(t)=U(t)$.
Hence $[X_1,X_0] \in \mathcal{B}$ (because $X_1 < X_0$) satisfies both estimates. By \cref{lm:malabar}, when $n(b) = 1$, there exist $m, \overline{m} \in \N$ such that $b = \ad^m_{X_0} \dad^{\overline m}_{X_0} (X_1)$. 
Since $X_1$ is minimal, if $b \neq X_1$, $\overline m > 0$.
Thus, by the previous invariant properties, we get the conclusion with $c_1:=2$.

\medskip

\noindent \emph{Induction $(k-1)\rightarrow k$.} Let $k \geq 2$ and let us assume that the two estimates are proved for every $b \in \mathcal{B} \setminus X$ with $n(b) \leq (k-1)$. 
Let $b \in \mathcal{B}$ with $n(b)=k$.
By \cref{lm:malabar} and the previous invariant properties, we may assume that $b = \ad^m_{b_1}(b_2)$ with $m \in \N^*$, $b_1 < b_2 \in \mathcal{B}$, $b_1 \neq X_0$, ($b_2 \in X$ or $\lambda(b_2) < b_1$) and $(b_2 \neq X_0$ or $m > 1$).

\begin{itemize}
    \item If $b_1 = X_1$, then $b_2 = X_0$ (otherwise, if $b_2 \notin X$, $\lambda(b_2) < X_1$, which is impossible since $X_1$ is minimal). Thus
    \begin{equation}
        |\dot \xi_b(t)| = \frac{|U(t)|^m}{m!}
    \end{equation}
    so \eqref{Coord2.0:borne_Sk_xi'_u1} with $c_k = 1$ holds since $n_0(b) = 1$ and $k = m$.
    
    \item If $b_1 \neq X_1$, then $b_1$ satisfies \eqref{Coord2.0:borne_Sk_xi'_u1} for some $k_1 \in \intset{1, k-1}$. Moreover, either $b_2 = X_0$ or $b_2 \notin X$ (because it cannot be $X_1$). The case ($b_2 = X_0$ and $m > 1$) is easier and left to the reader. Thus we are left with the case where $b_2$ satisfies \eqref{Coord2.0:borne_Sk_xi'_u1} for some $k_2 \in \intset{1, k-1}$. One has $k = m k_1 + k_2$ and $\nu := n_0(b) = m n_0(b_1) + n_0(b_2) =: m \nu_1 + \nu_2$.
    Thus,
    \begin{equation}
        |\dot \xi_b(t)|
        \leq \frac{c_{k_1}^m \| U \|_{L^{k_1}}^{m k_1}}{m!} \frac{(c_{k_1} t)^{m \nu_1-m}}{(\nu_1-1)!^m}
        \left(c_{k_2} |U(t)|^{k_2} \frac{(c_{k_2} t)^{\nu_2-1}}{(\nu_2-1)!}
        + c_{k_2}^2 \|U\|_{L^{k_2}}^{k_2} \frac{(c_{k_2} t)^{\nu_2-2}}{(\nu_2-2)!}\mathbf{1}_{\nu_2 \geq 2}\right).
    \end{equation}
    Thanks to Hölders' inequality,
    \begin{equation}
        \| U \|_{L^{k_1}}^{m k_1} \| U \|_{L^{k_2}}^{k_2} 
        \leq \| U \|_{L^k}^{k} t^m.
    \end{equation}
    Thanks to Hölder's inequality and Young's inequality,
    \begin{equation}
        \| U \|_{L^{k_1}}^{m k_1} |U(t)|^{k_2}
        \leq t^m \left( \| U \|_{L^k}^k t^{-1} + |U(t)|^k \right).
    \end{equation}
    Moreover, thanks to \eqref{factorielle_multiple}, for $i \in \intset{1,2}$,
    \begin{equation}
        \frac{1}{m!}\frac{1}{(\nu_1-1)!^m}\frac{1}{(\nu_2-i)!} 
        \leq 2^{(m+1)(\nu-i)} \frac{1}{(\nu-i)!}.
    \end{equation}
    Combining these inequalities proves \eqref{Coord2.0:borne_Sk_xi'_u1} with $c_k := 2^{k+2} \max \{ c_j; j \in \intset{1,k-1} \}$. 
\end{itemize}
\end{proof}

These enhanced estimates yield the following result.

\begin{proposition} \label{Prop:error_u1_Prod}
    Let $X=\{X_0,X_1\}$, $\mathcal{B}$ a \GHB\ of $\mathcal{L}(X)$ for which $X_1$ is the minimal element and $(\xi_b)_{b\in\mathcal{B}}$ the associated coordinates of the second kind. 
    Let $r,\delta>0$, $f_0,f_1 \in {\CC^{\omega,r}(B_{4\delta};\K^d)}$.
    For each $M \in \N^*$, there exist $\eta_M, C_M > 0$ such that, for every $u \in L^1((0,T);\K)$ with $T \leq \eta_M$ and $\|U\|_{L^{M+1}(0,T)} \leq \eta_M$, the ordered product of the $e^{\xi_b(t,1,u) f_b}$ over the infinite set {$\mathcal{B} \cap S_M = \{ b \in \mathcal{B} ; n(b) \leq M \}$ (using \cref{def:SM})} converges uniformly on $B_\delta$ and, for each $t \in [0,T]$ and $p\in B_\delta$,
 	\begin{equation} \label{eq:Prop:error_u_Prod_U1}
 		\left| x(t;f,u,p) - \underset{b \in \mathcal{B} \cap S_M}{\overset{\rightarrow}{\Pi}} e^{\xi_b(t,1,u) f_b } p \right| \leq C_M \|U\|_{L^{M+1}(0,t)}^{M+1}. 
 	\end{equation}
\end{proposition}

\begin{proof}
    The proof is the same as the proof of \cref{Prop:error_u_Prod}. 
    The only difference is that we use estimates of \cref{Lem:Coord2.0:borne_Sk_u1} instead of those of \cref{Lem:Coord2.0:borne_Sk}.
    The fact that these enhanced estimates are not valid for $b \in X$ doesn't come into play. 
    Indeed, neither $X_0$ nor $X_1$ are involved in the final error term \eqref{eq:final-error}.
\end{proof}

\subsection{Failure of the primitive estimate for multiple inputs}
\label{sec:u1.multi}

\cref{Magnus3_Cor1} relying only on the primitive of the input is specific to the scalar-input case and fails for multiple inputs. As an illustration, for $\delta > 0$ and $f_0, f_1 \in {\CC^\infty(B_\delta;\K^d)}$, in the degenerate case $M = 0$ and the particular case $f_0(0) = 0$, $p = 0$, estimate \eqref{Magnus_2_Cor} implies that, for every $T > 0$, there exists $C_T > 0$ such that, for $t \in [0,T]$ and $u \in L^1(0,T)$ with $\| U \|_{L^\infty} \leq 1$,
\begin{equation}
 | x(t; u, 0) | \leq C_T \| U \|_{L^\infty}.
\end{equation}
As illustrated by the following example, even this very crude estimate fails for multiple inputs, because the $W^{-1,\infty}$ norms are not sufficient to bound the nonlinear terms arising in the dynamic.

\begin{example}
 Let $T > 0$ and consider the following system on $\R^2$:
 \begin{equation}
  \left\{ \begin{aligned}
   \dot x_1 & = u, \\
   \dot x_2 & = v x_1,
  \end{aligned} \right.
 \end{equation}
 where $u$ and $v$ are two scalar inputs. There exists $u_n, v_n \in L^1(0,T)$ such that
 \begin{equation} \label{eq:failure_unvn}
  \| U_n \|_{L^\infty} + \| V_n \|_{L^\infty} \to 0
  \quad \text{and} \quad 
  | x(t;(u_n,v_n),0) | \not\to 0,
 \end{equation}
 where $U_n$ is the primitive of $u_n$ and $V_n$ the primitive of $V_n$.
 Indeed, let $n \in \N^*$ and define $u_n(t) := n \cos n^2 t$ and $v_n(t) := n \sin n^2 t$. Then one has
 \begin{equation}
  \|U_n\|_{L^\infty} + \|V_n\|_{L^\infty} \leq \frac{2}{n}.
 \end{equation}
 Moreover, $x_1(t) = U_n(t) = (\sin n^2 t) / n$ and
 \begin{equation}
  x_2(T) = \int_0^T v_n(t) U_n(t) \dd t
  = \int_0^T \sin^2 (n^2 t) \dd t
  \to \frac{T}{2},
 \end{equation}
 as $n \to +\infty$. This proves \eqref{eq:failure_unvn}.
\end{example}

\begin{remark}
 Although \cref{Magnus3_Cor1} does not hold for multiple inputs, we expect that the proof method can be adapted to obtain asymmetric estimates, involving for example $\|U\|_{L^\infty} + \|v\|_{L^\infty}$ in the two-inputs case (or the converse). Such asymmetric estimates have been used successfully to obtain sharp results for particular control systems in \cite{giraldi2019necessary}.
\end{remark}

\section{On direct intrinsic representations of the state}
\label{sec:diffeo}

The expansions studied above in this article unfortunately don't provide a direct intrinsic representation of the state. 
The Magnus and Sussmann expansions are given with intrinsic quantities (Lie brackets of the vector fields) but they require to compute one or multiple flows in order to recover the state. 
The Chen-Fliess expansion gives directly a formula for the state, but it depends on non-intrinsic quantities (see \cref{rk:cf-drawbacks} and \cref{rk:chen.gb.fb}). 
In this section, we investigate the possibility of finding a direct intrinsic formula for the state. 
We discuss this possibility in the context of affine systems.

\subsection{Approximate direct intrinsic representations}
\label{subsec:X/Z}
 
We prove in this section approximate  direct intrinsic representations which achieve the desired goal up to a small error.
We believe that the formulas we derive can be of interest for applications to control theory as they give approximate expressions for the state in terms of the inputs and Lie brackets of the vector fields evaluated at the origin.

We start with an elementary result, which bounds the error when replacing a flow by the value of the vector field. 

\begin{lemma} \label{Lem_oZ}
 Let $\delta>0$ and $z \in {\CC^1(B_{\delta};\K^d)}$  such that $\|z\|_{\CC^0} \leq \delta$. Then
 \begin{equation}
     |e^{z}(0)-z(0)| \leq |z(0)| \|Dz\|_{\CC^0} e^{\|Dz\|_{\CC^0}}.
 \end{equation}
 \end{lemma}
 
 \begin{proof}
    Let $x(t):=e^{tz}(0)$ for $t \in [0,1]$. 
    Then, for every $t\in[0,1]$, 
 \begin{equation}
     |x(t)-tz(0)| \leq 
 \int_0^t |z(x(\tau)) - z(0)| \dd\tau \leq
 \frac{t^2}{2} \|Dz\|_{\CC^0} |z(0)| +
 \int_0^t \|Dz\|_{\CC^0} |x(\tau)-\tau z(0)| \dd\tau
 \end{equation}
 and by Gr\"onwall's lemma,
 $|x(t)-tz(0)|\leq \frac{t^2}{2} \|Dz\|_{\CC^0} |z(0)| e^{t \|Dz\|_{\CC^0}}$.
 \end{proof}
 
This elementary estimate allows to obtain approximate direct intrinsic representations from the various Magnus expansions described above.

\begin{proposition} \label{Prop:Approx_intrinsic_repr}
Let $M\in\N^*$, $\delta>0$ and $q\in\N^*$.
 \begin{enumerate}
     \item Let $I = \intset{0,q}$ or $I = \intset{1,q}$. 
     Let $f_i \in {\CC^{M^2}(B_{\delta};\K^d)}$ for $i \in I$. 
     For $T>0$ and $u \in L^\infty((0,T);\K^q)$, if $x(t;f,u,0)$ denotes the solution to \eqref{affine_syst_q} with $p=0$ and $Z_M(t,f,u)$ denotes the vector field defined in \cref{Prop:Magnus_1} (called $Z_M(t,\sum_{i\in I} u_i f_i)$ in this statement), then, as $T \to 0$,
     \begin{equation}
        x(t;f,u,0)=Z_M(t,f,u)(0)+O\left(t^{M+1} + t |x(t;f,u,0)| \right). 
    \end{equation}
     in the following sense: there exist $C, \eta>0$ such that, for every $T \in (0,\eta]$ and $u \in L^\infty((0,T);\K^q)$ with $\|u\|_{L^\infty}\leq 1$, for each $t \in [0,T]$,
     \begin{equation} \label{approx_repr_1.0}
     \left|x(t;f,u,0)-Z_M(t,f,u)(0)\right|\leq C \left( t^{M+1} + t |x(t;f,u,0)| \right).\end{equation}
     \item Let $T>0$, $f_0,\dotsc,f_q \in {\CC^{M^2+1}(B_{2\delta};\K^d)}$ with $f_0(0)=0$ and $T\|f_0\|_{\CC^0} \leq \delta$. 
     For $u \in L^1((0,T);\K^q)$, if $x(t;f,u,0)$ denotes the solution to \eqref{eq:affine} with $p=0$ and $\mathcal{Z}(t;f,u)$ denotes the vector field defined in \cref{thm:Magnus_2_u}, then, as $\|u\|_{L^1} \to 0$,
     \begin{equation}x(t;f,u,0)=\mathcal{Z}_M(t,f,u)(0)+O\left(\|u\|_{{L^1(0,t)}}^{M+1} + |x(t;f,u,0)|^{{1+\frac{1}{M}}} \right).
     \end{equation}
     \item Let $f_0, f_1 \in {\CC^{\omega}(B_{3\delta};\K^d)}$ with $f_0(0)=0$.
     Let $T > 0$ as in \cref{Prop:trick}.
     For $u \in L^1((0,T);\K)$, if $x(t;f,u,0)$ denotes the solution to \eqref{Scalar_affine_syst} with $p=0$ and $\mathcal{Z}(t,f,u)$ denotes the vector field defined in \cref{thm:Magnus_2_u} (with $q=1$), then, as $(T,\|U\|_{L^\infty}) \to 0$,
     \begin{equation} \label{eq:xzmUMespilon}
     x(t;f,u,0)=\mathcal{Z}_M(t,f,u)(0)+O\left( \|U\|_{{L^{M+1}(0,t)}}^{M+1} + |x(t;f,u,0)|^{{1+\frac{1}{M}}} \right),
     \end{equation}
 \end{enumerate}
\end{proposition}
 
 \begin{proof}
 \noindent \emph{Proof of the first statement.}
 By \cref{Prop:Magnus_1}, there exists $C_1>0$ and $T^* >0$ 
such that for every $u\in L^\infty((0,T^*);\K^q)$ with $\|u\|_{L^\infty} \leq 1$ and $t \in [0,T^*]$,
 \begin{equation}
 \left|x(t;f,u,0)-e^{Z_M(t,f,u)}(0) \right| \leq C_1 t^{M+1}. 
 \end{equation}
By the explicit expression of $Z_M(t,f,u)$, there exists $C_2>0$ such that for every $u\in L^\infty((0,T^*);\K^q)$ with $\|u\|_{L^\infty} \leq 1$ and $t \in [0,T^*]$,
\begin{equation}\|Z_M(t,f,u)\|_{\CC^1} \leq C_2 t.\end{equation}
Thus, by \cref{Lem_oZ}, there exists $C_3>0$ such that, for every for every $u\in L^\infty((0,T^*);\K^q)$ with $\|u\|_{L^\infty} \leq 1$ and $t \in [0,T^*]$,
\begin{equation}\left| e^{Z_M(t,f,u)}(0)- Z_M(t,f,u)(0)\right| \leq C_3 t \left|Z_M(t,f,u)(0) \right|. \end{equation}
Then, by triangular inequality, for every $u\in L^\infty((0,T^*);\K^q)$ with $\|u\|_{L^\infty} \leq 1$ and $t \in [0,T^*]$
\begin{equation}\left|x(t;f,u,0)- Z_M(t,f,u)(0)\right| \leq C_1 t^{M+1} + C_3 t |Z_M(t,f,u)(0)|\end{equation}
and in particular, for $t \leq T \leq 1/(2C_3)$
 \begin{equation}\left|Z_M(t,f,u)(0)\right| \leq 2 \left|x(t;f,u,0)\right| + 2 C_1 t^{M+1}.
 \end{equation}
This gives (\ref{approx_repr_1.0}) with $C=\max\{2C_1;2C_3\}$ and $\eta:=\min\{T^*,1/(2C_3)\}$.
 
 \medskip
 
 \noindent \emph{Proof of the second statement.}
 The strategy is the same: one starts from the estimate in \cref{thm:Magnus_2_u}, then applies \cref{Lem_oZ} to $\mathcal{Z}_M(t,f,u)$ and concludes thanks to the following estimate, implied by the explicit expressions of the vector field
 \begin{equation}
 \| \mathcal{Z}_M(t,f,u) \|_{\CC^1}=\underset{\|u\|_{L^1} \to 0}{O}\left( \|u\|_{L^1(0,t)} \right),
 \end{equation}
 {and Young's inequality.}

\medskip

\noindent \emph{Proof of the third statement.}  
First, one can assume that $f_1(0) \neq 0$. 
Indeed, otherwise, both $x$ and $\mathcal{Z}_M$ vanish identically, so the desired estimate is void.
Using \cref{Prop:Magnus1.1/1.2/chpvect} and the explicit expression of the vector field $\CBHD_M\left(\mathcal{Y}_M(t,f,U) , U(t) f_1 \right)$, we obtain in the asymptotics $(t,\|U\|_{L^\infty})\rightarrow 0$
\begin{equation}
    \left\| \mathcal{Z}_{M}(t,f,u) \right\|_{\CC^1} = O \left( |U(t)|+\|U\|_{L^1(0,t)} \right).
\end{equation}
Thus, using $f_0(0)=0$, \cref{Magnus3_Cor1} and the same strategy as above, we obtain in the asymptotics $(t,\|U\|_{L^\infty})\rightarrow 0$
\begin{equation}x(t;f,u,0)=\mathcal{Z}_M(t,f,u)(0)+O\left(|U(t)|^{M+1} + \int_0^t |U|^{M+1} +  \left( |U(t)|+\|U\|_{L^1} \right)|x(t;f,u,0)| \right). \end{equation}
The following proposition and {Young's} inequality give the conclusion.
\end{proof}

\begin{proposition}
    Let $\delta>0$, $f_0,f_1\in {\CC^{\omega}(B_{\delta};\K^d)}$ with $f_0(0)=0$ and $f_1(0)\neq 0$. There exists $T,\eta,C>0$ such that, for every $u \in L^1((0,T),\K)$ with $\|U\|_{L^\infty}<\eta$ and $t \in [0,T]$,
    \begin{equation}
        |U(t)| \leq C \left( |x(t;f,u,0)|+ \|U\|_{{L^1(0,t)}} \right).
    \end{equation}
\end{proposition}

\begin{proof}
    With the notations of \cref{Prop:trick},  $x(t;f,u,0)=e^{U(t) f_1} x_1(t;F,U,0)$ tends to zero when $\|U\|_{L^\infty} \rightarrow 0$. 
A Taylor expansion of order 2 in $x(t;f,u,0)=e^{U(t) f_1} x_1(t;F,U,0)$ provides $C_1>0$ such that, for every $t\in[0,T]$ and $u \in L^1((0,T);\K)$ such that $\|U\|_{L^\infty} \leq \eta^*$,
\begin{equation} \label{x(t)=x1+Uf10}
\left|x(t;f,u,0)-x_1(t;F,U,0)-U(t) f_1(0) \right| \leq C_1|U(t)|^2
+ C_1 |U(t)| |x_1(t;F,U,0)|.
\end{equation}
Moreover, by Gr\"onwall's lemma, there exists $C_2>0$ such that
\begin{equation} \label{x1<UL1}
\left|x_1(t;F,U,0)\right| \leq C_2 \|U\|_{L^1(0,t)}.
\end{equation}
Let $P:\K^d \rightarrow \K^d$ defined by $P(y)=\langle y , f_1(0) \rangle/|f_1(0)|^2$. Applying $P$ to the vector in the left-hand side of (\ref{x(t)=x1+Uf10}) and using (\ref{x1<UL1}), we get the conclusion, when $\|U\|_{L^\infty}$ is small enough.
\end{proof}

{
Under additional nilpotency assumptions, one can omit the truncation errors in the representation formulas of \cref{Prop:Approx_intrinsic_repr}.

\begin{corollary}
    Under the same assumptions as in \cref{Prop:Approx_intrinsic_repr}.
    \begin{enumerate}
        \item Assume moreover that $\mathcal{L}(\{ f_i ; i \in I \})$ is nilpotent of index at most $M+1$. Then, as $T \to 0$,
        \begin{equation}
            x(t;f,u,0) = Z_M(t,f,u)(0) + O\left(t|x(t;f,u,0)|\right).
        \end{equation}
        
        \item Assume moreover that $f_i \in \CC^\omega(B_{4\delta};\K^d)$ for $i \in I := \intset{0,q}$ and that $\{ f_i ; i \in \intset{1,q} \}$ is semi-nilpotent of index at most $M+1$ with respect to $f_0$. 
        Then, as $\|u\|_{L^1} \to 0$,
        \begin{equation}
            x(t;f,u,0) = \mathcal{Z}_M(t,f,u)(0) + O\left(\|u\|_{L^1(0,t)} |x(t;f,u,0)|\right).
        \end{equation}
        
        \item Assume moreover that $f_0,f_1 \in \CC^\omega(B_{4\delta};\K^d)$ and that $\{ f_1 \}$ is semi-nilpotent of index at most $M+1$ with respect to $f_0$. Then, as $(T,\|U\|_{L^\infty}) \to 0$,
        \begin{equation}
            x(t;f,u,0) = \mathcal{Z}_M(t,f,u)(0) + O\left((\|U\|_{L^1(0,t)}+|U(t)|)|x(t;f,u,0)|\right).
        \end{equation}
    \end{enumerate}
\end{corollary}

\begin{proof}
    These are straightforward consequences of \cref{thm:magnus1.nilpotent} (for the first item) and \cref{thm:analytic+brackets-nilpotent=>magnus11-equality} (for the second and third item, thanks to the analyticity assumption), using the same approach as in the proof of \cref{Prop:Approx_intrinsic_repr}.
\end{proof}
}
 
\begin{remark}
    Estimate \eqref{eq:xzmUMespilon} proves that, for a situation in which $\int_0^t |U|^{M+1}$ is negligible, the state is well approximated by $\mathcal{Z}_M(t,f,u)(0)$, which is a convergent series of iterated Lie brackets of $f_0$ and $f_1$ evaluated at $0$. 
    We expect that this representation can be useful for applications to control theory, where one tries to relate controllability of the system with geometric relations on the Lie brackets evaluated at zero.
\end{remark}

\subsection{Diffeomorphisms and Lie brackets}
\label{sec:diffeo.invariant}

Lie brackets behave very nicely with respect to local changes of coordinates.
Let $f_i$ be smooth vector fields for $i \in I$, $p \in \K^d$ and $\theta$ be a smooth local diffeomorphism near $p$.
If $x(t)$ denotes the solution to \eqref{affine_syst_q}, we define $y(t) := \theta(x(t))$.
Then, one checks that $y$ is the solution to
\begin{equation} \label{eq:yg0g1}
    \dot{y}(t) = \sum_{i \in I} u_i(t) g_i(y(t))
    \quad \textrm{and} \quad
    y(0) = {p'},
\end{equation}
where $g_i := \theta_* f_i$ and ${p'} := \theta(p)$.
By iterating \cref{thm:push.lie}, Lie brackets of the vector fields defining the dynamics for~$y$ can be computed explicitly from those of $x$.
More precisely, for every $b \in \Br(X)$,
\begin{equation}
    g_b = \theta_* f_b
\end{equation}
with the notation of \cref{Def:evaluated_Lie_bracket}. 
In particular, there exists a linear invertible map $L_p : \K^d \to \K^d$, $L_p := D\theta(p)$, such that, for every $b \in \Br(X)$,
\begin{equation} \label{eq:gb=Lfb}
    g_b({p'}) = L f_b(p).
\end{equation}

Conversely, if the $f_i$ and $g_i$ for $i \in I$ are analytic vector fields, the existence of points $p$ and ${p'}$ and a linear invertible map $L_p$ such that~\eqref{eq:gb=Lfb} holds is a sufficient condition for the existence of a local smooth diffeomorphism $\theta$ with $\theta(p) = {p'}$ and such that, for {all} controls $u_i$, there holds $y(t) = \theta(x(t))$ where $x$ and $y$ denote the solutions to \eqref{affine_syst_q} and \eqref{eq:yg0g1} for the same set of controls. 
This nice property is proved in \cite[Theorem 1]{zbMATH03385496} and was then extended with a more general geometric viewpoint in \cite{zbMATH03471325} (see also \cite[Theorem 5.5]{zbMATH02105696} for a modern presentation).

When \eqref{eq:gb=Lfb} only holds for brackets up to some length $M \in \N$ and the controls are uniformly bounded in $L^\infty$, one can prove (see \cite{zbMATH03435355}) the existence of a local smooth diffeomorphism $\theta$ and a constant $C$ such that 
\begin{equation} \label{eq:yxttm}
    |y(t) - \theta(x(t))| \leq C t^{M+1}.
\end{equation}
Up to our knowledge, the converse, which is conjectured to be true in \cite{zbMATH03435355}, is a nice open problem.

\begin{open}
    Let $I = \intset{1,q}$ and $X = \{ X_1, \dotsc X_q \}$.
    Let $p, {p'} \in \K^d$. 
    Assume that there exists a smooth diffeomorphism $\theta$ from a neighborhood of $p$ to a neighborhood of ${p'}$ and $M \in \N$ such that, for {all} controls $u_1, \dotsc u_q \in L^\infty(0,T)$ with $\|u_i\| \leq 1$, estimate \eqref{eq:yxttm} holds for {the} trajectories $x$ and $y$ corresponding to the same controls. 
    Does this imply that there exists a linear invertible map such that, for each $b \in \Br(X)$ with $|b|\leq M$, \eqref{eq:gb=Lfb} holds?
\end{open}

\begin{open}
    Same question in the context of affine systems with drift, i.e.\ when $I = \intset{0,q}$, $X = \{ X_0, X_1, \dotsc X_q \}$ and the first control $u_0$ is constrained to be identically equal to $1$. 
    This question might be harder because one gets less information from \eqref{eq:yxttm} as it is valid for less choices of controls since $u_0$ is heavily constrained.
\end{open}

\begin{remark} \label{rk:chen.gb.fb}
    Property \eqref{eq:gb=Lfb} is specific to Lie brackets and does not hold for products of differential operators. 
    {As an illustration, consider the case $\K = \R$, $d = 2$, $X = \{ X_0, X_1 \}$, $p = p' = 0$ with $L_0 := D\theta(0) = \operatorname{Id}_2$.
    Then, for every $\phi \in \CC^\infty(\R^2;\R)$, $(g_b \phi)(0) = (f_b \phi)(0)$, but this relation does not extend to a similar relation between products of $f_0$ and $f_1$ and those of $g_0$ and $g_1$.
    For example, with the vector fields $f_0(x) := (0,x_1)$ and $f_1(x) := (1,0)$ and the smooth diffeomorphism $\theta(x) := (x_1, x_2+x_1^2)$, one has $g_0(y) = (0,y_1)$ and $g_1(y) = (1, 2y_1)$.
    In particular, $(f_1^2 \phi)(0) = \partial_{11} \phi(0)$ but $(g_1^2 \phi)(0) = \partial_{11} \phi(0) + 2 \partial_{12} \phi(0)$.}
    This explains why we consider that the Chen-Fliess expansion is not an intrinsic representation of the state, as it depends on quantities which are not invariant through local changes of coordinates.
\end{remark}

\subsection{Replacing the Magnus flow by a diffeomorphism}

Let $f_i$ for $i \in I$ be smooth vector fields.
We consider the solution~$x(t;u)$ to \eqref{affine_syst_q} with $p=0$.
Let $Z_M(t,u)$ be the vector field defined in \cref{Prop:Magnus_1}
(and called $Z_M(t,\sum_{i\in I} u_i f_i)$ in this statement).
By \cref{Prop:Magnus_1}, for each $M \in \N$, $x(t;u)$ is given by the time-one flow of the autonomous vector field $Z_M(t,u)$, up to an error scaling like $t^{M+1}$ when the controls $u_i$ are uniformly bounded in $L^\infty$. 

In this paragraph, inspired by the nice properties of Lie brackets with respect to diffeomorphisms recalled above, we attempt to replace the computation of the time-one flow by a diffeomorphism.
This can be seen as {being related with the} converse of the classical question of whether a given diffeomorphism can be represented as the time-one flow of an autonomous vector field (see e.g.\ \cite{zbMATH01274714,zbMATH01882113} {for positive answers in particular cases, \cite[Section~2]{zbMATH01274714} for an elementary necessary condition, and \cite{zbMATH04090406} or
\cite{zbMATH03464540} for statements highlighting that the answer is only rarely positive}).

This also corresponds to replacing the terms $x(t;u)+o(|x(t;u)|)$ in \cref{Prop:Approx_intrinsic_repr} by $\theta(x(t;u))$, where $\theta$ is a smooth local diffeomorphism of $\K^d$.

We start with a definition.

\begin{definition} \label{def:homo.time}
    Let $T > 0$ and $n \in \N$. 
    We say that a functional $\beta : [0,T] \times L^\infty((0,T);\K^q) \to \K$ is \emph{homogeneous of degree $n$ with respect to time} when, for every $u \in L^\infty((0,T);\K^q)$, $\lambda \in (0,1]$ and $t \in [0,T]$,
    \begin{equation}
        \beta(\lambda t, u^\lambda) = \lambda^n \beta(t, u)
    \end{equation}
    where $u^\lambda$ is defined by $u^\lambda(\lambda t) := u(t)$ for $t \in [0,T]$ and $u^\lambda(\lambda t) := 0$ for $t > T$.
\end{definition}

In particular, the product of two homogeneous functionals of degree $n$ and $m$ with respect to time is an homogeneous functional of degree $n+m$.
The coordinates of the first kind $\zeta_b(t,u)$, pseudo-first kind $\eta_b(t,u)$ and second kind $\xi_b(t,u)$ are all homogeneous of degree $|b|$ with respect to time. An interesting property of homogeneous functionals is given by the following statement.

\begin{lemma} \label{thm:homo.equal}
    Let $T > 0$, $n \in \N$ and $\beta : [0,T] \times L^\infty((0,T);\K^q) \to \K$, homogeneous of degree $n$ with respect to time. Assume that there exists $C > 0$ such that, for every $u \in L^\infty((0,T);\K^q)$ with $\|u\|_{L^\infty(0,T)} \leq 1$ and each $t \in [0,T]$,
    \begin{equation}
        | \beta(t,u) | \leq C t^{n+1}.
    \end{equation}
    Then $\beta \equiv 0$.
\end{lemma}

\begin{proof}
    Let $t \in [0,T]$ and $u \in L^\infty((0,T);\K^q)$ such that $\|u\|_{L^\infty(0,T)} \leq 1$. On the one hand, for each $\lambda \in (0,1]$, $\beta(\lambda t, u^\lambda) = \lambda^n \beta(t,u)$. On the other hand, $|\beta(\lambda t, u^\lambda)| \leq C \lambda^{n+1} t^{n+1}$ because $\|u^\lambda\|_{L^\infty} = \|u\|_{L^\infty}\leq 1$. Hence $|\beta(t,u)| \leq C \lambda t^{n+1}$ for each $\lambda \in (0,1]$ so $\beta(t,u)=0$. 
\end{proof}

One could wonder if the following proposition holds.

\begin{false}
    Let $X=\{X_i;i\in I\}$, $\mathcal{B}$ be a monomial basis of $\mathcal{L}(X)$. 
    Let $T > 0$. 
    There exists a family $(\beta_b)_{b \in \mathcal{B}}$ of functionals from $[0,T] \times L^\infty((0,T);\K^q)$ to $\K$, with $\beta_b$ homogeneous of degree $|b|$ with respect to time, such that the following statement holds.
    Let $\delta > 0$ and $f_i \in {\CC^\infty(B_{\delta};\K^d)}$ for $i \in I$. 
    There exists a smooth diffeomorphism $\theta$ of $\K^d$ near $p  = 0$ such that, for each $M > 0$, there exists $C_M, T_M > 0$ such that, for every $u \in L^\infty((0,T);\K^q)$ with $\|u\|_{L^\infty} \leq 1$, for each $t \in [0,T_M]$,
    \begin{equation} \label{eq:xthetayctm}
        |\theta(x(t;u)) - y_M(t;u)| \leq C_M t^{M+1},   
    \end{equation}
    and
    \begin{equation} \label{eq:y.betab.fb}
        y_M(t;u) = \theta(0)+\sum_{|b| \leq M} \beta_b(t,u) g_b(\theta(0)),
    \end{equation}
    where $g_b = \theta_* f_b$ and $x(t;u)$ is the solution to \eqref{affine_syst_q} starting from $p = 0$.
\end{false}

The functionals $\beta_b$ would be the analog of the coordinates of the first and second kind described earlier.
A formula such as \eqref{eq:y.betab.fb} would be ideal for applications to control theory for example, since it is expressed on intrinsic quantities (Lie brackets) and allows to compute $x(t;u)$ directly without solving for flows (one recovers $x(t;u) \approx \theta^{-1}(y(t;u))$). 
In some sense, it corresponds to asking if there exists a local change of coordinates for which the Chen-Fliess expansion only involves Lie bracket terms (and all the non-Lie bracket terms vanish).

Unfortunately, it is impossible in general, as illustrated by the following counter-example.

\begin{proposition} \label{Prop:c-ex_diffeo}
    Let $X = \{ X_0, X_1 \}$.
    Let $T > 0$ and consider, in $\R^3$, $f_0(x) := (0, x_1 + x_1^2, x_1x_2)$ and $f_1(x) := (1, 0, 0)$, i.e.\ the following affine system with drift
    \begin{equation} \label{eq:x.cex.formenormale}
        \begin{cases}
             \dot{x}_1 = u, \\
             \dot{x}_2 = x_1 + x_1^2, \\
             \dot{x}_3 = x_1 x_2,
        \end{cases}
    \end{equation}
    together with the initial data $x(0) = 0$.
    There exists a monomial basis $\mathcal{B}$ of $\mathcal{L}(X)$, such that, for {all} functionals $\beta_b : [0,T] \times L^\infty((0,T);\R) \to \R$ for $b \in \mathcal{B}$, {homogeneous} of degree $|b|$ with respect to time and for every local $\CC^6$ diffeomorphism~$\theta$ of $\R^3$, there exists $M \in \llbracket 1, 6 \rrbracket$ and a control $u \in L^\infty((0,T);\R)$ with $\|u\|_{L^\infty} \leq 1$ such that \eqref{eq:xthetayctm} does not hold, even for small times.
\end{proposition}

\begin{proof}
    Let $\mathcal{B}$ be a length-compatible Hall basis of $\mathcal{L}(X)$ with $X_0<X_1$. 
    
    \medskip
    
    \noindent \emph{Step 1: Computation of $y_6(t)$.} 
    We define $\dot{\mathcal{B}}_\ell=\{ b\in\mathcal{B} ; n_1(b)= \ell\}$ for every $\ell\in\N$.
    Then $\dot{\mathcal{B}}_1=\{\ad_{X_0}^k(X_1);k\in\N\}$. The computation shows that the only elements $b \in \dot{\mathcal{B}}_1$ such that $f_b \neq 0$ are
    \begin{align}
        b_1 &= X_1, 
        & b_2 &= [X_0,X_1], 
        & c_1 &= [X_0,[X_0,X_1]], \\
        f_{b_1}(x) &= e_1, 
        & f_{b_2}(x) &= -(1+2x_1)e_2 - x_2 e_3,
        & f_{c_1}(x) &= x_1^2 e_3.
    \end{align}
    Thus, the only elements $b\in \dot{\mathcal{B}}_2$ that could satisfy $f_b \neq 0$ are $[b_1,b_2]$, $[b_1,c_1]$, $[b_2,c_1]$. The computation shows that, among them, only the two first ones do satisfy the condition:
    \begin{align}
        b_3&=[X_1,[X_0,X_1]], 
        & c_2&=[X_1,\ad_{X_0}^2(X_1)],  \\
        f_{b_3}(x) &= -2 e_2,
        & f_{c_2}(x) &= 2x_1e_3.
    \end{align}
    Thus, the only elements $b\in \dot{\mathcal{B}}_3$ with length at most $6$ that could satisfy $f_b \neq 0$ are $[b_1,b_3]$, $[b_1,c_2]$, $[b_2,b_3]$, $[b_2,c_2]$, $[c_1,b_3]$. The computation shows that, among them, only the second and the third ones do satisfy the condition:
    \begin{align}
        b_4&=\ad_{X_1}^2\ad_{X_0}^2(X_1) ,
        & b_5&=[[X_0,X_1],[X_1,[X_0,X_1]]] ,  \\
        f_{b_4}(x)&= 2e_3,
        & f_{b_5}(x)&=-2e_3.
    \end{align}
    Thus the only elements $b\in \dot{\mathcal{B}}_4$ with length at most $6$ that could satisfy $f_b \neq 0$ are $[b_1,b_4]$ and $[b_1,b_5]$, but the computation shows that they satisfy $f_b=0$. Therefore, for every $b \in \dot{\mathcal{B}}_4 \cup \dot{\mathcal{B}}_5 \cup \dot{\mathcal{B}}_6$, $f_b=0$. In conclusion, $b_1, \dotsc, b_5$ are the only elements $b \in \mathcal{B}$ such that $f_b(0) \neq 0$. In particular, none of them have length $4$ or $6$, thus 
    \begin{equation} \label{eq:y6}
        y_6(t)= \theta(0)+ D\theta(0) \Big(\beta_1(t,u) e_1 -\beta_2(t,u) e_2 - 2\beta_3(t,u) e_2+2(\beta_4(t,u)-\beta_5(t,u))e_3\Big)
    \end{equation}
    is the sum of 4 homogeneous functionals of degree $1, 2, 3$ and $5$.
    Here and below we write $\beta_j$ instead of $\beta_{b_j}$ for brevity.
    
    \medskip
    
    \noindent \emph{Step 2: Computation of homogeneous terms with degree $4$ and $6$ in $\theta(x(t))$.}
    In this step, we consider a local $\CC^6$ diffeomorphism $\theta$ of $\R^3$ defined on a neighborhood of $p=0$.
    For $u \in L^\infty((0,T);\R)$, we denote by $U$ the primitive of $u$ such that $U(0) = 0$ and $V$ the primitive of $U$ such that $V(0) = 0$. Straightforward explicit integration of \eqref{eq:x.cex.formenormale} yields
    \begin{equation}
        x(t;u) = U(t) e_1 + V(t) e_2 + \int_0^t U^2(s) \dd s e_2 + \frac{1}{2} V^2(t) e_3 + \int_0^t U(s) \int_0^s U^2(s') \dd s' \dd s e_3,
    \end{equation}
    where the five terms are respectively functionals homogeneous of degree $1$ through $5$ with respect to time in the sense of \cref{def:homo.time}. Using a Taylor expansion of $\theta$ at $0$, one obtains (vector-valued) functionals $\gamma_k$ for $k \in \llbracket 1, 6 \rrbracket$, homogeneous of degree $k$ with respect to time such that for every $M\in \llbracket 1, 6\rrbracket$
        \begin{equation}
            \theta(x(t)) = \theta(0)+ \sum_{k=1}^M \gamma_k(t,u) + \mathcal{O}(t^{M+1}).
        \end{equation}
    In particular
     \begin{equation}
            \begin{split}
                 \gamma_4(t,u) & = {\frac 1 2} V^2(t) \partial_3\theta(0)
             + U(t) \int_0^t U^2 \partial_{12} \theta(0)
             + {\frac 1 2} V^2(t)  \partial_{22} \theta(0) 
            \\ &  + {\frac 1 2} U^2(t) V(t)  \partial_{112} \theta(0)
            + \frac{1}{4!} U^4(t)  \partial_{1}^4 \theta(0)
            \end{split}
        \end{equation}
    and
    \begin{equation}
            \begin{split}
                \gamma_6(t,u) & = 
                U(t) \int_0^t U(s) \int_0^s U(s')^2 ds' ds \partial_{13} \theta(0) 
                + \frac{1}{2} V^3(t) \partial_{23}\theta(0)
                + \frac{1}{2} \left(\int_0^t U^2\right)^2 \partial_{22}\theta(0)
                \\
                & + \frac{1}{4} U^2(t)V^2(t) \partial_{113} \theta(0)
                + \frac{1}{2} U(t)V(t) \int_0^t U^2 \partial_{122}\theta(0)
                + \frac{1}{6} V^3(t) \partial_{222} \theta(0)
                \\
                & + \frac{1}{6} U^3(t) \int_0^t U^2 \partial_{1112}\theta(0)
                + \frac{1}{4} U^2(t) V^2(t) \partial_{1122}\theta(0)
                \\
                & + \frac{1}{4!} V(t)U^4(t) \partial_{1}^4 \partial_2 \theta(0)
                + \frac{1}{6!} U^6(t) \partial_1^6 \theta(0)
                .
            \end{split}
        \end{equation}
    
    \medskip
    
\noindent \emph{Step 3: Denying \eqref{eq:xthetayctm}.}
We proceed by contradiction, assuming that there exists a local $\CC^6$ diffeomorphism $\theta$ of $\R^3$ such that, for each $M \in \llbracket 1, 6 \rrbracket$,  there exists $C_M,T_M>0$ such that \eqref{eq:xthetayctm} holds for every $t \in [0,T_M]$ and $u \in L^\infty((0,T_M);\R)$ with $\|u\|_{L^\infty} \leq 1$. 

By induction on $M$, estimate \eqref{eq:xthetayctm}, \cref{thm:homo.equal} and (\ref{eq:y6}) imply that 
$\gamma_1 = \beta_1 \partial_1 \theta(0)$, 
$\gamma_2 = - \beta_2 \partial_2 \theta(0)$, 
$\gamma_3 = - 2 \beta_3 \partial_2 \theta(0)$, 
$\gamma_4 = 0$, 
$\gamma_5 = 2(\beta_4-\beta_5) \partial_3 \theta(0)$ and
$\gamma_6 = 0$.
        
On the one hand, by choosing $u$ such that $U(t) = 0$ but $V(t) \neq 0$, the relation $\gamma_4(t,u)=0$ implies that $\partial_{22}\theta(0) = - \partial_3 \theta(0) \neq 0$ because $\theta$ is a local diffeomorphism. On the other hand, by choosing $u$ such that $U(t)=V(t)=0$ but $\int_0^t U^2 \neq 0$, the relation $\gamma_6(t,u)=0$ implies that $\partial_{22}\theta(0) = 0$. This concludes the proof, since we have found incompatible conditions on $\partial_{22} \theta(0)$.
\end{proof}

\begin{remark}
    This section is written with a focus on time-based estimates. 
    However, a similar ``false proposition'' could be stated for control-based estimates. The same counter-example also negates this possibility.
\end{remark}


\subsection*{Acknowledgements}

The authors would like to thank Matthias Kawski for the time taken to read carefully this long article, many relevant bibliographical and notational suggestions, and his enthusiasm concerning some of the results contained in this work.

The authors {also} thank Joackim Bernier for discussions about the numerical literature and the importance of the convergence issues discussed in this article for numerical splitting schemes.

The authors benefit from the support of ANR project LEBESGUE, grant ANR-11-LABX-0020.
Karine Beauchard and Fr\'ed\'eric Marbach benefit from the support of ANR project TRECOS, grant ANR-20-CE40-0009-01.

\appendix

\section{Proving that the logarithm of the flow is a Lie series}

\subsection{Using shuffle relations and Ree's theorem}
\label{sec:ree}

{In this paragraph, we describe a proof of \cref{thm:log-lie} relying on Ree's theorem and shuffle relations satisfied by the coefficients of the Chen series.
This approach is notably used in \cite{MR1972790,MR1603219}.
We start with some definitions.

\begin{definition}[Shuffle product]
    The \emph{shuffle product} is the map from $I^* \times I^*$ to the free vector space over $I^*$ defined by induction on the length of the words by $\emptyset \shuffle \sigma = \sigma \shuffle \emptyset := \sigma$ for every $\sigma \in I^*$ (where $\emptyset$ denotes the empty word) and, for every $\sigma,\sigma' \in I^*$ and $\ell, \ell' \in I$, 
    \begin{equation} \label{eq:def-shuffle}
        (\sigma \ell) \shuffle (\sigma' \ell') := (\sigma \shuffle (\sigma' \ell')) \ell + ((\sigma \ell) \shuffle \sigma') \ell'
    \end{equation}
\end{definition}

Intuitively, the shuffle product of two words is the sum of all the ways of riffle shuffling these two words together, interleaving their letters (exactly as one would riffle shuffle two packets of a card deck).
For example, the shuffle product of the words $ab$ and $cd$ (over the Latin alphabet) is $abcd + acbd + acdb + cabd + cadb + cdab$.

The following result was introduced in \cite[Theorem~2.5]{MR100011} to prove \cref{thm:log-lie}.

\begin{lemma}[Ree's theorem]
	Let $\gamma : I^* \to \K$ with $\gamma_{\emptyset} = 1$.
	We still denote by $\gamma$ its linear extension to the free vector space over $I^*$.
	Consider the formal series $x := \sum_{\sigma \in I^*} \gamma_\sigma X_\sigma$.
	Then $\log x \in \widehat{\mathcal{L}}(X)$ iff (the linear extension of) $\gamma$ satisfies the so-called ``shuffle relations'', i.e.\ iff for every $\sigma, \sigma' \in I^*$,
	\begin{equation} \label{eq:shuffle-rel}
		\gamma_{\sigma \shuffle \sigma'} = \gamma_\sigma \gamma_{\sigma'}.
	\end{equation}
\end{lemma}

\begin{proof}
    This statement is item \emph{(iii)} in \cite[Theorem~3.2]{MR1231799}.
\end{proof}

Therefore, to show that $\log x(t)$ is a Lie series, it suffices to check that the coefficients $\int_0^t a_\sigma$ (defined in \eqref{a.sigma}) of the Chen series satisfy these shuffle relations.
We proceed as in \cite[Section~2]{MR100011}, by induction on $|\sigma|+|\sigma'|$ in \eqref{eq:shuffle-rel}.
Definition \eqref{a.sigma} can also be written as, for every $\sigma \in I^*$ and $\ell \in I$,
\begin{equation} \label{eq:a.sigma-rec}
    \int_0^t a_{\sigma \ell} = \int_0^t a_{\sigma}(s) a_{\ell}(s) \dd s.
\end{equation}

Since we set $\int_0^t a_\emptyset = 1$ by definition for the empty word $\emptyset$, \eqref{eq:shuffle-rel} holds for every $\sigma, \sigma' \in I^*$ when $|\sigma|+|\sigma'| = 1$.
Assume now that it holds for $|\sigma|+|\sigma'|\leq n$ for some $n \in \N^*$.
Let $\sigma, \sigma' \in I^*$ and $\ell, \ell' \in I$ such that $|\sigma \ell|+|\sigma' \ell'| = n+1$.
Applying successively, \eqref{eq:def-shuffle} and the linearity of the extension of $\gamma$, \eqref{eq:a.sigma-rec}, the induction hypothesis and eventually \eqref{eq:a.sigma-rec} again, we obtain, for every $t \geq 0$,
\begin{equation}
    \begin{split}
        \gamma_{(\sigma \ell) \shuffle (\sigma' \ell')}(t)
        & = \gamma_{(\sigma \shuffle (\sigma' \ell')) \ell}(t)
        + \gamma_{((\sigma \ell) \shuffle \sigma') \ell'}(t) \\
        & = \int_0^t \gamma_{\sigma \shuffle (\sigma' \ell')}(s) a_{\ell}(s) \dd s 
        + \int_0^t \gamma_{(\sigma\ell) \shuffle \sigma'}(s) a_{\ell'}(s) \dd s \\
        & = \int_0^t \gamma_{\sigma}(s) \gamma_{\sigma' \ell'}(s) a_{\ell}(s) \dd s 
        + \int_0^t \gamma_{\sigma\ell}(s) \gamma_{\sigma'}(s) a_{\ell'}(s) \dd s \\
        & = \int_0^t \dot{\gamma}_{\sigma \ell}(s) \gamma_{\sigma' \ell'}(s) + \gamma_{\sigma \ell}(s) \dot{\gamma}_{\sigma' \ell'}(s) \dd s,
    \end{split}
\end{equation}
which proves that $\gamma_{(\sigma \ell) \shuffle (\sigma' \ell')} = \gamma_{\sigma \ell} \gamma_{\sigma' \ell'}$.
}

\subsection{Using Friedrich's criterion}
\label{sec:friedrichs}

{
In this paragraph, we describe a proof of \cref{thm:log-lie} relying on Friedrich's criterion.
This approach is notably used in \cite[Section~3]{MR886816}.
We start with some definitions.
}

Let $\mathcal{A}(X) \otimes \mathcal{A}(X)$ be the tensor product of {the} algebra $\mathcal{A}(X)$ with itself (i.e.\  the tensor product of $\mathcal{A}(X)$ and $\mathcal{A}(X)$, endowed with the product rule $(a\otimes b)(a'\otimes b') := (aa') \otimes (bb')$, see \cite[Chapter~3, Section~4.1, Definition~1]{MR1727844} for a precise construction).
{The algebra $\mathcal{A}(X)$ is the universal enveloping algebra of the Lie algebra $\mathcal{L}(X)$, and as such is a Hopf algebra (see \cite{zbMATH03747352}). 
The coproduct homomorphism $\Delta$ : $\mathcal{A}(X) \to \mathcal{A}(X) \otimes \mathcal{A}(X)$ is defined by setting the values $\Delta(1) := 1 \otimes 1$ and $\Delta(X_i) := X_i \otimes 1 + 1 \otimes X_i$ for $1 \leq i \leq q$. 
This defines a unique homomorphism because $\mathcal{A}(X)$ is freely generated by $X$ as an algebra (see \cite[Proposition 1.2]{MR1231799} for more detail).} 
The {coproduct} $\Delta$ can then be used to characterize Lie elements, as in the following result, which was proposed by Friedrichs in~\cite{MR0056467}, then proved by multiple authors in the same period \cite{MR0066359,MR0082643,MR0083103,MR0067873}.

\begin{lemma}[Friedrichs' criterion]
 \label{thm:friedrichs}
 For $a \in \mathcal{A}(X)$, $a \in \mathcal{L}(X)$ if and only if the condition $\Delta(a) = a \otimes 1 + 1 \otimes a$ holds.
\end{lemma}

\begin{proof}
 This statement is the equivalence between \emph{(i)} and \emph{(iii)} in \cite[Theorem~1.4]{MR1231799}.
\end{proof}

\begin{example}
 The element $X_1 X_2$ does not belong to $\mathcal{L}$. And indeed,
 \begin{equation}
  \begin{split}
    \Delta(X_1 X_2) = \Delta(X_1) \Delta(X_2)
    & = (X_1 \otimes 1 + 1 \otimes X_1) (X_2 \otimes 1 + 1 \otimes X_2) \\
    & = X_1 X_2 \otimes 1 + X_1 \otimes X_2 + X_2 \otimes X_1 + 1 \otimes X_1 X_2 \\
    & \neq X_1 X_2 \otimes 1 + 1 \otimes X_1 X_2.
  \end{split}
 \end{equation}
 On the contrary, the element $[X_1, X_2] = X_1 X_2 - X_2 X_1$ belongs to $\mathcal{L}$. And indeed,
 \begin{equation}
  \begin{split}
    \Delta([X_1, X_2]) & = \Delta(X_1 X_2) - \Delta(X_2 X_1) \\
    & = (X_1 X_2 \otimes 1 + X_1 \otimes X_2 + X_2 \otimes X_1 + 1 \otimes X_1 X_2) \\
    & \quad \quad - (X_2 X_1 \otimes 1 + X_2 \otimes X_1 + X_1 \otimes X_2 + 1 \otimes X_2 X_1) \\
    & = [X_1, X_2] \otimes 1 + 1 \otimes [X_1, X_2].  
  \end{split}
 \end{equation}
\end{example}

The tensor product $\mathcal{A}(X) \otimes \mathcal{A}(X)$ also has a graded structure, with $(\mathcal{A}(X) \otimes \mathcal{A}(X))_n = \bigoplus_{i=0}^n \mathcal{A}_i(X) \otimes \mathcal{A}_{n-i}(X)$. Since the homomorphism $\Delta$ is linear and degree preserving, it can be extended as an homomorphism from $\widehat{\mathcal{A}}(X)$ to $\widehat{\mathcal{A}(X) \otimes \mathcal{A}(X)}$, the formal power series over $\mathcal{A}(X) \otimes \mathcal{A}(X)$. For such series with zero constant term, one can define, as in \eqref{def:exp}, an exponential, say $\exp_\otimes$, which also verifies a uniqueness property such as \cref{lemma:log_is_unique}. One can then derive a criterion to determine whether the logarithm of a formal power series is a Lie element.

\begin{corollary} \label{thm:friedrichs2}
 Let $a \in \widehat{\mathcal{A}}(X)$ with $a_0 = 1$. Then $\log (a) \in \widehat{\mathcal{L}}(X)$ if and only if $\Delta(a) = a \otimes a$. 
\end{corollary}

\begin{proof}
 We follow \cite[Theorem 3.2]{MR1231799}. By linearity and degree preservation, \cref{thm:friedrichs} implies that, for $a \in \widehat{\mathcal{A}}(X)$, $a \in \widehat{\mathcal{L}}(X)$ if and only if $\Delta(a) = a \otimes 1 + 1 \otimes a$. For $a \in \widehat{\mathcal{A}}(X)$ with constant term $1$,
 \begin{equation}
  \begin{split}
   \log a \in \widehat{\mathcal{A}}(X) 
   & \Longleftrightarrow \Delta(\log(a)) = \log(a) \otimes 1 + 1 \otimes \log(a) \\
   & \Longleftrightarrow \exp_\otimes \left(\Delta(\log(a))\right) = \exp_\otimes \left(\log(a) \otimes 1 + 1 \otimes \log(a)\right) \\
   & \Longleftrightarrow \Delta\left(\exp(\log(a))\right) = \exp_\otimes (\log(a) \otimes 1) \exp_\otimes (1 \otimes \log(a)) \\
   & \Longleftrightarrow \Delta (a) = ((\exp \log a) \otimes 1) (1 \otimes (\exp \log a)) = a \otimes a,
  \end{split}
 \end{equation}
 where we used the equality $\Delta (\exp (\cdot)) = \exp_\otimes (\Delta(\cdot))$, because $\Delta$ is an homomorphism, and the fact that $\exp_\otimes (b \otimes 1 + 1 \otimes c) = \exp_\otimes (b \otimes 1) \exp_\otimes(1 \otimes c)$, because $b \otimes 1$ and $1 \otimes c$ commute.
\end{proof}

{
Therefore, to show that $\log(x(t))$ is a Lie series, it suffices to check that $\Delta(x(t)) = x(t) \otimes x(t)$. 
This can be checked as in \cite[Section~3]{MR886816} using the following argument.}
At the initial time $\Delta(x(0)) = \Delta(1) = 1 \otimes 1 = x(0) \otimes x(0)$. 
On the one hand
\begin{equation}
 \frac{\dd}{\dd t} \Delta(x) = \Delta( \dot{x} ) = \Delta(xa)
 = \Delta(x) \Delta(a)
 = \Delta(x) (a\otimes 1 + 1 \otimes a).
\end{equation}
On the other hand,
\begin{equation}
 \frac{\dd}{\dd t} (x \otimes x)
 = \dot{x} \otimes x + x \otimes \dot{x}
 = (x a) \otimes x  + x \otimes (x a)
 = (x \otimes x) (a \otimes 1 + 1 \otimes a).
\end{equation}
Hence, both quantities satisfy the same formal differential equation with the same initial condition, so they are equal for every $t \geq 0$.

\section{Elementary numerical identities}

\subsection{Bernoulli numbers}

We use the notation $(B_n)_{n\in\N}$ to denote the Bernoulli numbers, which are defined (using the modern NIST sign and indexing convention) by the identity 
\begin{equation} \label{eq:def:bernoulli}
 \forall z \in \C, |z|<2\pi, \qquad 
 \frac{z}{e^{z}-1} = \sum_{n=0}^{+\infty} B_n \frac{z^n}{n!} = 1 - \frac{z}{2} + \sum_{n=1}^{+\infty} B_{2n} \frac{z^{2n}}{(2n)!}.
\end{equation}

\begin{lemma}
 The Bernoulli numbers satisfy, for every $n \geq 2$
 \begin{align}
  \label{eq:bernoulli.1}
  \sum_{k=0}^{n-1} \binom{n}{k} B_k
  & = 0, \\
  \label{eq:bernoulli.2}
  \sum_{k=0}^{n} \binom{n}{k} \frac{B_k}{n+1-k}
  & = 0, \\
  \label{eq:bernoulli.4}
  \sum_{k=0}^n \frac{B_{n-k}}{(n-k)!(k+2)!}
  & = - \frac{B_{n+1}}{(n+1)!}.
 \end{align}
 Moreover, the odd Bernoulli numbers except $B_1$ vanish and, for every $n \geq 1$,
 \begin{equation} \label{eq:bernoulli.3}
   B_{2n}
   = (-1)^{n+1} \frac{2 (2n)!}{(2\pi)^{2n}} \zeta(2n)
   \sim
   (-1)^{n+1} 2 \sqrt{2\pi n} \left(\frac{n}{\pi}\right)^{2n},
 \end{equation}
 where $\zeta$ is the Riemann zeta function.
\end{lemma}

\begin{proof}
    {The first two identities} are classical and can be proved using the generating series of the Bernoulli numbers of \eqref{eq:def:bernoulli}, respectively by identification in $z = (e^z-1) \times (z / (e^z-1))$ for \eqref{eq:bernoulli.1} and in $1 = ((e^z-1)/z) \times (z/(e^z-1))$ for \eqref{eq:bernoulli.2}.
    
    {The third identity} \eqref{eq:bernoulli.4} follows from \eqref{eq:bernoulli.2} and the computation
    \begin{equation}
     \begin{split}
     \sum_{k=0}^n \frac{B_{n-k}}{(n-k)!(k+2)!}
     & = \frac{1}{(n+1)!} \sum_{\ell=0}^n
     \binom{n}{n-\ell}
     \frac{n+1}{n+1-\ell} \frac{B_\ell}{n-\ell+2} \\
     & = \frac{1}{(n+1)!} \sum_{\ell=0}^n
     \binom{n+1}{\ell}
     \frac{B_\ell}{(n+1)-\ell+1} 
     = - \frac{B_{n+1}}{(n+1)!}.
     \end{split}
    \end{equation}
 
    {Eventually, the} relationship with the Riemann zeta function is proved in \cite[equation (12.38)]{zbMATH05948059}. The asymptotic is a consequence of the Stirling's approximation and $\zeta(s) \to 1$ as $s \geq 1$ tends to $+ \infty$ (which is a direct consequence of the formula $\zeta(s) = \sum n^{-s}$). 
\end{proof}

\bibliographystyle{plain}
\bibliography{biblio}

\end{document}